\documentclass[12pt]{amsart}

\usepackage[colorlinks=true,urlcolor=blue,
bookmarks=true,bookmarksopen=true,citecolor=blue
]{hyperref}

\usepackage{pdfsync}
\usepackage[all, cmtip]{xy}
\usepackage{graphicx}

\usepackage{epsfig}
\usepackage{amscd}
\usepackage[mathscr]{eucal}
\usepackage{amssymb}
\usepackage{amsxtra}
\usepackage{amsmath}
\usepackage{latexsym} 
\usepackage[all]{xy}
\usepackage{enumerate}
\usepackage{mathrsfs}

\usepackage{soul}

\usepackage{color}
\theoremstyle{plain}

\newtheorem{mainthm}{Theorem}
\newtheorem{thm}{Theorem}[subsection]

\newtheorem{cor}[thm]{Corollary}

\newtheorem{lem}[thm]{Lemma}
\newtheorem{prop}[thm]{Proposition}

\theoremstyle{definition}
\newtheorem{dfn}[thm]{Definition}

\theoremstyle{remark}
\newtheorem{rem}[thm]{Remark}

\newtheorem{rems}[thm]{Remarks}

\theoremstyle{plain}

\newcommand{\Id}{{{\mathchoice {\rm 1\mskip-4mu l} {\rm 1\mskip-4mu l}
      {\rm 1\mskip-4.5mu l} {\rm 1\mskip-5mu l}}}}

\newcommand{\cobto}{\leadsto}
\newcommand{\id}{\textnormal{id}}

\newcommand{\cobc}{\mathcal{C}ob}
\newcommand{\fk}{\mathcal{F}uk}
\newcommand{\lag}{\mathcal{L}}
\newcommand{\gcob}{G_{cob}(M)}

\newcommand{\R}{\mathbb{R}}
\newcommand{\Z}{\mathbb{Z}}

\newcommand{\C}{\mathbb{C}}

\newcommand{\La}{\Lambda}

\newcommand{\cob}{\mathcal{C}ob}
\newcommand{\fuk}{\mathcal{F}uk}
\newcommand{\fukd}{\mathcal{F}uk^d}
\newcommand{\fukcb}{\mathcal{F}uk^d_{cob}}
\newcommand{\fukcben}{\mathcal{F}uk^d_{{cob, 1/2}}}

\newcommand{\mor}{{\textnormal{Mor\/}}}

\newcommand{\form}{\Theta}

\newcommand{\pbred}{}
\newcommand{\pbredb}{}

\newcommand{\ocgreen}{}

\newcommand{\corr}{}

\newcommand{\pbcorr}{}

\newcommand{\pbaddress}{biran@math.ethz.ch}
\newcommand{\ocaddress}{cornea@dms.umontreal.ca}

\begin{document}

\title[Lagrangian cobordism and Fukaya categories]{Lagrangian
  cobordism and Fukaya categories} \date{\today}

\thanks{The second author was supported by an NSERC Discovery grant
  and a FQRNT Group Research grant}

\author{Paul Biran and Octav Cornea}

\address{Paul Biran, Department of Mathematics, ETH-Z\"{u}rich,
  R\"{a}mistrasse 101, 8092 Z\"{u}rich, Switzerland}
\email{\pbaddress}

\address{Octav Cornea, Department of Mathematics and Statistics
  University of Montreal C.P. 6128 Succ.  Centre-Ville Montreal, QC
  H3C 3J7, Canada} \email{\ocaddress}

\bibliographystyle{alphanum}

%


\maketitle

%
%

\tableofcontents 


\section{Introduction} \label{s:intro} The purpose of this paper is to
show that geometric Lagrangian cobordisms translate algebraically, in
a functorial way, into iterated triangular decompositions in the
derived Fukaya category.

Fix a symplectic manifold $M$ as well as a class $\lag$ of Lagrangian
submanifolds of $M$ (the precise class will be made explicit
in~\S\ref{sb:monotonicity}).  Floer homology, introduced in Floer's
seminal work~\cite{Fl:Morse-theory} and extended in subsequent
works~\cite{Oh:HF1, FO3:book-vol1, FO3:book-vol2}, associates to a
pair of Lagrangians $L, L'$ in our class a $\Z_{2}$-vector space
$HF(L,L')$.  Assuming for the moment that $L$ is transverse to $L'$,
this is the homology of a chain complex $CF(L,L')$, called the Floer
complex, whose generators are the intersection points of $L$ and $L'$.
The differential counts ``strips'' $u:\R\times [0,1]\to M$ that join
these intersection points, have boundaries along $L$ and $L'$ and
satisfy a Cauchy-Riemann type equation $\overline{\partial}_{J}u = 0$
with respect to an almost complex structure $J$ on $M$ which is
compatible with the symplectic structure. The fact that for generic
$J$ such $J$-holomorphic curves form well-behaved moduli spaces
originates in the fundamental work of Gromov~\cite{Gr:phol}
(see~\cite{McD-Sa:jhol} for the foundations of the theory). The Floer
complex thus depends on additional choices, in particular on $J$,
however its homology $HF(L,L')$ is invariant.  Floer homology together
with other additional structures, also based on counting
$J$-holomorphic curves, are central tools in symplectic topology with
wide-reaching applications.
 
In what Lagrangian topology is concerned, the most efficient way to
aggregate these structures is provided by the derived Fukaya category
whose definition we now sketch - a rigorous, detailed treatment that
serves as foundation for our paper is contained in Seidel's book
\cite{Se:book-fukaya-categ} - see also \S\ref{sec:fuk-M}. Given a
third Lagrangian $L''$ in our class, there is a product, due to
Donaldson: $CF(L,L')\otimes CF(L',L'')\to CF(L,L'')$. This is defined
by counting $J$-holomorphic triangles whose edges are mapped to $L$,
$L'$ and $L''$. This operation descends to homology where it is
associative. It is therefore possible to define a category, called the
Donaldson category of $M$, whose objects are the Lagrangians in
$\lag$ and with morphisms $\mor_{Don}(L,L')=HF(L,L')$. The
composition of morphisms is given by the triangle product. It was
discovered by Fukaya~\cite{Fu:morse-homotopy-ai, Fu:Morse-homotopy}
that, by taking into account the chain level data involving moduli
spaces of $J$-holomorphic polygons with arbitrary number of edges, one
can define a much richer algebraic structure, nowadays called the
Fukaya $A_{\infty}$-category, $\fuk(M)$.  The objects are the same as
those of the Donaldson category, however, this is no longer a category
in the classical sense (bur rather an $A_{\infty}$-category) because
the triangle product is not associative at the chain level.  Moreover,
while the data contained in the Fukaya $A_{\infty}$-category is
extremely rich, working directly with this $A_{\infty}$-category
itself is quite difficult.  Kontsevich~\cite{Ko:ICM-HMS} discovered
that there is a triangulated completion of the Donaldson category to
a true category called the derived Fukaya category and denoted by
$D\fk(M)$. Moreover, the derived Fukaya category is independent of the
auxiliary structures used to define it up to appropriate equivalences
and some of the finer information present at the level of $\fuk(M)$ survives
the passage to $D\fk(M)$.

Starting from the Fukaya $A_{\infty}$-category, the construction of
$D\fk(M)$ is algebraic, based on the fact that at the
$A_{\infty}$-level it is possible to define cone-attachments (or in a
different terminology, exact triangles) by a formula similar to the
definition of the cone over a chain map in classical homological
algebra.  As a consequence, the triangulated structure of the derived
category is somewhat mysterious and non-geometric in its definition.
At the same time, it is precisely this triangulated structure that is
often useful in the study of the Lagrangian submanifolds of $M$.

Lagrangian cobordism was introduced by Arnold~\cite{Ar:cob-1,
  Ar:cob-2}, see also~\S\ref{sub:cob-def} for the specific variant
used here. Consider one such cobordism $$(V; L_{1}\cup\ldots \cup
L_{k}, L)~.~$$ This is a Lagrangian submanifold of $V \subset
\mathbb{R}^2\times M$ with $k+1$-cylindrical ends so that there are
$k$ negative ends, each identified with $(-\infty,0]\times\{i\}\times
L_{i}$, $i=1, \ldots, k$, and one positive end identified with
$[1,\infty)\times \{1\}\times L$. The projection of such a cobordism
to $\mathbb{R}^2$ is like in Figure \ref{fig:MorphCob2b}.

\begin{figure}[htbp]
   \begin{center} 
      \epsfig{file=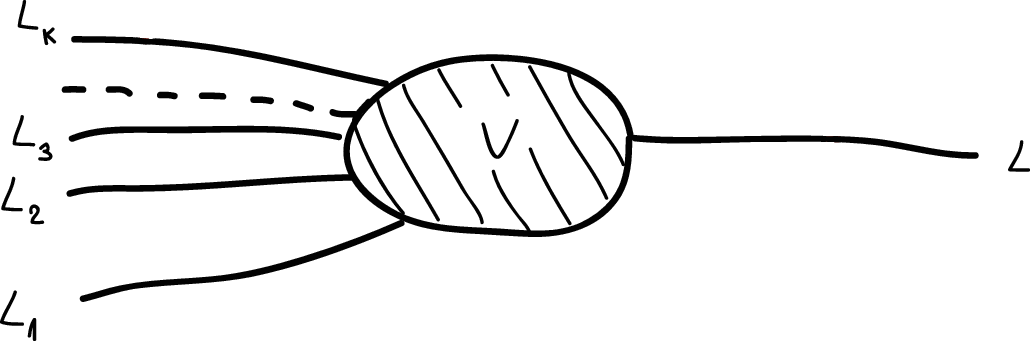, width=0.7\linewidth}
   \end{center}
   \caption{A cobordism $V\subset \mathbb{R}^2\times M$ 
     with a single positive end $L$ and multiple
     negative ends $ L_{1},\ldots, L_{k}$  projected
     on $\mathbb{R}^2$.
     \label{fig:MorphCob2b}}
\end{figure}

Of course, we will have to further restrict the class of Lagrangian
cobordisms, the relevant constraints coming from the class
$\lag$ of Lagrangian submanifolds of $M$ that we have already fixed.
We denote the class of admissible cobordisms by $\lag_{cob}$. Its
precise definition will be given in~\S\ref{subsec:category-cob}.

In this paper we establish the following fundamental correspondence between
cobordism and the triangulated structure of the derived Fukaya
category:
\begin{mainthm} \label{cor:multiple-triang} If $V$ is a Lagrangian
   cobordism as above, then there exist $k$ objects $Z_1, \ldots, Z_k$
   in $D\fk(M)$ with $Z_1 = L_1$ and $Z_k \simeq L$ which fit into
   $k-1$ exact triangles as follows:
   $$L_{i}[1] \to Z_{i-1}\to Z_{i} \to L_i, \ \forall \,  2\leq i \leq k.$$
   In particular, $L$ belongs to the triangulated subcategory of
   $D\fk(M)$ generated by $L_{1}, L_{2}, \ldots, L_{k}$.
\end{mainthm}

The notation $L_i[1]$ stands for a shift by one in the grading for the
object $L_i$. In fact we will work in an ungraded
setting (thus $L_i[1]$ is the same as $L_i$). We left the grading
shift in the notation, only in order to indicate the expected
statement in the graded framework.

The first indication that a result like
Theorem~\ref{cor:multiple-triang} holds appeared in~\cite{Bi-Co:cob1}
where we showed that for any fixed $N\in \lag$, the Floer complexes
$$CF(N,L_{1}),\ldots, CF(N,L_{i}), \ldots, CF(N,L_{k}), CF(N,L)$$ fit
into a sequence of chain cone-attachments as implied by
Theorem~\ref{cor:multiple-triang}.

We deduce Theorem~\ref{cor:multiple-triang} as an immediate
consequence of a stronger result, conjectured in~\cite{Bi-Co:cob1},
that provides a more complete and conceptual description of the
relationship between cobordisms and triangulations. In particular, we
will see that not only cobordisms provide triangular decompositions
but, moreover, concatenation of cobordisms corresponds to refinement
of the respective decompositions.  More elaboration is needed to
formulate this stronger result more precisely.

We use a category $\cobc(M)$ from
\cite{Bi-Co:cob1}\footnote{{$\cobc(M)$ was denoted by
    $\mathcal{C}ob^d_0(M)$ in~\cite{Bi-Co:cob1} as it will be in the rest
    of the paper, starting with~\S\ref{sec:prerequisites}. The role of
    the decorations $d$ and $0$ will be explained
    in~\S\ref{sec:prerequisites}.}}. We remark that an alternative,
somewhat different categorical point of view on Lagrangian cobordism
has been independently introduced by Nadler and Tanaka
in~\cite{Na-Ta}. 

The objects of $\cobc(M)$ are finite ordered families of Lagrangian
submanifolds of $M$ that belong to the class $\lag$. The morphisms are
isotopy classes of certain Lagrangian cobordisms, possibly
multi-ended. (Again, the precise definitions are given
in~\S\ref{subsec:category-cob}.) In particular, the cobordism $V$
considered earlier represents such a morphism.

The geometric category $\cobc(M)$ is monoidal under (essentially)
disjoint union but is not triangulated.  To relate the morphisms in
$\cobc(M)$ to the triangular decompositions in $D\fk(M)$ we consider a
category $T^{S}D\fk(M)$ that is obtained from $D\fk(M)$ by a general
construction, introduced in \cite{Bi-Co:cob1} and further detailed
in~\S\ref{subsec:cones}, that associates to any triangulated category
$\mathcal{C}$ a new category $T^{S}\mathcal{C}$ that is monoidal and
whose morphisms sets, $\hom(x, -)$, parametrize the ways in which $x$
can be resolved by iterated exact triangles.  The main purpose
of $T^{S} D\fk(M)$ is to encode the triangular decompositions in
$D\fk(M)$ as morphisms in a category that can serve as target to a
functor defined on $\cobc(M)$.

Here is the main result of the paper.
\begin{mainthm} \label{thm:main} There exists a monoidal functor
   $$\widetilde{\mathcal{F}}: \cobc(M) \longrightarrow
   T^{S}D\fk(M),$$ with the property that $\widetilde{\mathcal{F}}(L)
   = L$ for every Lagrangian submanifold $L \in \lag$.
\end{mainthm}
Given that $V$ represents a morphism in $\cobc(M)$,
Theorem \ref{cor:multiple-triang} follows immediately from Theorem
\ref{thm:main} and the definition of $T^{S}(-)$, the sequence of
exact triangles in the statement being provided by
$\widetilde{\mathcal{F}}(V)$.

\subsubsection*{Organization of the paper}
The plan for the rest of the paper is as follows:
~\S\ref{sec:prerequisites} contains extensive prerequisites, most
importantly the basic cobordism definitions in~\S\ref{sub:cob-def}, a
short review in \S\ref{sec:fuk-M} of the construction of the Fukaya
$A_{\infty}$-category basically following~\cite{Se:book-fukaya-categ}
and, in~\S\ref{subsec:cones}, the definition of the category
$T^{S}(-)$.  Additionally, for completeness, in
Appendix~\ref{subsec:more-alg} we recall basic $A_{\infty}$-category
notions.  In~\S\ref{sec:FukCob} we set up a Fukaya
$A_{\infty}$-category whose objects are cobordisms in
$\mathbb{R}^2\times M$. As it will be explained below, this is an
essential step in the proof of Theorem~\ref{thm:main}. It might also
be of independent interest.  The proof of Theorem~\ref{thm:main}
appears in~\S\ref{s:proof-main}.
 
In the remainder of the introduction we pursue with some more
technical remarks and corollaries of Theorem~\ref{thm:main}. We then
summarize the main steps in the proof Theorem~\ref{thm:main}.
  
We refer to~\cite{Bi-Co:cob1} for more extensive background and
examples of Lagrangian cobordisms.
 
\subsection{Further context and Corollaries of Theorem~\ref{thm:main}}
\label{subsec:cor}

To further outline the properties of the functor
$\widetilde{\mathcal{F}}$ from Theorem~\ref{thm:main} it is useful to
consider the commutative diagram below.

\begin{eqnarray}\label{eq:commut-diag1}
   \begin{aligned}
      \xymatrix{ \cobc(M)\ar[r]^{\widetilde{\mathcal{F}}}
        \ar[d]_{\mathcal{P}}
        & T^{S}D\fk(M)\ar[d]^{\mathcal{P}}  \\
        S\cobc(M) \ar[r]^{\mathcal{F}}
        \ar[rd]_{\mathcal{H}F_{N}} & D\fk(M)\ar[d]^{\hom(N,-)}\\
        & (\mathcal{V},\times)}
   \end{aligned}
\end{eqnarray}
 
We explain next the ingredients in this diagram.

\subsubsection{A simple version of $\widetilde{\mathcal{F}}$ and the
  top square in~\eqref{eq:commut-diag1}}
\label{subsubsec:simplification}

We describe here the functor $\mathcal{F}$, that appears in the middle
row in Diagram~\eqref{eq:commut-diag1}.  For this, we introduce
another cobordism category, denoted $S\cobc(M)$, which is simpler than
$\cobc(M)$. Its objects are Lagrangians in the class $\lag$ and the
morphisms relating two such objects, $L$ and $L'$, are horizontal
isotopy classes of cobordisms $V$ in $\mathcal{L}_{cob}$ (see
Definitions~\ref{def:Lcobordism} and~\ref{d:isotopies}) so that $L$ is
the unique ``positive'' end and $L'$ is the ``top'' negative end of
$V$, as in Figure~\ref{fig:MorphCob2b} (with $L'=L_{k}$).  There is a canonical functor
$\mathcal{P}:\cobc(M) \to S \cobc(M)$ that associates to a family
$(L_{1}, \ldots, L_{k})$ the last Lagrangian in the family, $L_{k}$,
and has a similar action on morphisms. Directly out of the definition
of $T^{S}D\fk(M)$ - see \S\ref{subsec:cones}, there also is a
projection functor $\mathcal{P}:T^{S}D\fk(M)\to D\fk(M)$ that again
associates to each family of Lagrangians the last object in the
family. The construction of $\widetilde{\mathcal{F}}$ implies that
there is an induced functor $\mathcal{F}:S\cobc(M)\to D\fk(M)$
which is the identity on objects and makes the top square in the
Diagram (\ref{eq:commut-diag1}) commutative.

As it will be seen in more detail in~\S\ref{sb:diag-cor}, the functor $\mathcal{F}$
has the advantage that it can be explicitly described on morphisms as
follows.  Fix a cobordism $V$ representing a morphism between $L$ and
$L'$ in $S\cobc(M)$.  The class $\mathcal{F}([V])\in HF(L,L')=\hom_{D\fuk}(L,L')$ is the
image of the unity in $HF(L,L)$ (induced by the fundamental class of
$L$) through a morphism $\phi_{V}:HF(L,L)\to HF(L',L)$ that is given
by counting Floer strips in $\mathbb{R}^2\times M$ with boundary
conditions along $V$ on one side and on $\gamma\times L$ on the other
side, $\mathcal{F}([V])=\phi_{V}([L])$.  Here $\gamma\subset
\mathbb{R}^2$, $V$ are as in Figure~\ref{fig:MorphCob2}
(again with $L'=L_{k}$) where are also depicted the planar projections of
the strips whose count provides the morphism $\phi_{V}$.
\begin{figure}[htbp]
   \begin{center} 
      \epsfig{file=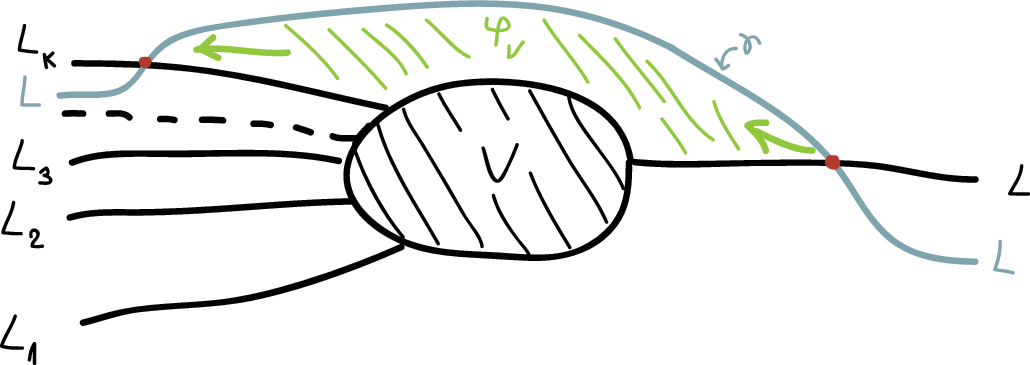, width=0.7\linewidth}
   \end{center}
   \caption{A cobordism $V\subset \mathbb{R}^2\times M$ 
     with a positive end $L$ and with $L'=L_{k}$ together with the projection of the 
     $J$-holomorphic strips that define the morphism $\phi_{V}$.
     \label{fig:MorphCob2}}
\end{figure}
\begin{rem} \label{rem:ham-Seidel-rep} a. It has been verified by
   Charette-Cornea (\S 3.4 in \cite{Cha:thesis}) that $\phi_{V}$ is an
   extension of the Lagrangian version of the Seidel morphism
   \cite{Se:pi1} introduced by Hu-Lalonde \cite{Hu-La:Seidel-morph}
   (see also \cite{Leclercq:spectral} and \cite{Hu-La-Le:monodromy})
   to which $\phi_{V}$ reduces when $V$ is the Lagrangian suspension
   associated to a Hamiltonian isotopy.  Thus, from this perspective,
   Theorem \ref{thm:main} shows that the Seidel morphism extends to a
   natural correspondence $V \to \phi_{V}$ that satisfies the
   properties needed to define the functor $\mathcal{F}$.

   b. \pbred{In fact a stronger version of the remark at point a is
     true.  It is proved in~\cite{Ch-Co:cob-Seidel} that the Seidel
     representation admits a categorification in the following sense:
     the fundamental groupoid of $Ham (M)$, $\Pi (Ham(M))$, viewed as
     a monoidal category, acts on both $\cobc(M)$ and $T^{S}D\fk(M)$
     and $\widetilde{\mathcal{F}}$ is equivariant with respect to this
     action.}
\end{rem}

We use the functor $\mathcal{F}$ to illustrate Theorem
\ref{cor:multiple-triang} in a particular case where we can also make
the statement more precise by identifying the morphisms involved.
\begin{cor}\label{cor:special-short} 
   If the Lagrangian cobordism $(V; L_{1}\cup L_{2}, L)$ has just two
   negative ends $L_{1}, L_{2}\subset M$ (for instance, this happens
   if $L$ is obtained by surgery on $L_{1}$ and $L_{2}$
   \cite{Bi-Co:cob1}), then there is an exact triangle in $D\fuk(M)$
   \begin{eqnarray} \label{eq:simple-exact}
      \begin{aligned}
         \xymatrix{ L_{2} \ar[dd]_{\mathcal{F}(V'')} & & \\
           &  &L \ar[ull]_{\mathcal{F}(V)}\\
           L_{1} \ar[urr]_{\mathcal{F}(V')} & &}
      \end{aligned}
   \end{eqnarray}
   where $V'$ and $V''$ are the cobordisms obtained by bending the
   ends of $V$ as in the Figure \ref{fig:hairy-cob} below.

   \begin{figure}[htbp]
      \begin{center} \vspace{-0.4in} \epsfig{file=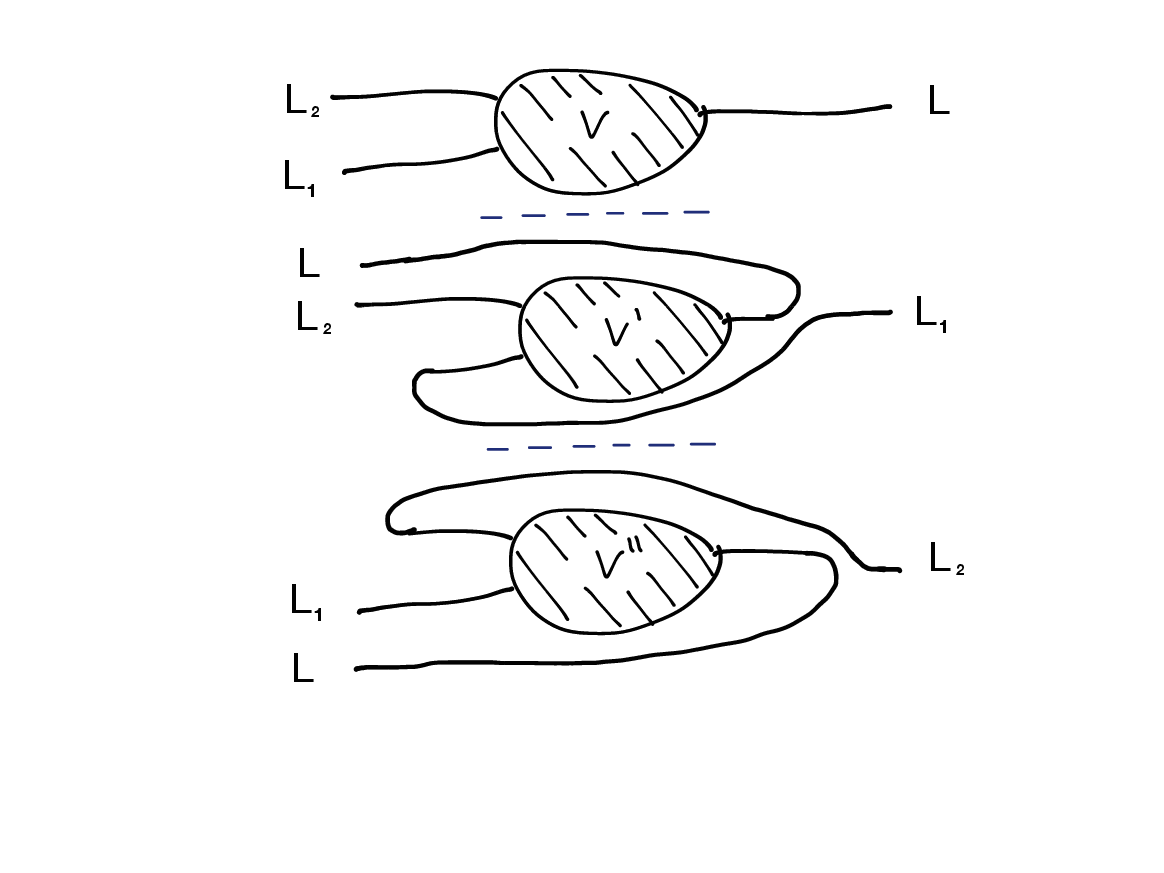,
           width=1\linewidth, height=0.8\linewidth }
      \end{center}
      \vspace{-1in}
      \caption{ The cobordisms $V$ and $V'$, $V''$ obtained by bending
        the ends of $V$ as indicated. \label{fig:hairy-cob}}
   \end{figure}
\end{cor}
The proof of Corollary~\ref{cor:special-short} is given
in~\S\ref{subsubsec:proof-cor}. 

\subsubsection{Floer homology and the bottom triangle
  in~\eqref{eq:commut-diag1}}

Let $N$ be an object in $D\fk(M)$.  There is an obvious functor
$$h_{N}:D\fk(M)\xrightarrow{\hom(N, -)} (\mathcal{V}, \times)$$ 
with values in the monoidal category $(\mathcal{V}, \times)$ of
ungraded vector spaces over $\mathbb{Z}_2$, with the monoidal
structure $\times$ being direct product. (We thus ignore all the
issues related to grading and orientations.) The functor
$\mathcal{H}F_{N}$ in the diagram is the composition
$\mathcal{H}F_{N}=h_{N}\circ \mathcal{F}$.  Assuming now that $N\in
\lag$, we remark that the functor $\mathcal{H}F_{N}$ associates to
each object $L$ in $S\cobc(M)$ the Floer homology $HF(N,L)$.  Thus the
functor $\mathcal{H}F_{N}$ encodes Floer homology as a sort of
``Lagrangian Quantum Field Theory'': it is a vector space valued
functor defined on a cobordism category that associates to each
Lagrangian $L\in\lag$ the Floer homology $\mathcal{H}F_{N}(L)=HF(N,L)$
(one could also complete $S\cobc(M)$ to a monoidal
category over which $\mathcal{H}F_{N}$ extends monoidally thus
bringing the formal properties of $\mathcal{H}F_{N}$ even closer to
the axioms of a TQFT).

From this perspective the existence of Diagram \ref{eq:commut-diag1}
can be seen as a statement concerning the properties of the Floer
homology functor. In particular, the existence of  $\mathcal{F}$ reflects the 
naturality properties of $\mathcal{H}F_{N}$
with respect to $N$. Further properties involve the triangulated structure of $D\fk(M)$ and they translate into the existence of the lift $\widetilde{\mathcal{F}}$.

In more concrete terms, as a consequence of Diagram \eqref{eq:simple-exact},
of Remark \ref{rem:ham-Seidel-rep} and
Theorem~\ref{cor:multiple-triang}, we immediately see that:
\begin{cor} \label{cor:super-simple}For any $N\in
   \lag$ the Floer homology functor
   $$\mathcal{H}F_{N}:S\cobc(M)\to  (\mathcal{V},\times)$$ 
   defined above has the following three properties:
   \begin{itemize}
     \item[i.]$\mathcal{H}F_{N}$ restricts to the Seidel
      representation on those cobordisms $V$ that are given as the
      Lagrangian suspension associated to a Hamiltonian isotopy acting
      on a given Lagrangian submanifold of $M$.
     \item[ii.]If $V$ has just two negative ends $L_{1}$, $L_{2}$ and
      $V'$, $V''$ are as in Corollary~\ref{cor:special-short}, then
      there is a long exact sequence that only depends on the
      horizontal isotopy type of $V$
      $$\ldots \longrightarrow \mathcal{H}F_{N}(L_{2})
      \xrightarrow{\mathcal{H}F_{N}(V'')} \mathcal{H}F_{N}(L_{1})
      \xrightarrow{\mathcal{H}F_{N}(V')} \mathcal{H}F_{N}(L)
      \xrightarrow{\mathcal{H}F_{N}(V)} \mathcal{H}F_{N}(L_{2})
      \longrightarrow \ldots$$ and this long exact sequence is natural
      in $N$.
     \item[iii.] \pbredb{More generally, if $V$ has negative ends
        $L_{1}, L_{2},\ldots, L_{k}$ with $k\geq 2$, then there exists
        a spectral sequence $\mathcal{E}_{N}(V) = \{E^r_p, d^r: E^r_p
        \longrightarrow E^r_{p-1}\}_{r \geq 0, p \in \mathbb{Z}}$,
        each page of which is graded by a single index, with $E^r_p = 0$
        for $p\leq 0$, and so that:
      \begin{itemize}
        \item[a.] the first page of the spectral sequence satisfies:
         $E^1_p = HF(N, L_p)$ for every $p \geq 1$. Moreover, the
         differential $d^1: E^1_p \longrightarrow E^1_{p-1}$ is given
         by $d^1 = H(\pi_{p-1}) \circ H(\psi_p)$, where $H(\psi_p)$
         and $H(\pi_{p-1})$ are the Floer homological maps induced by
         the morphisms $\psi_p$ and $\pi_{p-1}$ from the exact
         triangles in Theorem~\ref{cor:multiple-triang} for $i=p-1, p$
         \begin{equation*}
            \begin{aligned}
               L_{p-1}[1] \xrightarrow{\, \psi_{p-1}\, } & Z_{p-2}
               \longrightarrow Z_{p-1}
               \xrightarrow{\, \pi_{p-1}\, } L_{p-1}, \\
               L_{p}[1] \xrightarrow{\, \psi_p \,} & Z_{p-1} \longrightarrow
               Z_{p} \xrightarrow{\, \pi_{p} \, } L_p.
            \end{aligned}
         \end{equation*}
        \item[b.]  from the first page on, the terms of the spectral
         sequence only depend on $N$ and the horizontal isotopy type
         of $V$. Furthermore, the sequence is natural in $N$.
        \item[c.] $\mathcal{E}_{N}(V)$ collapses at page $r = k$ and
         converges to $\mathcal{H}F_{N}(L)$.
      \end{itemize}
      }
   \end{itemize}
\end{cor}


\subsection{Relation to $K$-theory} \label{sbsb:K-intro}

The cobordism category $\cobc$ of a symplectic manifold $M$ gives rise
to a group $\gcob$ somewhat analogous to cobordism groups in
differential topology. For this end we first consider the free abelian
group $F_{\lag}$ generated by the Lagrangian submanifolds $L \in
\lag$. We then define a subgroup of relations $R_{\lag} \subset
F_{\lag}$ as the subgroup generated by all elements of the form $L_1 +
\cdots + L_r$ for which there exists a Lagrangian cobordism $V \in
\lag_{cob}$ without any positive end and whose negative ends consist
of $(L_1, \ldots, L_r)$ (i.e. $(L_1, \ldots, L_r)$ is null Lagrangian
cobordant). We define $\gcob = F_{\lag} / R_{\lag}$.

One can alter the above and make other meaningful definitions. For
example, one can consider also a non-abelian version of $G_{cob}$ in
which the order of the Lagrangians $(L_1, \ldots, L_r)$ on the
positive end plays a role (see e.g.~\cite{Bi-Co:cob1}). Another
possible variation is to consider the Lagrangians $L \in \lag$
together with additional structures such as orientations, spin
structures, grading, local systems etc. One defines then the relations
as above by requiring in addition that these structures extend over
the cobordism $V$.

Next consider the $K$-theory group (or Grothendieck group)
$K_0(D\fk(M))$ of the derived Fukaya category of $M$. Recall that this
group is the abelian group generated by the objects of $D\fk(M)$
modulo the following collection of relations: every exact triangle
$$X \longrightarrow Y \longrightarrow Z \longrightarrow X[-1]$$ contributes the 
relation $X - Y + Z = 0$. (In our case the Fukaya category is not
graded, hence $Y = -Y$ for every object $Y$.)

\begin{cor} \label{c:K-groups} The mapping $\lag \longrightarrow
   K_0(D\fk(M))$ given by $L \longmapsto L$ induces a well defined
   homomorphism of groups
   \begin{equation} \label{eq:theta-K-groups}
      \Theta : \gcob \longrightarrow K_0(D\fk(M)).
   \end{equation}
\end{cor}

The proof follows immediately from Theorem~\ref{cor:multiple-triang}.
Indeed, let $L_1 + \cdots + L_r$ be a generator of $R_{\lag}$, so
that we have a Lagrangians cobordism $V$ without positive ends and
with negative ends $(L_1, \ldots, L_r)$. By
Theorem~\ref{cor:multiple-triang} we have objects $Z_1 \simeq L_1$,
$Z_2, \ldots, Z_{r-1}$, $Z_r \simeq 0$ in $D\fk(M)$ and a sequence of
exact triangles:
$$L_i[1] \longrightarrow Z_{i-1} \longrightarrow Z_i \longrightarrow
L_i, \quad \forall \; 2 \leq i \leq r.$$ From this we obtain the
following identities in $K_0(D\fk(M))$:
$$Z_1 = L_1, \quad Z_i = L_i + Z_{i-1} \;\; 
\forall \; 2 \leq i \leq r, \quad Z_r=0.$$ Summing these identities
up, it readily follows that $L_1 + \cdots + L_r = 0$ in $K_0(D
\fk(M))$.  This proves that $R_{\lag}$ is sent to $0 \in K_0(D\fk(M))$
under the mapping $F_{\lag} \longrightarrow K_0(D\fk(M))$, induced by
$L \longmapsto L$, hence the homomorphism $\Theta$ is well defined.

An interesting question is when is the homomorphism $\Theta$ an
isomorphism. This question can sometimes be studied with the help of
homological mirror symmetry at least in those cases where it has been
established. More precisely, if there exists a triangulated
equivalence between an appropriate completion $D^c\fk(M)$ of $D\fk(M)$
and the bounded derived category of coherent sheaves $D^b
Coh(M^{\vee})$ on the mirror manifold $M^{\vee}$, then there is an
isomorphism between the associated $K$-groups: $K_0(D^c\fuk(M)) \cong
K(M^{\vee})$. The point is that in some cases the latter group is well
known. One can then combine such algebro-geometric information
together with constructions of Lagrangian cobordisms on the
``$M$''-side, e.g. surgery (see~\S6 in~\cite{Bi-Co:cob1}), in order to
study the homomorphism $\Theta$.

For such an approach one needs sometimes to adjust a bit the
definitions of $\cobc$ and $D\fk(M)$ to include more structures, as
indicated above, or to work with particular classes of Lagrangian
submanifolds $\lag$. The simplest non-trivial example seems to be $M =
\mathbb{T}^2$ (and $M^{\vee} = $ elliptic curve). \pbred{Recent
  results of Haug~\cite{Haug:K-cob, Haug:PhD-thesis} show} that in
this case an appropriate version of $\Theta$ (defined for a suitable
class $\lag$) is indeed an isomorphism.

\subsection{Outline of the proof of Theorem~\ref{thm:main}}
\label{subsec:outline-proof}
 
The proof has essentially two main steps.
 
\subsubsection{The Fukaya category of cobordisms}

As mentioned before, the first step is to define a Fukaya category of
cobordisms in $\mathbb{R}^2\times M$, which we denote
$\fk_{cob}(\mathbb{R}^2\times M)$. The construction follows the set-up
in Seidel's book~\cite{Se:book-fukaya-categ}.  In particular, the
regularity of the relevant moduli spaces is insured by perturbing the
Cauchy-Riemann equation by Hamiltonian terms.  Compared to the
construction in~\cite{Se:book-fukaya-categ}, there are two additional
major issues that have to be addressed in our setting: the first is
that we work in a monotone situation and no longer an exact one. The
second is that Lagrangian cobordisms are embedded in a non-compact
ambient manifold, $\mathbb{R}^2\times M$, and the total
spaces of these cobordisms are non compact Lagrangians. Thus we need
to deal with compactness issues as well as with regularity at
$\infty$. Adapting the construction from the exact setting to the
monotone one is in fact non-problematic: it uses the same type of
arguments as in our previous work~\cite{Bi-Co:rigidity}. The
non-compactness issue turns out to be considerably more delicate.  As
in~\cite{Bi-Co:cob1}, the main tool that we use to insure the
compactness of moduli spaces of perturbed $J$-holomorphic curves
$\mathbf{u}$ is based on the open mapping theorem for holomorphic
functions in the plane. There is however a difficulty in implementing
this strategy directly. On one hand, the
Hamiltonian perturbations needed to construct the
$A_{\infty}$-category have to be picked in such a way as to insure
regularity, including at infinity, which requires perturbations that
are not compactly supported. On the other hand, to apply the
compactness argument based on the open mapping theorem we need that,
outside of a compact in $\mathbb{R}^2\times M$, the curves ${\bf u}$
satisfy a horizontally homogeneous equation in the sense that the
projections to $\mathbb{R}^2 \cong \mathbb{C}$ of the curves ${\bf u}$
are {\em holomorphic}. These two constraints: perturbations that are
non-trivial at $\infty$ and horizontally homogeneous equations are in
general incompatible!  To deal with this point we define the relevant
moduli spaces using curves ${\bf u}$ that satisfy perturbed
$J$-holomorphic equations with Hamiltonian perturbation terms that do
not vanish at $\infty$ but that have a special behavior away from a
fixed, sufficiently big compact.  The Hamiltonian perturbations are so
that the curves ${\bf u}$ transform by a specific change of variable - also
known as a naturality transformation - to curves ${\bf v}$ that are
horizontally homogeneous at infinity.  Compactness for the curves
${\bf v}$ implies then the desired compactness for the curves ${\bf
  u}$.  The naturality transformation can be implemented in a coherent
way along all the moduli spaces used to define the
$A_{\infty}$-multiplications (see~\S\ref{subsec:strip-ends}) but then
two further problems arise.  First, the boundary conditions for the
curves ${\bf u}$ are slightly different from those considered e.g.
in~\cite{Se:book-fukaya-categ} and thus energy bounds have to be
verified explicitly as they do not directly result from the
calculations in~\cite{Se:book-fukaya-categ}. A second and more serious
issue is that the boundary conditions for the curves ${\bf v}$ are not
fixed but rather moving ones. As a consequence, proving compactness
for the curves ${\bf v}$ is not quite immediate and requires
additional precision in the choice of perturbations. This is
implemented in~\S\ref{subsec:strip-ends} and~\S\ref{subsec:energy},
where we use the term {\em bottleneck} to indicate the particular
profile of the Hamiltonians that are adapted to this purpose. (See
e.g. Figure~\ref{fig:kinks} in~\S\ref{subsec:equation}.) These choices
of particular Hamiltonian perturbations come back with a vengeance and
complicate to a large extent the proofs of various properties of the
resulting $A_{\infty}$-category such as invariance.

\subsubsection{Inclusion, triangles and
  $\widetilde{\mathcal{F}}$}\label{subsubsec:intro-triang}

To construct the functor $\widetilde{\mathcal{F}}$ we first compare
the two categories $\fk(M)$, $\fk_{cob}(\mathbb{R}^2\times M)$.
Namely, we show that if $\gamma:\R\to \mathbb{R}^2$ is a curve in the
plane with horizontal ends, then there is an induced functor of
$A_{\infty}$-categories:

$$\mathcal{I}_{\gamma}:\fk(M)\to \fk_{cob}(\mathbb{R}^2\times M)$$
defined on the corresponding Fukaya category of $M$.  On objects this
functor is defined by $\mathcal{I}_{\gamma}(L)=\gamma\times L$.

Denote $\mathcal{A}=\fk(M)$ and
$\widetilde{\mathcal{A}}=\fk_{cob}(\mathbb{R}^2\times M)$. There is a
Yoneda embedding functor $$\mathcal{Y}:\mathcal{A}\to \textnormal{\sl
  fun}(\mathcal{A},Ch^{opp}),$$ where the right-hand side stands for
$A_{\infty}$-functors from $\mathcal{A}$ to the opposite category of
chain complexes viewed as a dg-category. A similar Yoneda embedding is
defined also on $\widetilde{\mathcal{A}}$.

Fix now a cobordism $(V; L_{1}\cup \ldots \cup L_{k}, L)$ as in
Figure~\ref{fig:cob-curv}.  Given any curve $\gamma$ as above, there
is a functor $\mathcal{M}_{V,\gamma}:\mathcal{A}\to Ch^{opp}$ defined
by
$$\mathcal{M}_{V,\gamma}=\mathcal{Y}(V)\circ \mathcal{I}_{\gamma}~.~$$
At the derived level, this functor only depends on the horizontal
isotopy classes of $V$ and $\gamma$.  We consider a particular set of
curves $\alpha_{1},\ldots, \alpha_{k}\subset \mathbb{R}^2$ basically
as in Figure \ref{fig:cob-curv}.  Therefore, we get a sequence of
functors $\mathcal{M}_{V,i}:=\mathcal{M}_{V,\alpha_{i}}$, $i=1,
\ldots, k$.
\begin{figure}[htbp]
   \begin{center}
      \epsfig{file=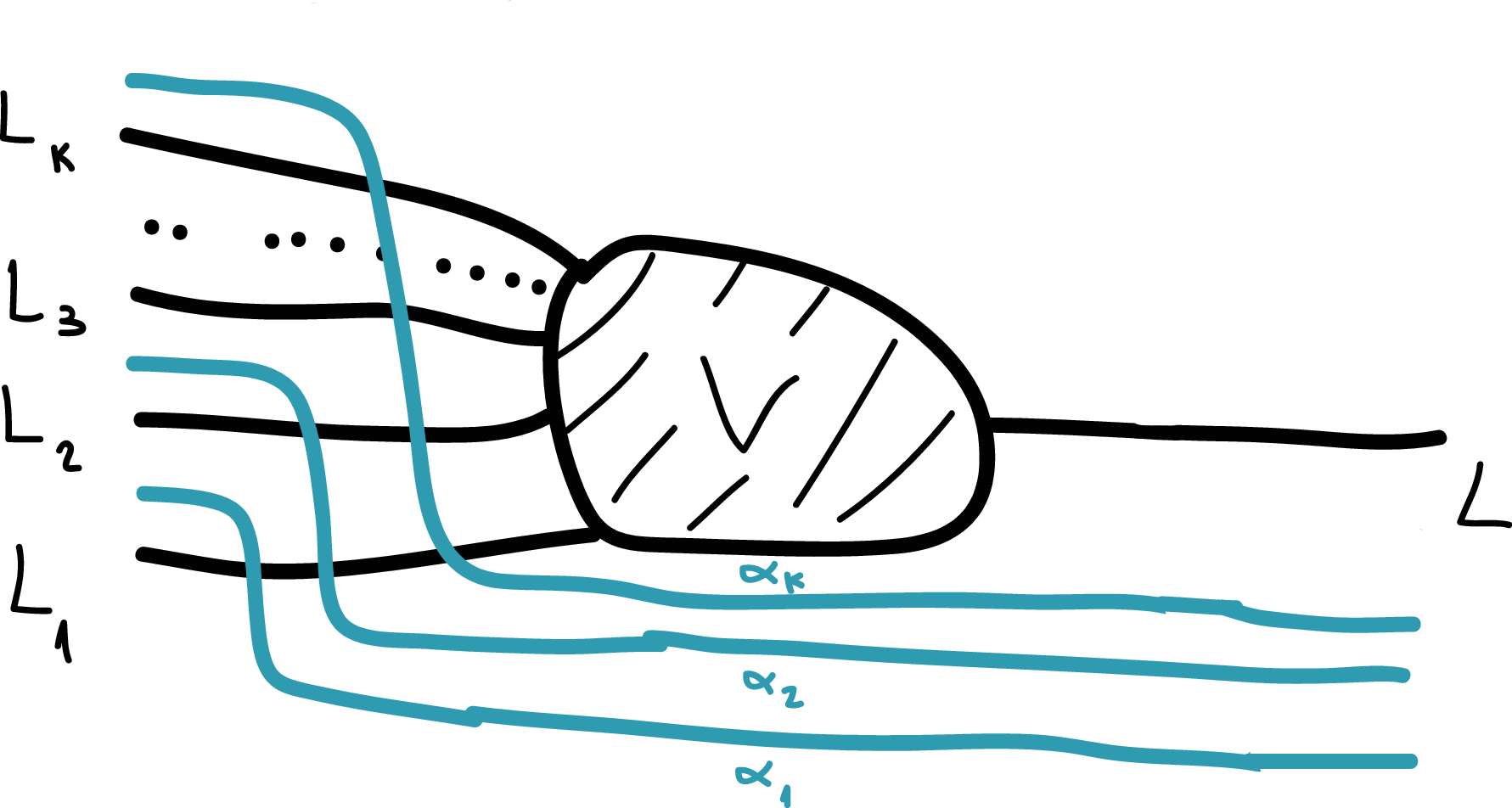, width=0.6\linewidth}
   \end{center}
   \caption{\label{fig:cob-curv} A cobordism $V$ together with curves
     of the type $\alpha_{i}$'s.}
\end{figure}

We then show that these functors are related by exact triangles
(in the sense of triangulated $A_{\infty}$ categories):
\begin{equation} \label{eq:funct-exact} T^{-1}\mathcal{Y}(L_{s})\to
   \mathcal{M}_{V,s-1}\to \mathcal{M}_{V,s}\to \mathcal{Y}(L_{s}) \;\;
   \forall\ 2 \leq s \leq k.
\end{equation}
Moreover, there is a quasi-isomorphism $\phi_{V}:\mathcal{Y}(L)\to
\mathcal{M}_{V,k}$. The proof of this fact requires the same type of
arguments that appeared earlier in constructing the Fukaya cobordism
category together with some new geometric ingredients.  \ocgreen{In
  particular, the key ingredient to show the existence of the exact
  triangles (\ref{eq:funct-exact}) is the fact that, with appropriate
  choices of data, the relevant perturbed holomorphic curves
  $\mathbf{u}$ that contribute to the $A_{\infty}$ operations
  transform by naturality into curves $\mathbf{v}$ whose projection is
  holomorphic around the intersections of the curves $\alpha_{i}$ and
  the projection of $V$. By taking into account orientations and using
  again the open mapping theorem it then follows that if such a curve
  (viewed as punctured polygon) has as entries intersection points
  involving some of the first $s$ ends of $V$, then it has as exit an
  intersection point also involving one of these ends. The exact
  sequences (\ref{eq:funct-exact}) are an algebraic translation of
  this fact.}

With the exact sequences (\ref{eq:funct-exact}) established, the
definition of $\widetilde{\mathcal{F}}$ is relatively direct, by
translating the preceding structure to the derived setting.  Finally,
we verify that $\widetilde{\mathcal{F}}$ respects composition which is
again a non-trivial step.

\begin{rem} 
   a. Apriori, a different approach to the construction of the Fukaya
   category of cobordisms, that avoids the difficult perturbation
   issues above, would be to use for the definition of all the
   relevant moduli spaces only horizontally homogeneous equations.  In
   this case, compactness is automatic but the algebraic output of the
   construction is not an $A_{\infty}$-category but a weaker structure
   sometimes called a pre-$A_{\infty}$-category. For instance, the
   Floer complex $CF(V,V')$ for two cobordisms $V$ and $V'$ is only
   defined if $V$ and $V'$ are distinct at $\infty$.  This leads to a
   plethora of further complications. It is not clear whether this
   other approach can lead to a proof of Theorem \ref{thm:main} and,
   even more, to one shorter than the proof here.

   b. \ocgreen{As explained in \S\ref{subsubsec:intro-triang},
     transforming the curves $\mathbf{u}$ by naturality to curves
     $\mathbf{v}$ whose projection is holomorphic outside of a certain
     compact $\subset \C$ is important for the proof of Theorem
     \ref{thm:main} not only to define the Fukaya category of
     cobordisms but also in the second step, where specific properties
     of planar holomorphic curves enter the argument.  Indeed, we
     actually need at that point rather fine control on the region of
     holomorphicity of the projection of $\mathbf{v}$ in the sense
     that it is not sufficient for holomorphicity to take place at
     infinity but also in regions where various cylindrical
     projections of the Lagrangians involved intersect.
   
     Fukaya categories in a variety of other non-compact situations
     have appeared before in the literature, in particular
     in~\cite{AbouSeidel:wrapped} and in~\cite{Se:Lefschetz-Fukaya}.
     The construction in Seidel's paper \cite{Se:Lefschetz-Fukaya} is
     closest to the construction here and a number of results from
     that paper are used here. Moreover, a rather straightforward
     adaptation of the methods in \cite{Se:Lefschetz-Fukaya} leads to
     a category with objects cobordisms with ends only on one side
     (that is cobordisms of the type $V:\emptyset \to (L_{1},\ldots,
     L_{k})$).  Compactness, is insured in \cite{Se:Lefschetz-Fukaya}
     by a variant of the maximum principle for harmonic functions and,
     while it does require a special form of ``disjoining'' planar
     hamiltonian perturbations, all naturality issues are bypassed.
     However, this setup is not applicable, at least directly, to the
     proof of Theorem \ref{thm:main} not only because we need to deal
     with cobordisms with arbitrary ends but, more importantly,
     because implementing the second step of the proof in this setup
     does not seem immediate.  In short, our choice here is to
     construct the category $\fuk_{cob}(\C\times M)$ in a form that is
     directly applicable to the proof of Theorem \ref{thm:main}.  The
     construction in itself provides an alternative approach to that
     in \cite{Se:Lefschetz-Fukaya} and is potentially of some
     independent interest.  }
\end{rem}

\subsection*{Acknowledgments} Part of this work was accomplished
during a stay at the Institute for Advanced Study in Princeton. We
thank Helmut Hofer and the IAS for their gracious hospitality.
\pbred{We would also like to thank the referee for a very careful reading
  of the paper and for making many comments helping to improve the
  quality of the exposition.}


\section{Prerequisites}\label{sec:prerequisites}

Here we fix the setting of the paper, in particular the definition of
the Lagrangian cobordism category that we use, the relevant Fukaya
category as well as all the auxiliary constructions and conventions
needed in the paper. {{Note that from now on and through the remainder
    of the paper, a part of the notation from the Introduction will
    change. Namely, the class of Lagrangian submanifolds $\lag$ will
    be denoted $\mathcal{L}^*_d$, the class of admissible cobordisms
    by $\lag_d(\mathbb{C} \times M)$, the category $\cobc$ will be
    denoted by $\cob^d_0$ and $\fk$ by $\fuk^d$. The meaning of the
    decorations $d$, $*$ and $0$ in this notation will be explained
    below.}}

We assume here that the manifold $(M^{2n}, \omega)$ is compact.
Lagrangian submanifolds $L \subset M$ will be generally assumed to be
closed unless otherwise indicated.

The
subsections~\S\ref{sb:monotonicity},~\S\ref{sub:cob-def},~\S\ref{subsec:category-cob}
and~\S\ref{sb:rev-HF} are just recalls of various definitions and
constructions from~\cite{Bi-Co:cob1} and \S\ref{sec:fuk-M} concerns the Fukaya category. 
 Subsection~\ref{subsec:cones}
contains a description of the $T^{S}(-)$ construction that is more
detailed and precise than the one in~\cite{Bi-Co:cob1}.

\subsection{Monotonicity} \label{sb:monotonicity} 

All families of Lagrangian submanifolds in our constructions have to
satisfy a monotonicity condition in a uniform way as described below.
Given a Lagrangian submanifold $L \subset M$ let
$$\omega : \pi_{2}(M,L)\to \R \ , \ \mu:\pi_{2}(M,L)\to \Z$$
be the morphism given, respectively, by integration of $\omega$ and by
the Maslov index. The Lagrangian $L$ is \emph{monotone} if there
exists a positive constant $\rho>0$ so that for all $\alpha\in
\pi_{2}(M,L)$ we have $\omega(\alpha)=\rho\mu(\alpha)$ and moreover
the minimal Maslov number $$N_{L}: =\min\{\mu(\alpha) : \alpha\in
\pi_{2}(M,L) \ ,\ \omega(\alpha)>0\}$$ satisfies $N_{L}\geq 2$.

We will use $K=\mathbb{Z}_2$ as the ground ring. 
However, we mention here that most of the discussion
generalizes to arbitrary rings under additional assumptions on the Lagrangians. 

For a  closed, monotone Lagrangian $L$ there is an associated basic 
Gromov-Witten type invariant $d_{L}\in \Z_{2}$ given 
as the number (mod $2$) of $J$-holomorphic disks of Maslov
index $2$ going through a generic point $P\in L$ for $J$ a generic
almost complex structure that is compatible with $\omega$.

A family of Lagrangian submanifolds $L_{i}$, $i\in I$, is 
uniformly monotone if each $L_{i}$ is monotone and the
following condition is satisfied: there exists $d\in K$ so that for
all $i\in I$ we have $d_{L_{i}}=d$ and there
exists a positive real constant $\rho$ so that the monotonicity
constant of $L_{i}$ equals $\rho$ for all $i\in I$.
All the Lagrangians used in the
paper will be assumed monotone and, similarly, the Lagrangian families
will be assumed uniformly monotone.

For $d\in \Z_{2}$ and $\rho\geq 0$, we let 
$\mathcal{L}_{d}(M)$ be the family of closed, connected Lagrangian
submanifolds $L\subset M$ that are monotone with monotonicity constant
$\rho$ and with $d_{L}=d$ (we thus suppress $\rho$ from the notation).

\subsection{Cobordism: main definitions}\label{sub:cob-def}
The plane $\mathbb{R}^2$  is endowed with the symplectic
structure $\omega_{\mathbb{R}^2} = dx \wedge dy$, $(x,y) \in
\mathbb{R}^2$.  The product $\widetilde{M}=\mathbb{R}^2 \times M$ is endowed with 
the symplectic form $\omega_{\mathbb{R}^2} \oplus \omega$. We denote by
$\pi: \mathbb{R}^2 \times M \to \mathbb{R}^2$ the projection. For a
subset $V \subset \mathbb{R}^2 \times M$ and $S \subset \mathbb{R}^2$
we let $V|_{S} = V \cap \pi^{-1}(S)$.

\begin{dfn}\label{def:Lcobordism}
   Let $(L_{i})_{1\leq i\leq k_{-}}$ and $(L'_{j})_{1\leq j\leq
     k_{+}}$ be two families of closed Lagrangian submanifolds of
   $M$. We say that that these two (ordered) families are Lagrangian
   cobordant, $(L_{i}) \simeq (L'_{j})$, if there exists a smooth
   compact cobordism $(V;\coprod_{i} L_{i}, \coprod_{j}L'_{j})$ and a
   Lagrangian embedding $V \subset ([0,1] \times \mathbb{R}) \times M$
   so that for some $\epsilon >0$ we have:
   \begin{equation} \label{eq:cob_ends}
      \begin{aligned}
         V|_{[0,\epsilon)\times \mathbb{R}} = & \coprod_{i} 
         ([0, \epsilon) \times \{i\})  \times L_i \\
         V|_{(1-\epsilon, 1] \times \mathbb{R}} = 
         & \coprod_{j} ( (1-\epsilon,1]\times \{j\}) \times L'_j~.~
      \end{aligned}
   \end{equation}
   The manifold $V$ is called a Lagrangian cobordism from the
   Lagrangian family $(L'_{j})$ to the family $(L_{i})$. We 
   denote such a cobordism by $V:(L'_{j}) \cobto (L_{i})$ or
   $(V; (L_{i}), (L'_{j}))$.
\end{dfn}
\begin{figure}[htbp]
   \begin{center}
      \epsfig{file=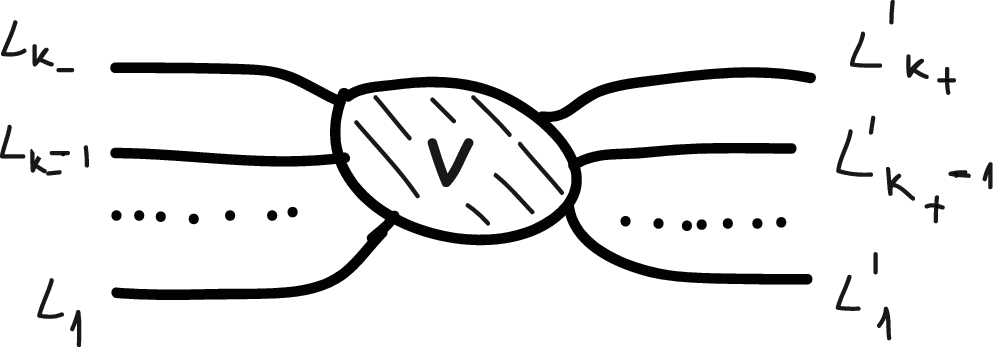, width=0.6\linewidth}
   \end{center}
   \caption{A cobordism $V:(L'_{j})\cobto (L_{i})$ projected
     on $\mathbb{R}^2$.}
\end{figure}

 A cobordism is called {\em monotone} if
$$V\subset ([0,1]\times \mathbb{R})\times M$$ is a 
monotone Lagrangian submanifold.    

It is
often more convenient to view cobordisms as embedded in $\mathbb{R}^2
\times M$. Given a cobordism $V\subset ([0,1] \times \mathbb{R})
\times M$ as in Definition~\ref{def:Lcobordism} we can extend
trivially its negative ends towards $-\infty$ and its positive ends to
$+\infty$ thus getting a Lagrangian $\overline{V} \subset \mathbb{R}^2
\times M$. We will in general not distinguish between $V$ and
$\overline{V}$ but if this distinction is needed we will call
\begin{equation}\label{eq:R-extension}
   \overline{V} = \Bigl(\coprod_{i} (-\infty, 0] \times \{i\} 
   \times L_i \Bigr) \cup V \cup 
   \Bigl(\coprod_{j} [1,\infty) \times \{j\} 
   \times L'_j \Bigr)
\end{equation}
The $\mathbb{R}$-extension of $V$.  At certain points in the paper we will 
also use Lagrangians in $\R^{2}\times M$ that are $\R$-extensions of cobordisms
$V\subset ([a,b]\times\R)\times M$. The definition of such cobordims is identical
with the one above except with the interval $[a,b]$ replacing $[0,1]$.

More generally, by a {\em Lagrangian submanifold with cylindrical ends} we mean a
Lagrangian submanifold $\overline{V} \subset \widetilde{M}$ without
boundary that has the following properties:
\begin{enumerate}
  \item[i.] For every $a<b$ the subset $\overline{V}|_{[a,b] \times
     \mathbb{R}}$ is compact.
  \item[ii.] There exists $R_+$ such that $$\overline{V}|_{[R_+, \infty)
     \times \mathbb{R}} = \coprod_{i=1}^{k_+} [R_+, \infty) \times
   \{a^+_i\} \times L^{+}_i$$ for some $a^+_1 < \cdots < a^+_{k_+}$
   and some Lagrangian submanifolds $L^+_1, \ldots, L^+_{k_+} \subset
   M$.
  \item[iii.] There exists $R_{-} \leq R_+$ such that
   $$\overline{V}|_{(-\infty, R_{-}] \times
     \mathbb{R}} = \coprod_{i=1}^{k_-} (-\infty, R_-] \times \{a^-_i\}
   \times L^{-}_i$$ for some $a^{-}_1 < \cdots < a^-_{k_-}$ and some
   Lagrangian submanifolds $L^{-}_1, \ldots, L^{-}_{k_-} \subset M$.
\end{enumerate}
We allow $k_+$ or $k_-$ to be $0$ in which case $\overline{V}|_{[R_+,
  \infty)\times \mathbb{R}}$ or $\overline{V}|_{(-\infty, R_{-}]
  \times \mathbb{R}}$ are void.

For every $R \geq R_+$ write $E^+_{R}(\overline{V}) =
\overline{V}|_{[R, \infty) \times \mathbb{R}}$ and call it a positive
cylindrical end of $\overline{V}$. Similarly, we have for $R \leq
R_{-}$ a negative cylindrical end $E^{-}_{R}(\overline{V})$.

If $\overline{W}$ is a Lagrangian submanifold with
cylindrical ends then by an obvious modification of the ends (and a
possible symplectomorphism on the $\mathbb{R}^2$ component) it is easy
to obtain a Lagrangian cobordism between the families of Lagrangians
corresponding to the positive and negative ends of $\overline{W}$.

In order to simplify terminology, we will say that a Lagrangian with
cylindrical ends $\overline{V}$ is cylindrical outside of a compact
subset $K \subset \mathbb{R}^2$ if $\overline{V}|_{\mathbb{R}^2
  \setminus K}$ consists of horizontal ends, i.e. it is of the form
$E^{-}_{R_{-}}(\overline{V}) \cup E^+_{R_+}(\overline{V})$.

We will also need the following notion.
\begin{dfn} \label{d:cyl-dist} Two Lagrangians with cylindrical ends
   $\overline{V}, \overline{W} \subset \widetilde{M}$ are said to be
   cylindrically distinct at infinity if there exists $R>0$ such that
   $\pi(E^+_R(\overline{V})) \cap \pi(E^+_R(\overline{W})) =
   \emptyset$ and $\pi(E^-_{-R}(\overline{V})) \cap
   \pi(E^-_{-R}(\overline{W})) = \emptyset$.
\end{dfn}

Finally, here is a class of Hamiltonian isotopies that will be
useful in the following.
\begin{dfn}[Horizontal isotopies] \label{d:isotopies} Let
   $\{\overline{V_t}\}_{t \in [0,1]}$ be an isotopy of Lagrangian
   submanifolds of $\widetilde{M}$ with cylindrical ends. We call this
   isotopy horizontal if there exists a (not necessarily compactly
   supported) Hamiltonian isotopy $\{\psi_t\}_{t \in [0,1]}$ of
   $\widetilde{M}$ with $\psi_0 = \Id$ and with the following
   properties:
   \begin{enumerate}
     \item[i.] $\overline{V_t} = \psi_t(\overline{V_0})$ for all $t
      \in [0,1]$.
     \item[ii.] There exist real numbers $R'_{-}< R_- < R_+<R'_{+}$ such that for all
      $t\in [0,1]$, $x\in E^{\pm}_{R'_{\pm}}(\overline{V_0})$ we have
      $\psi_t(x)\in E^{\pm}_{R_{\pm}}(\overline{V_0})$.
     \item[iii.] There is a constant $K>0$ so that for all $x\in
      E^{\pm}_{R_{\pm}}(\overline{V_0})$, $|d\pi_{x}(X_t(x))|< K$.  Here
      $X_t$ is the (time dependent) vector field of the flow
      $\{\psi_t\}_{t \in [0,1]}$.
   \end{enumerate}
   We say that two Lagrangians $\overline{V},
   \overline{V'} \subset \widetilde{M}$ with cylindrical ends are {\em
     horizontally} isotopic if there exists an isotopy as above
   $\{\overline{V_t}\}_{t \in [0,1]}$ with
   $\overline{V_0}=\overline{V}$ and $\overline{V_1} = \overline{V'}$.
   We will sometimes say that an ambient Hamiltonian isotopy
   $\{\psi_t\}_{t \in [0,1]}$ as above is horizontal with respect to
   $\overline{V_0}$.
\end{dfn}

\subsection{The category $\mathcal{C}ob^{d}_{0}(M)$}
\label{subsec:category-cob}
Consider first the following category
$\widetilde{{\mathcal{C}ob}^d}(M)$, $d \in K$. Its objects are
families $(L_{1}, L_{2},\ldots, L_{r})$ with $r \geq 1$, $L_{i}\in
\mathcal{L}_{d}(M)$. (Recall that $\mathcal{L}_d(M)$ stands for
the class of uniformly monotone Lagrangians $L$ with $d_L = d$
 with the same monotonicity constant $\rho$ which is
omitted from the notation.) We will denote by $\mathcal{L}_{d}(\C\times M)$
the Lagrangians in $\C\times M$ that satisfy the same conditions: they 
are uniformly monotone with the same $d_{V}=d$ and the same monotonicity
constant $\rho$.

To describe the morphisms in this category we proceed in two steps.
First, for any two horizontal isotopy classes of cobordisms $[V]$ and
$[U]$ with $V:(L'_{j}) \cobto (L_{i})$ (as in
Definition~\ref{def:Lcobordism}) and $U:(K'_{s})\cobto (K_{r})$ we
define the sum $[V]+[U]$ to be the horizontal isotopy class of a  cobordism 
$W:(L'_{j})+(K'_{s})\cobto (L_{i})+(K_{r})$ so that $W=V \coprod
\widetilde{U}$ with $\widetilde{U}$ a suitable translation up the 
$y$-axis of a cobordism
horizontally isotopic to $U$ so that $\widetilde{U}$ is disjoint
from $V$.

The morphisms in $\widetilde{{\mathcal{C}ob}^d}(M)$ are now defined as
follows. A morphism $$[V] \in \mor \bigl( (L'_{j})_{1 \leq j \leq S},
(L_{i})_{1\leq i\leq T} \bigr)$$ is a horizontal isotopy class that
is written as a sum $[V]=[V_{1}] + \cdots + [V_{S}]$ with each
$V_{j}\in \mathcal{L}_{d}(\C\times M)$ a cobordism from the Lagrangian
family formed by the {\em single} Lagrangian $L'_{j}$ and a subfamily
$(L_{r(j)}, \ldots, L_{r(j)+s(j)})$ of the $(L_{i})$'s, and so that
$r(j)+s(j)+1=r(j+1)$.  In other words, $V$ decomposes as a union of
$V_{i}$'s each with a single positive end but with possibly many
negative ones. We will often denote such a morphism by $V:
(L'_{j})\longrightarrow (L_{i})$.

The composition of morphisms is induced by concatenation followed by a
rescaling to reduce the ``width'' of the cobordism to the interval
$[0,1]$. 

We consider here the void set as a Lagrangian of arbitrary
dimension. We now intend to
factor both the objects and the morphisms in this category by
equivalence relations that will transform this category in a strict
monoidal one. For the objects the equivalence relation is induced by
the relations
\begin{equation}
(L, \emptyset) \sim (\emptyset,L) \sim (L).
\label{eq:eq-objects}
\end{equation}  
 
At the level of the morphisms a bit more care is needed. For each $L
\in \mathcal{L}_{d}(M)$ we will define two particular cobordisms
$\Phi_{L}:(\emptyset, L)\cobto (L,\emptyset)$ and $\Psi_{L}:
(L,\emptyset)\cobto (\emptyset,L)$ as follows. Let $\gamma: [0,1] \to
[0,1]$ be an increasing, surjective smooth function, strictly
increasing on $(\epsilon,1-\epsilon)$ and with $\gamma'(t)=0$ for
$t\in [0,\epsilon]\cup [1-\epsilon, 1]$. We now let
$\Phi(L)=graph(\gamma)\times L$ and $\Psi(L)=graph(1-\gamma)\times L$.
The equivalence relation for morphisms is now induced by the following
two identifications:
\begin{enumerate}[(Eq 1)]
  \item For every cobordism $V$ we identify $V+\emptyset \sim
   \emptyset+V \sim V$, where $\emptyset$ is the void cobordism
   between two void Lagrangians. \label{i:eq-1}
  \item If $V:L \longrightarrow (L_{1},..., L_{i}, \emptyset,
   L_{i+2},\ldots, L_{k})$, then we identify $V \sim V' \sim V''$,
   where $V'=\Phi_{L_{i+2}} \circ V$, $V''=\Psi_{L_{i}}\circ V.$
   \label{i:eq-2}
\end{enumerate}

\begin{figure}[htbp]
   \begin{center}
      \epsfig{file=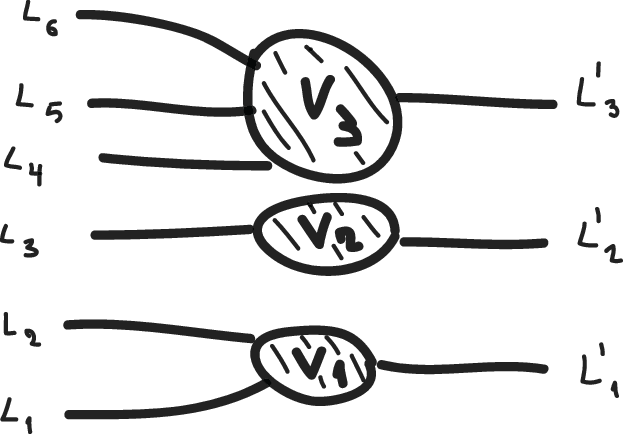, width=0.5\linewidth}
   \end{center}
   \caption{A morphism $V:(L'_{1},L'_{2},L'_{3}) \longrightarrow
     (L_{1}, \ldots, L_{6})$, $V=V_{1}+V_{2}+V_{3}$, projected to
     $\mathbb{R}^2$.}\label{fig:MorphCob}
\end{figure}

To construct the category $\mathcal{C}ob^{d}(M)$ we now consider the
full subcategory $\mathcal{S}\subset \widetilde{\cob}^{d}(M)$ obtained
by restricting the objects only to those families $(L_{1}, \ldots,
L_{k})$ with $L_{i}$ non-narrow for all $1\leq i\leq k$. Recall that a
monotone Lagrangian is non-narrow if its quantum homology $QH(L)$
(with $\La=\Z_{2}[t,t^{-1}]$ coefficients) does not vanish.  Then
$\mathcal{C}ob^{d}(M)$ is obtained by the quotient of the objects of
$\mathcal{S}$ by the equivalence relation in~\eqref{eq:eq-objects} and
the quotient of the morphisms of $\mathcal{S}$ by the equivalence
relation in~(Eq~\ref{i:eq-1}), (Eq~\ref{i:eq-2}).

This category is called the {\em Lagrangian cobordism category} of
$M$.  As mentioned before, it is a strict monoidal category, {where
  the monoidal structure is defined on objects by concatenating
  \pbred{tuples} of Lagrangians and on morphisms by taking disjoint
  unions of cobordisms, possibly after a suitable Hamiltonian
  isotopy.}

The Theorem~\ref{thm:main} requires an additional assumption on all
the Lagrangians in our constructions.  Every Lagrangian $L$ is
required to satisfy:
\begin{equation}\label{eq:Hlgy-vanishes}
    \pi_{1}(L)\stackrel{i_{\ast}}{\longrightarrow} 
   \pi_{1}(M) \ 
   \ \ \mathrm{vanishes},
\end{equation}
where $i_{\ast}$ is induced by the inclusion $L\subset M$. 
Alternatively, in case the first Chern class $c_{1}$ and $\omega$ are proportional
as morphisms defined on $H_{2}(M;\Z)$ (and not only on $\pi_{2}(M)$) it is enough
to assume that the image of $i_{\ast}$ is torsion.
 An
analogous constraint is imposed also to the Lagrangian cobordisms
involved. 

We denote by $\mathcal{L}^{\ast}_{d}(M)$ the Lagrangians in
$\mathcal{L}_{d}(M)$ that are non-narrow and additionally
satisfy~\eqref{eq:Hlgy-vanishes}. There is a subcategory of
$\mathcal{C}ob^{d}(M)$, that will be denoted by
$\mathcal{C}ob^{d}_{0}(M)$, whose objects consist of families of
Lagrangians each one belonging to $\mathcal{L}^{\ast}_{d}(M)$ and
whose morphisms are represented by Lagrangian cobordisms $V$
satisfying the analogous condition to~\eqref{eq:Hlgy-vanishes}, but in
$\mathbb{R}^2 \times M$. This is again a strict monoidal category.

\subsection{Floer homology}
\label{sb:rev-HF} 
In this subsection we recall some basic notation and  definitions concerning
Lagrangian Floer homology. We refer the
reader to~\cite{Oh:HF1, Oh:HF1-add, Oh:spectral} for the foundations
of Floer homology for monotone Lagrangians, and
to~\cite{FO3:book-vol1, FO3:book-vol2} for the general case. 

\subsubsection{Lagrangian Floer homology} \label{sbsb:lag-hf} Let
$L_0, L_1 \subset M$ be two monotone Lagrangian submanifolds with
$d_{L_0} = d_{L_1} = d$. We assume in addition that
$L_0$ and $L_1$ have the same monotonicity constant (or in other words
that the pair $(L_0, L_1)$ is uniformly monotone).

We assume that condition (\ref{eq:Hlgy-vanishes}) is satisfied by all
the Lagrangians and cobordisms in the paper.  An observation due to Oh
\cite{Oh:HF1} shows that in this case one can construct Floer
complexes and the associated homology with coefficients in $\Z_{2}$ as
summarized below.

Denote by $\mathcal{P}(L_0, L_1) = \{ \gamma \in C^0([0,1],M) \mid
\gamma(0) \in L_0, \gamma(1) \in L_1\}$ the space of paths in $M$
connecting $L_0$ to $L_1$. For  $\eta \in \pi_0(\mathcal{P}(L_0,
L_1))$ we denote the path connected component of $\eta$ by
$\mathcal{P}_{\eta}(L_0,L_1)$. 

Fix $\eta \in \pi_0(\mathcal{P}(L_0, L_1))$ and let $H: [0,1]\times M
\to \mathbb{R}$ be a Hamiltonian function with Hamiltonian flow
$\psi_t^H$. We assume that $\psi_1^H(L_0)$ is transverse to $L_1$.
Denote by $\mathcal{O}_{\eta}(H)$ the set of paths $\gamma \in
\mathcal{P}_{\eta}(L_0, L_1)$ which are orbits of the flow $\psi^H_t$.
Finally, we choose also a generic $1$-parametric family of almost
complex structures $\mathbf{J} = \{J_t\}_{t \in [0,1]}$ compatible
with $\omega$.

The Floer complex
$CF(L_0, L_1; \eta; H, \mathbf{J})$ with coefficients in
$\Z_{2}$ is
generated as a $\Z_{2}$-vector space by the elements of
$\mathcal{O}_{\eta}(H)$. The Floer differential $$\partial: CF(L_0,
L_1; \eta; H, \mathbf{J}) \longrightarrow CF(L_0, L_1; \eta; H,
\mathbf{J})$$ is defined as follows. For a generator $\gamma_{-} \in
\mathcal{O}_{\eta}(H)$ we put
$$\partial(\gamma_{-}) = 
\sum_{\scriptscriptstyle \gamma_+ \in \mathcal{O}_{\eta}(H)} \#_{2}(
\mathcal{M}_0(\gamma_{-}, \gamma_{+};H, \mathbf{J})) \gamma_{+}.$$
Here $\mathcal{M}_0(\gamma_{-}, \gamma_{+}; H, \mathbf{J})$ stands for
the $0$-dimensional components of the moduli space of finite energy strips
$u:\mathbb{R}\times [0,1] \longrightarrow M$ connecting $\gamma_{-}$
to $\gamma_{+}$ that satisfy the Floer equation
\begin{equation}\label{eq:Floer-equation}
\partial_{s}u+J\partial_{t}u+\nabla H(t,u)=0
\end{equation}
modulo the $\mathbb{R}$-action coming from translation in the $\R$
coordinate; the number of elements in $\mathcal{M}_0(\gamma_{-},
\gamma_{+}; H, \mathbf{J})$ is finite due to condition
(\ref{eq:Hlgy-vanishes}) and is counted over $\Z_{2}$.

\begin{rem}\label{rem:non-grading}
   In this paper the Floer complexes, $CF(-)$, are defined over
   $\Z_{2}$ and are not graded. Hence the associated Floer homology
   $HF(-)$ is also un-graded. In special situations one can endow
   $CF(-)$ with some grading though not always over $\mathbb{Z}$ (e.g.
   when $L_0$ and $L_1$ are both oriented, then there is a
   $\mathbb{Z}_2$-grading).  See~\cite{Se:graded} for a systematic
   approach to these grading issues.
\end{rem}

Standard arguments show that the homology $HF(L_0, L_1; \eta; H,
\mathbf{J})$ is independent of the additional structures $H$ and
$\mathbf{J}$ up to canonical isomorphisms. We will therefore omit $H$
and $\mathbf{J}$ from the notation.

We often consider all components $\eta \in
\pi_0(\mathcal{P}(L_0,L_1))$ together i.e. take the direct sum complex
\begin{equation} \label{eq:cf-all-eta} CF(L_0,L_1;H,\mathbf{J}) =
   \bigoplus_{\eta} CF(L_0,L_1;\eta; H,\mathbf{J})
\end{equation}
with total homology which we denote $HF(L_0,L_1)$. There is an obvious
inclusion map $i_{\eta}: HF(L_0,L_1;\eta) \longrightarrow
HF(L_0,L_1)$.

\begin{rems} \label{r:auton-J} 
 If $L_0$ and $L_1$ are transverse we can take
      $H=0$ in $CF(L_0, L_1; H, \mathbf{J})$. 
      When $H=0$ we will omit it from the notation and
      just write $CF(L_0, L_1; \mathbf{J})$. We will sometimes omit
      $\mathbf{J}$ too when its choice is obvious.
  \end{rems}

\subsubsection{Moving boundary conditions}
\label{subsubsec:movingbdry}

Assume that $L_{0}$ and $L_{1}$ are two transverse
Lagrangians. Fix the component $\eta$ and the almost complex structure
$\mathbf{J}$. We also fix once and for all a path $\gamma_{0}$ in the
component $\eta$. Now let $\varphi = \{\varphi_t\}_{t \in [0,1]}$ be a
Hamiltonian isotopy starting at $\varphi_0 = \Id$. The isotopy
$\varphi$ induces a map $$\varphi_* : \pi_0(\mathcal{P}(L_0, L_1))
\longrightarrow \pi_0(\mathcal{P}(L_0, \varphi_1(L_1)))$$ as follows.
If $\eta \in \pi_0(\mathcal{P}(L_0, L_1))$ is represented by
$\gamma:[0,1] \to M$ then $\varphi_* \eta$ is defined to be the
connected component of the path $t \mapsto \varphi_t(\gamma(t))$ in
$\mathcal{P}(L_0, \varphi_1(L_1))$.

The isotopy $\varphi$ induces a canonical isomorphism
\begin{equation} \label{eq:iso-c-phi}
   c_{\varphi}: HF(L_0, L_1; \eta)
   \longrightarrow HF(L_0, \varphi_1(L_1); \varphi_* \eta)
\end{equation}
coming from a chain map defined using moving boundary conditions (see
e.g.~\cite{Oh:HF1}). The isomorphism $c_{\varphi}$ depends only on the
homotopy class (with fixed end points) of the isotopy $\varphi$.

\begin{rem} These constructions also apply without modification to cases
   when $M$ is not compact (but e.g. tame), if we have some way to
    insure that all solutions $u$ of finite energy as above have
   their image inside a fixed compact set $K \subset M$.
Finally, the constructions recalled here can also be adapted to the case of
non-compact Lagrangians with cylindrical ends. This will be pursued in much more detail in this
paper following essentially the approach from \cite{Bi-Co:cob1}.
Other variants appear in slightly different settings in the literature (see for
instance the works of Seidel~\cite{Se:book-fukaya-categ},
Abouzaid~\cite{Abouzaid:homog-coord},
Auroux~\cite{Aur:fuk-cat-sym-prod}, as well as earlier work of
Oh~\cite{Oh:hf-non-compact}).
\end{rem}


\subsection{The Fukaya category $\fuk^{d}(M)$} \label{sec:fuk-M} In
this section we discuss the Fukaya $A_{\infty}$-category $\fuk^{d}(M)$
of uniformly monotone Lagrangians in $M$. We \pbred{refer}
to~\S\ref{subsec:more-alg} for the basic algebraic background on
$A_{\infty}$ categories. We emphasize that we work here in an ungraded
context and over $\Z_{2}$.  We also recall that
$\mathcal{L}^{\ast}_{d}(M)$ is the set of the Lagrangian submanifolds
$L$ of $M$ that are uniformly monotone with $d_{L}=d$, so that $L$ is
non-narrow (in other words, $QH(L)\not=0$) and, moreover,
condition~\eqref{eq:Hlgy-vanishes} is satisfied.  We will use the
Floer constructions with the conventions in~\S\ref{sbsb:lag-hf}.

In the paper we follow the definition and construction of the Fukaya
$A_{\infty}$-category from~\cite{Se:book-fukaya-categ} with the
following differences:
\begin{itemize}
  \item[i.]The objects of $\fuk^{d}(M)$ are the elements of
   $\mathcal{L}^{\ast}_{d}(M)$. Thus we work with monotone Lagrangians
   rather than exact ones.
  \item[ii.] The morphism space $\hom_{\fuk^{d}(M)}(L_0,L_1)$ between $L_0,
   L_1 \in \mathcal{L}^{\ast}_{d}(M)$ is taken as
   in~\cite{Se:book-fukaya-categ} to be the Floer complex $CF(L_0,
   L_1; H_{L_0, L_1}, J_{L_0, L_1})$. However, there are two
   differences concerning the $A_{\infty}$ operations. First,
   unlike~\cite{Se:book-fukaya-categ} we work with homology rather
   than cohomology. Thus our Floer differential and higher composition
   maps $\mu_k$ differ from~\cite{Se:book-fukaya-categ} in the
   following way. If $\gamma_1, \ldots, \gamma_k$ are Hamiltonian
   chords, $\gamma_i \in CF(L_{i-1}, L_i)$, then $\mu_k(\gamma_1,
   \ldots, \gamma_k) \in CF(L_0, L_k)$ counts perturbed holomorphic
   disks $u: D \setminus \{z_1, \ldots, z_{k+1}\} \longrightarrow M$
   with negative punctures at $z_i \in \partial D$, $i=1, \ldots, k$
   (corresponding to in-going strip like ends) asymptotically
   emanating from the chords $\gamma_1, \ldots, \gamma_k$, and one
   positive puncture at $z_{k+1} \in \partial D$ corresponding to the
   output chord counted by $\mu_k(\gamma_1, \ldots, \gamma_k)$. In
   contrast, in~\cite{Se:book-fukaya-categ} the punctures $z_1,
   \ldots, z_k$ are positive and $z_0$ is negative. As we work in an
   ungraded framework, our homological conventions have no effect on
   grading. The second difference is that we place the punctured points $z_1,
   z_1, \ldots, z_{k+1} \in \partial D$ in clockwise order, whereas
   in~\cite{Se:book-fukaya-categ} the ordering is counterclockwise.
   (Also note that we number the punctures with indices $1, \ldots,
   k+1$ rather than $0, \ldots, k$.) In our case, the arc connecting
   $z_i$ to $z_{i+1}$ (clockwise oriented) is mapped by $u$ to $L_i$,
   $i=1, \ldots, k$, and the arc connecting $z_{k+1}$ to $z_1$ is
   mapped by $u$ to $L_0$.
  \item[iii.] In the definition of the Floer data $(H_{L_{0}, L_{1}},
   J_{L_{0}, L_{1}})$ for each pair of Lagrangians $L_{0}, L_{1}$
   (see~\cite{Se:book-fukaya-categ}~Chapter 9,~(9j)) we add the
   following requirement. Write $J_{L_0, L_1} = \{J(t)\}_{t \in
     [0,1]}$. Let $\mathcal{M}_{1}(L_{i};\alpha, J(i))$, $i=0,1$, be
   the moduli space of $J(i)$-holomorphic disks with boundary on
   $L_{i}$ belonging to the homotopy class $\alpha\in
   \pi_{2}(M,L_{i})$ and with one marked point on the boundary. Let
   $ev_1^{i}:\mathcal{M}_{1}(L_{i};\alpha, J(i)) \to L_{i}$ be the
   evaluation at the marked point. Recall that the generators of the
   Floer complex $CF(L_{0},L_{1}; H_{L_{0}, L_{1}}, J_{L_{0}, L_{1}}
   )$ are Hamiltonian chords $\gamma:[0,1]\to M$ of $H_{L_{0},L_{1}}$
   with $\gamma(i)\in L_{i}$, $i=0,1$.  We require that
   $J_{L_{0},L_{1}}$ be so that for all $\alpha\in \pi_{2}(M,L_{i})$
   with $\mu(\alpha)=2$ and all chords $\gamma$, the points
   $\gamma(i)$, $i=0,1$, are regular values of the evaluation map
   $ev_1^{i}$. A generic choice of Floer datum $(H_{L_{0}, L_{1}},
   J_{L_{0}, L_{1}})$ satisfies this constraint.
  \item[iv.] The definition of the $\mu_{k}$ composition maps of the
   $A_{\infty}$-category is given in terms of counts of maps from
   punctured disks with boundary conditions along a sequence of
   Lagrangians $L_{0}, \ldots, L_{k}\in \mathcal{L}^{\ast}_{d}(M)$
   that satisfy perturbed Cauchy-Riemann equations. In the
   verification of the relations among the $\mu_{k}$'s appear only
   moduli spaces of such disks of dimension $1$ or less. The condition
   $N_{L}\geq 2$ implies that, in the moduli spaces involved in the
   construction of the $\mu_{k}$'s and in verifying the $A_{\infty}$
   relations, no bubbling of disks or of spheres is possible except
   for one case: moduli spaces of Floer strips of Maslov index $2$
   with the $-\infty$ end coinciding with the $+\infty$ end.  In this
   case, assuming that the side Lagrangians are $L_{0}$ and $L_{1}$,
   there are two types of ``bubbled'' configurations: a disk of Maslov
   $2$ with boundary on $L_{0}$ that passes through the start of a
   Hamiltonian chord $\gamma$ or a similar disk with boundary on
   $L_{1}$ that passes through the end of $\gamma$. In both cases, one
   should view the configuration as a pair consisting of a degenerate
   Floer strip $u$ concentrated on $\gamma$ (i.e. $u(s,t) =
   \gamma(t)$) together with a bubbled holomorphic disk, with boundary
   either on $L_0$ or on $L_1$.  The fact that the maps $ev_1^{0}$ and $ev_{1}^{1}$
   above are of the same degree $d=d_{L_{0}}=d_{L_{1}}$ implies that the
   exceptional ``bubble'' configurations discussed above cancel out
   algebraically so that $\mu_{1}$ remains a differential.
\end{itemize}
More details on our specific conventions appear in~\S\ref{sec:FukCob}
where part of the construction of the Fukaya category is described in
more detail (and in a more general situation).

Once the geometric constructions above are accomplished this leads to
an $A_{\infty}$-category which is homologically unital. We denote this
$A_{\infty}$-category  by $\fuk^d(M)$.

Of course, $\fuk^d(M)$ depends on many choices of auxiliary structures
(e.g. the perturbation data etc.). Thus we have here in fact a family
of $A_{\infty}$-categories, parametrized by a huge collection of
choices of data. However, any two such categories are quasi-isomorphic
by a quasi-isomorphism which is canonical in homology. In particular
the associated derived categories are equivalent. In the next
subsection we will briefly discuss these equivalences
following~\cite{Se:book-fukaya-categ}. The construction will be
repeated in more detail later on in~\S\ref{sb:inv-fuk-2} when we
discuss the same issues for the Fukaya category of cobordisms.

\subsubsection{Invariance properties of $\fuk^d(M)$}
\label{sbsb:inv-fuk-1}

Here we summarize the construction from Chapter~10
of~\cite{Se:book-fukaya-categ}, where more details and proofs can be
found. See also~\S\ref{sbsb:families-A-infty}.

The Fukaya category constructed above depends on a choice of auxiliary
structures such as a choice of strip-like ends, Floer and perturbation
data etc. We denote by $\mathcal{I}$ the collection of all admissible
such choices. For every $i \in \mathcal{I}$ we denote by $\fuk^d(M;i)$
the corresponding Fukaya category. As explained
in~\cite{Se:book-fukaya-categ} one can construct one big
$A_{\infty}$-category $\fuk^d(M)^{\textnormal{tot}}$ together with a
family of full and faithful embeddings $\fuk^d(M;i) \to
\fuk^d(M)^{\textnormal{tot}}$, $i \in \mathcal{I}$. The outcome is
that the family $\fuk^d(M;i)$, $i \in \mathcal{I}$, becomes a coherent
system of $A_{\infty}$-categories. Moreover, in this case the
comparison functors $\mathcal{F}^{i_1,i_0}: \fuk^d(M;i_0) \to
\fuk^d(M;i_1)$ are in fact quasi-isomorphisms acting as identity on
objects and their corresponding homology functors $F^{i_1, i_0}:
H(\fuk^d(M;i_0)) \to H(\fuk^d(M;i_1))$ are canonical.

We will go into more details of this type of construction
in~\S\ref{sb:inv-fuk-2} when dealing with comparison between different
Fukaya categories of cobordisms.

In view of the above we denote by abuse of notation any of the
categories above $\fuk^d(M;i)$ by $\fuk^d(M)$ omitting the choice of
structures $i$ from the notation.

\subsubsection{The derived Fukaya category} \label{sbsb:derived-fuk}

We continue to use here the notation from~\S\ref{sbsb:inv-fuk-1}. 

We denote by $D\fuk^{d}(M;i)$ the derived category associated to
$\fuk^{d}(M;i)$, $i \in \mathcal{I}$, following the construction
recalled in~\S\ref{sbsb:triang-derived}.  Namely, we take the
triangulated closure $\fuk^{d}(M;i)^{\wedge}$ inside
$mod(\fuk^{d}(M;i))$ under the Yoneda image of $\fuk^{d}(M;i)$. The
derived Fukaya category $D\fuk^{d}(M;i)$ is now obtained from
$\fuk^{d}(M;i)^{\wedge}$ by replacing the morphisms with their values
in homology.  These are triangulated categories in the usual
sense (no longer just $A_{\infty}$ ones). Moreover, the functors
$F^{i_1, i_0}$ from~\S\ref{sbsb:inv-fuk-1} extend to canonical
isomorphisms of the derived categories $F^{i_1, i_0}:D\fuk^d(M;i_0)
\to D\fuk^d(M;i_1)$. (See also~\S\ref{sbsb:families-A-infty}.) The
outcome is a strict system of categories $D\fuk^d(M;i)$, $i \in
\mathcal{I}$, in the sense of~\cite{Se:book-fukaya-categ}.

In view of the above we denote all these derived Fukaya categories by
$D\fuk^d(M)$, omitting the choice of auxiliary structures $i$.

\begin{rem} We emphasize that our variant of the derived Fukaya category does not
involve completion with respect to idempotents.
\end{rem}

\subsection{Cone decompositions over a triangulated category}
\label{subsec:cones}
The purpose of the construction discussed in this section (again
following \cite{Bi-Co:cob1}) is to parametrize the various ways to
decompose an object by iterated exact triangles inside a given
triangulated category.  This will be applied later in the paper to the
category $D\fuk^{d}(M)$.

We recall~\cite{Weibel:book-hom-alg} that a triangulated category
$\mathcal{C}$ - that we fix from now on - is an additive category
together with a translation automorphism $T:\mathcal{C} \to
\mathcal{C}$ and a class of triangles called {\em exact triangles}
$$ T^{-1}X\stackrel{u}{\longrightarrow}X\stackrel{v}{\longrightarrow}Y
\stackrel{w}{\longrightarrow}Z$$ that
satisfy a number of axioms due to Verdier and to Puppe (see
e.g.~\cite{Weibel:book-hom-alg}).
  
A cone decomposition of length $k$ of an object $A\in \mathcal{C}$ is
a sequence of exact triangles:
$$T^{-1}X_{i}\stackrel{u_{i}}{\longrightarrow}Y_{i}
\stackrel{v_{i}}{\longrightarrow}Y_{i+1}\stackrel{w_{i}}{\longrightarrow}
X_{i}$$ with $1\leq i\leq k$, $Y_{k+1}=A$, $Y_{1}=0$. (Note that
$Y_2\cong X_1$.) Thus $A$ is obtained in $k$ steps from $Y_{1}=0$.  To
such a cone decomposition we associate the family $l(A)=(X_{1},
X_{2},\dots , X_{k})$ and we call it the {\em linearization} of the
cone decomposition. This definition is an abstract form of the
familiar iterated cone construction in case $\mathcal{C}$ is the
homotopy category of chain complexes. In that case $T$ is the
suspension functor $TX = X[-1]$ and the cone decomposition simply
means that each chain complex $Y_{i+1}$ is obtained from $Y_i$ as the
mapping cone of a morphism coming from some chain complex, in other
words $Y_{i+1} = \textnormal{cone}(X_{i}[1]
\stackrel{u_i}{\longrightarrow} Y_i)$ for every $i$, and $Y_1=0$,
$Y_{k+1}=A$.  Two cone decompositions $\{T^{-1}X_{i}\to Y_{i}\to
Y_{i+1}\to X_{i}\}_{1\leq i\leq k}$ and $\{T^{-1}X_{i}\to Y'_{i}\to
Y'_{i+1}\to X_{i}\}_{1\leq i\leq k}$ of two different objects $A$ and,
respectively, $A'$, are said equivalent if there are isomorphisms
$I_{i}:Y_{i}\to Y'_{i}$, $1\leq i \leq k+1$, making the squares in the
diagram below commutative:

\begin{equation} \label{eq:cones-equiv}
   \begin{aligned}
      \xymatrix@-2pt{ T^{-1}X_{i}\ar[d]^{id}\ar[r]&
        Y_{i}\ar[d]^{I_{i}}\ar[r] & Y_{i+1}\ar[d]^{I_{i+1}}
        \ar[r] & X_{i}\ar[d]^{id} \\
        T^{-1}X_{i}\ar[r] & Y'_{i}\ar[r] & Y'_{i+1}\ar[r]& X_{i} }
   \end{aligned}
\end{equation}
Such a family of isomorphisms is called an isomorphism of
cone-decompositions. We say that an isomorphism $I: A\to A'$ extends
to an isomorphism of the respective cone-decompositions if there is
\pbred{an} isomorphism of cone-decomposition\pbred{s} with the last
term $I_{k+1}=I$.  In particular, two equivalent cone decompositions
have the same linearization.

We will now define a category $T^{S} \mathcal{C}$ called the {\em
  category of (stable) triangle (or cone) resolutions over}
$\mathcal{C}$.  The objects in this category are finite, ordered
families $(x_{1}, x_{2},\ldots, x_{k})$ of objects $x_{i}\in
\mathcal{O}b(\mathcal{C})$.

We will first define the morphisms in $T^{S} \mathcal{C}$ with  domain being
a family formed by a single object $x \in \mathcal{O}b(\mathcal{C})$
and target $(y_1, \ldots, y_q)$, $y_{i}\in \mathcal{O}b(\mathcal{C})$.
For this, consider triples $(\phi, a, \eta)$, where $a \in
\mathcal{O}b(\mathcal{C})$, $\phi: x \to T^s a$ is an isomorphism (in
$\mathcal{C}$) for some index $s$ and $\eta$ is a cone decomposition
of the object $a$ with linearization $(T^{s_1}y_1,
T^{s_{2}}y_{2},\ldots, T^{s_{q-1}}y_{q-1}, y_q)$ for some family of
indices $s_1, \ldots, s_{q-1}$. Below we will also sometimes denote by $s_{q}$
 the shift index attached to the last element $y_{q}$ with the
understanding that $s_{q}=0$.  A morphism $\Psi: x \longrightarrow
(y_1, \ldots, y_q)$ is an equivalence class of triples $(\phi,a,\eta)$
as before up to the equivalence relation given by $(\phi,a,\eta)\sim
(\phi',a',\eta')$ if there is an isomorphism $I_{a}:a\to a'$ which
extends to an isomorphism of the cone decompositions $\eta$ and
$\eta'$ and so that $\phi'=(T^{s}I_{a})\circ \phi$.

The identity morphism $\id : x \to x$, where $x \in
\mathcal{O}b(\mathcal{C})$, is given by $(\id, x, \eta_x)$, where
$\eta_x$ is the trivial cone decomposition of $x$ given by the obvious
exact triangle $x \stackrel{\id}{\longrightarrow} x \longrightarrow 0
\longrightarrow Tx$.

We now define the morphisms between two general objects.  
A morphism
$$\Phi\in\mor_{T^{S}\mathcal{C}}((x_{1},\ldots x_{m}),
(y_{1},\ldots, y_{n}))$$ is a sum $\Phi =\Psi_{1}\oplus \cdots \oplus
\Psi_{m}$ where $\Psi_{j}\in \mor_{T^{S}\mathcal{C}}(x_{j},
(y_{\alpha(j)},\ldots, y_{\alpha(j)+\nu(j)}))$, and $\alpha(1)=1$,
$\alpha(j+1) = \alpha(j) + \nu(j) + 1$, $\alpha(m) + \nu(m) = n$. The
sum $\oplus$ means here the obvious concatenation of morphisms. With
this definition this category is strict monoidal, the unit element
being given by the void family.

Next we make explicit the composition of the morphisms in
$T^{S}\mathcal{C}$.  We consider first the case of two morphisms
$\Phi'$, $\Phi$,
\begin{equation} \label{eq:Phi-circ-Phi'-1}
   \begin{aligned}
      & \Phi': x \longrightarrow
      (y_{1},\ldots, y_{k}), \quad \Phi'=(\phi',a',\eta'), \\
      & \Phi:
      (y_1, \ldots, y_{h-1}, y_h, y_{h+1}, \ldots, y_k) \longrightarrow
      (y_1, \ldots, y_{h-1}, z_1, \ldots, z_l, y_{h+1}, \ldots, y_k),
   \end{aligned}
\end{equation}
where $$\Phi = \id \oplus \cdots \oplus\id \oplus \Psi_h \oplus \id \oplus \cdots \oplus \id, \quad
\Psi_h : y_h \longrightarrow (z_1, \ldots, z_l), \quad \Psi_h = (\phi,
a, \eta).$$

We will now define $\Phi'' = \Phi \circ \Phi'$.  We will assume for
simplicity that the families of shifting degrees indices for both
$\Phi'$ and $\Phi$ are $0$ (the general case is a straightforward
generalization of the argument below). From the morphism $\Phi'$ we
get an isomorphism $\phi': x \to a'$ and a cone decomposition of $a'$
with linearization $(y_1, \ldots,y_k)$, i.e. objects $a'_1=0, a'_2,
a'_3, \ldots, a'_k, a'_{k+1}=a'$ and exact triangles:
\begin{equation} \label{eq:ex-tr-eta'} T^{-1}y_i \longrightarrow a'_i
   \longrightarrow a'_{i+1} \longrightarrow y_i, \quad i=1, \ldots, k.
\end{equation}
For $i=h-1, h, h+1$ we have:
\begin{equation} \label{eq:ex-tr-yh} 
   \begin{aligned}
      & T^{-1}y_{h-1} \longrightarrow a'_{h-1}
      \longrightarrow a'_{h} \longrightarrow y_{h-1} \\
      & T^{-1}y_{h} \longrightarrow a'_{h}
      \longrightarrow a'_{h+1} \longrightarrow y_{h} \\
      & T^{-1}y_{h+1} \longrightarrow a'_{h+1} \longrightarrow
      a'_{h+2} \longrightarrow y_{h+1}.
   \end{aligned}
\end{equation}
Similarly, from the morphism $\Phi$ we obtain an isomorphism $\phi:y_h
\to a$, a sequence of objects $a_1=0, a_2, a_3, \ldots, a_l,
a_{l+1}=a$ and exact triangles:
\begin{equation} \label{eq:ex-tr-a_j} T^{-1}z_j \longrightarrow a_j
   \longrightarrow a_{j+1} \longrightarrow z_j, \quad j=1, \ldots, l.
\end{equation}
The composition $\Phi'' = \Phi \circ \Phi': x \longrightarrow (y_1,
\ldots, y_{h-1}, z_1, \ldots, z_l, y_{h+1}, \ldots, y_k)$ is defined
as the triple $(\phi'',a'',\eta'')$ given as follows.  First, there is
an isomorphism of the middle line in (\ref{eq:ex-tr-yh}) with the
exact triangle:
\begin{equation}\label{eq:triang-new}
   T^{-1}a\to a'_{h} \to a''_{h+1}\to a
\end{equation}
where the map $T^{-1}a\to a'_{h}$ is defined as $T^{-1}a
\stackrel{T^{-1}\phi^{-1}}{\longrightarrow} T^{-1}y_{h}\to a'_{h}$.
Using (\ref{eq:triang-new}) we construct a sequence of triangles
\begin{equation}\label{eq:traing-new2}
   T^{-1}y_{j}\to a''_{j}\to a''_{j+1}\to y_{j}
\end{equation}
that are isomorphic with the respective triangles in
(\ref{eq:ex-tr-eta'}) as follows: for $j<h$ these coincide with the
triangles in (\ref{eq:ex-tr-eta'}); for $j=h$ we use
(\ref{eq:triang-new}); finally for $j>h$ these triangles are
constructed inductively by using as the first map in each of these
triangles the composition $T^{-1}y_{j}\to a'_{j}\to a''_{j+1}$ where
$a'_{j}\to a''_{j}$ is the isomorphism constructed at the previous
stage. This map can be completed to an exact triangle by the axioms of
a triangulated category.  We then put $a''=a''_{k+1}$. This is endowed
with an isomorphism $\phi''_{0}:a'\to a''$ and we let
$\phi''=\phi''_{0}\circ \phi'$. It remains to describe the cone
decomposition $\eta''$.

We will construct below
new objects $a'_{h,1}, \ldots, a'_{h,l+1}$ with $a'_{h,1} = a'_h=a''_{h}$,
$a'_{h,l+1} = a''_{h+1}$ and exact triangles:
\begin{equation} \label{eq:triangles-a'_hq} T^{-1} z_q \longrightarrow
   a'_{h,q} \longrightarrow a'_{h,q+1} \longrightarrow z_q, \quad q=1,
   \ldots, l.
\end{equation}
With this at hand the cone decomposition $\eta''$ is defined by taking
the cone decomposition (\ref{eq:traing-new2}) and replacing the line
$i=h$ in it (i.e.  \eqref{eq:triang-new}) by the list of triangles
from~\eqref{eq:triangles-a'_hq}.

We now turn to the construction of the objects $a'_{h,q}$ and the
triangles~\eqref{eq:triangles-a'_hq}.  Given $1 \leq q \leq l$, let
$\beta_q: a_q \to a$ be the morphism obtained by successive
composition of the middle arrows of~\eqref{eq:ex-tr-a_j} for $j=q,
\ldots, l$.  Consider now the composition $\alpha_q = u'_h \circ
(T^{-1}\phi^{-1}) \beta_q$ of the following three morphisms:
$$\alpha_q \, : \, T^{-1}a_q \stackrel{\beta_q}{\longrightarrow}T^{-1}a 
\xrightarrow{T^{-1}\phi^{-1}} T^{-1}y_h
\stackrel{u'_h}{\longrightarrow} a'_h,$$ where the last arrow $u'_l$
here is the first arrow in the middle line of~\eqref{eq:ex-tr-yh}.

By the axioms of a triangulated category the morphism $\alpha_q$ can
be completed into an exact triangle, i.e. there exists an object
$a'_{h,q} \in \mathcal{O}b(\mathcal{C})$ and an exact triangle:
\begin{equation} \label{eq:ex-tr-alpha_q}
   T^{-1}a_q \stackrel{\alpha_q}{\longrightarrow} a'_h \longrightarrow 
   a'_{h,q} \longrightarrow a_q.
\end{equation}
By the octahedral axiom and standard results on triangulated
categories (see e.g.~\cite{Weibel:book-hom-alg}) the triangles
in~\eqref{eq:ex-tr-alpha_q} (for $q$ and $q+1$) and those
in~\eqref{eq:ex-tr-a_j} (for $j=q$ and after a shift) fit into the
following diagram:
\begin{equation}\label{eq:squares-diag}
   \begin{aligned}
      \xymatrix@-2pt{
        T^{-2}z_{q}\ar[r]\ar[d]& 0\ar[d]\ar[r]&
        T^{-1}z_{q}\ar[r]\ar[d]&T^{-1}z_{q}\ar[d]\\
        T^{-1}a_{q}\ar[r]^{\alpha_{q}}\ar[d]& 
        a'_{h}\ar[d]\ar[r]&a'_{h,q}\ar[d]\ar[r]&a_{q}\ar[d]\\
        T^{-1}a_{q+1}\ar[r]^{\alpha_{q+1}}\ar[d] &
        a'_{h}\ar[r]\ar[d]&a'_{h,q+1}\ar[r]\ar[d]&a_{q+1}\ar[d]\\
        T^{-1}z_{q}\ar[r] &0 \ar[r] \ar[r] &z_{q}\ar[r]&z_{q}}
   \end{aligned}
\end{equation}
in which all rows and columns are exact triangles.  (In fact, the
diagram is determined up to isomorphism by the upper left square, from
which the upper two triangles and left two triangles are extended. The
existence of the third column follows from octahedral axioms applied a
few times.)  Moreover, all squares in the diagram are commutative (up to a sign).

Note that the third column is precisely the triangle that we needed
in~\eqref{eq:triangles-a'_hq} in order to complete the construction of
the composition in~\eqref{eq:Phi-circ-Phi'-1}.

This definition is seen to immediately pass to equivalence classes and it 
remains to define the composition $\Phi \circ \Phi'$ of more
general morphisms than~\eqref{eq:Phi-circ-Phi'-1}.  The case when the
domain of $\Phi'$ consists of a tuple of objects in $\mathcal{C}$ and
$\Phi$ is as in~\eqref{eq:Phi-circ-Phi'-1} is an obvious
generalization of the preceding construction. Next, the case when the
rest of the components of $\Phi$ are not necessarily $\id$ (but rather
general cone decompositions too) is done by reducing to the case
discussed above by successive compositions. Namely, assume that
$\Phi': x \to (y_1, \ldots, y_k)$ and $\Phi = \Psi_1 + \cdots +
\Psi_k$ with $\Psi_j: y_j \to w_j$, where $x, w_j$ are tuples of objects in $\mathcal{C}$. We define
$$\Phi \circ \Phi' = (\Psi_1 + \id + \cdots + \id) \circ \cdots \circ 
(\id + \cdots + \id + \Psi_k) \circ \Phi',$$ noting that each step of
this composition is of the type already defined.

This completes the definition of composition of morphisms in the
category $T^S \mathcal{C}$. It is not hard to see that the this
composition of morphisms is associative - here it is important that we are
in fact using equivalence classes of cone-decompositions. 

To conclude this discussion we remark that there is a projection
functor
\begin{equation} \label{eq:proj} \mathcal{P}:T^{S}\mathcal{C}
   \longrightarrow \Sigma\mathcal{C}
\end{equation}
Here $\Sigma\mathcal{C}$ stands for the {\em stabilization category} of $\mathcal{C}$: 
$\Sigma\mathcal{C}$ has the same objects as
$\mathcal{C}$ and the morphisms in $\Sigma \mathcal{C}$ from $a$ to
$b\in \mathcal{O}b(\mathcal{C})$ are morphisms in $\mathcal{C}$ of the
form $a\to T^{s} b$ for some integer $s$.

The definition of $\mathcal{P}$ is as follows:
$\mathcal{P}(x_{1},\ldots x_{k})=x_{k}$ and on morphisms it associates to
$\Phi\in\mor_{T^{S}\mathcal{C}}(x, (x_{1},\ldots, x_{k}))$,
$\Phi=(\phi,a,\eta)$, the composition:
$$\mathcal{P}(\Phi): x\stackrel{\phi}{\longrightarrow} T^{s}a 
\stackrel{w_{k}}{\longrightarrow} T^{s}x_{k}$$ with $w_{k}:a \to x_{k}$
defined by the last exact triangle in the cone decomposition $\eta$ of
$a$,
$$T^{-1}x_{k}\longrightarrow a_{k}
\longrightarrow a\stackrel{w_{k}}{\longrightarrow} x_{k}~.~$$ A
straightforward verification shows that $\mathcal{P}$ is indeed a
functor. Moreover, we see that any isomorphism $\phi: x\to a$ in
$\mathcal{C}$ is in the image of $\mathcal{P}$: if
$\overline{\phi}=(\phi, a,\eta_{a} ): x\to a$ is the morphism in
$T^{S}(\mathcal{C})$ defined by putting $\eta_{a}$ to be the cone
decomposition formed by a single exact triangle $T^{-1}a\to 0 \to a
\to a$, then $\mathcal{P}(\overline{\phi})=\phi$.


\section{The Fukaya category of cobordisms} \label{sec:FukCob}

The purpose of this section is to construct a Fukaya type category
$\fuk^{d}_{cob}(\C\times M)$ whose objects are cobordisms $V\subset
[0,1]\times\R \times M$ as defined in Definition~\ref{def:Lcobordism}
that satisfy the condition~\eqref{eq:Hlgy-vanishes} and are uniformly
monotone with $d_{V}=d$ and with the same monotonicity constant $\rho$
fixed before.  Most of the time we identify such a cobordism $V$ and
its $\R$-extension $\overline{V}\subset \C\times R$.  We denote the
collection of these cobordisms by $\mathcal{CL}_{d}(\mathbb{C} \times
M)\subset\mathcal{L}_{d}(\C\times M)$.  We follow in this construction
Seidel's scheme from~\cite{Se:book-fukaya-categ} -- as recalled
in~\S\ref{sec:fuk-M} -- but with some significant modifications that
are necessary to deal with compactness issues due to the fact that our
Lagrangians are non-compact.

\subsection{Strip-like ends and an associated family of transition
  functions} \label{subsec:strip-ends}
We first recall the notion of a consistent choice of strip-like ends
from~\cite{Se:book-fukaya-categ}. Fix $k \geq 2$. Let
$\mathrm{Conf}_{k+1}(\partial D)$ be the space of configurations of
$(k+1)$ distinct points $(z_1, \ldots, z_{k+1})$ on $\partial D$ that
are ordered clockwise. Denote by $Aut(D) \cong PLS(2,\mathbb{R})$ the
group of holomorphic automorphisms of the disk $D$. Put
$\mathcal{R}^{k+1}=\mathrm{Conf}_{k+1}(\partial D)/Aut(D)$.  Next, put
$$\widehat{\mathcal{S}}^{k+1}=\bigl(\mathrm{Conf}_{k+1}(\partial
D)\times D\bigr)/Aut(D).$$ The projection $\widehat{\mathcal{S}}^{k+1}
\to \mathcal{R}^{k+1}$ has the following sections $\zeta_i[z_1,
\ldots, z_{k+1}] = [(z_1, \ldots, z_{k+1}), z_i]$, $i=1, \ldots, k+1$.
Put $\mathcal{S}^{k+1} = \widehat{\mathcal{S}}^{k+1} \setminus
\bigcup_{i=1}^{k+1} \zeta_i(\mathcal{R}^{k+1}).$ The fibre bundle
$$\mathcal{S}^{k+1} \to \mathcal{R}^{k+1}$$ is called a universal family of 
$(k+1)$-pointed disks. Its fibres $S_r$, $r \in \mathcal{R}^{k+1}$,
are called $(k+1)$-pointed (or punctured) disks.

Let $Z^{+}= [0,\infty)\times [0,1]$, $Z^{-}=(-\infty, 0]\times [0,1]$
be the two infinite semi-strips. Let $S$ be a $(k+1)$ pointed disk
with punctures at $(z_1, \ldots, z_{k+1})$. A choice of strip-like
ends for $S$ is a collection of embeddings: $\epsilon^{S}_{i}:Z^{-}\to
S$, $ 1\leq i \leq k$, $\epsilon^{S}_{k+1}:Z^{+}\to S$ that are proper
and holomorphic and
\begin{equation*} \label{eq:strip-like-ends}
   \begin{aligned}
      (\epsilon^{S}_{i})^{-1}(\partial S) & = (-\infty, 0] \times \{0,
      1\}, \quad \lim_{s \to -\infty}
      \epsilon^{S}_{i}(s, t) = z_i, \quad \forall \, 1 \leq i \leq k, \\
      (\epsilon^{S}_{k+1})^{-1}(\partial S) & = [0, \infty) \times
      \{0, 1\}, \quad \lim_{s \to \infty} \epsilon^{S}_{k+1}(s, t) =
      z_{k+1}.
   \end{aligned}
\end{equation*}
Moreover, we require the $\epsilon^{S}_i$, $i=1, \ldots, k+1$, to have
pairwise disjoint images. A universal choice of strip-like ends for
$\mathcal{S}^{k+1}\to \mathcal{R}^{k+1}$ is a choice of $k+1$ proper
embeddings $\epsilon^{\mathcal{S}}_i: \mathcal{R}^{k+1} \times Z^- \to
\mathcal{S}^{k+1}$, $i=1, \ldots, k$,\; $\epsilon^{\mathcal{S}}_{k+1}:
\mathcal{R}^{k+1} \times Z^+ \to \mathcal{S}^{k+1}$ such that for every $r
\in \mathcal{R}^{k+1}$ the restrictions ${\epsilon^{\mathcal{S}}_i}|_{r
  \times Z^{\pm}}$ consists of a choice of strip-like ends for $S_r$.
See~\cite{Se:book-fukaya-categ} for more details.

We extend the definition for the case $k=1$, by setting
$\mathcal{R}^2 = \textnormal{pt}$ and $\mathcal{S}^{2} = D \setminus
\{-1,1\}$. We endow $D \setminus \{-1, 1\}$ by strip-like ends with 
identifying it holomorphically with the strip $\mathbb{R} \times
[0,1]$ endowed with its standard complex structure.

Pointed disks endowed with strip-like ends can be glued in the
following way. Let $\rho \in (0,\infty)$ and let $S'$, $S''$ be two
pointed disks with punctures at $(z'_1, \ldots, z'_{k'+1})$ and
$(z''_1, \ldots, z''_{k''+1})$ respectively. Fix $q \in \{1, \ldots,
k''\}$. Define a new $(k'+k'')$-pointed surface $S' \#_{\rho} S''$ by
taking the disjoint union
$$S' \setminus \epsilon^{S'}_{k'+1}([\rho, \infty) \times [0,1]) \coprod
S'' \setminus \epsilon^{S''}_{q}((-\infty, -\rho] \times [0,1])$$ and
identifying $\epsilon^{S'}_{k'+1}(s, t) \sim
\epsilon^{S''}_{q}(s-\rho, t)$ for $(s,t) \in (0,\rho)\times[0,1]$.
The family of glued surfaces $S' \#_{\rho} S''$ inherits $k'+k''$
punctures on the boundary from $S'$ and $S''$. The disks $S'$, $S''$
induce in an obvious way a complex structure on each $S' \#_{\rho}
S''$, $\rho>0$, so that $S'\#_{\rho} S''$ is biholomorphic to a
$(k'+k'')$-pointed disk. Thus we can holomorphically identify each $S'
\#_{\rho} S''$, $\rho>0$, in a unique way, with a fibre $S_{r(\rho)}$
of the universal family $\mathcal{S}^{k'+k''} \to
\mathcal{R}^{k'+k''}$.  Finally, there are strip-like ends on $S'
\#_{\rho} S''$ which are induced from those of $S'$ and $S''$ by
inclusion in the obvious way.

The space $\mathcal{R}^{k+1}$ has a natural compactification
$\overline{\mathcal{R}}^{k+1}$ described by parametrizing the elements
of $\overline{\mathcal{R}}^{k+1}\backslash \mathcal{R}^{k+1}$ by
trees~\cite{Se:book-fukaya-categ}. The family $\mathcal{S}^{k+1} \to
\mathcal{R}^{k+1}$ admits a partial compactification
$\overline{\mathcal{S}}^{k+1} \to \overline{\mathcal{R}}^{k+1}$ which
can be endowed with a smooth structure. Moreover, the fixed choice of
universal strip-like ends for $\mathcal{S}^{k+1}\to \mathcal{R}^{k+1}$
admits an extension to $\overline{\mathcal{S}}^{k+1}\to
\overline{\mathcal{R}}^{k+1}$. Further, these choices of universal
strip-like ends for the spaces $\mathcal{R}^{k+1}$ for different $k$'s
can be made in a way that is consistent with these compactifications
(see Lemma~9.3 in~\cite{Se:book-fukaya-categ}). 

For every $k \geq 2,$ we fix a Riemannian metric $\rho_{k+1}$ on
$\overline{\mathcal{S}}^{k+1}$ so that this metric reduces on each
surface $S_r$ to a metric compatible with all the splitting/gluing
operations - it is consistent with respect to these operations in the
same sense as the choice of strip-like ends.  Moreover we require the
metrics $\rho_{k+1}$ to have the following property: for every $k \geq
2$ there exists a constant $A_{k+1}$ such that
\begin{equation} \label{eq:metric-rho}
   \textnormal{length}_{\rho_{k+1}}(\partial S_r) \leq A_{k+1}, \;
   \forall r \in \overline{\mathcal{R}}^{k+1}.
\end{equation}
This requirement will be useful later for energy estimates as in
Lemma~\ref{lem:energy-bound1} below.

Our construction requires an additional auxiliary structure which can
be defined once a choice of universal strip-like ends is fixed as
above. This structure is a smooth function $\mathbf{a}:
\mathcal{S}^{k+1} \to [0,1]$ with some properties which we describe
now.  We start with $k=1$. In this case $\mathcal{S}^2 = D \setminus
\{-1, 1\} \cong \mathbb{R} \times [0,1]$ and we define
$\mathbf{a}(s,t) = t$, where $(s,t)\in \R\times [0,1]$. To describe
$\mathbf{a}$ for $k \geq 2$ write $a_r := \mathbf{a}|_{S_r}$, $r \in
\mathcal{R}^{k+1}$. We require the functions $a_r$ to satisfy the
following for every $r \in \mathcal{R}^{k+1}$ - see
Figure~\ref{fig:transition}:
\begin{figure}[htbp]
   \begin{center}
      \epsfig{file=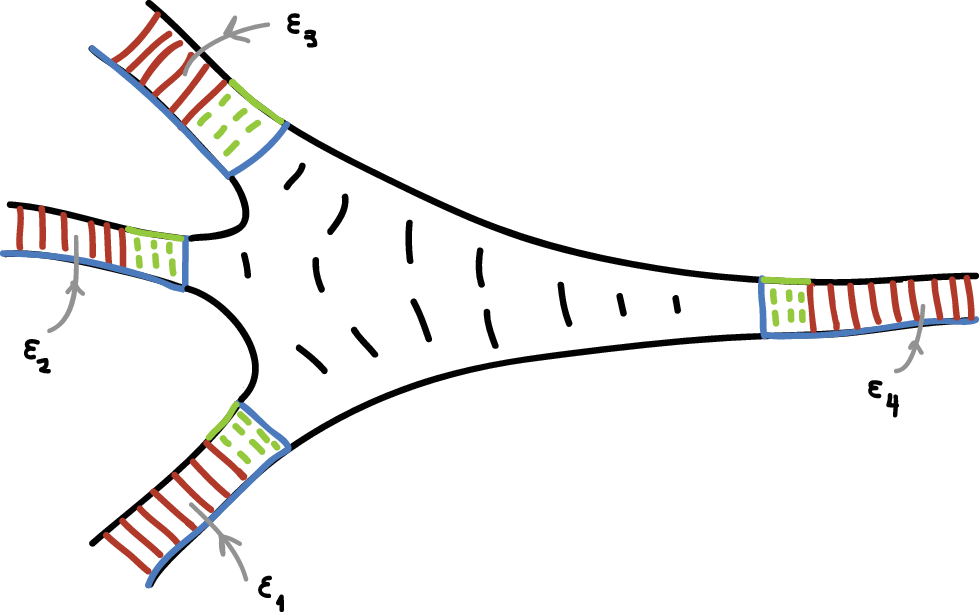, width=0.8\linewidth}
   \end{center}
   \caption{\label{fig:transition} The constraints imposed on a
     transition function for a domain with three entries and one exit:
     in the red region the function $a$ equals $(s,t)\to t$; along the
     blue arcs the function $a$ vanishes; the green region is a
     transition region. There  are no additional constraints in the
     black region.}
\end{figure}

\begin{itemize}
  \item[i.] For each entry strip-like end $\epsilon_{i}:Z^{-}\to
   S_{r}$, $1\leq i\leq k$, we have:
   \begin{itemize}
     \item[a.]$a_{r} \circ \epsilon_{i}(s,t)=t$, $\forall \ (s,t)\in
      (-\infty, -1]\times [0,1]$.
     \item[b.] $\frac{\partial }{\partial
        s}(a_{r}\circ\epsilon_{i})(s,1)\leq 0$ for $s\in [-1,0]$.
     \item[c.] $a_{r}\circ \epsilon_{i}(s,t)=0$ for $(s,t)\in
      ((-\infty, 0]\times \{0\})\cup (\{0\}\times [0,1])$.
   \end{itemize}
  \item[ii.] For the exit strip-like end $\epsilon_{k+1}:Z^{+}\to
   S_{r}$ we have:
   \begin{itemize}
     \item[a'.]$a_{r} \circ \epsilon_{k+1}(s,t)=t$, $\forall \
      (s,t)\in [1, \infty)\times [0,1]$.
     \item[b'.] $\frac{\partial }{\partial
        s}(a_{r}\circ\epsilon_{k+1})(s,1)\geq 0$ for $s\in [0,1]$.
     \item[c'.] $a_{r}\circ \epsilon_{k+1}(s,t)=0$ for $(s,t)\in
      ([0,+\infty)\times \{0\}) \cup (\{0\}\times [0,1])$.
   \end{itemize}
\end{itemize} 
The total function $\mathbf{a}: \mathcal{S}^{k+1} \to [0,1]$ will be
called a global transition function. When it is important to emphasize
its dependence on $k$ we will denote it also by $\mathbf{a}^{k+1}$.

Because the strip-like ends are picked consistently, the functions
$\mathbf{a}^{k+1}$ can also be picked consistently for different
values of $k$ . This means that $\mathbf{a}$ extends smoothly to
$\overline{\mathcal{S}}^{k+1}$ and moreover along the boundary
$\partial \overline{\mathcal{S}}^{k+1}$ it coincides with the
corresponding pairs (or tuples) of functions $\mathbf{a}^{k'+1}:
\mathcal{S}^{k'+1} \to [0,1]$, $\mathbf{a}^{k''+1}:\mathcal{S}^{k'+1}
\to [0,1]$ with $k'+k'' = k+1$, corresponding to trees of split
pointed disks. The proof of this follows the same principle as that
used in~\cite{Se:book-fukaya-categ} to show consistency for the
strip-like ends. The key point is compatibility with gluing/splitting.
In other words, given two pointed disks $S'$ and $S''$, with $k'$ and
$k''$ punctures respectively, consider the family of glued pointed
disks $S(\rho)= S' \#_{\rho} S''$, where the gluing is done between
the positive puncture of $S'$ and the $q$'th negative puncture of
$S''$.  We need to define the transition function $a^{(\rho)}: S(\rho)
\to [0,1]$ for large $\rho$, given the two transition functions $a'$
and $a''$ associated to $S'$ and $S''$. For this note that $a'$ and
$a''$ satisfy $a'\circ \epsilon^{S'}_{k'+1}(s,t) = t$ for $s>1$ and
$a'' \circ \epsilon^{S''}_q (s,t) = t$ for $s<-1$. Thus the functions
$a'$ and $a''$ glue together to form a new function $a^{(\rho)}:
S(\rho) \to [0,1]$ which restricts to $a'$ and $a''$ when $\rho \to
\infty$.  This procedure defines the function $\mathbf{a}$ near the
boundary of $\overline{\mathcal{S}}^{k'+k''}$, and one can extend it
to the rest of $\mathcal{S}^{k'+k''}$ keeping the requirements
i(a)-i(c) and ii(a')-ii(c') above. It follows that global transition
functions exist and extend to $\mathbf{a}:
\overline{\mathcal{S}}^{k+1}\to [0,1]$ for all $k \geq 1$, in a way
which is consistent with splitting/gluing.  Note that given a fixed
consistent choice of universal strip-like ends, the choices of
associated global transition functions form a contractible set.

The following result will be of use later in the paper.

\begin{lem} \label{lem:der-bound} Let $\mathbf{a}^{j+1}:
   \overline{\mathcal{S}}^{j+1}\to [0,1]$, $1 \leq j \leq k$, be a
   choice of consistent global transition functions. Then there exists
   a constant $C^{k}_{\mathbf{a}} > 0$ (which depends on the metric
   $\rho_{k+1}$) so that for any $r\in \overline{\mathcal{R}}^{k+1}$
   and any tangent vector $\xi \in T(\partial S_{r})$ we have:
   $$|da^{k+1}_{r}(\xi)|\leq  C^{k}_{\mathbf{a}}\ |\xi|_{\rho_{k+1}}~.~$$
\end{lem}

The proof of this lemma is straightforward and is based on the
following facts:
\begin{itemize}
  \item[-] The metrics $\rho_{k+1}$ are compatible with
   gluing/splitting. Note that the property~\eqref{eq:metric-rho} of
   the metrics $\rho_{k+1}$ plays no role here.
  \item[-] For $x\in \bigcup_{i=1}^k\epsilon_{i}((-\infty, -1]\times
   \{0,1\}) \cup \epsilon_{k+1}([1,\infty)\times \{0,1\})$ and $\xi\in
   T_x(\partial S_{r})$ we have $da_{r}(\xi)=0$.
  \item[-] $\overline{\mathcal{R}}^{k+1}$ is compact so that if
   $U\subset \overline{\mathcal{S}}^{k+1}$ is an arbitrarily small
   neighborhood of the union of all the punctures - both those
   associated to the strata in $\partial \overline{\mathcal{R}}^{k+1}$
   as well as those corresponding to the entries and the exit - then
   $\overline{\mathcal{S}}^{k+1}\backslash U$ is compact.
  \item[-] The strip-like ends as well as the global transition
   function $\mathbf{a}$ are compatible with respect to gluing so that
   if the neighborhood $U$ is sufficiently small, then
   $da_{r}(\xi)=0$ for all $x\in \partial S_{r} \cap U$, $\xi\in
   T_x(\partial S_{r})$.
\end{itemize}
We leave the details of the argument to the reader.

\subsection{$J$-holomorphic curves} \label{subsec:equation}
This section describes the particular perturbed $J$-holomorphic curves
that will be used in the definition of the higher compositions in the
category $\fuk^{d}_{cob}(\C\times M)$. To simplify notation we write
$\widetilde{M} = \mathbb{C} \times M$ and endow it with the symplectic
form $\widetilde{\omega} = \omega_{\mathbb{R}^2} \oplus \omega$, where
$\omega$ is the symplectic structure of $M$. We denote by $\pi:
\widetilde{M} \to \mathbb{C}$ the projection.

We fix a function $h:\R^{2}\to \R$ so that (see also
Figure~\ref{fig:kinks}):
\begin{itemize}
  \item[i.] The support of $h$ is contained in the union of the sets
   $$W_{i}^{+}= [2,\infty)\times [i-\epsilon, i+\epsilon] 
   \quad \textnormal{and } \; W_{i}^{-}= (-\infty, -1]\times
   [i-\epsilon, i+\epsilon], \; i \in \mathbb{Z}~,~$$ where $0<
   \epsilon<1/4$.
  \item[ii.] The restriction of $h$ to each set
   $T_{i}^{+}=[2,\infty)\times [i-\epsilon/2, i+\epsilon/2]$ and
   $T_{i}^{-}=(-\infty, -1] \times [i-\epsilon/2, i+\epsilon/2]$ is
   respectively of the form $h(x,y)=h_{\pm}(x)$, where the smooth
   functions $h_{\pm}$ satisfy:
   \begin{itemize}
     \item[a.] $h_{-}:(-\infty, -1]\to \R$ has a single critical point
      in $(-\infty,-1]$ at $-\frac{3}{2}$ and this point is a
      non-degenerate local maximum. Moreover, for all $x\in (-\infty,
      -2)$, we have $h_{-}(x)=\alpha^{-}x+\beta^{-}$ for some
      constants $\alpha^{-},\beta^{-}\in\R$ with $\alpha^->0$.
     \item[b.] $h_{+}:[2,\infty)\to \R$ has a single critical point in
      $[2,\infty)$ at $\frac{5}{2}$ and this point is also a
      non-degenerate maximum.  Moreover, for all $x\in (3,\infty)$ we
      have $h_{+}(x)=\alpha^{+}x+\beta^{+}$ for some constants
      $\alpha^{+},\beta^{+}\in\R$ with $\alpha^+ < 0$.
   \end{itemize}
  \item[iii.] The Hamiltonian isotopy $\phi_t^h: \mathbb{R}^2 \to
   \mathbb{R}^2$ associated to $h$ exists for all $t \in \mathbb{R}$.
   Moreover, the derivatives of the functions $h_{\pm}$ are
   sufficiently small so that the Hamiltonian isotopy $\phi^{h}_{t}$
   keeps the sets $[2,\infty)\times \{i\}$ and $(-\infty, -1]\times
   \{{i}\}$ inside the respective $T_{i}^{\pm}$ for $-1\leq t\leq 1$.
  \item[iv.] The Hamiltonian isotopy $\phi_t^h$ preserves the strip
   $[-\tfrac{3}{2}, \tfrac{5}{2}] \times \mathbb{R}$ for all $t$, in
   other words $\phi_t^h \bigl([-\tfrac{3}{2}, \tfrac{5}{2}] \times
   \mathbb{R} \bigr) = [-\tfrac{3}{2}, \tfrac{5}{2}] \times
   \mathbb{R}$ for every $t$.
\end{itemize}
\begin{figure}[htbp]\vspace{-0.9in}
   \begin{center}
      \epsfig{file=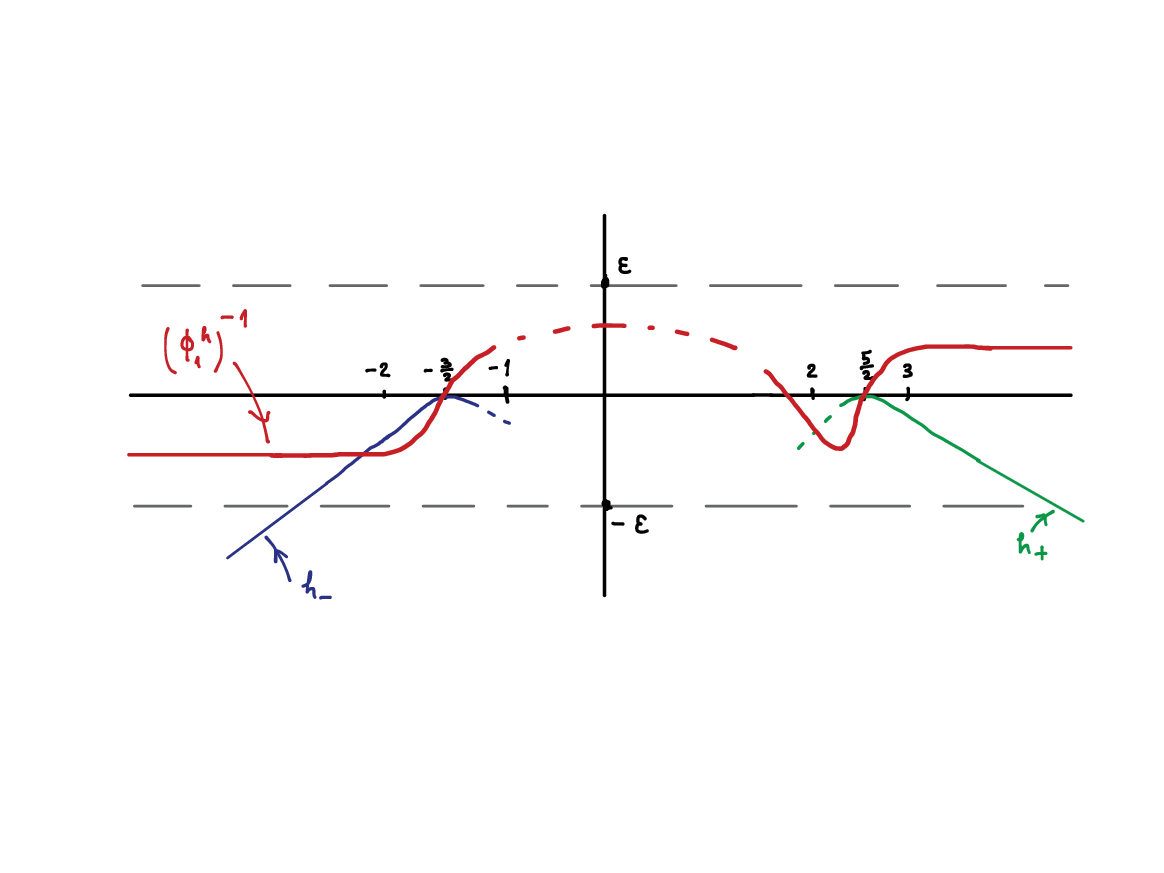, width=1\linewidth}
   \end{center}
   \vspace{-1.5in}\caption{ \label{fig:kinks} The graphs of $h_{-}$
     and $h_{+}$ and the image of $\R$ by the Hamiltonian
     diffeomorphism $(\phi^{h}_{1})^{-1}$. The profile of the
     functions $h_{-}$ at $-3/2$ and $h_{+}$ at $5/2$ are the
     ``bottlenecks''.}
\end{figure}
For example, we can define $h(x,y) = h_{-}(x)\sigma_i(y)$ on
$W_i^{-}$, where $h_-$ is a function as above with the constant
$\beta_-$ adjusted so that $h_{-}(-\tfrac{3}{2}) = 0$ and
$\sigma_i:[i-\epsilon, i+\epsilon] \to \mathbb{R}$ is a function that
vanishes near $i \pm \epsilon$ and equals $1$ near $i$. We define $h$
on $W_{i}^+$ in an analogous way. Next extend $h$ by $0$ on the rest
of $(-\infty, -1] \times \mathbb{R}$ and the rest of $[2, \infty)
\times \mathbb{R}$. Finally extend $h$ to the rest of $[-1,2]\times \mathbb{R}$
 in an arbitrary way (so that $\phi_t^h$ exists for all
$t$). It is easy to see that such an $h$ has properties i-iv above.

Notice that there are infinitely many sets $W_{i}^{\pm}$, two for each
integer $i$. Also, the signs of the constants $\alpha^{+}$,
$\alpha^{-}$ are prescribed by the points ii.a and ii.b. It follows
from the description of the functions $h_{\pm}$ that for every $i \in
\mathbb{Z}$ and $t \in [-1,1]$ we have:
\begin{equation*}
   \begin{aligned}
      & \phi_{t}^{h}\bigl( (2,\infty)\times\{i\} \bigr) \cap \bigl(
      (2,\infty)\times\{i\} \bigr) = (\tfrac{5}{2}, i), \\
      & \phi^{h}_{t} \bigl( (-\infty, -1)\times \{i\} \bigr) \cap
      \bigl( (-\infty, -1)\times \{i\} \bigr)= (-\tfrac{3}{2}, i).
   \end{aligned}
\end{equation*}
These properties of the functions $h_{\pm}$ will be used in
establishing compactness for $J$-holomorphic curves. The function $h$
itself will be referred to in the following as a {\em profile
  function} and the two critical points $(-3/2, i)$ and $(5/2,i)$ will
be called its {\em bottlenecks}.

\pbred{The functions $h$, $h_{\pm}$ are an essential ingredient in
  defining Floer perturbation data for a pair of cobordisms $(V, V')$.
  The construction of these perturbations will be explained in detail
  further below. For the moment, we explain in the next remark the role of
  the particular shape of $h$.}

\begin{rem}\label{rem:bottlenecks}
   The behavior of $h$ at the points $-3/2$ and $5/2$ is essential in
   two respects. On one hand, the shape of $h_{\pm}$ at these points
   will play the role of a ``bottleneck'', limiting the behavior of
   holomorphic curves in the neighborhood of $(-3/2, i)$ and
   $(5/2,i)$. This insures compactness for the relevant perturbed
   $\widetilde{J}$-holomorphic curves for certain restricted almost
   complex structures $\widetilde{J}$ on $\C\times M$.  On the other
   hand, the fact that the critical points of $h_{-}$ and $h_{+}$ are
   {\em maxima} implies that the almost complex structures
   $\widetilde{J}$ in question are regular in the sense required to
   define the $A_{\infty}$-category.
   \pbred{Additionally, the self Floer homology $HF(V,V)$, defined
     using a perturbation based on $h$, is isomorphic to the relative
     Lagrangian quantum homology $QH(V, \partial V)$ and is therefore
     unital. See Remark~\ref{rem:comp-Floer-compl} for more on that.}
\end{rem}

For each pair of cobordisms $V,V'\in \mathcal{CL}_{d}(\widetilde{M})$
we pick a {\em Floer datum} $\mathscr{D}_{V,V'} = (\bar{H}_{V, V'},
J_{V, V'})$ consisting of a Hamiltonian $\bar{H}_{V,V'}:[0,1]\times
\widetilde{M} \to \R$ and a (possibly time dependent) almost complex
structure $J_{V,V'}$ on $\widetilde{M}$ which is compatible with
$\widetilde{\omega}$.  We will also assume that each Floer datum
$(\bar{H}_{V,V'},J_{V,V'})$ satisfies the following conditions:
\begin{itemize}
  \item[i.] $\phi^{\bar{H}_{V,V'}}_{1}(V)$ is transverse to $V'$.
  \item[ii.] Write points of $\widetilde{M} = \mathbb{C} \times M$ as
   $(x,y,p)$ with $x+iy \in \mathbb{C}$, $p \in M$. We require that
   there exists a compact set $K_{V,V'}\subset
   (-\frac{5}{4},\frac{9}{4})\times \R \subset \C$ so that
   $\bar{H}_{V,V'}(t,(x,y,p))= h(x,y)+H_{V,V'}(t,p)$ for $(x+iy, p)$
   outside of $K_{V,V'}\times M$, for some $H_{V,V'}: [0,1]\times M\to
   \R$.
  \item[iii.] The projection $\pi: \widetilde{M} \to \mathbb{C}$ is
   $(J_{V,V'}(t), (\phi_t^h)_* i)$-holomorphic outside of $K_{V, V'}
   \times M$ for every $t \in [0,1]$.
\end{itemize}

Note that these conditions imply that the time-$1$ Hamiltonian chords
$\mathcal{P}_{\bar{H}_{V,V'}}$ of $\bar{H}_{V,V'}$ that start on $V$
and end on $V'$, form a finite set. Indeed, the chords that project to
$(-\frac{3}{2}, i)$ or to $(\frac{5}{2}, i)$ are finite in number due
to condition i. Further, if $\gamma$ is a chord that does not project
to one of these points, it follows from conditions ii and iii in the
definition of $h$, that such a chord $\gamma$ is contained in
$[-\frac{5}{4},\frac{9}{4}]\times \R\times M$ and so there can only be
a finite number of such chords too.  Thus the number of elements in
$\mathcal{P}_{\bar{H}_{V,V'}}$ is finite.

For a $(k+1)$-pointed disk $S_{r}$, we denote by $C_{i}\subset
\partial S_{r}$ the connected components of $\partial S_{r}$ indexed
so that $C_{1}$ goes from the exit to the first entry, $C_{i}$ goes
from the $(i-1)$-th entry to the $i$, $1\leq i\leq k$, and $C_{k+1}$
goes from the $k$-th entry to the exit. (Recall that we order the punctures
on the boundary of $S_{r}$ in a clockwise orientation.)

In order to describe the version of the Cauchy-Riemann equation
relevant for our purposes we need to choose additional perturbation
data. We follow partially the scheme from~\cite{Se:book-fukaya-categ}.
To every collection $V_{i}\in \mathcal{CL}_{d}(\mathbb{C} \times M)$,
$1\leq i\leq k+1$ we choose a perturbation datum $\mathscr{D}_{V_1,
  \ldots, V_{k+1}} = (\form, \mathbf{J})$ consisting of:
\begin{enumerate}
  \item[i.] A family $\form = \{ \form^r \}_{r \in
     \mathcal{R}^{k+1}}$, where $\form^r \in \Omega^{1}(S_{r},
   C^{\infty}(\widetilde{M}))$ is a $1$-form on $S_r$ with values in
   smooth functions on $\widetilde{M}$ (considered as autonomous
   Hamiltonian functions).  We write $\form^r(\xi): \widetilde{M} \to
   \mathbb{R}$ for the value of $\form^r$ on $\xi \in T S_r$. When
   there is no risk of confusion we write $\form$ instead of
   $\form^r$.
  \item[ii.] $\mathbf{J} = \{J_z \}_{z \in \mathcal{S}^{k+1}}$ is a family
   of $\widetilde{\omega}$-compatible almost complex structure on
   $\widetilde{M}$, parametrized by $z \in S_r$, $r \in
   \mathcal{R}^{k+1}$. We will sometimes write $J(z,\widetilde{p})$
   for the complex structure $J_z$ acting on $T_{\widetilde{p}}
   \widetilde{M}$.
\end{enumerate}
The forms $\form^r$ induce forms $Y^r=Y^{\form^r} \in
\Omega^{1}(S_{r}, C^{\infty}(T\widetilde{M}))$ with values in
(Hamiltonian) vector fields on $\widetilde{M}$ via the relation
$Y(\xi)=X^{\form(\xi)}$ for each $\xi\in T S_{r}$ (in other words,
$Y(\xi)$ is the Hamiltonian vector field on $\widetilde{M}$ associated
to the autonomous Hamiltonian function $\form(\xi): \widetilde{M} \to
\mathbb{R}$).

The perturbation data $\mathscr{D}_{V_1, \ldots, V_{k+1}}$ is required
to satisfy additional conditions which we will explain shortly.
However, before doing so, here is the relevant Cauchy-Riemann equation
associated to $\mathscr{D}_{V_1, \ldots, V_{k+1}}$
(see~\cite{Se:book-fukaya-categ} for more details):
\begin{equation}\label{eq:jhol1} \ u:S_{r}\to \C\times M, \quad Du +
   J(z,u)\circ Du\circ j = Y + J(z,u)\circ Y\circ j, \quad 
   u(C_{i})\subset V_{i} ~.~
\end{equation}
Here $j$ stands for the complex structure on $S_r$. The $i$-th entry
of $S_r$ is labeled by a time-$1$ Hamiltonian orbit $\gamma_{i}\in
\mathcal{P}_{\bar{H}_{V_{i},V_{i+1}}}$ and the exit is labeled by a
time-$1$ Hamiltonian orbit $\gamma_{k+1}\in
\mathcal{P}_{\bar{H}_{V_{1}, V_{k+1}}}$. The map $u$ verifies
$u(C_{i})\subset V_{i}$ and $u$ is required to be asymptotic - in the
usual Floer sense - to the Hamiltonian orbits $\gamma_{i}$ on each
respective strip-like end.

The perturbation data $\mathscr{D}_{V_1, \ldots, V_{k+1}}$ are
constrained by a number of additional conditions that we now list.
Most of these conditions are similar to the analogous ones in the
setting of~\cite{Se:book-fukaya-categ} but there are also some
significant differences so that we go through this part in detail. We
indicate the different points with $^{\ast}$. For further use we
denote by $s_{V_{1}, \ldots, V_{k+1}}\in \mathbb{N}$ the smallest $l
\in \mathbb{N}$ so that $\pi(V_1 \bigcup \cdots \bigcup
V_{k+1})\subset \R\times (-l,l)$. Write $\bar{h} = h \circ \pi:
\widetilde{M} \to \mathbb{R}$, where $h: \mathbb{R}^2 \to \mathbb{R}$
is the function described at the beginning of this section. We also
write
\begin{equation*}
   \begin{aligned}
      & U^r_i = \epsilon_i^{S_r} \bigl( (-\infty, -1] \times [0,1]
      \bigr) \subset S_r, \quad i = 1, \ldots, k, \\
      & U^r_{k+1} = \epsilon_{k+1}^{S_r} \bigl( [1, \infty) \times [0,1]
      \bigr) \subset S_r, \\
      & \mathcal{W}^r = \bigcup_{i=1}^{k+1} U^r_i.
   \end{aligned}
\end{equation*}

The conditions on $\mathscr{D}_{V_1, \ldots, V_{k+1}}$ are the
following:
\begin{itemize}
  \item[i.] {\em Asymptotic conditions.} For every $r \in
   \mathcal{R}^{k+1}$ we have $\form|_{U_i^r} = \bar{H}_{V_{i},
     V_{i+1}} dt$, $i=1, \ldots, k$ and $\form|_{U_{k+1}^r} =
   \bar{H}_{V_{1}, V_{k+1}} dt$. (Here $(s,t)$ are the coordinates
   parametrizing the strip-like ends.) Moreover, on each $U_i^r$,
   $i=1, \ldots, k$, $J_z$ coincides with $J_{V_{i},V_{i+1}}$ and on
   $U_{k+1}^r$ it coincides with $J_{V_{1}, V_{k+1}}$, i.e.
   $J_{\epsilon^{\mathcal{S}_r}_i(s,t)} = J_{V_i, V_{i+1}}(t)$ and
   similarly for the exit end. Thus, over the part of the strip-like
   ends $\mathcal{W}^r$ the perturbation datum $\mathscr{D}_{V_1,
     \ldots, V_{k+1}}$ is compatible with the Floer data
   $\mathscr{D}_{V_i, V_{i+1}}$, $i=1, \ldots, k$ and
   $\mathscr{D}_{V_1, V_{k+1}}$.
  \item[ii.$^{\ast}$ ]{\em Special expression for $\form$.} The
   restriction of $\form$ to $S_r$ equals $$\form|_{S_r} = da_r \otimes
   \bar{h} +\form_0$$ for some $\form_0 \in \Omega^{1}(S_{r},
   C^{\infty}(\widetilde{M}))$ which depends smoothly on $r \in
   \mathcal{R}^{k+1}$. Here $a_r:S_r \to \mathbb{R}$ are the
   transition functions from~\S\ref{subsec:strip-ends}. The form
   $\form_0$ is required to satisfy the following two conditions:
   \begin{itemize}
     \item[a.] $\form_0(\xi)=0$ for all $\xi\in TC_{i}\subset
      T\partial S_{r}$.
     \item[b.]  There exists a compact set $K_{V_{1},\ldots ,
        V_{k+1}}\subset (-\frac{3}{2},\frac{5}{2})\times \R$ which is
      independent of $r \in \mathcal{R}^{k+1}$ so that
      $K_{V_{1},\ldots , V_{k+1}}\times M$ contains all the sets
      $K_{V_{i}, V_{j}}$ involved in the Floer datum
      $\mathscr{D}_{V_i, V_j}$, and with
      $$K_{V_{1},\ldots, V_{k+1}}\supset\ ([-\frac{5}{4},
      \frac{9}{4}]\times [-s_{V_{1},\ldots, V_{k+1}},+s_{V_{1},\ldots,
        V_{k+1}}])$$ so that outside of $K_{V_{1},\ldots ,
        V_{k+1}}\times M$ we have $D \pi (Y_0)=0$ for every $r$, where
      $Y_{0}=X^{\form_0}$.
   \end{itemize}
  \item[iii$^{\ast}$] Outside of $K_{V_{1},\ldots , V_{k+1}}\times M$
   the almost complex structure $\mathbf{J}$ has the property that the
   projection $\pi$ is
   $(J_z,(\phi_{a_r(z)}^{h})_{\ast}(i))$-holomorphic for every $r \in
   \mathcal{R}^{k+1}$, $z \in S_r$.
\end{itemize}

\begin{rem}
   \begin{itemize}
     \item[a.]  The form $Y_{0}$ coincides with $dt\otimes
      X^{H_{V_{i},V_{i+1}}}$ (respectively $dt\otimes
      X^{H_{V_{1},V_{k+1}}}$) on each strip-like end. \pbred{This will
        play a role when proving compactness for Floer polygons (i.e.
        solutions of equation~\eqref{eq:jhol1}), for example in
        Lemma~\ref{lem:compact1}. The point is that after applying a
        naturality transformation - as given explicitely in \eqref{eq:nat-v} - to solutions
        of~\eqref{eq:jhol1} we would like to obtain maps $v:S_r
        \longrightarrow \mathbb{C} \times M$ whose projection to
        $\mathbb{C}$ is holomorphic outside a prescribed region in the
        plane. The fact that $Y_0$ is vertical along the strip-like
        ends (i.e. $D\pi(Y_0)=0$) is important in this respect.}
     \item[b.] The only difference of substance concerning the point
      ii$^{\ast}$ in comparison with the setting
      in~\cite{Se:book-fukaya-categ} is that {\em we do not require}
      $Y$ to satisfy $Y(\xi)=0$ for all $\xi\in T(\partial S_{r})$.
      In fact, the form $da\otimes X^{\bar{h}}$ already does not, in
      general, satisfy this condition on the strip-like ends.  Still,
      the usual Fredholm theory remains valid in this case.  There are
      some changes however related to the treatment of compactness. We
      will discuss explicitly this point further below.
     \item[c.] The role of the transition functions
      $\mathbf{a}:\mathcal{S}^{k+1} \to [0,1]$ is to identify the
      solutions of equation~\eqref{eq:jhol1} to solutions of an
      equation for which compactness can be easily verified. In
      particular, this involves the constraint on the almost complex
      structure at the point iii$^{\ast}$. The notation there means
      $(\phi^h_{a_r(z)})_*(i) := D \phi^h_{a_r(z)} \circ i \circ (D
      \phi^h_{a_r(z)})^{-1}$.
   \end{itemize}
\end{rem}

Using the above choices of data we will construct the
$A_{\infty}$-category $\fuk^{d}_{cob}(\C\times M)$ by the usual method
from~\cite{Se:book-fukaya-categ}. The objects of this category are
Lagrangians cobordisms $V\subset \C\times M$, the morphisms space
between the objects $V$ and $V'$ will be $CF(V, V';
\mathscr{D}_{V,V'})$, the $\Z_{2}$-vector space generated by the
Hamiltonian chords $\mathcal{P}_{\bar{H}_{V,V'}}$.  The $A_{\infty}$
structural maps $\mu_k$, $k \geq 1$, are defined by counting pairs
$(r,u)$ with $r \in \mathcal{R}^{k+1}$ and $u$ a solution
of~\eqref{eq:jhol1} as in~\cite{Se:book-fukaya-categ}.  There are a
few more ingredients needed for this construction to work in our
setting: we need to establish apriori energy estimates and prove that
they are sufficient for Gromov compactness to apply. We also need to
show that it is possible to choose the various data involved so as to
insure regularity. We will deal in~\S\ref{subsec:energy} with
compactness and in~\S\ref{subsec:transversality} with regularity.  As
mentioned in Remark \ref{rem:bottlenecks}, both points depend on our
choice of bottlenecks.  We sum up the construction
in~\S\ref{subsec:fin-cob}.
 
\subsection{Energy bounds and compactness}\label{subsec:energy}
We begin with the following general Proposition which will be useful
in the sequel.

\begin{prop} \label{p:open-map} Let $\Sigma, \Gamma$ be Riemann
   surfaces, not necessarily compact, $\Sigma$ possibly with boundary
   and $\Gamma$ without boundary. Let $w: \Sigma \to \Gamma$ be a
   continuous map, and $U \subset \Gamma$ an open connected subset.
   Assume that:
   \begin{enumerate}
     \item $\textnormal{image\,}(w) \cap U \neq \emptyset$.
     \item $w$ is holomorphic over $U$, i.e. $w|_{w^{-1}(U)} :
      w^{-1}(U) \to U$ is holomorphic.
     \item $w(\partial \Sigma) \cap U = \emptyset$.
     \item $\bigl(\overline{\textnormal{image\,}(w)} \setminus
      \textnormal{image\,}(w) \bigr) \cap U = \emptyset$.
   \end{enumerate}
   Then $\textnormal{image\,}(w) \supset U$. In particular, if
   $\overline{\textnormal{image\,}(w)} \subset \Gamma$ is compact then
   so is $\overline{U}$.
\end{prop}
\begin{proof}
   We have $\textnormal{image\,}(w) \cap U =
   \overline{\textnormal{image\,}(w)} \cap U =$ closed subset of $U$.
   At the same time, our assumptions and the open mapping theorem for
   holomorphic functions imply that $\textnormal{image\,}(w) \cap U =
   w(\textnormal{Int\,} \Sigma) \cap U =$ open subset of $U$. As $U$
   is connected, we obtain $\textnormal{image\,}(w) \cap U = U$.
\end{proof}
We will mainly use Proposition~\ref{p:open-map} with $\Sigma = S_r$ or
an open subset of $S_r$ and $\Gamma = \mathbb{C}$.

We now turn to compactness in our specific situation. The role of the
transition functions $\mathbf{a}:\mathcal{S}^{k+1} \to [0,1]$ in our
construction is reflected in the following result.
\begin{lem}\label{lem:compact1} 
   Let $V_{1},\ldots V_{k+1}\in \mathcal{CL}_{d}(M)$ be $k+1$
   cobordisms. Fix Floer and perturbation data as
   in~\S\ref{subsec:equation}. Then there exists a constant $C =
   C_{V_1, \ldots, V_{k+1}}$ that depends only on the cobordisms $V_1,
   \ldots, V_{k+1}$, the Floer data and the perturbation data, such
   that for every $r \in \mathcal{R}^{k+1}$ and every solution
   \pbred{$u:S_{r}\to \mathbb{C} \times M$} of~\eqref{eq:jhol1}, we have
   $u(S_{r})\subset B_{V_1, \ldots, V_{k+1}} \times M$, where $B_{V_1,
     \ldots, V_{k+1}} = [-\tfrac{3}{2}, \tfrac{5}{2}] \times [-C, C]$.
\end{lem}

\begin{proof}
   First note that, due to our assumptions on $h$, the subset
   $[-\tfrac{3}{2}, \tfrac{5}{2}] \times \mathbb{R} \times M$ as well
   as its complement are both invariant under the flow
   $\phi_t^{\bar{h}}$ of $\bar{h}$.

   We now define an auxiliary constant. Let $C'>0$ be large enough
   such that the set $B' = [-\tfrac{3}{2}, \tfrac{5}{2}] \times [-C',
   C']$ satisfies the following conditions for every $t \in [0,1]$:
   \begin{enumerate}
     \item $B' \supset (\phi_t^h)^{-1} \Bigl( \pi(V_i) \cap \bigl(
      [-\tfrac{3}{2}, \tfrac{5}{2}] \times \mathbb{R} \bigr) \Bigr)$.
     \item $B' \supset (\phi_t^h)^{-1}(K_{V_1, \ldots, V_{k+1}})$.
     \item $B' \supset (\phi_t^h)^{-1}(\pi(\gamma(t)))$ for every
      chord $\gamma:[0,1] \to \mathbb{C} \times M$ in
      $\mathcal{P}_{\bar{H}_{V_i,V_{i+1}}}$, $i=1, \ldots, k$, and in
      $\mathcal{P}_{\bar{H}_{V_1,V_{k+1}}}$.
   \end{enumerate}
   The proof is based on the following auxiliary Lemma..

   \noindent \textbf{Auxiliary Lemma.}  {\sl Let $r \in \mathcal{R}^{k+1}$ and
     $u: S_r \to \mathbb{C} \times M$ a solution of~\eqref{eq:jhol1}.
     Define $v:S_r \to \mathbb{C} \times M$ by the formula:
     \begin{equation} \label{eq:nat-v} u(z) =
        \phi_{a_r(z)}^{\bar{h}}(v(z)),
     \end{equation}
     where $a_r:S_r \to [0,1]$ is the transition function. 
     We have $\textnormal{image\,}(v) \subset B'$.}

   \noindent \textbf{Proof of the Auxiliary Lemma.} A straightforward
   calculation shows that the Floer equation~\eqref{eq:jhol1} for $u$
   transforms into the following equation for $v$:
   \begin{equation}\label{eq:jhol2}
      Dv + J'(z,v)\circ Dv\circ j = Y' + J'(z,v)\circ Y'\circ j.
   \end{equation}
   Here $Y'\in \Omega^{1}(S_{r}, C^{\infty}(TM))$ and $J'$ are defined
   by:
   \begin{equation} \label{eq:nat-Y'-J'} Y=D
      \phi_{a(z)}^{\bar{h}}(Y')+da_r \otimes X^{\bar{h}}, \quad
      J_z=(\phi_{a_r(z)}^{\bar{h}})_{\ast} J'_z.
   \end{equation}
   The map $v$ satisfies the following moving boundary conditions:
   \begin{equation}\label{eq:mov-bdry} \forall \ z \in C_{i}, \quad 
      v(z)\in (\phi_{a(z)}^{\bar{h}})^{-1}(V_{i}).
   \end{equation}
   The asymptotic conditions for $v$ at the punctures of $S_r$ are as
   follows. For $i=1, \ldots, k$, $v(\epsilon_i(s,t))$ tends as $s \to
   -\infty$ to a time-$1$ chord of the flow $(\phi_t^{\bar{h}})^{-1}
   \circ \phi_t^{\bar{H}_{V_i, V_{i+1}}}$ starting on $V_i$ and ending
   on $(\phi_1^{\bar{h}})^{-1}(V_{i+1})$.  (Here $\epsilon_i(s,t)$ is
   the parametrization of the strip-like end at the $i$'th puncture.)
   Similarly, $v(\epsilon_{k+1}(s,t))$ tends as $s \to \infty$ to a
   chord of $(\phi_t^{\bar{h}})^{-1} \circ \phi_t^{\bar{H}_{V_1,
       V_{k+1}}}$ starting on $V_1$ and ending on
   $(\phi_1^{\bar{h}})^{-1}(V_{k+1})$. Note that there exists
   $\delta>0$ (that depends on the function $h$) such those chords
   that lie outside of $[-\tfrac{3}{2} + \delta,
   \tfrac{5}{2}-\delta]\times \mathbb{R} \times M$ have constant
   projection under $\pi$ to points of the type $(-\tfrac{3}{2}, q)$
   or $(\tfrac{5}{2}, q)$, $q \in \mathbb{Z}$.
   The proof of the Auxiliary Lemma is based on three claims, as follows.

   \smallskip
   \noindent \textbf{Claim 1.} {\sl Put $v' = \pi \circ v: S_r \to
     \mathbb{C}$. The map $v'$ is holomorphic over $\mathbb{C}
     \setminus ([-\tfrac{3}{2} + \delta', \tfrac{5}{2}-\delta'] \times
     [-C', C'])$ for small enough $\delta'>0$.}
   \smallskip

   Indeed, let $z\in S_r$ be a point with $v'(z) \notin [-\tfrac{3}{2}
   + \delta', \tfrac{5}{2}-\delta'] \times [-C', C']$.  The vector
   field $Y$ from equation~\eqref{eq:jhol1} is of the form
   $Y=Y_{0}+da_r\otimes X^{\bar{h}}$ with $D\pi (Y_{0})=0$ so that for
   $Y'$ we have $Y_{0}=D \phi^{\bar{h}}_{a(z)}(Y')$.  The Hamiltonian
   vector field $X^{\bar{h}}$ is horizontal with respect to $\C\times
   M$, hence $D\pi(Y')=0$. Further, by property iii$^{\ast}$ and the
   definition of $J'$ from~\eqref{eq:nat-Y'-J'} we have that at the
   point $z$ the projection $\pi$ is $(J',i)$-holomorphic. This proves
   Claim 1.

\

   Define now the following subsets (see
   Figure~\ref{f:compactness-1}):
   \begin{equation*}
      \begin{aligned}
         & R = \Bigl( \bigcup_{i=1}^{k+1} \bigcup_{t \in [0,1]}
         (\phi_{t}^{h})^{-1}(\pi(V_i)) \Bigr) \cap (\mathbb{C}
         \setminus
         B'), \quad Q = \mathbb{C} \setminus (R \cup B'), \\
         & \mathcal{I} = \Bigl( \bigl \{ (-\tfrac{3}{2}, q) \mid q \in
         \mathbb{Z} \bigr \} \cup \bigl \{ (\tfrac{5}{2}, q) \mid q
         \in \mathbb{Z} \bigr \} \Bigr) \cap B'.
      \end{aligned}
   \end{equation*}
   It is easy to see that $Q \subset \mathbb{C}$ is open and that
   $\overline{R} \subset R \cup \mathcal{I}$.
   \vspace{0.6in}
   \begin{figure}[htbp]
      \begin{center}
         \epsfig{file=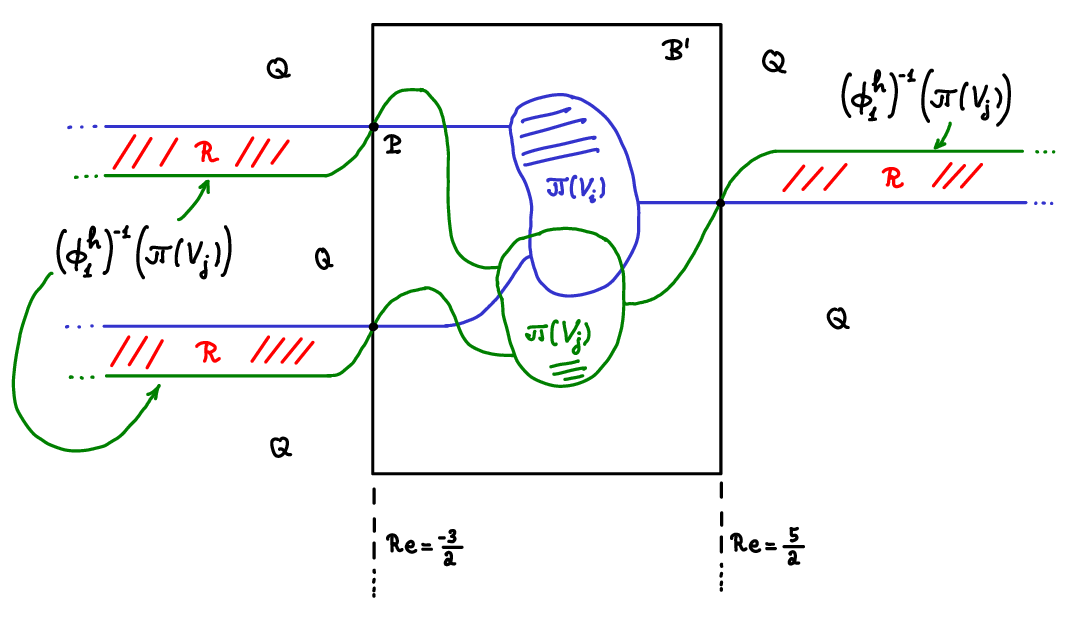, width=0.9 \linewidth}
      \end{center}
      \caption{\label{f:compactness-1} The projection to $\mathbb{C}$
        of $V_i$, $V_j$, $\phi_t^{\bar{h}}(V_j)$ and the regions $R$,
        $Q$.}
   \end{figure}
   
   \smallskip
   \noindent \textbf{Claim 2.} {\sl $v'(S_r) \cap Q = \emptyset$.}
   \smallskip
   
   To see this, first note that all connected components of $Q$ are
   unbounded, hence none of them has compact closure. Now assume by
   contradiction that $v'(S_r)$ intersects $Q$. Let $Q_{0}$ be one of
   the connected components of $Q$ that intersects $v'(S_{r})$. Note
   that $v'(\partial S_r) \subset B' \cup R$ hence $v'(\partial S_r)
   \cap Q_0 = \emptyset$.  Note also that all the chords corresponding
   to the ends of $v'$ are disjoint from $Q_0$, hence
   $(\overline{v'(S_r)} \setminus v'(S_r)) \cap Q_0 = \emptyset$.
   Finally, $\overline{v'(S_r)}$ is clearly compact. Thus we can apply
   Proposition~\ref{p:open-map} with $w = v'$, $\Sigma = S_r$, $\Gamma
   = \mathbb{C}$ and $U = Q_0$ and obtain that $\overline{Q}_0$ is
   compact. A contradiction. This concludes the proof of Claim 2.

   \smallskip Note that the intersection points between $\pi(V_i) \cap
   (\mathbb{C} \setminus \textnormal{Int\,} B')$ and
   $\pi(\phi_t^h(V_j)) \cap (\mathbb{C} \setminus \textnormal{Int\,}
   B')$ are all in $\mathcal{I}$.
   
   \noindent \textbf{Claim 3.} {\sl Let $P \in \mathcal{I}$. Then
     for every $z \in \textnormal{Int\,}(S_r)$, $v'(z) \neq P$.}
   \smallskip
   
   Indeed, if there were a point $z$ with $v'(z) = P$, then as $v'$ is
   holomorphic in a neighborhood of $z$ the open mapping would imply
   that there is a neighborhood of $P$ which is contained in the image
   of $v'$. But Any neighborhood of $P$ must intersect $Q$, hence we
   obtain a contradiction to Claim 2. This concludes the proof of
   Claim 3.
   
   \smallskip

   We are now ready to prove the main statement of the Auxiliary Lemma, i.e.
   $\textnormal{image\,}(v') \subset B'$. By what we have proved so
   far, it is enough to show that $\textnormal{image\,}(v') \cap R =
   \emptyset$ which we now proceed to show. Indeed, suppose by contradiction that
   $\textnormal{image\,}(v') \cap R \neq \emptyset$. As all the
   Hamiltonian chords corresponding to $v$ are inside $B' \times M$ we
   have only two possibilities:

   \begin{enumerate}
     \item $\textnormal{image\,}(v') \subset \mathcal{R} \cup
      \mathcal{I}$ and $\textnormal{image\,}(v')$ is not constant at
      one of the points of $\mathcal{I}$.
     \item There exists $z_0, z_1 \in S_r$ with $v'(z_0) \in R$ and
      $v'(z_1) \in B' \setminus \mathcal{I}$.
   \end{enumerate}

   We begin by ruling out possibility 2. Let $\{z_t\}_{t \in [0,1]}$
   be a path in $S_r$ connecting $z_0$ to $z_1$ and such that $z_t \in
   \textnormal{Int\,}(S_r)$ for every $0<t<1$. As $$\partial B' \cap
   \partial R \subset \mathcal{I}$$ it follows that there exists $0<
   t_0 < 1$ such that $v'(z_{t_0}) \in \mathcal{I}$. However, this
   contradicts the Claim 3. This rules out possibility 2.

   Assume now that possibility 1 occurs. Then
   $\textnormal{image\,}(v')$ is entirely contained inside
   $\overline{R_0}$, where $R_0 \subset R$ is one of the connected
   components. Now $\overline{R_0} = R_0 \cup \{P\}$, where $P \in
   \mathcal{I}$. Without loss of generality assume that $P =
   (-\tfrac{3}{2}, q)$ for some $q \in \mathbb{Z}$ (the case $P =
   (\tfrac{5}{2}, q)$ can be dealt with in an analogous way). Note
   that all the Hamiltonian chords that lie in $\overline{R_0} \times
   M$ project via $\pi$ to the point $P$. In particular, at the exit
   puncture of $S_r$, the map $v'$ tends to $P$, i.e.
   \begin{equation} \label{eq:v'-exit} \lim_{s \to \infty}
      v'(\epsilon_{k+1}(s,t)) = P.
   \end{equation}
   Put $v'' = v' \circ \epsilon_{k+1} : [0, \infty) \times [0,1] \to
   \mathbb{C}$. We have $v''(s, 0) \in (-\infty, \tfrac{3}{2}] \times
   \{q\}$ for every $s \geq 0$. We claim that for every $s \geq 0$ we
   have $\partial_s v'' (s,0) \leq 0$. Indeed, if \pbred{$\partial_s
     v''(s_0, 0)>0$} for some $s_0 \geq 0$ then as $v''$ is
   holomorphic we have $\partial_t v''(s_0, 0) = i \partial_s
   v''(s_0,0) \in i \mathbb{R}_{\geq 0}$. It follows that for $t$
   close enough to $0$, $v''(s_0,t) \in Q$. But this contradicts the statement
   at Claim 2. Thus, $\partial_s v''(s,0) \leq 0$ for every $s \geq 0$. On
   the other hand we have $\lim_{s \to \infty} v''(s,0) = P$. It
   follows that $v''(s,0) = P$ for every $s \geq 0$. From this it is
   not hard to see (e.g. by a reflection argument) that $v''$ is the
   constant map at $P$. A contradiction. This rules out possibility 1,
   and concludes the proof of the Auxiliary Lemma.
   
   \
   
   We now return to the proof of Lemma \ref{lem:compact1}.
 Let $C = C_{V_1, \ldots,
     V_{k+1}}>0$ to be a large enough constant so that the subset $B =
   [-\tfrac{3}{2}, \tfrac{5}{2}] \times [-C, C]$ satisfies $$B \supset
   \bigcup_{\tau \in [0,1]} \phi_{\tau}^h(B'),$$ where $B'$ is the
   subset  in the Auxiliary Lemma.  For every $r \in
     \mathcal{R}^{k+1}$ and every solution $u:S_r \to \mathbb{C}
     \times M$ of~\eqref{eq:jhol1} we have $\textnormal{image\,}(u)
     \subset B \times M$.

   Indeed, if $u$ is such a solution and we define $v$
   using~\eqref{eq:nat-v} then by the Auxiliary Lemma  we have $$v(z) \in B' \times
   M, \; \forall z \in S_r.$$ It follows that $$u(z) =
   \phi_{a_r(z)}^{\bar{h}}(v(z)) \in \phi_{a_r(z)}^h (B') \times M
   \subset B \times M, \; \forall z \in S_r.$$
\end{proof}

To be able to apply Gromov compactness to the moduli spaces of
solutions of~\eqref{eq:jhol1} we also need to have some apriori energy
estimates that we now proceed to establish.

Let $\mathcal{M}(\gamma_{1},\ldots, \gamma_{k+1}; \mathscr{D})$ be the
space of pairs $(r,u)$ with $r \in \mathcal{R}^{k+1}$ and $u:S_r \to
\widetilde{M}$ a solution of the equation~\eqref{eq:jhol1} with
asymptotic conditions at the punctures prescribed by the chords
$\gamma_1, \ldots, \gamma_{k+1}$. Here $\mathscr{D} =
\mathscr{D}_{V_1, \ldots, V_{k+1}} = (\form, \mathbf{J})$ is a
perturbation datum as described in~\S\ref{subsec:equation}.

Recall that the energy of a map $u:S_r \to \widetilde{M}$ is defined
as follows. Pick a volume form $\sigma$ on $S_r$. Let $Y$ be the
$1$-form on $S_r$ with values in Hamiltonian vector fields associated
to $\form$. Then $$E(u) := \tfrac{1}{2} \int_{S_r} |Du -
Y|_{\mathbf{J}}^2 \sigma.$$ Recall also that the energy is independent
of the form $\sigma$ (note that the norm $| \cdot |_J$ on $1$-forms on
$S_r$ does depend on $\sigma$).

\begin{lem}\label{lem:energy-bound1}
   There exists a constant, $C$, depending on the transition functions
   $\mathbf{a}$, the Hamiltonians $H_{V_{1}, V_{2}},\ldots,
   H_{V_{k},V_{k+1}}, H_{v_{1}, V_{k+1}}$, $h$ and the perturbation data
   $\mathscr{D}$ so that for any Hamiltonian chords $\gamma_1, \ldots,
   \gamma_{k+1}$ and every $(r,u) \in \mathcal{M}(\gamma_{1},\ldots,
   \gamma_{k+1}; \mathscr{D})$ we have
   $$E(u)\leq \int_{S_r} u^* \widetilde{\omega} + C.$$ 
\end{lem}
\begin{proof} Let $r\in \mathcal{R}^{k+1}$ and a $u:S_{r}\to \C\times
   M$ a solution of~\eqref{eq:jhol1}.

   Let $s+it$ be local conformal coordinates on $S_r$. We write the
   form $\sigma$ in these coordinates as $\sigma = \lambda ds \wedge
   dt$, where $\lambda$ is a real function defined locally in the
   chart of the coordinates $(s,t)$. Write $$\form^r_{s,t} = F_{s,t}ds
   + G_{s,t}dt,$$ where $F_{s,t}, G_{s,t}: \widetilde{M} \to
   \mathbb{R}$ are smooth functions (that depend also on $r \in
   \mathcal{R}^{k+1}$).

   A simple calculation shows that
   \begin{equation} \label{eq:energy-1} \tfrac{1}{2} |Du -
      Y|_{\mathbf{J}}^2 \lambda(s,t) = \widetilde{\omega}(\partial_s
      u, \partial_t u) + dF_{s,t}(\partial_t u) - dG_{s,t}(\partial_s
      u) - \{F_{s,t}, G_{s,t} \}(u(s,t)).
   \end{equation}
   Here $\{\cdot, \cdot\}$ is the Poisson bracket, and we use the sign
   convention that $$\{F, G\} = -\widetilde{\omega}(X^F, X^G) =
   dF(X^G).$$
   
   Denote by $\widetilde{u}: S_r \to S_r \times \widetilde{M}$ the
   graph of $u$, $\widetilde{u}(z) = (z, u(z))$. Let
   $\widetilde{\form}^r = \textnormal{pr}^* \form^r$, where
   $\textnormal{pr}: S_r \times \widetilde{M} \to \widetilde{M}$ is
   the projection. We have
   \begin{equation*}
      \begin{aligned}
         & \widetilde{\form}^r_{(s,t,p)} = F_{s,t}(p)ds +
         G_{s,t}(p)dt, \quad \textnormal{where} \, (s,t) \in S_r, p \in M, \\
         & (\widetilde{u}^* \widetilde{\form}^r)_{s,t} =
         F_{s,t}(u(s,t)) ds + G_{s,t}(u(s,t))dt.
      \end{aligned}
   \end{equation*}
   It follows that $$d(\widetilde{u}^* \widetilde{\form}^r) =
   \Bigl(-\tfrac{\partial F_{s,t}}{\partial t}(u(s,t)) +
   \tfrac{\partial G_{s,t}}{\partial s}(u(s,t)) + dG_{s,t}(\partial_s
   u) - dF_{s,t}(\partial_t u) \Bigr) ds \wedge dt.$$ Substituting
   this into~\eqref{eq:energy-1} we obtain:
   \begin{equation} \label{eq:energy-2} 
      \begin{aligned}
         \tfrac{1}{2} |Du - Y|_{\mathbf{J}}^2 \sigma = & u^*
         \widetilde{\omega} - d(\widetilde{u}^* \widetilde{\form}^r) \\
         & + \Bigl(-\tfrac{\partial F_{s,t}}{\partial t}(u(s,t)) +
         \tfrac{\partial G_{s,t}}{\partial s}(u(s,t)) - \{F_{s,t},
         G_{s,t}\}(u(s,t)) \Bigr)ds \wedge dt.
      \end{aligned}
   \end{equation}

   We start by estimating the integral of the third summand
   in~\eqref{eq:energy-2}. The estimate here has already been outlined
   in~\cite{Se:book-fukaya-categ} (see Sections~8g and~9l). It
   basically follows from the fact that the form $\form$ coincides
   with the Floer datum on the strip like ends (our different choice
   of $\form$ with respect to~\cite{Se:book-fukaya-categ} plays no
   role in this estimate).

   Consider the form $R^{\form^r} \in \Omega^2(S_r,
   C^{\infty}(\widetilde{M}))$ defined by $R^{\form^r} = d \form^r +
   \tfrac{1}{2}\{\form^r \wedge \form^r \}$, where $d$ stands for the
   exterior differential on $S_r$. In the local coordinates $(s,t)$ we
   have:
   $$R^{\form^r}_{s,t} = \Bigl(-\tfrac{\partial F_{s,t}}{\partial t} +
   \tfrac{\partial G_{s,t}}{\partial s} - \{F_{s,t}, G_{s,t}\}
   \Bigr)ds \wedge dt.$$ This form does not depend on the coordinates
   $(s,t)$. It is in fact the curvature form of the connection on $S_r
   \times \widetilde{M}$ induced by $\form^r$.  See Chapter~8
   of~\cite{McD-Sa:Jhol-2} for more details.  (But note that we use
   here different sign conventions for Hamiltonian vector fields and
   for the Poisson bracket. Note also that our $\widetilde{M}$ is
   $\mathbb{C} \times M$, whereas the $\widetilde{M}$
   in~\cite{McD-Sa:Jhol-2} corresponds in our case to $S_r \times
   \mathbb{C} \times M$.)

   Using the curvature $R^{\form^r}$ we can write the third summand
   in~\eqref{eq:energy-2} as
   \begin{equation} \label{eq:energy-R-1} \int_{S_r} R^{\form^r}(u).
   \end{equation}

   Note that the form $R^{\form^r}$ vanishes identically over the part
   of the strip-like ends $\mathcal{W}_r \subset S_r$. This follows
   from the fact that $\form^r$ coincides with $\bar{H}_{V_{i},
     V_{i+1}} dt$ on $U^r_i \subset S_r$, $i=1, \ldots, k$, and with
   $\bar{H}_{V_{1}, V_{k+1}} dt$ on $U^r_{k+1} \subset S_r$, where
   $(s,t)$ are conformal coordinates adapted to the strip-like ends.
   Moreover, due to the consistency of the family of forms $\form$
   with respect to gluing/splitting, the forms $\form^r$ and
   $R^{\form^r}$, $r \in \mathcal{R}^{k+1}$, extend continuously in $r
   \in \overline{\mathcal{R}}^{k+1}$ to the partial compactification
   $\overline{\mathcal{S}}^{k+1}$ of $\mathcal{S}^{k+1}$.

   The extended form $R^{\form}$ vanishes identically in a small
   neighborhood $\mathcal{W}$ of the union of all the punctures of the
   surfaces in $\overline{S}^{k+1}$. Note that $\overline{S}^{k+1}
   \setminus \mathcal{W}$ is compact. It follows that for every
   compact subset $K \subset \mathbb{C} \times M$ there exists a
   constant $C_K$ for which we have the following uniform bound:
   $$\int_{S_r} 
   \max_{p \in K} R^{\form^r}(p) \leq C_K, \quad \forall r \in
   \overline{\mathcal{R}}^{k+1}.$$

   By Lemma~\ref{lem:compact1}, there exists a compact subset $B
   \subset \mathbb{C}$ (that does not depend on $(r,u)$),
   such that $u(S_r) \subset B \times M$.  It follows that
   $$\int_{S_r} R^{\form^r}(u) \leq 
   \int_{S_r} \max_{p \in B \times M} R^{\form^r}(p) \leq C_{B \times
     M}.$$ This provides a uniform bound (independent of $(r,u)$) for the
   integral of the third summand of~\eqref{eq:energy-2}.

   We now turn to bounding the integral of the second summand
   in~\eqref{eq:energy-2}. Here there are some differences in
   comparison to~\cite{Se:book-fukaya-categ} due to our choice of the
   form $\form$. We have:
   \begin{equation} \label{eq:energy-T2-1}
      \begin{aligned}
         \int_{S_r} -d(\widetilde{u}^* \widetilde{\form}^r) =
         & -\int_{\partial S_r} \widetilde{u}^* \widetilde{\form}^r \\
         & + \sum_{i=1}^k \int_0^1 \bar{H}_{V_i, V_{i+1}}(t,
         \gamma_i(t)) dt - \int_0^1 \bar{H}_{V_1, V_{k+1}}(t,
         \gamma_{k+1}(t))dt.
      \end{aligned}
   \end{equation}
   As all the Hamiltonian chords $\gamma_i$ lie entirely inside the
   compact subset $B_{V_1, \ldots, V_{k+1}} \times M$ (see
   Lemma~\ref{lem:compact1}, the last two terms
   in~\eqref{eq:energy-T2-1} are uniformly bounded by a bound that
   depends only on the maximum of $|\bar{H}_{V_i, V_{i+}}(t,p)|$ and
   $|\bar{H}_{V_1, V_{k+1}}(t,p)|$ over the compact set $[0,1] \times
   B_{V_1, \ldots, V_{k+1}} \times M$.

   It remains to bound (uniformly in $r$) the first term
   in~\eqref{eq:energy-T2-1}, namely $\int_{\partial S_r}
   \widetilde{u}^* \widetilde{\form}^r$. (As remarked above, in
   contrast to our case, in the setting of~\cite{Se:book-fukaya-categ}
   this term vanishes because $\form$ is assumed
   in~\cite{Se:book-fukaya-categ} to vanish on $T(\partial S_r)$.)

   Recall that $\form^r = \form^r_0 + da_r \otimes \bar{h}$ and that
   $\form^r_0 (\xi)=0$ for every $\xi \in T(\partial S_r)$. Thus we
   have:
   \begin{equation} \label{eq:eq:energy-T2-2} \int_{\partial S_r}
      \widetilde{u}^* \widetilde{\form}^r = \int_{\partial S_r}
      \form^r(u) = \int_{\partial S_r} \bar{h}(u) da_r.
   \end{equation}
   Put $C_{\bar{h}} = \max_{p \in B_{V_1, \ldots, V_{k+1}}\times M}
   |\bar{h}(p)|$. By Lemma~\ref{lem:compact1} we have $|\bar{h}(u(z))|
   \leq |C_{\bar{h}}|$ for every $z \in S_r$. It now follows from
   Lemma~\ref{lem:der-bound} and~\eqref{eq:metric-rho} that
   $$\Bigl | \int_{\partial S_r}  \widetilde{u}^* \widetilde{\form}^r \Bigr |
   \leq C_{\bar{h}} C^{k}_{\mathbf{a}} A_{k+1}.$$
   This yields a uniform bound for the integral of the second summand
   in~\eqref{eq:energy-2}, and concludes the proof of the lemma.
\end{proof}

It remains to bound the symplectic area of the solutions
of~\eqref{eq:jhol1}. For this we make use of
condition~\eqref{eq:Hlgy-vanishes}.

\begin{lem}\label{lem:symp-area-bound}
   Let $(r,u), (r',u') \in \mathcal{M}(\gamma_{1},\ldots, \gamma_{k},
   \gamma_{k+1};\mathscr{D})$. Then
   $$\int_{S_r} u^* \widetilde{\omega} - \int_{S_{r'}} u'^* \widetilde{\omega} = 
   \rho(\mu(u)-\mu(u'))$$ where $\mu$ is the Maslov index and $\rho$
   is the monotonicity constant.
\end{lem}

\begin{proof} We start with the case when the inclusions
   $\pi_{1}(V_{i})\to \pi_{1}(\C\times M)$ are all null. We will then
   modify the argument to cover the case when the image of these
   inclusions is torsion and under the additional assumption that
   $c_{1}$, $\int \omega$ are proportional on $H_{2}(M;\Z)$.  The
   argument is essentially the same as the one originally used by
   Oh~\cite{Oh:HF1} for strips. We first glue $u$ and $-u'$ along the
   chords $\gamma_{i}$. This produces a surface $N$ of genus $0$ with
   boundary components $\partial_{i} N $, $1\leq i\leq k+1$, together
   with a map $j:N\to \C\times M$ whose image is the union of the
   closures of the images of $u$ and $u'$ and whose $i$-th boundary
   component is mapped to $u(C_{i})\cup u'(C_{i})\subset V_{i}$.  We
   now use trivializations of $T(\C\times M)$ over $u$ and $u'$ to get
   a trivialization $\xi$ of $j^{\ast}T(\C\times M)$ and we compute
   $$\mu (u)-\mu(u')= \mu_{\xi}(\partial_{k+1} N)-
   \sum_{i=1}^{k}\mu_{\xi}(\partial _{i} N) ~.~$$ This difference can
   also be computed using other trivializations of $T(\C\times M)$
   over $N$.  Such trivializations are not unique but it is easy to
   see that the difference above is independent of this choice.

   By (\ref{eq:Hlgy-vanishes}) $[\partial_{i} N]=0\in \pi_{1}(M)$ for
   all $i$.  Notice that $N$ is homeomorphic to a sphere $S^{2}$ with
   $k+1$ disjoint disks removed from it.    Let
   $\Gamma_{i}$, $1\leq i\leq k+1$ be $k+1$ copies of the $2$- disk.
   Consider the surface $N'= N\cup _{i=1}^{k+1} \cup_{\partial_{i} N}
   \Gamma_{i}$. In other words, $N'$ is obtained by gluing the disks
   $\Gamma_{i}$ to $N$ along the boundary components $\partial_{i}N$.
   Clearly, $N'$ is homeomorphic to a sphere.

   We now notice that the map $j:N\to M$ extends to a map $x:
   N'\to M$.  We fix a trivialization
   $\theta$ of $x^{\ast}T(\C\times M)$.  We restrict this
   trivialization to $N$ and we get:
   $$\mu(u)-\mu(u')=\mu_{\xi}(\partial_{k+1} N)-
   \sum_{i=1}^{k}\mu_{\xi}(\partial _{i} N)=
   \mu_{\theta}(\partial_{k+1} N)-\sum_{i=1}^{k}\mu_{\theta}(\partial
   _{i} N) ~.~$$ On the other hand $\mu_{\theta}(-)$ can also be
   calculated using the disks $\Gamma_{i}$, $1\leq i\leq k+1$, and $x$
   (as the trivialization $\theta$ restricts to trivializations of the
   $\Gamma_{i}$) and so we have, by uniform monotonicity
   $\mu_{\theta}(\partial_{i}N)=\mu(\Gamma_{i})=\frac{1}{\rho}(\omega_{0}\oplus
   \omega) (\Gamma_{i})$, $1\leq i\leq k$, and
   $\mu_{\theta}(\partial_{k+1}N)=2c_{1}(x)-\mu(\Gamma_{k+1}) =
   \frac{1}{\rho}((\omega_{0}\oplus \omega) (x)-(\omega_{0}\oplus
   \omega) (\Gamma_{k+1}))$.  We obtain:
   \begin{eqnarray}\mu(u)-\mu(u')=\frac{1}{\rho}
      ((\omega_{0}\oplus\omega)(x)-\sum_{i=1}^{k+1}
      (\omega_{0}\oplus\omega)(\Gamma_{i}))=\\ \nonumber
      =\frac{1}{\rho}(\omega_{0}\oplus
      \omega)(N)=\frac{1}{\rho}((\omega_{0}\oplus\omega)(u)-
      (\omega_{0}\oplus\omega)(u'))
   \end{eqnarray}
   
   It now remains to consider the case when the inclusions
   $\pi_{1}(V_{i})\to \pi_{1}(\C\times M)$ are torsion but
   additionally $c_{1}$ and $\int \omega$ are proportional on
   $H_{2}(M;\Z)$.  Let $l\in \mathbb{N}$ be big enough so that $l
   \cdot [\partial_{i} N]=0$ in $\pi_{1}(\C\times M)$, $\forall i$.
   Let $\widetilde{\Gamma}_{i}\in\pi_{2}(\C\times M,V_{i})$ so that
   $\partial \widetilde{\Gamma}_{i} = l\cdot\partial_{i}N$. Define the
   singular chain in $\C\times M$
   $$\widetilde{x}= l\cdot N+\sum_{i=1}^{k+1}\widetilde{\Gamma}_{i}~.~$$ 
   Notice that this is a cycle and thus we can write:
   \begin{eqnarray}l(\mu(u)-\mu(u'))=\frac{1}{\rho}
      ((\omega_{0}\oplus\omega)(\widetilde{x})-\sum_{i=1}^{k+1}
      (\omega_{0}\oplus\omega)(\widetilde{\Gamma}_{i}))=\\ \nonumber
      =\frac{1}{\rho}(\omega_{0}\oplus \omega)(l\cdot
      N)=\frac{l}{\rho}((\omega_{0}\oplus\omega)(u)-
      (\omega_{0}\oplus\omega)(u'))
   \end{eqnarray}
 
\end{proof}

Let now $\mathcal{M}_{\nu}(\gamma_{1},\ldots, \gamma_{k},
\gamma_{k+1}; \mathscr{D}) \subset \mathcal{M}(\gamma_{1},\ldots,
\gamma_{k}, \gamma_{k+1}; \mathscr{D})$ be the moduli space formed by
those $u$'s with $\mu(u)=\nu$. It follows from the previous lemma that
all elements in $$\mathcal{M}_{\nu}(\gamma_{1},\ldots, \gamma_{k},
\gamma_{k+1}; \mathscr{D})$$ have the same $\widetilde{\omega}$-area,
hence by Lemma~\ref{lem:energy-bound1} the energy of such curves $u$
is uniformly bounded.  In our applications these are the moduli spaces
that will be used and $\nu$ will only take the values $0$ and $1$.

\subsection{Transversality} \label{subsec:transversality} In order to
define the Fukaya category $\fuk_{cob}^{d}(\C\times M)$ we need to be
able to choose regular Floer data, $\mathscr{D}_{V',V'} = (\bar{H}_{V,
  V'}, J_{V, V'})$ as well as regular perturbation data
$\mathscr{D}_{V_1, \ldots, V_{k+1}} = (\form, \mathbf{J})$ for all
Lagrangian cobordisms involved in the construction.

The basic arguments follow those in~\cite{Se:book-fukaya-categ}. Below
we explain how they adapt to our case.

We consider the solutions $u$ of~\eqref{eq:jhol1} associated to the
Floer data $\bar{H}_{V,V'}$, $J_{V,V'}$ as well as to the perturbation
data $\mathscr{D}_{V_1, \ldots, V_{k+1}} =(\form, \mathbf{J})$. By
Lemma~\ref{lem:compact1}, these solutions are contained inside the
sets $B_{V,V'}\times M$ and, respectively, $B_{V_1, \ldots, V_{k+1}}
\times M$. We also have that in the interior of these sets, and away
from the strip-like ends, $\bar{H}_{V,V'}$, $J_{V,V'}$ and
respectively $\form_{0}$, $\mathbf{J}$ are not restricted. However,
the Floer data have to satisfy conditions $ii$, $iii$ and the
perturbation data have to satisfy the conditions $ii^{\ast}$ and
$iii^{\ast}$ from \S\ref{subsec:equation}.  In other words, the
admissible perturbations - for the Hamiltonian part as well as the
almost complex structure - have to be ''vertical'' (or contained in
$M$) in a small neighborhood of each of the bottlenecks.  \pbredb{To
  show that the data can be chosen to insure regularity, we first make
  generic choices of Floer and perturbation data in $M$ so that all
  curves $u$ with $\pi\circ u$ constant equal to a single bottleneck
  are regular in $M$ (this means, in particular, to choose generically
  $H_{V,V'}$ and $J_{V,V'}$ outside of $K_{V,V'}$ and similarly for
  the vertical part of the forms $\Theta_{0}$ outside of
  $K_{V_{1},\ldots ,V_{k+1}}\times M$ - see the point $ii^{\ast}$ in
  \S\ref{subsec:equation} - as well as for the relevant almost complex
  structures). We claim that such curves $u$ continue to be regular
  also when viewed in $\mathbb{C} \times M$ . To see this we appeal to
  the specific form of our bottlenecks, namely to the fact that they
  correspond to maxima of the function $h_{\pm}$. The linearized
  operator corresponding to equation~\eqref{eq:jhol1} at such
  solutions $u$ split into horizontal and vertical components. By the
  results of~\S\ref{sb:ind-reg} (more specifically,
  Corollary~\ref{c:ind-0} applied with $f_1 = \cdots = f_{k+1} = h_-$
  or $h_+$) it follows that the index of the horizontal operator is
  $0$ and that it is surjective.  Since the perturbation data were
  chosen generically in the fiber the vertical part of the operator is
  also surjective.  Consequently the total linearized operator at
  solutions $u$ with constant projection at the bottleneck is
  surjective. This proves regularity for curves $u$ with constant
  projection at the bottlenecks.}

\pbredb{It remains to analyze the moduli spaces of curves that do not
  project to a constant at a bottleneck. At this point it is useful to
  consider also the curves $v$ that correspond to $u$ via the
  transformation~\eqref{eq:nat-v}. Clearly, regularity for $u$ is
  equivalent to regularity for its corresponding $v$. Note also that
  $u$ projects to a constant at a bottleneck iff its corresponding $v$
  has the same property. Now, if $v$ does not project to a constant at
  a bottleneck then its projection can not remain only in a small
  neighborhood of any fixed bottleneck: otherwise $\pi\circ v$ would
  have the same fixed bottleneck point as entry as well as exit which
  in turn shows that $\pi\circ v$ is constant (indeed, it is easy to
  see by an application of the open mapping theorem that a bottleneck
  can only be an exit if the curve $\pi\circ v$ is constant). Thus,
  these curves reach the region where all the data can be chosen
  freely.  The arguments from \cite{Se:book-fukaya-categ} together
  with the standard regularity argument (as for instance
  in~\cite{McD-Sa:Jhol-2}) imply that the Floer and perturbation data
  can be chosen to be both regular and consistent.}


\subsection{Summing it up} \label{subsec:fin-cob} In view of the
compactness results in~\S\ref{subsec:energy} and the regularity
properties discussed in~\S\ref{subsec:transversality}, the
construction of the Fukaya $A_{\infty}$ category
$\fuk_{cob}^{d}(\C\times M)$ can proceed as in the case covered
in~\cite{Se:book-fukaya-categ}, with the modifications summarized
in~\S\ref{sec:fuk-M} for the monotone setting.  As
in~\cite{Se:book-fukaya-categ}, this process actually leads to a
coherent system of $A_{\infty}$-categories.  In our case, the
categories in this system also depend on the choice of transition
functions $\mathbf{a}$ as well as on the function $h$.  However, as
described in~\cite{Se:book-fukaya-categ} this dependence can be dealt
with and the resulting $A_{\infty}$ categories are all
quasi-equivalent. In particular, 
the associated homology categories are all equivalent in a canonical way.
These invariance properties will be made explicit in the next section.

For the convenience of the reader, we collect below the main
ingredients in the construction of this category and fix relevant
notation.

\subsubsection*{A. Objects} The objects of $\fuk_{cob}^{d}(\C\times
M)$ are the Lagrangian cobordisms $V\in \mathcal{CL}_{d}(M)$ (those
cobordisms $V\subset \C\times M$, as given in
Definition~\ref{def:Lcobordism}, which are uniformly monotone with
monotonicity constant $\rho>0$ and with the number
$d_{V}=d\in\Z_{2}$).

\subsubsection*{B. Morphisms} To define the morphisms we assume
regular Floer data $(\bar{H}_{V,V'}, J_{V,V'})$ fixed for any pair
$V,V'\in\mathcal{CL}_{d}(M)$ - subject to the conditions
in~\S\ref{subsec:equation} - and we put
$$\hom(V,V')= CF(V,V'; \bar{H}_{V,V'}, J_{V,V'})= \Z_{2} 
\langle \mathcal{P}_{V,V'} \rangle~.~$$

\subsubsection*{C. The Floer complex} The differential $\mu_1$ on
$C(V,V')$, is defined by counting index-$0$ solutions $u:\mathbb{R}
\times [0,1] \longrightarrow \widetilde{M}$ of the
equation~\eqref{eq:jhol1}, modulo reparametrization by $\R$, in the
special case when $k=1$ and $K=dt\otimes \bar{H}_{V,V'}$, $J=J_{V,V'}$
 all over the domain of $u$. The compactness properties
in~\S\ref{subsec:energy} obviously apply to this case too and, as a
consequence, we deduce that the complex $(C(V,V'), \mu_1)$ satisfies
all the usual properties of the standard Lagrangian Floer complex. We
denote the resulting Floer homology by $HF(V,V')$.

\begin{rem}\label{rem:comp-Floer-compl} 
   Given two cobordisms $V$, $V'$ as above as well as the profile
   function $h:\R^{2}\to \R$, consider the complex $C=CF(V,V';
   (H',h);\widetilde{\mathbf{J}})$ as defined in \cite{Bi-Co:cob1}.
   This complex coincides with a complex $CF(V,V'; \bar{H}, \bar{J})$
   as defined above for appropriate $\bar{H}$, $\bar{J}$.  Thus, for
   cobordisms, the definition of
   the Floer complex in the current paper generalizes the one in \S
   4.3 of \cite{Bi-Co:cob1}.  It is important to note that the
   particular choice of the shape of the profile function away from
   $[0,1]\times \R$, in particular the fact that the bottlenecks
   correspond to local maxima of the functions $h_{\pm}$, implies, as
   shown in \cite{Bi-Co:cob1}, that $HF(V,V)\cong QH(V,\partial V)$
   (here $QH(V, \partial V)$ is understood without grading 
   and over $\Z_{2}$). This isomorphism was obtained in \S 5.2
   \cite{Bi-Co:cob1} by a PSS-argument that makes use of moving
   boundary conditions.  Alternatively, one could use a similar
   PSS-type morphism
   $$PSS: \mathcal{C}(f, J_{V,V})\to CF(V,V; \bar{H}_{V,V}, J_{V,V}),$$
   where $\mathcal{C}(f, J_{V,V})$ is the pearl complex associated to
   a function on $V$ whose negative gradient points outward along the
   boundary of $V$.  The advantage of this morphism is that the
   boundary conditions are now fixed. Either way, by adapting the
   usual argument from the compact case, it is easy to see that this
   PSS morphism is compatible with multiplication. Thus the $\mu_{2}$
   product on $HF(V,V)$ has a unit because $QH(V,\partial V)$ has as
   unit the fundamental class of $V$ relative to $\partial V$.
\end{rem}

\subsubsection*{D. The cobordism Fukaya category} Finally, the
definition of the higher multiplications, $\mu_{k}$, $k \geq 2$, in
$\fuk_{cob}^{d}(\C\times M)$ is given by counting solutions of the
equation (\ref{eq:jhol1}).  By Remark \ref{rem:comp-Floer-compl} the
resulting category is homologically unital.


\subsection{Invariance properties of $\fuk^{d}_{cob}(\C\times M)$}
\label{sb:inv-fuk-2}

Similarly to $\fuk^d(M)$, the Fukaya category of cobordisms
$\fuk^d_{cob}(\mathbb{C} \times M)$ depends on auxiliary choices. The
goal of this section is to explain how to compare these different
$A_{\infty}$-categories. Recall that the auxiliary choices required in
the construction consist of strip-like ends, Floer data for each pair
of objects, compatible perturbation data as well as one additional
choice, specific to our construction, namely the profile function $h:
\mathbb{R}^2 \to \mathbb{R}$ introduced in~\S\ref{subsec:equation}.

Given a profile function $h$ we denote by $\mathcal{I}_{cob}^{h}$ the
set of all the possible regular auxiliary structures as above, with
Floer and perturbation data corresponding to $h$. For $i \in
\mathcal{I}_{cob}^h$ we write $\fuk^d(\mathbb{C} \times M; i,h)$ for
the corresponding Fukaya category (the $h$ in $\fuk^d(\mathbb{C}
\times M; i,h)$ is superfluous since it is already encoded in $i$, but
we write it nevertheless).

The main purpose of this section is to prove the following result.
\begin{prop}\label{prop:inv1} The family of
   categories $\fuk^d_{cob}(\mathbb{C} \times M; i,h)$, $h \in
   \mathcal{H}$, $i \in \mathcal{I}^h_{cob}$ forms a coherent system
   in the sense of~\cite{Se:book-fukaya-categ}. In particular they are
   all quasi-equivalent, and in fact quasi-isomorphic.
\end{prop}

The proof of this proposition will occupy \S\ref{subsubsec:tot-profile},
\S\ref{subsubsec:def-tot}, \S\ref{subsubsec:quasi-eq} and
\S\ref{subsubsec:inf-prof} below.  We then deal with another
invariance issue: our definition of a profile function involves some
additional choices that were made to simplify the compactness arguments in the
construction of $\fuk_{cob}^{d}(\C\times M)$. Mainly, the real
coordinates of the bottlenecks have been fixed at $\frac{3}{2}$ and
$\frac{5}{2}$ - see Figure \ref{fig:kinks}. We will show in
\S\ref{subsubsec:bottle} an invariance result with respect to this
choice. Finally, in \S\ref{subsubsec:ham} we discuss the action of
horizontal Hamiltonian isotopies on the cobordism Fukaya category.

\subsubsection{A {\bf tot} category integrating the profile functions:
  basic ingredients} \label{subsubsec:tot-profile} For a profile
function $h^0 : \mathbb{R}^2 \longrightarrow \mathbb{R}$. We denote by
$\mathcal{H}(h^0)$ the set of all profile functions $h: \mathbb{R}^2
\longrightarrow \mathbb{R}$ which coincide with $h^0$ outside of
$[-2,3] \times \mathbb{R}$ and have the same bottlenecks as $h^{0}$
(this last condition is automatic in view of our current definition of
profile function).

We will now turn the collection of $A_{\infty}$-categories
$\fuk^d_{cob}(\mathbb{C} \times M; i,h)$, $h \in \mathcal{H}(h^0)$, $i
\in \mathcal{I}_{cob}^{h}$, into a coherent system of
$A_{\infty}$-categories in the sense of~\S\ref{sbsb:families-A-infty}.
We follow closely the approach from~\cite{Se:book-fukaya-categ} which
has been also summarized (for $\fuk^d(M)$) in~\S\ref{sbsb:inv-fuk-1},
namely we will construct an $A_{\infty}$-category $\fuk^{d,
  tot}_{cob}(\mathbb{C} \times M; h^0)$ which contains all the
$\fuk^d_{cob}(\mathbb{C} \times M; i,h)$ as full subcategories, and
moreover the embeddings $\fuk^d_{cob}(\mathbb{C} \times M; i,h)
\longrightarrow \fuk^{d, tot}_{cob}(\mathbb{C} \times M; h^0)$ are all
quasi-equivalences.

The construction of $\fuk^{d, tot}_{cob}(\mathbb{C} \times M; h^0)$
will be analogous to the one used for $\fuk^{d}_{cob}(\mathbb{C}
\times M; i, h)$ (as described in~\S\ref{subsec:strip-ends}
-~\S\ref{subsec:fin-cob}) only that we will need more general
perturbation data. Here are the details of the construction.

Fix $k \geq 2$ and $k+1$ cobordisms $V_1, \ldots, V_{k+1} \in
\mathcal{CL}_{d}(\mathbb{C} \times M)$. Fix also $k+1$ profile
functions $h_{V_1, V_2}, \ldots, h_{V_k,V_{k+1}}, h_{V_1, V_{k+1}} \in
\mathcal{H}(h^0)$ and regular Floer data $\mathscr{D}_{V_1,V_2},
\ldots, \mathscr{D}_{V_k, V_{k+1}}, \mathscr{D}_{V_1, V_{k+1}}$ so
that $\mathscr{D}_{V_i, V_{i+1}}$ is defined according to the recipe
from~\S\ref{subsec:equation} using the profile function $h_{V_i,
  V_{i+1}}$, and similarly $\mathscr{D}_{V_1, V_{k+1}}$ is defined
using the profile function $h_{V_1, V_{k+1}}$.

We would like to define a higher composition
\begin{equation} \label{eq:mu-k-tot} \widehat{\mu}_k: CF(V_1, V_2;
   \mathscr{D}_{V_1, V_2}) \otimes \cdots \otimes CF(V_k, V_{k+1};
   \mathscr{D}_{V_k, V_{k+1}}) \longrightarrow CF(V_1, V_{k+1};
   \mathscr{D}_{V_1, V_{k+1}}).
\end{equation}

To  this end, we need to define new perturbation data
$\mathscr{P}_{V_1, \ldots, V_{k+1}} = (\form, \mathbf{J},
\mathbf{h})$. The first two structures are similar to those
in~\S\ref{subsec:equation} and we will indicate the differences later.
The third structure $\mathbf{h}$ consists of a family of profile
functions $\mathbf{h} = \{ h^{r, z}: \mathbb{R}^2 \to \mathbb{R}\}_{z
  \in S_r}$ parametrized by $r \in \mathcal{R}^{k+1}$, $z \in S_r$.
The functions $h^{r,z}$ are required to satisfy the following
properties for every $r \in \mathcal{R}^{k+1}$:
\begin{enumerate}
  \item[i.] $h^{r,z} \in \mathcal{H}(h^0)$ for every $z \in S_r$.
  \item[ii.] For every $z$ in the $i$'th in-going strip-like end of $S_r$
   we have $h^{r,z} = h_{V_i, V_{i+1}}$, $i=1, \ldots, k$, and over
   the out-going strip-like end we have $h^{r,z} = h_{V_1, V_{k+1}}$.
\end{enumerate}
Moreover we require the family $\mathbf{h}$ to be compatible with
splitting and gluing in an obvious way which involves the
corresponding families for lower values of $k$.  \pbred{Note that such
  families $\mathbf{h} = \{ h^{r, z}: \mathbb{R}^2 \to \mathbb{R}\}_{z
    \in S_r}$ always exist because for every $r \in
  \overline{\mathcal{R}}^{k+1}$, $z \in S_r$ the set of choices for
  $h_{r,z}$ is the space $\mathcal{H}(h^0)$ which is contractible, and
  therefore the arguments from~\cite{Se:book-fukaya-categ}
  (Sections~(9g) and~(9i)) easily extend in our case.
  \label{pb:red-1}}

From now on, when the value of $r \in \mathcal{R}^{k+1}$ is obvious we
will omit the $r$ from the notation and simply write $h^z$.

Next we choose a family of $\widetilde{\omega}$-compatible almost
complex structures $\mathbf{J} = \{J_{r,z}\}_{r \in \mathcal{R}^{k+1},
  z \in S_r}$ on $\widetilde{M}$ so that $\pi : \widetilde{M} \to
\mathbb{C}$ is $(J_{r,z}, (\phi_{a_r(z)}^{h^{r,z}})_*(i))$-holomorphic
for every $z \in S_r$ outside $K_{V_1, \ldots, V_{k+1}} \times M$ (we
are using here the notation from~\S\ref{subsec:equation}).

The forms $\form$ are defined in a similar way as
in~\S\ref{subsec:equation} via the formula:
$$\form|_z = da_r \otimes \bar{h}^{r,z} + \form_0, \;\; \forall z \in S_r,$$
where $\form_0$ is the same as in~\S\ref{subsec:equation}.

We will now describe the precise perturbed Cauchy-Riemann equation
relevant for defining the operations $\mu_k$. Let $\tau \in [0,1]$ and
consider the map
$$S_r \longrightarrow \textnormal{Ham}(\mathbb{R}^2), 
\quad z \longmapsto \phi_{\tau}^{h^{r,z}}.$$ Denote by $\beta_{\tau}^r
\in \Omega^1(S_r, C^{\infty}(T \mathbb{R}^2))$ the differential of
this map, viewed as a $1$-form on $S_r$ with values in the space of
Hamiltonian vector fields of $\mathbb{R}^2$.

Given $r \in \mathcal{R}^{k+1}$ we define a $1$-form $\mathcal{Y}$ on
$S_r$ with values in the space of Hamiltonian vector fields of
$\widetilde{M}$, whose value at $z \in S_r$ is given by:
\begin{equation} \label{eq:Y-extended} \mathcal{Y}_z =
   X^{\bar{h}^{r,z}} da_r + \beta^r_{a_r(z)} + Y_0,
\end{equation}
where all the terms on the right hand side are computed at $T_z S_r$.
Here $Y_0 = X^{\form_0}$ is the same as in~\S\ref{subsec:equation}.
Note that if all the profile functions $h_{i, i+1}$ and $h_{1, k+1}$
coincide and the family $\mathbf{h}$ is constant then $\mathcal{Y}$
coincides with the $1$-form $Y$ we had in~\S\ref{subsec:equation}.

We now consider the same equation as~\eqref{eq:jhol1} but with $Y$
replaced by $\mathcal{Y}$, namely:
\begin{equation}\label{eq:jhol-extended} \ u:S_{r}\to \C\times M, \quad Du +
   J(z,u)\circ Du\circ j = \mathcal{Y}_z + 
   J(z,u)\circ \mathcal{Y}_z\circ j, \quad 
   u(C_{i})\subset V_{i} ~.~
\end{equation}

We claim that Lemma~\ref{lem:compact1} continues to hold for solutions
of~\eqref{eq:jhol-extended} (possibly with a different constant $C$).
The proof of this is similar to the proof of Lemma~\ref{lem:compact1}
and uses the following additional ingredients:
\begin{enumerate}
  \item[i.] $\beta_{\tau}^{r} \equiv 0$ on the strip-like ends.
  \item[ii.] For every $z \in S^r$ and $\xi \in T_z S_r$ the vector field
   (on $\mathbb{R}^2$) $\beta_{\tau}^r (\xi)$ vanishes outside of
   $[-2, 3] \times \mathbb{R}$.
  \item[iii.] Instead of the naturality transformation in~\eqref{eq:nat-v}
   we now use $u(z) = \phi^{\bar{h}^{r,z}}_{a_r(z)}(v(z))$.
  \item[iv.] Being profile functions, the restrictions of each of the
   functions $h^{r,z}$ to $(-\infty, -1] \times \{i\}$ has a unique
   critical point at $-\tfrac{3}{2} \times \{i\}$ which is a local
   maximum and similarly the restriction to $[2,\infty) \times \{i\}$
   has a unique critical point at $\tfrac{5}{2}$ which is also a local
   maximum. Moreover, all the $h^{r,z}$ coincide with the same profile
   function $h^0$ outside of $[-2,3]\times \mathbb{R}$.
\end{enumerate}

The energy estimates from Lemma~\ref{lem:energy-bound1} can be
adjusted to this case too and we obtain a uniform energy bound for all
pairs $(r,u)$ with $r \in \mathcal{R}^{k+1}$ and $u$ satisfying
equation~\ref{eq:jhol-extended} with prescribed asymptotic conditions
at the punctures. The proof of this is similar to the proof of
Lemma~\ref{lem:energy-bound1} only that we have to use a uniform bound
on the functions in the family $\mathbf{h}$ and their derivatives.
Finally, Lemma~\ref{lem:symp-area-bound} continues to hold without any
changes.

The above imply the standard compactness result for the solutions
of~\eqref{eq:jhol-extended} with prescribed asymptotic conditions and
prescribed Maslov index. Namely, the space of solutions is compact up
to splitting and bubbling of holomorphic disks and spheres.

In order to define the operations $\widehat{\mu}_k$
from~\eqref{eq:mu-k-tot} we need to choose the perturbation data
$\mathscr{P}_{V_1, \ldots, V_{k+1}} = (\form, \mathbf{J}, \mathbf{h})$
in a consistent way (with respect to splitting and gluing of the
spaces $\mathcal{R}^{k+1}$ and $\mathcal{S}^{k+1}$ for different
values of $k$). \pbred{This is just a straightforward generalization
  of the case we had earlier (for $\fuk^d_{cob}$), when we had only
  $(\form, \mathbf{J})$ together with the fact (explained earlier)
  that the family of profile functions $\mathbf{h} = \{ h^{r, z}:
  \mathbb{R}^2 \to \mathbb{R}\}_{z \in S_r}$ can be chosen to be
  compatible with splitting and gluing.}

\subsubsection{Definition of the {\bf tot} category}
\label{subsubsec:def-tot} We are now ready to define the
$A_{\infty}$-category $\fuk^{d, tot}_{cob}(\mathbb{C}\times M; h^0)$
following the scheme from~\cite{Se:book-fukaya-categ}. The objects of
this category are triples $(V,i, h)$, where $V \in
\mathcal{CL}_{d}(\mathbb{C} \times M)$ is a Lagrangian cobordism, $h
\in \mathcal{H}(h^0)$ and $i \in \mathcal{I}_{cob}^h$. To describe the
morphisms, choose a regular Floer datum
$\mathscr{D}_{(V,i,h'),(W,j,h'')}$ for every pair of objects $(V, i,
h')$, $(W, j, h'')$, in an arbitrary way but subject to the following
two conditions: the Floer datum should be defined using a profile
function from $\mathcal{H}(h^0)$, and moreover, when $(i,h') =
(j,h'')$ then the Floer datum should coincide with the one prescribed
by the index $i=j$ and profile function $h'=h''$. The morphism space
$\hom \bigl( (V,i,h'), (W, j, h'') \bigr)$ is then defined as $CF(V,W;
\mathscr{D}_{(V,i,h'),(W,j,h'')})$ and the differential
$\widehat{\mu}_1$ is precisely the same as for the earlier definition
of $\fuk^d_{cob}(\mathbb{C} \times M)$. Namely, we do {\em not} use
equation~\eqref{eq:jhol-extended} to define $\widehat{\mu}_1$, but
rather equation~\eqref{eq:jhol1} (with $S_r = \mathbb{R} \times [0,1]$
of course).

The higher order compositions $\widehat{\mu}_k$ are defined as
follows. We choose regular consistent perturbation data
$\mathscr{P}_{V_1, \ldots, V_{k+1}} = (\form, \mathbf{J}, \mathbf{h})$
as described above for each $k+1$ tuple of objects $(V_1, i_1, h_1),
\ldots, (V_{k+1}, i_{k+1}, h_{k+1})$ subject to the condition that
when $h_1 = \cdots = h_{k+1} = h$ and $i_1= \cdots = i_{k+1} = i$ then
the family of profile functions $\mathbf{h}$ is constant and equals
$h$ and $(\form, \mathbf{J})$ are the same as prescribed by $i$.  The
operation $\widehat{\mu}_k$ is then defined by counting index-$0$
solutions of~\eqref{eq:jhol-extended} in the standard way. 

The fact that the operations $\widehat{\mu}_k$, $k \geq 1$, satisfy
the $A_{\infty}$-identities can be proved in the same way as for the
categories $\fuk^d_{cob}(\mathbb{C} \times M; i,h)$.

Note that when $\mathbf{h}$ is the constant family (i.e. $h^{r,z} = h$
for every $r, z$) then~\eqref{eq:Y-extended} coincides with the
$1$-forms $Y$ and equation~\eqref{eq:jhol-extended} becomes the same
as~\eqref{eq:jhol1}. It follows that for every $h \in
\mathcal{H}(h^0)$ and $i \in \mathcal{I}_{cob}^h$ we have an obvious
inclusion
\begin{equation} \label{eq:inclusion-fuk-cob} \fuk^d_{cob}(\mathbb{C}
   \times M; i,h) \longrightarrow \fuk^{d, tot}_{cob}(\mathbb{C}\times
   M; h^0)
\end{equation}
which is a full and faithful $A_{\infty}$-functor. Note also that
similarly to $\fuk^d_{cob}(\mathbb{C} \times M; i,h)$ the category
$\fuk^{d, tot}_{cob}(\mathbb{C}\times M; h^0)$ too is homologically
unital.

\subsubsection{The canonical inclusions are quasi-equivalences.}\label{subsubsec:quasi-eq}
It remains to show that the inclusions~\eqref{eq:inclusion-fuk-cob}
are quasi-equivalences. Let $V, W \in \mathcal{CL}_{d}(\mathbb{C}
\times M)$ be two cobordisms. Fix a profile function $h^0$ and let
$\mathscr{D}=\mathscr{D}_{V,W}$, $\mathscr{D}' = \mathscr{D}'_{V,W}$
be two arbitrary regular Floer data for $(V,W)$ both constructed using
profile functions $h, h' \in \mathcal{H}(h^0)$. Choose a perturbation
datum $\mathscr{P} = (\form, \mathbf{J}, \mathbf{h})$ on the strip $Z
= \mathbb{R} \times [0,1]$ which extends $\mathscr{D}$ at the negative
end of $Z$ and extends $\mathscr{D}'$ at the positive end of $Z$.
This perturbation datum $\mathscr{P}$ is constructed exactly as was
done above (for $S_r$ when $k \geq 2$), only that now it is defined
over $Z$.

We define a map
\begin{equation} \label{eq:continu-cob} \Phi_{\mathscr{P}} : CF(V, W;
   \mathscr{D}) \longrightarrow CF(V, W; \mathscr{D}'), \quad
   \Phi_{\mathscr{P}}(\gamma) = \sum_{\gamma'} \#
   \mathcal{M}_0(\gamma, \gamma'; \mathscr{P}) \gamma',
\end{equation}
where $\gamma' \in \mathcal{P}_{\bar{H}'}$ (i.e. $\gamma'$ is a chord
of the Hamiltonian $\bar{H}'$ (of $\mathscr{D}'$) starting on $V$ and
ending on $W$) and $\mathcal{M}_0(\gamma, \gamma'; \mathscr{P})$
stands for the $0$-dimensional component of the space of maps $u: Z
\to \widetilde{M}$ which satisfy the same equation
as~\eqref{eq:jhol-extended} with $S_r$ replaced by $Z$ and with the
boundary and asymptotic conditions replaced by the following:
\begin{equation*}
   \begin{aligned}
      & u(\mathbb{R} \times 0) \subset V, \quad u(\mathbb{R} \times 1)
      \subset W, \\
      & \lim_{s \to -\infty} u(s,t) = \gamma(t), \quad \lim_{s \to
        \infty} u(s,t) = \gamma'(t).
   \end{aligned}
\end{equation*}

The compactness arguments from~\S\ref{subsec:energy} apply to
$\mathcal{M}(\gamma, \gamma'; \mathscr{P})$ and together with standard
transversality arguments show that:
\begin{enumerate}
  \item[i.] $\mathcal{M}_0(\gamma, \gamma'; \mathscr{P})$ is a compact
   $0$-dimensional manifold, hence a finite set. Thus
   $\Phi_{\mathscr{P}}$ is well defined.
  \item[ii.] $\Phi_{\mathscr{P}}$ is a quasi-isomorphism of chain
   complexes.
  \item[iii.] The induced map in homology, which we denote by
   $$\Phi_{\mathscr{D}', \mathscr{D}}: H(CF(V,W;\mathscr{D}))
   \longrightarrow H(CF(V,W;\mathscr{D}')$$ is independent of
   $\mathscr{P}$. Moreover, the maps $\Phi_{\mathscr{D}',\mathscr{D}}$
   (for different $\mathscr{D}$, $\mathscr{D}'$) together with the
   product $\mu_2$ satisfy the TQFT composition laws (as described in
   Section~8b in~\cite{Se:book-fukaya-categ}). In particular, these
   maps are compatible with the triangle product.
\end{enumerate}
We call the maps $\Phi_{\mathscr{P}}$ as well as $\Phi_{\mathscr{D}',
  \mathscr{D}}$, continuation maps or comparison maps.

Next we consider the special case $(V,V)$, i.e. $W = V$.  Let
$\mathscr{D}$, $\mathscr{D}'$, $\mathscr{D}''$ be three Floer data for
the pair $(V,V)$, all constructed using profile function from
$\mathcal{H}(h^0)$. By choosing a perturbation datum
$\mathcal{P}_{\mathscr{D}', \mathscr{D}', \mathscr{D''}}$ on $S_3$,
which is compatible with these three Floer data (as described above)
we obtain the triangle product: $$\widehat{\mu}_2: CF(V,V;
\mathscr{D}) \otimes CF(V,V; \mathscr{D}'') \longrightarrow CF(V,V;
\mathscr{D}').$$ We denote the induced map in homology by $$* :
HF(V,V; \mathscr{D}) \otimes HF(V,V;\mathscr{D}'') \longrightarrow
HF(V,V; \mathscr{D}').$$ The isomorphism $\Phi_{\mathscr{D}',
  \mathscr{D}}$ can be expressed via the product $*$. More precisely,
there exists a canonical element $u \in HF(V,V;\mathscr{D}'')$ such
that
\begin{equation} \label{eq:cont-prod}
   \Phi_{\mathscr{D}', \mathscr{D}}(a) = a * u, \; \; \forall
   a \in HF(V,V;\mathscr{D}).
\end{equation}
The element $u$ can be described as follows. Let $PSS: QH(V, \partial
V) \to HF(V,V; \mathscr{D}'')$ be the PSS isomorphism as described
e.g. in Section~5.2 of~\cite{Bi-Co:cob1}. Recall that $QH(V, \partial
V)$ has a well defined unity $[V]$. The element $u$ is then defined by
$u = PSS([V])$. The proof of~\eqref{eq:cont-prod} is similar to the
case of closed Lagrangian submanifolds. It basically follows from the
fact that $\Phi_{\mathscr{D}', \mathscr{D}}$ is compatible with the
$\mu_2$-product and the TQFT composition identities (see
e.g.~\cite{Se:book-fukaya-categ} Section~8b). These extend to our
setting due to the compactness arguments from~\S\ref{subsec:energy}.
Alternatively, one can construct $u \in HF(V,V;\mathscr{D}'')$ by
counting perturbed holomorphic maps $D \setminus \{1\} \to \mathbb{C}
\times M$ from a disk punctured at one point, where the perturbation
extends the Floer datum $\mathscr{D}''$ on the strip-like end
corresponding to the puncture $1 \in D$. (See Section~8l
of~\cite{Se:book-fukaya-categ} for more on that.)

It follows that for every $V \in \mathcal{C}\mathcal{L}_d(\mathbb{C}
\times M)$, and $h', h'' \in \mathcal{H}(h^0)$, $i' \in
\mathcal{I}_{cob}^{h'}$, $i'' \in \mathcal{I}_{cob}^{h''}$ the objects
$(V, i', h')$ and $(V, i'', h'')$ are isomorphic in the homological
category $H \fuk_{cob}^{d, tot}(\mathbb{C} \times M); h^0)$. Standard
results in homological algebra~\cite{Ge-Ma:methods-hom-alg} imply that
the homological embeddings $$H \fuk^d_{cob}(\mathbb{C} \times M; i, h)
\longrightarrow H \fuk^{d, tot}_{cob}(\mathbb{C} \times M; h^0),$$
associated to~\eqref{eq:inclusion-fuk-cob}, are equivalences for every
$i$ and $h$.  By the discussion in~\S\ref{sbsb:families-A-infty} (see
also~\cite{Se:book-fukaya-categ} for more details) this gives the
family $\fuk_{cob}^d(\mathbb{C} \times M; i, h)$, $h \in
\mathcal{H}(h^0)$, $i \in \mathcal{I}_{cob}^{h}$, the structure of a
coherent system of $A_{\infty}$-categories. In particular we obtain
for every $h', h'' \in \mathcal{H}(h^0)$ and $i' \in
\mathcal{I}_{cob}^{h'}$, $i'' \in \mathcal{I}_{cob}^{h''}$ a
quasi-equivalence
\begin{equation} \label{eq:quasi-equiv-fuk-cob}
   \mathcal{F}_{(i'',h''), (i',h')} : \fuk_{cob}^d(\mathbb{C}\times M;
   i', h') \longrightarrow \fuk_{cob}^d(\mathbb{C}\times M; i'', h'')
\end{equation}
which is canonical in homology and satisfies $\mathcal{F}_{(i'',h''),
  (i',h')}(V) = V$ for every object $V$. {(In fact
  $\mathcal{F}_{(i'',h''), (i',h')}$ is a quasi-isomorphism.)}
Moreover, for every two cobordisms $V$ and $W$, the action of the
degree $1$ component of $\mathcal{F}_{(i'',h''), (i',h')}$ on
morphisms, namely $\mathcal{F}^{V,W}_{(i'',h''), (i',h')}: CF(V,W;
\mathscr{D}') \to CF(V,W; \mathscr{D}'')$, coincides in homology with
the continuation map $\Phi_{\mathscr{D}'', \mathscr{D}'}$. Here
$\mathscr{D}'$, $\mathscr{D}''$, are the Floer data for $(V,W)$
corresponding to the choices $i'$ and $i''$ respectively.

\subsubsection{ Changing the profile function at infinity.}
\label{subsubsec:inf-prof}
Summarizing the construction till now, we have an
$A_{\infty}$-category $\fuk^d_{cob}(\mathbb{C} \times M; h)$ which is
well defined up to {quasi-isomorphism}, but still depends on the
choice of the profile function at $\pm \infty$. It remains to get rid
of this dependence.

Denote by $\mathcal{H}$ the space of all profile functions (with all
possible constants $\alpha^{\pm}$, $\beta^{\pm}$, but subject to the
conditions described in~\S\ref{subsec:equation}, in particular
$\alpha^- > 0$, $\alpha^+ <0$ etc.). Let $h^0, h^1 \in \mathcal{H}$.
Denote by $\alpha^{\pm}_i$, $i=0,1$, the constants defining $h^i$ on
$(-\infty, 2] \times \mathbb{R}$ and $[3,\infty) \times \mathbb{R}$.
Fix $0< \epsilon < 1/100$ and choose a function $\sigma: \mathbb{R}
\to \mathbb{R}$ with the following properties:
\begin{enumerate}
  \item $\sigma(x) = 1$ for every $x \in [-\tfrac{3}{2}-\epsilon,
   \tfrac{5}{2}+\epsilon]$.
  \item $\sigma(x) = \tfrac{\alpha_1^-}{\alpha_0^-}$ for $x \leq -2$
   and $\sigma(x) = \tfrac{\alpha_1^+}{\alpha_0^+}$ for $x \geq 3$.
\end{enumerate}
Using this function we can construct a bijection $$\tau_{h^1,h^0}:
\mathcal{H}(h^0) \longrightarrow \mathcal{H}(h^1)$$ which is uniquely
characterized by the following property. For every $h \in \mathcal{H}(h^0)$
\begin{equation} \label{eq:tau-10}
   \begin{aligned}
      & (\tau_{h^1,h^0}h)(x,i) = h(x,i), \quad \forall x \in
      [-\tfrac{3}{2}-\epsilon,
      \tfrac{5}{2}+\epsilon], i \in \mathbb{Z}, \\
      & \frac{\partial (\tau_{h^1,h^0}h)}{\partial x} (x,i) = \sigma(x)
      \frac{\partial h}{\partial x} (x,i), \quad \forall x \in
      \mathbb{R}, i \in \mathbb{Z}.
   \end{aligned}
\end{equation}

For every $h \in \mathcal{H}(h^0)$, the bijection $\tau_{h^1,h^0}$ induces
in an obvious way a bijection $\tau_{h^1,h^0} : \mathcal{I}_{cob}^h
\longrightarrow \mathcal{I}_{cob}^{\tau_{h^1,h^0}h}$ which we also denote
by $\tau_{h^1,h^0}$. Note that the Hamiltonian chords corresponding to $i
\in \mathcal{I}_{cob}^h$ and to $\tau_{h^1,h^0}i \in
\mathcal{I}_{cob}^{\tau_{h^1,h^0} h}$ coincide. Moreover, due
to~\eqref{eq:tau-10} and in view of Lemma~\ref{lem:compact1} the
equation~\eqref{eq:jhol1} has precisely the same solutions for the
choice of data $i$ and $\tau_{h^1,h^0}i$. Thus the $A_{\infty}$-categories
$\fuk^d_{cob}(\mathbb{C} \times M; i; h)$ and $\fuk^d_{cob}(\mathbb{C}
\times M; \tau_{h^1,h^0}i; \tau_{h^1,h^0}h)$ can be identified by an obvious
chain-level isomorphism induced from $\tau_{h^1,h^0}$:
$$T^{i,h}_{h^1,h^0}: \fuk^d_{cob}(\mathbb{C}\times M; i, h)
\longrightarrow \fuk^d_{cob}(\mathbb{C}\times M; \tau_{h^1,h^0}i,
\tau_{h^1,h^0}h).$$ The action of $T_{h^1,h^0}$ on objects is of course the
``identity''.

A slight variation on Lemma~\ref{lem:compact1} shows that the
isomorphisms $T_{h^1,h^0}$ are compatible with the
quasi-equivalences~\eqref{eq:quasi-equiv-fuk-cob} in the sense that
for every $h', h'' \in \mathcal{H}(h^0)$ and $i' \in
\mathcal{I}^{h'}_{cob}$, $i'' \in \mathcal{I}^{h''}_{cob}$ we have:
$$T^{i'',h''}_{h^1,h^0} \circ \mathcal{F}_{(i'',h''), (i',h')} =
\mathcal{F}_{(\tau_{h^1,h^0} i'',\tau_{h^1,h^0} h''),
  (\tau_{h^1,h^0}i',\tau_{h^1,h^0}h')} \circ T^{i',h'}_{h^1,h^0}.$$

Finally, it follows from the definitions that the isomorphisms
$T^{i,h}_{h^1,h^0}$ have also the following compatibility properties:
$$T^{\tau_{h^1, h^0}i, \tau_{h^1, h^0}h}_{h^2, h^1} \circ
T^{i,h}_{h^1, h^0} = T^{i,h}_{h^2, h^0}.$$

This concludes the proof of Proposition \ref{prop:inv1}.  In view of
this result we will denote by abuse of notation each of the categories
$\fuk^d_{cob}(\mathbb{C} \times M; i,h)$ by $\fuk^d_{cob}(\mathbb{C}
\times M)$. However, notice for further use that each time an explicit
geometric construction takes place we generally need to indicate what
are the particular choices of almost complex structures, profile
functions etc.

\begin{rem}
   In the last step in the discussion above we have made a chain level
   comparison between $\fuk^d_{cob}(\mathbb{C} \times M; i, h)$ for
   certain pairs of functions $h \in \mathcal{H}$ which do not
   coincide at infinity. There is an alternative way to establish this
   comparison which in fact generalizes the construction of the quasi-equivalences
   $\mathcal{F}_{(i'',h''), (i',h')}$ described above. This goes as
   follows. One constructs directly an $A_{\infty}$-category
   $\widehat{\fuk}^{d, tot}_{cob}(\mathbb{C} \times M)$ which is
   defined in a similar way to $\fuk^{d,tot}_{cob}(\mathbb{C} \times
   M; h)$, with the following changes. One allows Floer data to have
   profile functions which do not necessarily coincide at infinity.
   Moreover, one takes perturbation data $\mathscr{P} = (\form,
   \mathbf{J}, \mathbf{h})$, where the family of profile functions $h
   \in \mathcal{H}$ interpolates the given choices of profile
   functions on the ends of the punctured disks $S_r$. The operations
   $\widehat{\mu}_k$ from~\eqref{eq:mu-k-tot} are then defined via the
   same equation~\eqref{eq:jhol-extended}. The only thing left to be
   checked is whether compactness holds for the spaces of solutions
   of~\eqref{eq:jhol-extended}. It turns out that a rather
   straightforward generalization of Lemma~\eqref{lem:compact1} holds
   for these solutions too. The main point is that the constants
   $\alpha^{\pm}$ (for the different profile functions involved) have
   the same signs and are small enough.
\end{rem}


\subsubsection{Independence of bottleneck position.}\label{subsubsec:bottle}
In this subsection we discuss the independence of the cobordism Fukaya
category on the positions along the real axis of the bottlenecks of
the profile function.  It is clear that our choice to put the
bottlenecks at the points $-\frac{3}{2}$ and $\frac{5}{2}$ is
completely arbitrary.  We get rid of this dependence below.

Denote by $\mathcal{C}\mathcal{L}_{d}^{[a,b]}(\C\times M)$ the
Lagrangians in $\C\times M$ that are $\R$-extensions of cobordisms
$V\subset [a,b]\times \mathbb{R} \times M$ with $-\infty< a<b <
\infty$, that satisfy condition~\eqref{eq:Hlgy-vanishes} and are
uniformly monotone with $d_{V}=d$.  For instance,
$\mathcal{C}\mathcal{L}^{[0,1]}_{d} (\C\times
M)=\mathcal{C}\mathcal{L}_{d}(\C\times M)$.  For fixed $c\leq a\leq 0
< 1\leq b\leq d$ we define a coherent system of categories $\fuk_{
  [a,b]}^{d}(\C\times M; [c,d])$ by the same method used before in
this section to define $\fuk^{d}_{cob}(\C\times M)$.  Their role is to
integrate all possible bottleneck positions for the profile functions.

The objects of $\fuk_{ [a,b]}^{d}(\C\times M; [c,d])$ are the
Lagrangians in the set $\mathcal{C}\mathcal{L}_{d}^{[a,b]}(\C\times
M)$. The morphisms and the multiplications $\mu_k$ are defined just as
in our construction of $\fuk^{d}_{cob}(\C\times M)$ but this time we
use profile functions $h^{[c,d]}$ similar to the ones considered till
now but rescaled to the interval $[c,d]$: thus such an $h^{[c,d]}$
satisfies the properties given in the definition of the function $h$
in \S\ref{subsec:equation} but relative to the interval $[c,d]$
instead of the interval $[0,1]$. For instance, we use the sets,
$$W_{i}^{+}= [d+1,\infty)\times [i-\epsilon, i+\epsilon] 
   \quad \textnormal{and } \; W_{i}^{-}= (-\infty, -1+c]\times
   [i-\epsilon, i+\epsilon], \; i \in \mathbb{Z}~,~$$
   and the functions $h_{\pm}$ have respectively critical points 
   \pbred{at $\frac{5}{2}+d$ and $-\frac{3}{2}+c$.}
In particular, $\fuk^{d}_{cob}(\C\times M)= \fuk^{d}_{[0,1]}(\C\times M; [0,1])$. 

It will be shown below that all these categories are in fact
{quasi-isomorphic}. Clearly, there is an inclusion of
 (coherent systems of) $A_{\infty}$ categories:
$$\widetilde{Id}_{a,b}:\fuk^{d}_{[0,1]}(\C \times M; [c,d])
\to \fuk^{d}_{[a,b]}(\C \times M; [c,d]).$$

\begin{prop}\label{prop:eq2} With the notation above we have:
   \begin{itemize}
     \item[i.] Any two cobordisms $V,V'\in
      \mathcal{CL}_{d}^{[a,b]}(\C\times M)$ that are horizontally
      isotopic are isomorphic in the homological category
      $H(\fuk^{d}_{[a,b]}(\C \times M); [c,d])$.  As a consequence
      $\widetilde{Id}_{a,b}$ is a quasi-equivalence.
     \item[ii.] Any two categories $\fuk_{[a,b]}^{d}(\C\times M;
      [c,d])$ are canonically quasi-equivalent,  independently of
      $c\leq a\leq 0 < 1\leq b\leq d$.
   \end{itemize}
\end{prop}
\begin{proof}
   Point {\em i} is very similar to the argument at the beginning of
   \S\ref{subsubsec:quasi-eq} except that we not only modify the data
   $\mathscr{D}$ but rather isotope from $V$ to $V'$.  Suppose that
   $V$ and $V'$ are horizontally isotopic via a horizontal isotopy
   $\Psi$.  As showed in \cite{Bi-Co:cob1} there is an associated
   isomorphism $\overline{\Psi}:HF(V,V)\to HF(V,V')$. Let $\alpha =
   \overline{\Psi}(u)$ where $u\in HF(V,V)$ is, as in
   \S\ref{subsubsec:quasi-eq}, the unit (that is the image of the
   fundamental class by the PSS-morphism $QH(V,\partial V)\to
   HF(V,V)$). A similar class $\beta$ can be defined in $HF(V',V)$.
   These classes satisfy $\alpha\ast \beta= u$ and $\beta\ast
   \alpha=u'$ thus showing that $V$ and $V'$ are isomorphic in the
   homology category as claimed. In homology the functor
   $\widetilde{Id}_{a,b}$ is full and faithful and as any cobordism in
   $\mathcal{CL}^{[a,b]}_{d}(\C\times M)$ can be rescaled by a
   horizontal isotopy to $[0,1]\times \R\times M$, it follows from the
   first part of the proof that $\widetilde{Id}_{a,b}$ is a
   quasi-equivalence.

   We now prove point {\em ii}.  In view of {\em i}, it is clearly
   enough to prove that given $c\leq 0, 1\leq d$ there is a quasi-equivalence
   of the two categories $\fuk^{d}_{[0,1]}(\C\times M; [0,1])$ and
   $\fuk^{d}_{[0,1]}(\C\times M; [c,d])$. For this we will use a
   special ambient hamiltonian isotopy induced by a hamiltonian
   isotopy in the plane, $\xi_{t}$, with the property that: $\xi_{t}
   :\C\to \C$ is a translation of the form $(x,y)\to (x+tv_{-},y)$ for
   $(x,y)\in (-\infty, \epsilon]\times \R$ and $(x,y)\to (x+tv_{+},
   y)$ for $(x,y)\in [1-\epsilon, \infty)\times \R$ with $\epsilon$
   very small and $v_{-}<0$ and $v_{+}>0$. Moreover, the choices of
   $v_{+}$ and $v_{-}$ are so that some fixed profile function
   $h_{0}=h^{[0,1]}$ has the property that $h_{0}\circ \xi_{1}^{-1}$
   is a profile function of the type $h^{[c,d]}$.
   We will now construct an appropriate {\bf tot} category
   $\fuk^{d}_{[0,1]}(\C\times M)^{\xi}$. The objects of this category
   are pairs $(V,t)$ where $V\in\mathcal{CL}_{d}^{[0,1]}(\C\times M)$
   and $t\in [0,1]$.  We now fix Floer and perturbation data for each
   fixed $t\in [0,1]$ in an arbitrary way except that at each $t$ we
   use as profile function $h^{t}=h_{0}\circ \xi_{t}^{-1}$.  This
   produces categories $\fuk^{d}_{[0,1]}(\C\times M; h^{t})$. This
   data, now defined for each fixed value of $t$ has now to be
   extended to more general families of the form $(V_{i}, t_{i})$ with
   different parameters $t_{i}$. The construction is similar to the
   construction of the {\bf tot} category in
   \S\ref{subsubsec:tot-profile} and \S\ref{subsubsec:def-tot}. We
   first fix the specific class of profile functions that we will use:
   $\mathcal{H}^{\xi}=\{h^{t} : \ t\in [0,1]\}$.  We then use
   $\mathcal{H}^{\xi}$ instead of $\mathcal{H}(h^{0})$ in the
   construction in \S\ref{subsubsec:tot-profile}.  In particular, we
   consider families $\bar{h}^{r,z}$ so that conditions i and ii after
   equation (\ref{eq:mu-k-tot}) are satisfied with $\mathcal{H}^{\xi}$
   instead of $\mathcal{H}(h^{0})$, in particular $h^{r,z}\in
   \mathcal{H}^{\xi}$ for each $r\in \mathcal{R}^{k+1}$ and $z\in
   S_{r}$. From here on the construction is just like in
   \S\ref{subsubsec:def-tot}: for each pair $(V, t_{1}), (W, t_{2})$
   we select Floer data $\mathscr{D}_{(V,t_{1}), (W,t_{2})}$ up to our
   usual conditions but so that the profile function is selected from
   $\mathcal{H}^{\xi}$ and, whenever $t_{1}=t_{2}$, the datum and the
   profile function coincides with the one prescribed for $t_{1}$.  We
   then define the higher compositions for each $k+1$-tuple of objects
   $(V_{1}, t_{1}),\ldots (V_{k+1}, t_{k+1})$ in the usual way but
   subject to the condition that the profile function is a family
   $\mathbf{h}=\{h^{r,z}\}$ as described above and is so that if
   $t_{1}=t_{2}=\ldots t_{k+1}$ then $h^{r,z}$ is the constant family
   equal to $h^{t_{1}}$ and, similarly, the rest of the data coincides
   with the one in $\fuk^{d}_{[0,1]}(\C\times M;h^{t_{1}})$.  On the
   technical level there are two points to check to finish the
   construction of $\fuk^{d}_{[0,1]}(\C\times M)^{\xi}$. The first
   concerns the localization of curves in the relevant moduli spaces -
   the analogue of Lemma \ref{lem:compact1}. The second second point
   is to verify energy bounds as in Lemma \ref{lem:energy-bound1}.
   This second point does not present any particular difficulty but
   the first one requires a new argument. The reason is that the
   functions $h^{t}$ have bottlenecks at different points along the
   real axis - these bottlenecks are simply translations along
   $\R\subset \C$ of the bottlenecks of $h$ but for various values of
   $t$ these are different points.  Thus the reasoning used in
   \S\ref{subsubsec:tot-profile} requires an adaptation that we now
   describe.  Consider a solution $u$ of the equation
   (\ref{eq:jhol-extended}) but, of course, with $h^{r,z}\in
   \mathcal{H}^{\xi}$.  Denote by $\psi_{t}$ the lift of $\xi_{t}$ to
   $\C\times M$.  We will again use a naturality transformation
   $u(z)=\phi_{a_{r}(z)}^{\bar{h}^{r,z}}(v(z))$ (we recall
   $\bar{h}^{r,z}=h^{r,z}\circ \pi$). In this case $\bar{h}^{r,z}$ can
   be written as $\bar{h}^{b_{r}(z)}$ where $b_{r}:S_{r}\to [0,1]$
   and, further, $\bar{h}^{b_{r}(z)}=\bar{h}_{0}\circ
   \psi_{b_{r}(z)}^{-1}$. By an elementary calculation we have:
   $$\phi_{\tau}^{\bar{h}^{r,z}}={\psi}_{b_{r}(z)} 
   \circ \phi_{\tau}^{\bar{h}_0}\circ {\psi}_{b_{r}(z)}^{-1}\ , \
   \forall \tau~.~$$ We now apply yet another naturality
   transformation to $v$ by putting
   $$w(z)={\psi}_{b_{r}(z)}^{-1} v(z)~.~$$
   We now discuss the properties of the planar maps $w'(z)=\pi \circ
   w(z)$. Because $\xi_{t}$ is a translation (for all $t$) in the real
   direction in the exterior of the region $[\epsilon, 1-\epsilon]$,
   the holomorphicity of $\pi\circ v$ around and on the sides of the
   bottlenecks is still true for $w$. Moreover, the boundary
   conditions for $w$ are such that for $w$ only the bottlenecks of
   $h_{0}$ itself are involved. In other words, the argument in Lemma
   \ref{lem:compact1} is applicable to $w$ and concludes the
   verification.

   The usual arguments used in the construction of the {\bf tot}
   category show that the ``fibre'' inclusion:
   $$\fuk^{d}_{[0,1]}(\C\times M;h^{t})\to \fuk^{d}_{[0,1]}(\C\times M)^{\xi} $$
   is a quasi-equivalence for all $t\in [0,1]$ and this implies that
   the first fibre which is $\fuk^{d}_{[0,1]}(\C\times M; [0,1])$ and
   the last one which is $\fuk^{d}_{[0,1]}(\C\times M; [c,d])$ are
   quasi-equivalent.
\end{proof}


\subsubsection{Action of horizontal Hamiltonian isotopies.}\label{subsubsec:ham}
Here we analyze the action of certain Hamiltonian diffeomorphisms
$\phi :\C\times M\to \C\times M$ on $\fuk^{d}_{cob}(\C\times M)$.  The
construction is parallel to a construction
in~\cite{Se:book-fukaya-categ}, specifically \S 10d, with some
modifications that we will outline.  The Hamiltonian diffeomorphisms
$\phi$ of interest to us satisfy a strong version of horizontality
that we define below.

\begin{dfn}\label{def:str-hor} Fix a hamiltonian isotopy $\phi_{t}$ ,
   $t\in [0,1]$, on $\C\times M$. Let $X_{t}$ be the time dependent
   vector field of $\phi_{t}$.  The isotopy $\phi_{t}$ is called an
   {\em ambient horizontal hamiltonian} isotopy if it has the property
   that there are sufficiently large constant $R_{\phi}, R'_{\phi}> 0$
   so that for each $t\in [0,1]$ there are two other {\em real}
   numbers $v^{t}_{-} , v^{t}_{+} \in \R$ so that for all $z=(x+iy,
   m)\in \C\times M$, we have \pbredb{$$X_{t}(z)= (v^{t}_{-}, 0) \in\C
     \oplus T_m M \ \textnormal{if \ } x \leq - R'_{\phi} \ , \
     X_{t}(z)= (v^{t}_{+},0) \in \C \oplus T_m M \ \textnormal{if \ }
     x \geq R_{\phi}~.~$$}
\end{dfn}

The definition above implies that an ambient horizontal Hamiltonian
isotopy moves all the ends of a Lagrangian cobordism by sliding them
along themselves.  Such an isotopy is horizontal in the sense of
Definition \ref{d:isotopies} with respect to any cobordism $V$.

The group of Hamiltonian diffeomorphisms $\phi$ that are time-$1$ maps
of ambient horizontal Hamiltonian isotopies will be denoted by
$Ham_{hor}(\C\times M)$.  This group only admits a partially defined
action on $\fuk_{cob}^{d}(\C\times M)$) because the objects of
$\fuk_{cob}^{d}(\C\times M)$ are $\R$-extensions of cobordisms in
$V\subset ([0,1]\times \R)\times M$ and $Ham_{hor}(\C\times M)$ does
not preserve this class of cobordisms.

However, with the notation on \S\ref{subsubsec:bottle}, notice that
for $|a|, |b|, |c|, |d|$ sufficiently large there are constants
$c',d'$ depending on $\phi$ and an $A_{\infty}$ functor
$$\widetilde{\phi}_{[a,b]}:  \fuk^{d}_{[0,1]}(\C \times M; [c,d])
\to \fuk^{d}_{[a,b]}(\C \times M;[c',d'])$$ with action on objects $V
\to \phi(V)$ and whose value on morphisms is defined by transporting
intersection points by $\phi$.  For the (higher) multiplication in the
target category, we use almost complex structures obtained by
transporting through $\phi$ the structures in the domain category.
Similarly, we use in the target a profile function $\bar{h}\circ
\phi^{-1}$ that corresponds to the profile function $h=h^{[c,d]}$ used
on the domain.  (Here $\bar{h}=\pi\circ h$.) This also determines the
values of $c'$ and $d'$. In both the case of the almost complex
structures and of the profile function, this construction is possible
because $\phi$ is horizontal so, for instance, it preserves the class
of almost complex structures for which the projection $\pi:\C\times
M\to \C$ is holomorphic at infinity.  There is an abuse of notation
here because the function $\bar{h}\circ\phi^{-1}$ is defined on
$\C\times M$ and not on $\C$ as usual profile functions are. However,
given that $\phi$ is a translation \pbred{in the real direction} in
$\C$ and by taking the bottlenecks of $h$ far enough, it is easy to
arrange that for $(x,m)\in \C\times M$ with $|Re(x)|$ large enough, we
have $\bar{h}\circ\phi^{-1}(x,m)=h\circ \tilde{\phi}^{-1}(x)$ where
$\tilde{\phi}$ is a horizontal translation of \pbred{speed
  $v^{1}_{\pm}$.}  In short, all arguments involving usual profile
functions can be applied as well to \pbred{the functions
  $\bar{h}^{t}=\bar{h}\circ\phi_t^{-1}$}.  In this subsection we will
further refer to these functions as {\em generalized profile
  functions}. The categories associated to them fit in the same
coherent system as described in~\S\ref{sb:inv-fuk-2} using the same
arguments as there.

It is easy to check that the functor described above is well-defined
and an inclusion.  Notice moreover that in view of Proposition
\ref{prop:eq2}, by composing with the appropriate quasi-equivalences,
we could treat simply $\widetilde{\phi}_{[a,b]}$ as a functor with
domain $\fuk^{d}_{[0,1]}(\C\times M; [c,d])$ and target
$\fuk^{d}_{[a,b]}(\C\times M; [c,d])$. In particular,
$\widetilde{\phi}_{[a,b]}$ and $\widetilde{Id}_{[a,b]}$ can be viewed
as functors between the same $A_{\infty}$-categories.

\begin{prop}\label{prop:ham}
   For any $\phi\in Ham_{hor}(\C\times M)$, and $|a|, |b|, |c|, |d|$
   large enough so that the functor $\widetilde{\phi}_{[a,b]}$ is well
   defined, we have that $\widetilde{\phi}_{[a,b]}$ is a
   quasi-equivalence and is quasi-isomorphic to
   $\widetilde{Id}_{[a,b]}$.
\end{prop}
\begin{proof} 
   The proof is based on the following {\bf tot} type construction.
   Fix an ambient horizontal Hamiltonian isotopy $\psi_{t}$, $t\in
   [0,1]$ so that $\phi=\psi_{1}$ and assume that $|a|, |b|, |c|, |d|$
   are sufficiently large so that the functor
   $\widetilde{\phi}_{[a,b]}$ is well defined for some $c',d'$. Fix
   also a profile function $h=h^{[c,d]}$.  We assume moreover that $h$
   is such that the compositions $\bar{h}\circ \psi_{t}^{-1}$ produce
   generalized profile functions giving rise to the category
   $\fuk^{d}_{[a,b]}(\C\times M; [c_{t}, d_{t}])$ for all $t\in
   [0,1]$, which belong to the coherent system.

   This can be easily achieved by assuming that the bottlenecks of $h$
   are far enough from the origin - in other words by taking $|c|,
   |d|$ large enough. Thus the functor $\widetilde{\phi}_{[a,b]}$ is
   well defined. It is easy to see that it is a quasi-equivalence: in
   homology it is faithful and full and each object in the target
   category is isomorphic to some object in the image of
   $\widetilde{\phi}_{[a,b]}$.

   We now show that $\widetilde{\phi}_{[a,b]}$ and
   $\widetilde{Id}_{[a,b]}$ are quasi-isomorphic as functors. We will
   first construct a category $\fuk^{d}_{\psi}(\C\times M)$.  The
   objects of this category are pairs $(V,t)$ where
   $V\in\mathcal{CL}_{d}^{[0,1]}(\C\times M)$ and $t\in [0,1]$.
   Geometrically $(V,t)$ is interpreted as $(V,t)=\psi^{t}\circ V$. We
   first fix the Floer and perturbation data for the category
   $\fuk^{d}_{[0,1]}(\C\times M;[c,d])$. We use in the construction
   the profile function $h$ fixed above.  This produces a specific
   $A_{\infty}$-category that we denote by $\fuk^{d}(\C\times M;
   i_{0})$. We now define categories $\fuk^{d}(\C\times M; i_{t})$
   that have the same objects as $\fuk^{d}(C\times M;i_{0})$ but with
   Floer and perturbation data transported from the choices in $i_{0}$
   by the isotopy $\psi_{t}$ - this uses the writing
   $(V,t)=\psi^{t}\circ V$. Simultaneously, in the definition of
   $i_{t}$ we use the generalized profile function, as explained
   above, $\bar{h}^{t}=\bar{h}\circ \psi_{t}^{-1}$. It is easy to see
   that this is compatible with all the relevant equations, in
   particular with~\eqref{eq:Y-extended}. This means, in particular,
   that there is a canonical identification between the categories
   $\fuk^{d}(\C\times M; i_{0})$ and $\fuk^{d}(\C\times M; i_{t})$ for
   all $t\in [0,1]$. To define $\fuk^{d}_{\psi}(\C\times M)$ we use
   for all families of objects $(V_{i}, t)$ with the same parameter
   $t$, the data being already collected in $i_{t}$.  We then extend
   the data to more general families in the usual way by using profile
   functions selected from the family
   $\mathcal{H}^{\psi}=\{\bar{h}\circ \psi_{t}^{-1} : \ t\in [0,1]\}$.
   We then pursue as in the proof of Proposition~\ref{prop:eq2} ii but
   with $\mathcal{H}^{\psi}$ instead of $\mathcal{H}(\xi)$.  It easy
   to see that, because $|a|, |b|, |c|, |d|$ are large enough all the
   arguments in the proof of Proposition~\ref{prop:eq2} apply also to
   our generalized profile functions.

   The usual arguments used in the construction of the {\bf tot}
   category show that the ``fibre'' inclusion:
   $$\fuk^{d}(\C\times M; i_{t})\to \fuk^{d}_{\psi}(\C\times M)$$
   is a quasi-equivalence for all $t\in [0,1]$.

   The next step in the proof is to construct in a similar way an
   $A_{\infty}$-category $\fuk^{d}_{[a,b]}(\C\times M)^{\psi}$.  The
   objects of this category are pairs $(V,t)\in
   \mathcal{CL}_{d}^{[a,b]}(\C\times M)\times [0,1]$. The Floer and
   perturbation data are defined in such a way that the following two
   conditions are satisfied:
\begin{itemize}
  \item[i.] for objects $(V_{1}, t),\ldots (V_{k+1}, t)$ the profile
   function in use is $\bar{h}^{t}$ and, for the rest, the
   perturbation choices etc. are as required to define the
   $A_{\infty}$-category $\fuk^{d}_{[a,b]}(\C\times M)$. We denote
   this category by $\fuk^{d}_{[a,b]}(\C\times M; t)$
  \item[ii.] the choices of Floer data, perturbation data and profile
   functions are so that the map
   $$Ob(\fuk^{d}_{\psi}(\C\times M)) \ni (V,t) \longrightarrow
   (\psi^{t}\circ V, t)\in Ob (\fuk^{d}_{[a,b]}(\C\times M)^{\psi})$$
   extends to an inclusion of $A_{\infty}$-categories.
\end{itemize}

Thus, $\fuk^{d}_{[a,b]}(\C\times M)^{\psi}$ is the analogue of the
category $\fuk^{d}_{[0,1]}(\C\times M)^{\xi}$ from the proof of
Proposition \ref{prop:eq2} with the additional restriction in the
choice of perturbation data as indicated at the point {\em ii} above.

To summarize what was achieved till now, we have constructed an
$A_{\infty}$ inclusion
$$\psi^{\ast}:\fuk^{d}_{\psi}(\C\times M)\to \fuk^{d}_{[a,b]}(\C\times M)^{\psi}$$
that has the property that $\psi^{\ast}$ \pbred{restricted to $i_0$,
  $\fuk^{d}(\C\times M; i_{0})\to \fuk^{d}_{[a,b]}(\C\times M; 0)$, is
  the inclusion $\widetilde{Id}_{[a,b]}$ and its restriction to $i_1$,
  $\fuk^{d}(\C\times M; i_{1})\to \fuk^{d}_{[a,b]}(\C\times M; 1)$, is
  $\widetilde{\phi}_{[a,b]}$.} Again, the usual arguments used in the
construction of the {\bf tot} category show that the ``fibre''
inclusions:
$$ \fuk^{d}_{[a,b]}(\C\times M; t)\to \fuk^{d}_{[a,b]}(\C\times M)^{\psi}$$  
are quasi-equivalences for all $t\in [0,1]$. \pbred{This shows that
  $\widetilde{\phi}_{[a,b]}$ is a quasi-equivalence (and the same for
  $\widetilde{Id}_{[a,b]}$, again). Recall also that (by composing
  with a certain quasi-equivalence) we can view
  $\widetilde{\phi}_{[a,b]}$ and $\widetilde{Id}_{[a,b]}$ as two
  functors between the same $A_{\infty}$-categories.}

\pbred{It remains to show that these two functors are
  quasi-isomoprhic, in the category of $A_{\infty}$-functors, i.e.
  there exists a natural transformation of $A_{\infty}$-functors
  between $\widetilde{\phi}_{[a,b]}$ and $\widetilde{Id}_{[a,b]}$
  which is a quasi-isomorphism. This can be proved using a moving
  boundary construction (based on the isotopy $\{\psi_s\}_{s \in
    [0,1]}$) which follows eactly the argument in Section~10(c)
  in~\cite{Se:book-fukaya-categ} (see also Proposition~10.3 in that
  section).} \pbred{Apart of working at all times with extended
  profile functions selected from $\mathcal{H}^{\psi}$ the argument
  from Secction~10(c) in~\cite{Se:book-fukaya-categ} applies here with
  almost no adjustments. We therefore skip the remaining details.}
\end{proof}

\begin{rem} \label{rem:equivariant} a. It is also possible to show
   that, in homology, the quasi-isomorphism between
   $\widetilde{\phi}_{[a,b]}$ and $\widetilde{Id}_{[a,b]}$ only
   depends on the homotopy class (with fixed end-points) of the path
   of Hamiltonian diffeomorphisms $\phi_{t}$, $t\in [0,1]$.

b. Notice that, except for the
   profile functions which are specific to our setting, the
   construction of $\fuk^{d}(\C\times M)^{\psi}$ is slightly simpler
   than the one in \S 10d \cite{Se:book-fukaya-categ}. This is because
   we are not actually constructing here an action of
   $Ham_{hor}(\C\times M)$ on $\fuk^{d}_{cob}(\C\times M)$.
   Constructing such an action is in fact possible, by a certain
   direct limit $\lim_{[a,b] \subset
     \mathbb{R}}\fuk^{d}_{[a,b]}(\C\times M)$ but goes beyond what is
   needed in this paper.
\end{rem}

\section{Proof of the main theorem} \label{s:proof-main}
The heart of the proof of Theorem~\ref{thm:main} consists in
constructing a sequence of exact triangles associated to a cobordism.
Before the proof, some preparation is required.


\subsection{Enrichment of $\fuk^d_{cob}(\mathbb{C} \times M)$}
\label{sb:enrich}

For convenience, we will use an enrichment,
$\fukcben(\mathbb{C} \times M)$, of the $A_{\infty}$-category
$\fukcb(\mathbb{C}\times M)$.

Fix a choice of all the auxiliary structures $(i,h)$ needed to define
a concrete model of $\fukcb(\mathbb{C} \times M)$, namely a profile
function $h$, and a choice $i$ of consistent Floer and perturbation
data as in~\S\ref{sec:FukCob}. We denote the resulting
$A_{\infty}$-category by $\fukcb(\mathbb{C} \times M; i,h)$. The
objects of $\fukcben(\mathbb{C} \times M)$ are Lagrangian
submanifolds with cylindrical ends $V \subset \mathbb{C} \times M$
with the same normalization as for Lagrangian cobordisms (i.e. $V$ is
cylindrical over $\mathbb{C} \setminus \{ 0 \leq \textnormal{Re} z
\leq 1\}$) with the only exception that the ends of $\pi(V)$ are
allowed to have $y$-coordinate in $\tfrac{1}{2}\mathbb{Z}$ (rather
than just $\mathbb{Z}$).

Further we assume that the profile function $h$ is small enough, i.e.
that the parameter $\epsilon$ in the definition of $h$ is small
enough, say $\epsilon < \tfrac{1}{10}$. See~\S\ref{subsec:equation}.
We now alter the definition of $h$ to a function $h': \mathbb{R}^2 \to
\mathbb{R}$ such that $$h'(x,y) = h(x, y-\tfrac{1}{2}), \; \;
\forall\; x \in (-\infty, -1] \cup [2,\infty), \; \; y \in
(j+\tfrac{1}{2}-\epsilon, j + \tfrac{1}{2} + \epsilon), \, j \in
\mathbb{Z}.$$ We also require that $h'(x,y) = h(x,y)$ for $x \in
(-\infty, -1] \cup [2,\infty)$, $y \in (j-\epsilon, j + \epsilon)$, $j
\in \mathbb{Z}$. We call the functions $h'$ of this type {\em extended
profile functions}.

Next, we extend the choices made in $i \in \mathcal{I}_{cob}^h$ in a
consistent way to accommodate all the new (as well as the old) objects
mentioned above and denote this extension by $i'$. We define
composition maps $\mu_k$, $k \geq 1$, in the same way we did for
$\fukcb(\mathbb{C} \times M)$. The result is a new
$A_{\infty}$-category $\fukcben(\mathbb{C} \times M; i', h')$ which
comes with an obvious full and faithful embedding $$j_{1/2}:
\fukcb(\mathbb{C} \times M; i, h) \longrightarrow \fukcben(\mathbb{C}
\times M; i', h').$$ The family of $A_{\infty}$-categories
$\fukcben(\mathbb{C} \times M; i', h')$ indexed by $(i', h')$ have the
same invariance properties as the family $\fukcb(\mathbb{C} \times M;
i,h)$ and forms a coherent system. Moreover, the collection of
functors $j_{1/2}$ (for different choices of $(i,h)$) are compatible
with the corresponding coherent systems of $\fukcb$ and $\fukcben$
categories. We will henceforth omit sometimes the $(i,h)$ or $(i',
h')$ from the notation when these choices are made clear.


\subsection{Inclusion functors} \label{subsec:inclusion} The purpose
of this subsection is to associate an $A_{\infty}$-functor
$$\mathcal{I}_{\gamma}:\fuk^{d}(M)\to \fukcben (\mathbb{C} \times M)~.~$$
to a curve $\gamma \subset \mathbb{R}^2 \cong \mathbb{C}$ that is
properly embedded, diffeomorphic to $\mathbb{R}$, is horizontal
outside $[0,1] \times \mathbb{R}$ and such that the ends of $\gamma$
have $y$-coordinates in $\tfrac{1}{2}\mathbb{Z}$. We can view $\gamma$
as a Lagrangian with cylindrical ends in the manifold $\mathbb{C}
\times \{\textnormal{pt}\}$ or as an object of $\fukcben(\mathbb{C}
\times \{\textnormal{pt}\})$.

\pbred{
In case the ends of the curve $\gamma$ have integral $y$-coordinates
the functor $\mathcal{I}_{\gamma}$ will factor through $$\fukd(M)
\longrightarrow \fukcb(\mathbb{C} \times M) \longrightarrow
\fukcben(\mathbb{C} \times M).$$ By abuse of notation we will sometime
denote the first functor $\fukd(M) \longrightarrow \fukcb(\mathbb{C}
\times M)$ in this composition also by $\mathcal{I}_{\gamma}$.
}

\pbred{
Below we will, in fact, define a family of $A_{\infty}$-functors
$$\mathcal{I}_{\gamma}: \fuk^d(M) \longrightarrow \fukcben(\mathbb{C} \times M)$$
that are quasi-isomorphic one to the other (as functors). We will call
them {\em inclusion functors}.
}

We first introduce an auxiliary $A_{\infty}$-category
$\mathcal{B}_{\gamma, h}$ that will be also used later in the paper.
We assume that $\gamma$ has the shape as in Figure~\ref{fig:curve}.
This choice is useful in the next section but is not restrictive. Any
other curve $\gamma$ works equally well for the purposes of the
current section.  We also pick $h$ and the extended profile function
$h'$ so that $(\phi_{1}^{h'})^{-1}(\gamma)$ looks as in
Figure~\ref{fig:curve}.
\begin{figure}[htbp]
   \begin{center}
      \epsfig{file=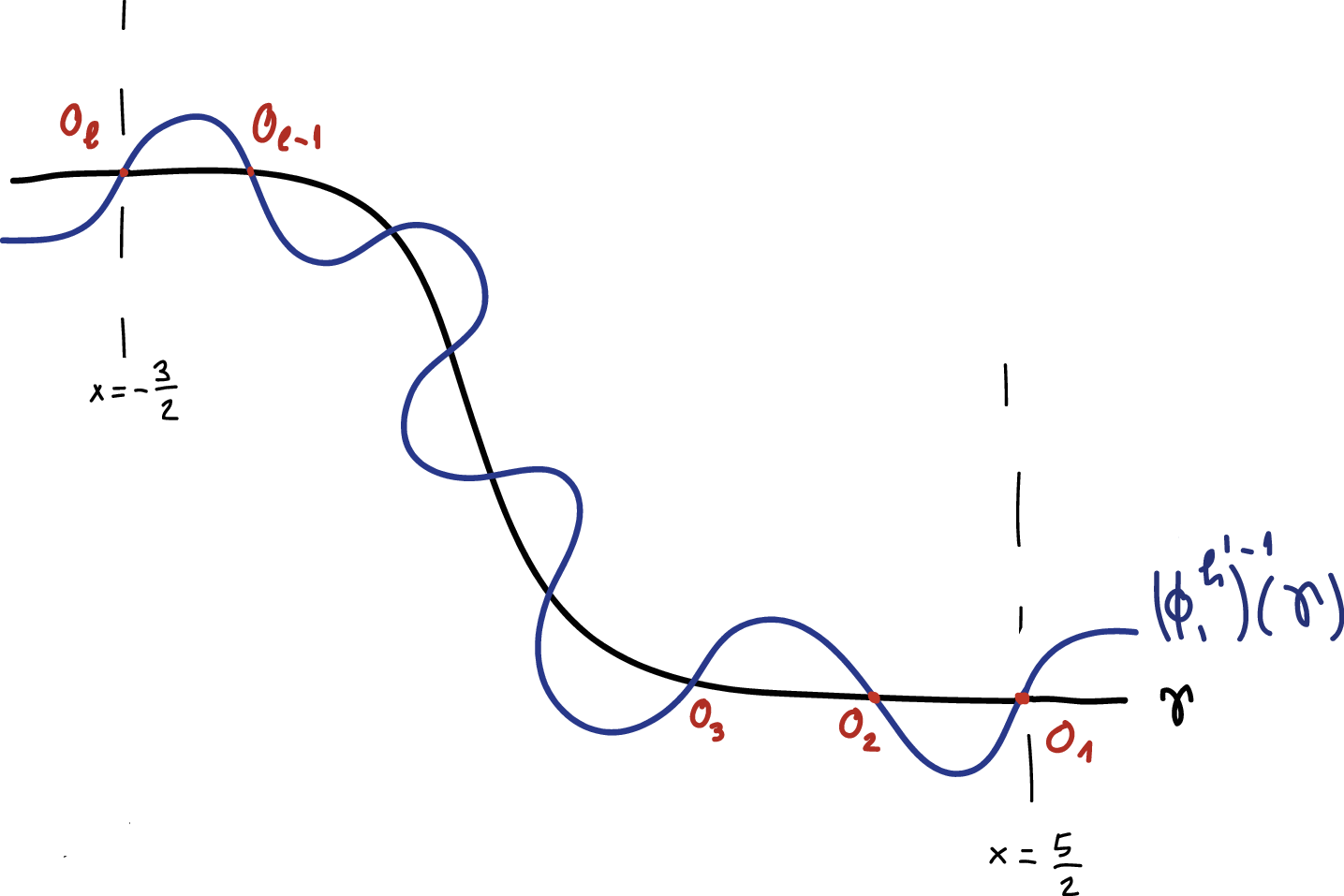, width=0.75\linewidth}
   \end{center}
   \caption{\label{fig:curve} The curves $\gamma$,
     $(\phi^{h'}_{1})^{-1}(\gamma)$ and the points $o_{i}$, $1\leq
     i\leq l$.}
\end{figure}
More precisely, $\gamma\cap (\phi_{1}^{h'})^{-1}(\gamma)$ consists of
an {\em odd} number $l$ of intersection points. We denote them by
$o_{1}, o_{2},\ldots, o_{l}$, as in the picture, and assume that
$l\geq 3$; $o_{1}$ is the intersection point with largest real
coordinate and $o_{l}$ the intersection point with the smallest real
coordinate. The intersection points $o_{i}$ with $i$ odd are called
{\em positive} (or odd) and the ones with $i$ even are called {\em
  negative} (or even).  Notice that $o_{1}$ and $o_{l}$ are
bottlenecks - the first along the positive end of $\gamma$ and the
second along the negative end.

The definition of $\mathcal{B}_{\gamma, h}$ is now as follows: the
objects are $\gamma\times L$ with $L\in \mathcal{L}^{\ast}_{d}(M)$. We
abbreviate $\widetilde{L}=\gamma\times L$. To define morphisms and
higher operations we first fix all the data needed to define a model
for $\fuk^{d}(M)$.  Let $\widetilde{L}=\gamma\times L$,
$\widetilde{L}'=\gamma\times L'$ be two objects of
$\mathcal{B}_{\gamma, h}$ and denote by $(H_{L,L'}, J_{L,L'})$ the
Floer datum that has already been associated to $(L,L')$ in
$\fuk^{d}(M)$.  We put $H_{\widetilde{L},\widetilde{L}'}=h'\oplus
H_{L,L'}$ and the time-dependent almost complex structures
$J_{\widetilde{L},\widetilde{L}'}(t)= i_{h'}(t)\oplus J_{L,L'}(t)$,
$t\in [0,1]$ where $$i_{h'}(t)=(\phi_{t}^{h'})_{\ast}i~.~$$ The Floer
datum for $(\widetilde{L},\widetilde{L}')$ is defined to be the pair
$(H_{\widetilde{L},\widetilde{L}'}, J_{\widetilde{L},\widetilde{L}'})$
and we put
$$\hom_{\mathcal{B}_{\gamma,h}}(\widetilde{L},\widetilde{L}')=
CF(\widetilde{L},\widetilde{L'};H_{\widetilde{L},\widetilde{L}'},
J_{\widetilde{L},\widetilde{L}'})$$ endowed with the standard Floer
differential.  \pbredb{It is not difficult to verify that this data is
regular. Indeed, our choices for $H_{\widetilde{L},\widetilde{L'}}$
and $J_{\widetilde{L},\widetilde{L}'}$ imply that the relevant
Fredholm operators split as a product of the obvious operators in $\C$
and in $M$. Due to the choice of $h'$, for each relevant Floer
trajectory $u=(u_{1}, u_{2})$ where $u_{1}=\pi\circ u$ and
$u_{2}=\pi_{2}\circ u$ (with $\pi_{2}:\C\times M\to M$ the projection
on the second factor) the respective operators are surjective at
$u_{1}$ and $u_2$ respectively. The surjectivity at $u_2$ is due to
the generic choice of $H_{L,L'}, J_{L,L'}$. For the surjectivity at
$u_1$ we need to separate to two cases according to whether or not
$u_1$ is constant. If $u_1$ is not constant then by automatic
regularity for holomorphic curves in $\mathbb{C}$ (see Section~13a
in~\cite{Se:book-fukaya-categ} and also~\cite{DS-Ro-Sa:comb-HF}) the
operator is surjective. If $u_1$ is constant then the surjectivity
follows from Corollary~\ref{c:ind-0}. Notice, in particular, that the
Floer strips that project to a constant in $\C$ equal to one of the
points $o_{i}$, be they positive or negative, remain regular as Floer
strips in $\C\times M$.}

\begin{rem} \label{rem:positive-negative} It is useful to contrast the
   regularity indicated above with the situation occurring for curves
   with boundary conditions along Lagrangians
   $\widetilde{L}_{1},\ldots, \widetilde{L}_{k+1}$ with $k\geq 2$,
   that project to one of the constant points $o_{i}$. Of course,
   these curves do not appear in the definition of the Floer complexes
   but rather in that of the higher compositions \pbredb{(see below)}.
   In this case, the curves that project to positive points $o_{i}$
   remain regular but those that project to negative points are not
   regular in $\C\times M$: \pbredb{indeed, by
     Corollary~\ref{c:ind-0}, at the negative points the planar
     operator has a negative index while at the positive points the
     index is $0$.}
\end{rem}

To construct the coherent system of categories
$\mathcal{B}_{\gamma,h}$ we can now apply the general scheme used to
construct the category $\fuk_{cob}^{d}(\C\times M)$ in
\S\ref{sec:FukCob}. This involves using quite general perturbations -
as in \S\ref{subsec:equation} - to the left of $o_{1}$ and to the
right of $o_{l}$. The techniques in \S\ref{subsec:energy} and
\S\ref{subsec:transversality} apply without change in view of the fact
that the points $o_{1}$ and $o_{l}$ are bottlenecks and the same
remains true for the invariance results in \S\ref{sb:inv-fuk-2}.

We use the category $\fukcben(\mathbb{C} \times M; i', h')$ as defined
in~\S\ref{sb:enrich} above using a choice of data $i$ that extends the
choices just made for $\mathcal{B}_{\gamma, h}$. We thus obtain an
obvious full and faithful embedding of $A_{\infty}$-categories:
\begin{equation} \label{eq:beta-gamma-h} \beta_{\gamma, h}:
   \mathcal{B}_{\gamma,h} \longrightarrow \fukcben(\mathbb{C} \times
   M; i', h').
\end{equation}

The next step is to define quasi-isomorphisms
$\mathcal{B}_{\gamma,h}\to \fuk^{d}(M)$ as well as certain explicit
comparison quasi-isomorphisms relating $\mathcal{B}_{\gamma,h_1}$ to
$\mathcal{B}_{\gamma, h_2}$ for different profile functions
$h_1 ,h_2$.  The desired inclusion functor will be defined by
composing $\beta_{\gamma, h}$ with the inverse of
$\mathcal{B}_{\gamma,h}\to \fuk^{d}(M)$. To construct these functors
and deduce some of their properties we need a special model for
$\mathcal{B}_{\gamma,h}$ associated to particular perturbation data
that we describe next.

This perturbation data is only required for the higher compositions
$\mu^{k}$, $k\geq 2$, the Floer data is fixed as above. Let
$L_{1},\ldots, L_{k+1}\in\mathcal{L}^{\ast}_{d}(M)$, $k\geq 2$ and
denote by $(\Theta_M, \mathbf{J}_M)$ the perturbation datum of
$(L_{1}, \ldots, L_{k+1})$ as fixed in $\fuk^{d}(M)$.  The
perturbation datum $(\Theta, \mathbf{J})$ associated to
$(\widetilde{L}_{1},\ldots,\widetilde{L}_{k+1})$ has the following
form: \pbredb{
\begin{equation}\label{eq:perturb-module}
   \Theta = da_{r}\otimes h'+\Theta_M + Q\ , \ 
   \mathbf{J}=\widetilde{i}_{h}\oplus \mathbf{J}_M
\end{equation}
} where $\widetilde{i}_{h}=\{ \widetilde{i}_{h}(z)\}_{z\in S_{r}}$ is
the family of almost complex structures on $\R^{2}$ defined by
$\widetilde{i}_{h}=(\phi^{h'}_{a_{r}(z)})_{\ast}i$, $z\in S_{r}$. The
term $Q$ is a $1$-form
$ Q\in \Omega^{1}(S_{r}, C^{\infty}(\C\times M))$ which is compactly
supported in the interior of $S_{r}$ and away from the strip like
ends. Moreover $Q$ is required to satisfy the following conditions for
all $z\in S_{r}$, $v\in T_{z}S_{r}$:
\begin{itemize}
\item[i.] The function $Q_{z}(v):\C\times M\to \R$ is of the from
  $Q_{z}(v)=q_{z}(v)\circ \pi$ with $q_{z}(v):\C\to\R$.
\item[ii.] The function $q_{z}(v)$ is supported in a union $U_{h}$ of
  small neighborhoods of the points $o_{i} \in \C$ with $i=$ even.
\end{itemize}

%
Similarly to what has been done before in the case of $\fuk^{d}(M)$
we define $(\Theta, \mathbf{J})$ so that it extends the Floer data
already defined on the strip-like ends of the surfaces $S_{r}$ and so
that it is consistent with splitting and gluing in the space
$\widehat{\mathcal{S}}^{k+1}$.

Note that if $Q$ is chosen such that for some $z, v$ the function
$q_z(v)$ does not have a critical point at $o_i$, where $i=$ even,
then no solution $u$ of~\eqref{eq:jhol1} can have a constant
projection at $o_i$. (In contrast, solutions with a constant
projection over $o_i$ with $i=$ odd definitely exist.)

\pbredb{The class of perturbations described above is sufficient for
  the regularity of the moduli spaces of perturbed $J$-holomorphic
  polygons \corr{as will be seen in
    Corollary \ref{c:ind-n-const}.  Basically}, due to our choices,
  the linearized operator corresponding to equation~\eqref{eq:jhol1}
  \pbcorr{has} at every solution $u$ \pbcorr{a horizontal component
    (corresponding to $\mathbb{R}^2$) and a vertical one
    (corresponding to $M$)}. The vertical part can be assumed
  surjective in view of the choices made for
  $\fuk^{d}(M)$. As for the horizontal part, the argument splits into
  two cases.  The first one is when $\pi \circ u$ is not constant and
  in that case Lemma~\ref{l:index-non-const} assures surjectivity of
  the respective operator. The second possible case is that
  $\pi \circ u$ is constant.  As remarked above a (generic) suitable
  choice of $Q$ assures that $\pi \circ u$ equals to one of the
  $o_i$'s with $i=$~odd.  Corollary~\ref{c:ind-0} then implies
  surjectivity for the horizontal operator.  \corr{The actual argument
    is slightly more complicated as the splitting mentioned above
    applies for the moduli spaces of polygons with fixed punctures and
    the moduli spaces we are interested in are associated to polygons
    with moving punctures - see Corollary \ref{c:ind-n-const}.}  }

Notice also that the Floer complex
$CF^{\mathcal{B}_{\gamma,h}}(\widetilde{L},\widetilde{L}')=
\hom_{\mathcal{B}_{\gamma, h}}(\widetilde{L},\widetilde{L}')$ splits
as a vector space:
$$CF^{\mathcal{B}_{\gamma,h}}(\widetilde{L},\widetilde{L}')=
\bigoplus_{i=1}^{l} CF(L,L'),$$ where $CF(L,L')=CF(L,L';H_{L,L'},
J_{L,L'})$ is the Floer complex in $\fuk^{d}(M)$. The $i$-th summand
at the right-hand side of the above equality corresponds to the point
$o_{i}$ and we will sometimes identify it as $CF(L,L')^{o_{i}}$.  We
use the following notation. For $x\in CF(L,L')$ and $1\leq i\leq l$ we
denote by $x^{(i)}\in
CF^{\mathcal{B}_{\gamma,h}}(\widetilde{L},\widetilde{L}')$ the element
$(0,\ldots, 0, x, 0 \ldots, 0)\in CF(L,L')^{o_{i}}$ where $x$ is
placed at the $i$-the spot.  Such an element $(0,\ldots, 0, x, 0
\ldots, 0)\in CF(L,L')^{o_{i}}\subset
CF^{\mathcal{B}_{\gamma,h}}(\widetilde{L},\widetilde{L}')$ will be
called a morphism {\em of type} $o_{i}$. Morphisms of some type
$o_{i}$ as before will be called homogeneous.  We also denote
$CF^{\mathcal{B}_{\gamma,h}}(\widetilde{L},\widetilde{L}')^{\leq
  \alpha}$ the partial sum $\oplus_{i=1}^{\alpha}
CF(L,L')^{o_{i}}\subset
CF^{\mathcal{B}_{\gamma,h}}(\widetilde{L},\widetilde{L}')$, $1\leq
\alpha\leq l$.

\begin{rem}\label{rem:snakes-chain}It is an easy exercise
   to describe also the differential $\overline{d}$ of
   $CF^{\mathcal{B}_{\gamma,h}}(\widetilde{L},\widetilde{L}')$.  For
   this let $d:CF(L,L')\to CF(L,L')$ be the differential of the Floer
   complex in $\fuk^{d}(M)$.  We then have, for $j$ even
   $\overline{d}x^{(j)}=(dx)^{(j)}$ and for $j$ odd
   $\overline{d}x^{(j)}=(dx)^{(j)}+x^{(j-1)}- x^{(j+1)}$ (where, by a
   slight abuse in notation, we put $x^{(j+1)}=0$ if $j=l$ and
   $x^{(j-1)}=0$ if $j=1$). In particular, the complexes
   $CF^{\mathcal{B}_{\gamma,h}}(\widetilde{L},\widetilde{L}')$ and
   $CF(L,L')$ are quasi-isomorphic.
\end{rem}
We will need a strong form of the statement in
Remark~\ref{rem:snakes-chain}.  For an odd index $1\leq j\leq l$ we
define a functor $c_{\gamma, h,j}: \mathcal{B}_{\gamma,h}\to
\fuk^{d}(M)$ by putting on objects $c_{\gamma,h,j}(\widetilde{L})=L$,
on morphisms $c_{\gamma,h,j}^{1}(x^{(i)})=x$ if $i=j$,
$c_{\gamma,h,j}^{1}(x^{(i)})=0$ if $i\not=j$ and
$c_{\gamma,h,j}^{r}=0$ for $r\geq 2$.

\begin{prop}\label{prop:htpy-proj} 
   With the definition above, the $c_{\gamma,h,j}$'s are
   $A_{\infty}$-functors and are quasi isomorphisms.  Moreover, for
   any two odd $1\leq j, j' \leq l$, the functors $c_{\gamma,h,j}$ and
   $c_{\gamma,h,j'}$ are homotopic.
\end{prop} 
\begin{proof}
   In view of Remark~\ref{rem:snakes-chain}, to prove the first part
   of the claim it is enough to show that $c_{\gamma,h,j}$ is an
   $A_{\infty}$-functor. For this we start by listing two properties
   of the higher compositions in $\mathcal{B}_{\gamma,h}$ in the model
   using the perturbations described above. Denote the compositions in
   $\mathcal{B}_{\gamma, h}$ by $\overline{\mu}_{k}$ and the
   compositions in $\fuk^{d}(M)$ by $\mu_{k}$.  We assume below $k\geq
   2$. Let $z_{i} \in CF(L_{i},L_{i+1})^{o_{j_{i}}}$, $1\leq i\leq k$,
   and let $u: S_r \to \widetilde{M}$ be an index-$0$ solution
   to~\eqref{eq:jhol1} with entries $z_1, \ldots, z_k$. Denote by $w
   \in \widetilde{M}$ the exit (or output) of $u$. We claim that one
   of the following possibilities occur:
   \begin{itemize}
   \item[i.] There exists an odd $s$ such that
     $j_1 = \cdots = j_k = s$ and $\pi \circ u$ is constant at
     $o_s$. In particular
     $\overline{\mu}_{k}( z_{1}, \ldots, z_{k}) = (0,\ldots,
     0,\mu_{k}( z_{1}, \ldots, z_{k}), 0,\ldots , 0)\in
     CF(L_{1},L_{k+1})^{o_{s}}$.
   \item[ii.] There is \corr{an} $i$ with $j_i=$ even and
     $w \in CF(L_1, L_{k+1})^{o_{j_i}}$. Moreover, for every
     $r \neq i$ we have $j_r = j_i \pm 1$ \corr{or $j_{r}=j_{i}$}. In
     particular
     $\overline{\mu}_{k}( z_{1}, \ldots, z_{k}) \in
     CF(L_{1},L_{k+1})^{o_{j_{i}}}$.
   \end{itemize}
   The justification of this property is a simple application of two
   arguments. First, the odd points play a typical ``bottleneck'' role
   because our perturbations are such that the curves $u$ transform by
   the naturality transformation (\ref{eq:nat-v}) to polygons that
   project holomorphically onto $\C$ around these points. Moreover,
   the odd points are ``entry'' (or input) points from the point of
   view of the strip-like ends of these projections. This shows that
   any such curve that has such a point as an ``exit'' (or output)
   point has to project to a constant. The second argument is based on
   Corollary~\ref{c:ind-n-const}. More specifically, if an even point
   corresponds to the ``exit'' then, it is necessary for \corr{at
     least} one even point to appear among the
   entries.


   Finally, the fact that $j_r = j_i \pm 1$ \corr{or $j_{r}=j_{i}$}
   for every $r$ follows from a bottleneck argument similar the ones
   used previously several times.

   The above description of $\overline{\mu}_{k}$ implies the following
   identity:
   $$c^1_{\gamma, h,j}(\overline{\mu}_k(z_1, \ldots, z_k)) = 
   \mu_k \bigl( c^1_{\gamma, h,j}(z_1), \ldots, c^1_{\gamma, h,j}(z_k)
   \bigr).$$ In view of the vanishing of the higher order components
   of $c_{\gamma, h, j}$ it follows from this identity that
   $c_{\gamma,h,j}$ is an $A_{\infty}$-functor.

   To conclude the proof of the Proposition it remains to show the
   statement about the homotopy. To simplify the exposition we will
   now assume that $l=3$. We will show now that $c_{\gamma, h, 1}$ and
   $c_{\gamma, h,3}$ are homotopic as $A_{\infty}$-functors.

   Let $k \geq 2$ and let $z_{i}\in CF(L_{i},L_{i+1})^{o_{j_{i}}}$,
   $1\leq i\leq k$, be so that all $j_{i} \in \{1,3\}$ for every $i$.
   By inspecting the points {\em i}, {\em ii} above we see that
   $\overline{\mu}(z_{1},\ldots, z_{k})=0$ except possibly if all
   $j_{i}$ are equal to $1$ or if they are all equal to $3$.
   Moreover, we have
   $$\overline{\mu}_{k}(x_{1}^{(1)},\ldots, x_{k}^{(1)})=
   (\mu_{k}(x_{1},\ldots, x_{k}))^{(1)}\ , \quad
   \overline{\mu}_{k}(x_{1}^{(3)},\ldots,
   x_{k}^{(3)})=(\mu_{k}(x_{1},\ldots, x_{k}))^{(3)}~.~$$ By Remark
   \ref{rem:snakes-chain} for $k=1$ we have
   $$\overline{\mu}_{1}(x^{(1)}+x^{(3)})=(\mu_{1}(x))^{(1)}+(\mu_{1}(x))^{(3)})~.~$$
   We thus define an $A_{\infty}$-functor:
   $$e_{\gamma,h}: \fuk^{d}(M) \longrightarrow \mathcal{B}_{\gamma,h}$$ 
   by defining it on objects as $e_{\gamma, h}(L)=\widetilde{L}$, on
   morphisms $e_{\gamma, h}^{1}(x)= x^{(1)}+x^{(3)}$ and setting the
   higher components to be null, $e_{\gamma,h}^{k}=0$ for all $k>1$.

   This morphism clearly satisfies $c_{\gamma, h, 1}\circ e_{\gamma,h}
   = c_{\gamma, h, 3}\circ e_{\gamma,h}=id$. Moreover, $e_{\gamma,h}$
   is a quasi-isomorphism and thus invertible up to homotopy. (Note
   that $e_{\gamma, h}$, in contrast to $c_{\gamma, h, j}$ does not
   depend on $j$.) Thus, $c_{\gamma, h,1}\simeq c_{\gamma, h, 3}$.
   The argument for $l>3$ is analogous: namely
   $$e_{\gamma,h}:\fuk^{d}(M)\to \mathcal{B}_{\gamma,h}$$ has the same
   definition on objects as in the case $l=3$, its action on morphisms
   is again null for $k\geq 2$ and
   $e_{\gamma,h}^{1}(x)=x^{(1)}+x^{(3)}+\ldots + x^{(l)}$.
\end{proof}

We now define the inclusion functor
$$\mathcal{I}_{\gamma, h} : \fukd(M) \longrightarrow
\fukcben(\mathbb{C} \times M; i', h')$$ as the composition
\begin{equation} \label{eq:functor_I-gamma-h} \mathcal{I}_{\gamma, h}
   = \beta_{\gamma, h} \circ e_{\gamma, h},
\end{equation}
where $\beta_{\gamma, h}$ is the functor defined
in~\eqref{eq:beta-gamma-h}.

\begin{rem}\label{rem:var-snakes}
   a. The functor $e_{\gamma,h}$ is a common homotopy inverse for all
   the projections $c_{\gamma,h,i}$. As all these projections are
   mutually chain homotopic we will sometimes omit the index $i$ from
   the notation in $c_{\gamma, h, i}$.

   b. If the ends of $\gamma$ have integral $y$-coordinates then
   $\mathcal{I}_{\gamma, h}$ factors through
   $$\fukd(M) \longrightarrow \fukcb(\mathbb{C} \times M; i, h) 
   \stackrel{j}{\longrightarrow} \fukcben(\mathbb{C} \times M; i',
   h').$$
\end{rem}

\begin{prop} \label{prop:inv-incl}
   Up to a quasi-isomorphism (canonical in homology), the 
   functor $\mathcal{I}_{\gamma,h}$ depends only
   on the horizontal isotopy type of $\gamma$. 
\end{prop}

\begin{proof}
   The proof follows the general method for proving invariance
   statements as in~\S\ref{sb:inv-fuk-2}. Thus, we first assume that
   $\gamma$ varies inside a fixed horizontal isotopy type and $h$
   changes while keeping fixed the bottlenecks $o_{1}$ and $o_{l}$ and
   $l$ remains fixed.  We then include the different categories
   $\mathcal{B}_{\gamma,h}$ in the same category $\mathcal{B}^{tot}$.
   This category comes with a projection $\bar{c}_{1}:\mathcal{B}^{tot}\to
   \fuk^{d}(M)$ associated to the first intersection point $o_{1}$. This functor
   $\bar{c}_{1}$ restricts to the projection $c_{\gamma,h,1}$ on each $\mathcal{B}_{\gamma,h}$
   and is a quasi-equivalence.  
   Consider now the functors:
   $$\fuk^{d}(M)\stackrel{e_{\gamma,h}}{\longrightarrow}
   \mathcal{B}_{\gamma,h}\to \mathcal{B}^{tot}~.~$$
   The composition of $\bar{c}_{1}$ with each one of these functors is the identity.  
   Thus any two such
   functors are quasi-isomorphic.  The category $\mathcal{B}^{tot}$ can
   be mapped by inclusion into $\fuk_{cob,1/2}^{d,tot}(\C\times M)$
   and thus any two functors $\mathcal{I}_{\gamma, h}$ are also quasi-isomorphic by
   a quasi-isomorphism that is canonical in homology.

   The same argument also applies if we vary $h$ by keeping its
   positive and negative bottlenecks all fixed while allowing the
   number $l \geq 3$ to change: because the bottlenecks are fixed the
   corresponding category $\mathcal{B}^{tot}$ remains well defined as
   well as its inclusion in $\fuk_{cob,1/2}^{d,tot}(\C\times M)$; the
   projection $c_{1}$ remains also well-defined.  The last step is to
   ``move'' the bottlenecks. This is achieved combining the argument
   above with that in~\S\ref{subsubsec:bottle}.
\end{proof}

\begin{rem}\label{rem:comparison}
   The description of the $A_{\infty}$-structure for the category
   $\mathcal{B}_{\gamma,h}$ in the proof of Proposition
   \ref{prop:htpy-proj} also allows us to compare explicitly the
   categories $\mathcal{B}_{\gamma,h}$ and $\mathcal{B}_{\gamma, g}$
   where the number of points $o_{i}\in\ C$ for the function $h$ is
   $l\geq 5$ and the number of points $o'_{i}\in \C$ for the
   function $g$ is $l'=l-2$. Indeed, it is easy to define:
   $$p^{h}_{g}:\mathcal{B}_{\gamma,h}\to \mathcal{B}_{\gamma, g}$$
   which is the identity on objects, it is given on morphisms by
   $(p^{h}_{g})^{1}(x^{(i)})=x^{(i)}$ if $i\leq l-2$ and
   $(p^{h}_{g})^{1}(x^{(i)})=0$ if $l-2< i \leq l$ and
   $(p^{h}_{g})^{k}=0$ for $k\geq 2$.  For suitable choices of $g$ and
   perturbation data it is easy to see that this is an $A_{\infty}$
   functor. For this, we pick $g$ so that
   $(\phi^{g'}_{1})^{-1}(\gamma)$ coincides with
   $(\phi_{1}^{h'})^{-1}(\gamma)$ to the right of $o_{l-2}$ and we
   also fix the perturbation data so that the perturbation data for
   $\mathcal{B}_{\gamma,h}$ extends that of $\mathcal{B}_{\gamma, g}$.
   Clearly, $p^{h}_{g}$ is compatible in the obvious sense with the
   projections $c_{\gamma, h}$ and $c_{\gamma, g}$ as well as with the
   functors $e_{\gamma, h}$, $e_{\gamma,g}$. 

\end{rem}


\subsection{Index and regularity} \label{sb:ind-reg} The purpose of
this subsection is to provide details on the index calculations used
above in constructing and analyzing the category $\mathcal{B}_{\gamma,
  h}$.  Recall that we had to calculate the index of the linearized
operator associated to equation~\eqref{eq:jhol1} at a solution $u$ of
that equation. We need to calculate the index for two types of
solutions $u$. The first is when $u$ has a constant projection to
$\mathbb{C}$, and the second one is when the projection is not
constant. The main part in this calculation amounts to computing the
index of the horizontal part of the operator. We will also address
regularity of the solutions.

\subsubsection{The case of solutions with constant projection}
\label{sbsb:u-const}

We will work here with the following slightly more general setting.
Let $\gamma \subset \mathbb{R}^2$ be a curve as
in~\S\ref{subsec:inclusion}. Fix $k \geq 2$. Let $f_1, \ldots, f_{k+1}
: \gamma \longrightarrow \mathbb{R}$ be Morse functions and assume
that $p \in \gamma$ is a critical point of all the $f_i$'s.  Denote by
$|p|_{f_i}$ the Morse index of $f_i$ at $p$. Fix a a punctured disk $S
= S_r$ for some $r \in \mathcal{R}^{k+1}$ and denote its complex
structure by $j$. Denote by $a = a_r: S \longrightarrow \mathbb{R}$
the transition function as in~\S\ref{subsec:strip-ends}. Let $\{
\hat{f}_z: \gamma \longrightarrow \mathbb{R} \}_{z \in S}$ be a family
of functions, parametrized by $S$, all having $p$ as a critical point
and such that on the $i'th$ strip-like end of $S$, $f_z$ coincides
with $f_i$ at infinity, i.e.  $\hat{f}_z = f_i$ for every $z \in
\epsilon_i^S(Z^- \setminus K^-)$, $i=1, \ldots, k$ and $\hat{f}_z =
f_{k+1}$ for every $z \in \epsilon_i^S(Z^+ \setminus K^+)$, where
$K^{\pm}$ are compact subsets of the semi-strips. We extend the
functions $f_i$ and $\hat{f}_z$ to functions on $\mathbb{R}^2$ by
identifying $T^*\gamma \cong \mathbb{R}^2$ and defining the extension
to be constant along each fibers of $T^*\gamma$. Denote by
$X^{\hat{f}} = \{X^{\hat{f}_z}\}_{z \in S}$ the ($S$-dependent)
Hamiltonian vector field associated to $\hat{f}$. Set $Z = da \otimes
X^{\hat{f}}$ viewed as a $1$-form on $S$ with values in the space of
Hamiltonian vector fields on $\mathbb{R}^2$. Finally define an
$S$-dependent complex structure $\hat{i} = \{\hat{i}_z \}_{z \in S}$
on $\mathbb{R}^2$ defined by $\hat{i}_z = \big(
\phi_{a(z)}^{\hat{f}_z} \big)_* i$.

We extend the choices above to the case $k=1$ as follows. In that case
$S$ can be identified with the strip $\mathbb{R} \times [0,1]$ and we
require that $f_0 = f_1$, the family $\hat{f}_z$ to be constant and
$a(s,t) = t$ for every $(s,t)$.

With the above data fixed, consider the following equation:
\begin{equation} \label{eq:jhol-plane} w:(S, \partial S)
   \longrightarrow (\mathbb{R}^2, \gamma), \quad Dw + \hat{i}_z(w)
   \circ Dw \circ j = Z + \hat{i}_z(w) \circ Z \circ j.
\end{equation}

Due to our choices we have $Z_z(p) = 0$ for every $z \in S$, hence the
constant map $w_0(z) \equiv p$ is a solution of~\eqref{eq:jhol-plane}.
Denote by $\mathcal{D}_{w_0}$ the linearized operator associated
to~\eqref{eq:jhol-plane} at $w_0$.

The following Lemma follows from the theory developed
in~\cite{Se:Lefschetz-Fukaya}.
\begin{lem} \label{l:index}
   The Fredholm index of $\mathcal{D}_{w_0}$ is given by
   $$\textnormal{ind} (\mathcal{D}_{w_0}) = |p|_{f_1} + \cdots + |p|_{f_k} - 
   |p|_{f_{k+1}} - (k-1).$$ Moreover, when the index is $0$ the
   operator $\mathcal{D}_{w_0}$ is surjective, hence the solution
   $w_0$ is regular.
\end{lem}

The following is an immediate consequence of the previous lemma.
\begin{cor} \label{c:ind-0} $\textnormal{ind}(\mathcal{D}_{w_0}) = 0$
   in any of the following two cases:
   \begin{enumerate}
     \item $k=1$ and $|p|_{f_1} = |p|_{f_2}$.
     \item $k \geq 2$ and $|p|_{f_1} = \cdots = |p|_{f_{k+1}} = 1$.
   \end{enumerate}
   Moreover, in both cases the constant solution $w_0$ is regular.
\end{cor}

\begin{proof}[Proof of Lemma~\ref{l:index}]
   The proof follows essentially the recipe from Sections~4.5 and~4.6
   in~\cite{Se:Lefschetz-Fukaya}, with a few minor modifications that
   we outline below.

   We begin with the index formula. Consider the same equation
   as~\eqref{eq:jhol-plane} but with $\hat{i}$ replaced by the
   standard (constant) complex structure $i$ in the plane
   $\mathbb{C}$. The constant map $w_0$ solves that equation too. The
   linearized operator $\mathcal{D}'_{w_0}$ of that equation is
   homotopic to $\mathcal{D}_{w_0}$ through Fredholm operators, hence
   its index is equals to that of $\mathcal{D}'_{w_0}$.

   The operator $\mathcal{D}'_{w_0}$ is precisely of the class of
   operators considered in Section~4e of~\cite{Se:Lefschetz-Fukaya}
   (see~(4.32) and~(4.33) in that paper for the precise description of
   that class). More precisely, the operator $\mathcal{D}'_{w_0}$ has
   the following form:
   \begin{equation} \label{eq:lin-D'}
      \begin{cases}
         & \mathcal{D}'_{w_0}(\xi) = \overline{\partial}{\xi} - \Bigl(
         i \circ \textnormal{Re}(\xi)
         \textnormal{diag} \bigl(\hat{f}_z''(p), 0 \bigr) 
         \otimes da \Bigr)^{0,1}, \\
         & \xi: S \longrightarrow \mathbb{R}^2, \\
         & \xi(z) \in T_{p} \gamma, \; \forall z \in \partial S.
      \end{cases}
   \end{equation}
   Here both sides of the equation are regarded as $1$-forms with
   values in $\hom_{\mathbb{R}}(\mathbb{C}, \mathbb{C}) =
   Mat_2(\mathbb{R})$. Also, $\hat{f}_z''(p)$ stands for the second
   derivative of the function $\hat{f}_z: \gamma \longrightarrow
   \mathbb{R}$ at the point $p$, where we identify $\gamma$ with
   $\mathbb{R}$ in the obvious way.

   From here on, the index calculation of $\mathcal{D}'_{w_0}$ follows
   a topological recipe developed in~\cite{Se:Lefschetz-Fukaya}.
   Denote by $\bar{S}$ the compactification of $S$ obtained by adding
   to each strip like end an interval of the type $\epsilon_i^S(\pm
   \infty \times [0,1])$. According to~\cite{Se:Lefschetz-Fukaya} we
   have $$\textnormal{ind}(\mathcal{D}'_{w_0}) = 1 +
   \deg(\hat{\lambda}),$$ where $\hat{\lambda}: \partial \bar{S}
   \longrightarrow {\mathbb{R}}P^1$ is a map defined as
   $\hat{\lambda}(z) = T_p\gamma \in {\mathbb{R}}P^1$ for every $z \in
   \partial S$ and over the intervals at infinity of $\partial
   \bar{S}$ by using an explicit homotopy class of loops in
   $\mathbb{R}P^1$ starting and ending at $T_p\gamma$. This homotopy
   class is explicitly determined on each end by the sign of
   $f_i''(p)$.

   A little care is needed when following the procedure
   from~\cite{Se:Lefschetz-Fukaya}, since Seidel's conventions differ
   from ours. The essential difference here is that
   in~\cite{Se:Lefschetz-Fukaya} Seidel parametrizes the strip-like
   ends in a different way than we do. Namely, the ``input'' (or
   ``entry'') ends $i=1, \ldots, k$ are parametrized in Seidel by
   $[0,\infty) \times [0,1]$ whereas we parametrize them by $(-\infty,
   0]\times [0,1]$. Similarly, the ``output'' (or ``exit'') end is
   parametrized in Seidel by $(-\infty, 0] \times [0,1]$ and we
   parametrize it by $[1, \infty) \times [0,1]$. The outcome is that
   in order to calculate the map $\bar{\lambda}$ using the conventions
   of Seidel we need to replace $f_i''(p)$ by $-f_i''(p)$ everywhere.
   Apart from that a straightforward substitution in the method
   from~\cite{Se:Lefschetz-Fukaya} yields the desired index formula.

   We now turn to regularity, under the assumption that the index is
   $0$. Here it is convenient to apply the naturality transformation,
   namely define $v:S \longrightarrow \mathbb{R}^2$ by the
   transformation $w(z) = \phi_{a(z)}^{\hat{f}_z}(v(z))$. Then the
   equation~\eqref{eq:jhol-plane} is transformed to the following
   equation with moving boundary conditions:
   \begin{equation} \label{eq:jhol-homogen-plane}
      \begin{cases}
         & Dv + i \circ Dv_z \circ j = 0, \\
         & v:S \longrightarrow \mathbb{R}^2, \\
         & v(z) \in \bigl( \phi_{a(z)}^{\hat{f}_z}
         \bigr)^{-1}(\gamma), \; \forall z \in \partial S.
      \end{cases}
   \end{equation}
   
   Note that the constant solution $w_0$ is sent by this
   transformation to the constant solution $v_0 \equiv p$ of
   equation~\eqref{eq:jhol-homogen-plane}. The linearization
   $\mathcal{D}''_{v_0}$ of~\eqref{eq:jhol-homogen-plane} at $v_0$ now
   looks as follows:
   \begin{equation} \label{eq:lin-D-homog}
      \begin{cases}
         & \mathcal{D}''_{v_0}(\xi) = \overline{\partial}{\xi} \\
         & \xi: S \longrightarrow \mathbb{R}^2, \\
         & \xi(z) \in D \bigl(\phi_{a(z)}^{\hat{f}_z}
         \bigr)^{-1}(T_{p} \gamma), \; \forall z \in \partial S.
      \end{cases}
   \end{equation}
   This operator still falls in the class of operators considered
   in~\cite{Se:Lefschetz-Fukaya} (see Section~4.5 in that paper). The
   index can now be independently (of the previous calculation)
   calculated via an analogous procedure:
   $\textnormal{ind}(\mathcal{D}''_{v_0}) = 1 + \deg (\hat{\lambda})$,
   where this time the map $\hat{\lambda}$ takes into account also the
   moving boundary conditions $D \bigl(\phi_{a(z)}^{\hat{f}_z}
   \bigr)^{-1}(T_{p} \gamma)$, $z \in \partial S$. Since the total
   index is assumed to be $0$ we have $\deg(\hat{\lambda}) = -1$. By
   Proposition~4.1 in~\cite{Se:Lefschetz-Fukaya} this implies that
   $\mathcal{D}''_{v_0}$ is injective. As the index is $0$ it must
   also be surjective.
\end{proof}

\subsubsection{The case of solutions with non-constant projection}
\label{sbsb:u-non-const}

Here we consider a very similar setting to the one
in~\S\ref{sbsb:u-const} with the following difference.  The functions
$f_1, \ldots, f_{k+1}: \gamma \longrightarrow \mathbb{R}$ are assumed
to be Morse but possibly with different critical points. The family
$\hat{f}$ is only assumed to coincide with the $f_i$'s on the ends at
infinity. The term $Z$ in equation~\eqref{eq:jhol-plane} is allowed
now to have following more general form:
$$Z = da \otimes X^{\hat{f}} + X^{q},$$ where $q$ is a $1$-form on $S$
with values in $C^{\infty}(\mathbb{R}^2)$ whose support (as a form on
$S$) is compact and lies inside the interior of $S$. The main example
we have in mind is when the $f_i$'s, as well as the $\hat{f}_z$'s all
coincide with one extended profile function $h'$, and the form $q$ is as
defined on page~\pageref{eq:perturb-module} after
equation~\eqref{eq:perturb-module}.

Let $w: (S, \partial S) \longrightarrow (\mathbb{R}^2, \gamma)$ be a
solution of equation~\eqref{eq:jhol-plane}, where $Z$ is now assumed
to be of the form just discussed. It is easy to see that along each
strip-like end $w$ converges to a critical point of the corresponding
$f_i$. Denote by $p_1, \ldots, p_k \in \gamma$ the critical points
corresponding to the input ends and by $p_{k+1}$ the one corresponding
to the output. We write $|p_i|_{f_i}$ for the Morse index of $f_i$ at
$p_i$.  We continue to denote the linearized operator
of~\eqref{eq:jhol-plane} at $w$ by $\mathcal{D}_{w}$. \pbcorr{Let
  $r \in \mathcal{R}^{k+1}$ be the point parametrizing the punctured
  disk $S$ (so that $S = S_r$) and denote by
  $\overline{\mathcal{D}}_{r,w}$ the extended linearized operator,
  that takes into account ``moving'' punctures}. The following Lemma
follows too from the methods of~\cite{Se:Lefschetz-Fukaya}.

\begin{lem} \label{l:index-non-const} Let $w$ be a {\em non-constant}
  solution of equation~\eqref{eq:jhol-plane}. Then the operator
  \pbcorr{$\overline{\mathcal{D}}_{r,w}$} is surjective and its index
  is given by:
  $$\pbcorr{\textnormal{ind}(\overline{\mathcal{D}}_{r,w})}
  = |p_1|_{f_1} + \cdots + |p_k|_{f_k} - |p_{k+1}|_{f_{k+1}} -\corr{
    1}.$$
\end{lem}

\begin{proof}
   The proof is very similar to the proof of Lemma~\ref{l:index},
   following the theory from Sections~4.5 and~4.6
   of~\cite{Se:Lefschetz-Fukaya}.

   By applying a suitable symplectomorphism of $\mathbb{R}^2$ we may
   assume without loss of generality that the curve $\gamma$ is
   entirely horizontal. The map $\hat{\lambda}$ mentioned in the proof
   of Lemma~\ref{l:index} then becomes constant over $\partial S$
   (since $\hat{\lambda}(z) = T_{w(z)} \gamma$ \, $\forall z \in
   \partial S$). Thus the only contribution of $\hat{\lambda}$ to the
   index of $\mathcal{D}_w$ comes from the ends of $S$. But the
   calculation over each end coincides with that done for the proof of
   Lemma~\ref{l:index}.

   \corr{The index of} \pbcorr{$\overline{\mathcal{D}}_{r,w}$}
   \corr{is bigger by $k-2$ compared to that of $\mathcal{D}_{w}$ and,
     finally, the surjectivity of}
   \pbcorr{$\overline{\mathcal{D}}_{r,w}$} follows from the theory
   developed in Section~13a of~\cite{Se:book-fukaya-categ}.
\end{proof}

\pbcorr{We now examine the situation for polygons in
  $\mathbb{R}^2 \times M$. Let
  $L_1, \ldots, L_{k+1} \in \mathcal{L}^*_d(M)$ and
  $\widetilde{L}_1, \ldots, \widetilde{L}_{k+1}$ as
  in~\S\ref{subsec:inclusion}. Let
  $\gamma_i \in CF(\widetilde{L}_i, \widetilde{L}_{i+1};
  H_{\widetilde{L}_i, \widetilde{L}_{i+1}}, J_{\widetilde{L}_i,
    \widetilde{L}_{i+1}})$, $i=1, \ldots, k$, and
  $\gamma_{k+1} \in CF(\widetilde{L}_1, \widetilde{L}_{k+1};
  H_{\widetilde{L}_1, \widetilde{L}_{k+1}}, J_{\widetilde{L}_1,
    \widetilde{L}_{k+1}})$ be generators.}
  
\begin{cor} \label{c:ind-n-const} \corr{For a generic choice of
    perturbation data}
  \pbcorr{$\mathscr{D}_M = (\Theta_M, \mathbf{J}_M)$ in $M$,}
  \corr{the moduli spaces}
  \pbcorr{$\mathcal{M}(\gamma_1, \ldots, \gamma_k, \gamma_{k+1};
    \mathscr{D}_{\mathbb{R}^2 \times M})$} \corr{of polygons}
  \pbcorr{in $\mathbb{R}^2 \times M$, defined using the perturbation
    data $\mathscr{D}_{\mathbb{R}^2 \times M} = (\Theta, \mathbf{J})$
    as in~\eqref{eq:perturb-module},} \corr{are regular. Moreover, for
    $k\geq 2$,} \pbcorr{if all the points
    $\pi(\gamma_1), \ldots, \pi(\gamma_k)$ are of odd type and the
    space
    $\mathcal{M}(\gamma_1, \ldots, \gamma_k, \gamma_{k+1};
    \mathscr{D}_{\mathbb{R}^2 \times M})$ has a $0$-dimensional
    component which is not empty, then $\pi(\gamma_{k+1})$ is also of
    odd type.}
\end{cor}
\begin{proof} \corr{Consider the ``horizontal'' moduli spaces, as in
    Lemma~\ref{l:index-non-const}, that we will denote} \pbcorr{for
    brevity by
    $\mathcal{M}_{k+1}(\mathbb{R}^2;\mathscr{D}_{\mathbb{R}^2})$.}
  \corr{Consider also the ``vertical'' moduli spaces}
  \pbcorr{$\mathcal{M}_{k+1}(M;\mathscr{D}_M)$} \corr{that are
    considered in the construction of the Fukaya category of $M$, as
    in \S\ref{sec:fuk-M}, and are associated to the system of
    perturbations} \pbcorr{$\mathscr{D}_M$ and strip-like ends,}
  \corr{compatible with gluing. Here $k$ is the number of inputs of
    the respective polygons. In both cases we will consider moduli
    spaces of arbitrary dimensions.  The horizontal ones are regular
    by Lemma~\ref{l:index-non-const}. Standard arguments imply that,
    for a generic choice of perturbations} \pbcorr{$\mathscr{D}_M$,
    the moduli spaces of the second type are also regular.} \pbcorr{We
    have two obvious projection maps:
    $p_{1}: \mathcal{M}_{k+1}(\mathbb{R}^2;
    \mathscr{D}_{\mathbb{R}^2}) \to \mathcal{R}^{k+1}$ and}
  \pbcorr{$p_{2}^{\mathscr{D}_M}:\mathcal{M}_{k+1}(M;\mathscr{D}_M)
    \to \mathcal{R}^{k+1}$.} \corr{The split nature of the
    perturbation datum in~\eqref{eq:perturb-module}} \pbcorr{implies
    that the moduli space that assembles the curves in
    $\mathcal{M}(\gamma_1, \ldots, \gamma_k, \gamma_{k+1};
    \mathscr{D}_{\mathbb{R}^2 \times M})$ is given by the fiber
    product
    $\mathcal{M}_{k+1}(\mathbb{R}^2; \mathscr{D}_{\mathbb{R}^2})
    \times_{\mathcal{R}^{k+1}} \mathcal{M}_{k+1}(M;\mathscr{D}_M)$
    over the two projections $p_{1}$ and $p_{2}^{\mathscr{D}_M}$. By
    standard arguments, a further generic choice of perturbations
    $\mathscr{D}_M$ is sufficient to make the projection
    $p_{2}^{\mathscr{D}_M}$ transverse to $p_{1}$.}
  \pbcorr{Consequently, for a generic choice of perturbations
    $\mathscr{D}_M$ the moduli space
    $\mathcal{M}(\gamma_1, \ldots, \gamma_k, \gamma_{k+1};
    \mathscr{D}_{\mathbb{R}^2 \times M})$ is regular.} \corr{Its
    dimension is given by the sum of the dimensions of the
    ``horizontal'' and ``vertical'' spaces} \pbcorr{$-(k-2)$.}

  \pbcorr{Assume by contradiction that, $k \geq 2$ and
    $\mathcal{M}(\gamma_1, \ldots, \gamma_k, \gamma_{k+1};
    \mathscr{D}_{\mathbb{R}^2 \times M})$ has a component of dimension
    $0$ such that $\pi(\gamma_1), \ldots, \pi(\gamma_k)$ are all of odd
    type but $\pi(\gamma_{k+1})$ is of even type. Let $(r,u)$ be an
    element of the $0$-dimensional component of
    $\mathcal{M}(\gamma_1, \ldots, \gamma_k, \gamma_{k+1};
    \mathscr{D}_{\mathbb{R}^2 \times M})$.  The total index of the
    extended operator $\overline{\mathcal{D}}_{r,u}$
    is
    $$0=ind(\overline{\mathcal{D}}_{r,u})=ind^{\mathrm{hor}}(r,u) +
    ind^{\mathrm{ver}}(r,u)-(k-2),$$ where $ind^{\mathrm{hor}}$ and
    $ind^{\mathrm{ver}}$ are the obvious horizontal and vertical
    indices. They coincide with the indices of the projections of $u$,
    respectively, to $\mathbb{R}^2$ and to $M$. By}
  \corr{Lemma~\ref{l:index-non-const}, the horizontal index is}
  \pbcorr{$ind^{\mathrm{hor}}(r,u)=k-1$} \corr{which implies that}
  \pbcorr{$ind^{\mathrm{ver}}(r,u)=-1$} \corr{and leads to a
    contradiction, thus ending the proof.}
 \end{proof}

\begin{rem}
  \corr{In the earlier} \pbcorr{published} \corr{version of this paper
    this  subsection contained} \pbcorr{a mistake confusing} \corr{between
    the moduli spaces with fixed punctures and those with ``moving
    punctures''. This led to an error in the statement of} \pbcorr{the
    original version of} \corr{Lemma~\ref{l:index-non-const} as well
    as to a wrong version of Corollary~\ref{c:ind-n-const}.}
  \pbcorr{The above Lemma~\ref{l:index-non-const} and
    Corollary~\ref{c:ind-n-const} should therefore replace the
    corresponding ones in the previous version of the paper.}
  \corr{The only impact of} \pbcorr{this correction} \corr{on the rest
    of the paper} \pbcorr{has to do with the proof of
    Proposition~\ref{prop:htpy-proj}. Namely,} \corr{the possibility at  point~ii in the proof of Proposition~\ref{prop:htpy-proj} needs
    to allow for curves with multiple even entries - as indicated in
    the present version - and not just with a single even entry, as
    claimed in the earlier version. However, the potential existence
    of these more general curves is immaterial to any of the other
    arguments,} \pbcorr{formulas and results} \corr{in the paper.}

  \corr{The errors in} \pbcorr{the earlier version} \corr{of this
    subsection were discovered by Emily Campling. We are grateful to her for pointing them}
  \pbcorr{out} \corr{to us.}
\end{rem}



\subsection{The exact triangles}

\subsubsection{$A_{\infty}$-modules associated to a cobordism}
\label{subsubsec:modules}

Let $V$ be a Lagrangian cobordism $V\in \mathcal{CL}_{d}(M)$,
$\gamma\subset \C$ a curve and $h$ a profile function, as
in~\S\ref{subsec:inclusion}. (See Figure~\ref{fig:curve}.) We assume
that $\#(\gamma\cap (\phi_{1}^{h'})^{-1}(\gamma))=l$ where $h'$ is the
extended profile function associated to $h$.

Consider the Yoneda embedding
$$\mathcal{Y}:\fuk^{d}_{cob,1/2}(\C\times M;i',h')\to 
\textnormal{\sl fun}(\fuk^{d}_{cob,1/2}(\C\times M;i',h'),
Ch^{\textnormal{opp}})$$ and the composition of
functors\begin{equation}\label{eq:composition} \mathcal{Y}(V)\circ
   \mathcal{I}_{\gamma,h}\ : \
   \fuk^{d}(M)\stackrel{\mathcal{I}_{\gamma,h}}{\longrightarrow}
   \fuk^{d}_{cob,1/2}(\C\times
   M;i',h')\stackrel{\mathcal{Y}(V)}{\longrightarrow}
   Ch^{\textnormal{opp}}
\end{equation}
where $\mathcal{I}_{\gamma,h}$ is the inclusion functor from equation
(\ref{eq:functor_I-gamma-h}).  As recalled in
\S\ref{subsubsec:morphisms}, $Ch^{\textnormal{opp}}$- valued functors
defined on an $A_{\infty}$-category $\mathcal{A}$ are naturally
identified with modules over $\mathcal{A}$.  Let
$\mathcal{M}_{V,\gamma,h}$ be the $\fuk^{d}(M)$-module that
corresponds to $\mathcal{Y}(V)\circ \mathcal{I}_{\gamma,h}$.
Explicitly, this module is given by
$\mathcal{M}_{V,\gamma,h}(N)=CF^{cob}(\gamma\times N, V)$ for all
$N\in\mathcal{L}^{\ast}_{d}$ and
$$\mu^{\mathcal{M}_{V,\gamma,h}}: CF(N_{1},N_{2})\otimes\ldots 
CF(N_{k-1}, N_{k})\otimes CF^{cob}(\gamma\times N_{k}, V)\to
CF^{cob}(\gamma\times N_{1}, V)$$ is given by
$\mu^{\mathcal{M}_{V,\gamma,h}}(a_{1},\ldots,
a_{k-1},b)=\mu^{cob}(e^{1}_{\gamma,h}(a_{1}),\ldots, e^{1}_{\gamma,
  h}(a_{k-1}),b)$.

Recall from~\eqref{eq:functor_I-gamma-h} that $\mathcal{I}_{\gamma, h}
= \beta_{\gamma, h} \circ e_{\gamma, h}$. Thus we also have a functor
$$\mathcal{Y}(V)\circ \beta_{\gamma,h}\ : \ \mathcal{B}_{\gamma, h}\to
\fuk^{d}_{cob,1/2}(\C\times M;i',h')\to Ch^{\textnormal{opp}}$$ with a
corresponding $\mathcal{B}_{\gamma,h}$-module
$\overline{\mathcal{M}}_{V,\gamma,h}$ and we have
$$\mathcal{M}_{V,\gamma,h} = 
e_{\gamma,h}^{\ast}(\overline{\mathcal{M}}_{V,\gamma,h})~.~$$ In view
of Remark \ref{rem:var-snakes} and of the invariance properties
of $\fuk_{cob}^{d}(\C\times M)$, it follows that the quasi-isomorphism
type of the module $\mathcal{M}_{V,\gamma,h}$ only depends on the
horizontal isotopy types of $\gamma$ and $V$.  The relevant
quasi-isomorphism can be constructed explicitly for $\gamma, h$ fixed
and $V$ horizontally isotopic to $V'$ via an isotopy $\Phi$. The
result is a quasi-isomorphism
$$\Phi^{V}_{V'}:\overline{\mathcal{M}}_{V,\gamma,h}\to 
\overline{\mathcal{M}}_{V',\gamma,h}$$ that is constructed by counting
$J$-holomorphic polygons with boundary conditions along $\gamma \times
N_{1}, \ldots, \gamma\times N_{k}$ and with the $k+1$ side satisfying a
moving boundary condition along $\Phi_{t}(V)$.

\subsubsection{Exact triangles associated to a cobordism}
\label{subsubsec:triangles}
In this subsection we use the theory developed above to construct the
exact triangles described in the introduction in
equation~\eqref{eq:funct-exact}. Even if we are interested in these
triangles at the derived level it is important to note that the actual
construction needs to be performed before derivation and thus it is by
necessity very explicit.

Let $V: L\cobto (L_{1},\ldots, L_{m})$ be a Lagrangian cobordism in
$\mathcal{CL}_{d}(\C\times M)$.  Consider a sequence of plane curves
$\gamma_{1},\ldots, \gamma_{m}\subset \C$ as in
Figure~\ref{fig:many-snakes}. We also choose profile functions $h_{i}$
and associated extended profile functions $h'_{i}$ so that
$(\phi_{1}^{h'_{i}})^{-1}(\gamma_{i})$ are as in Figure
\ref{fig:many-snakes} (where it is drawn for $m=4$).

\begin{figure}[htbp]
   \begin{center}
      \epsfig{file=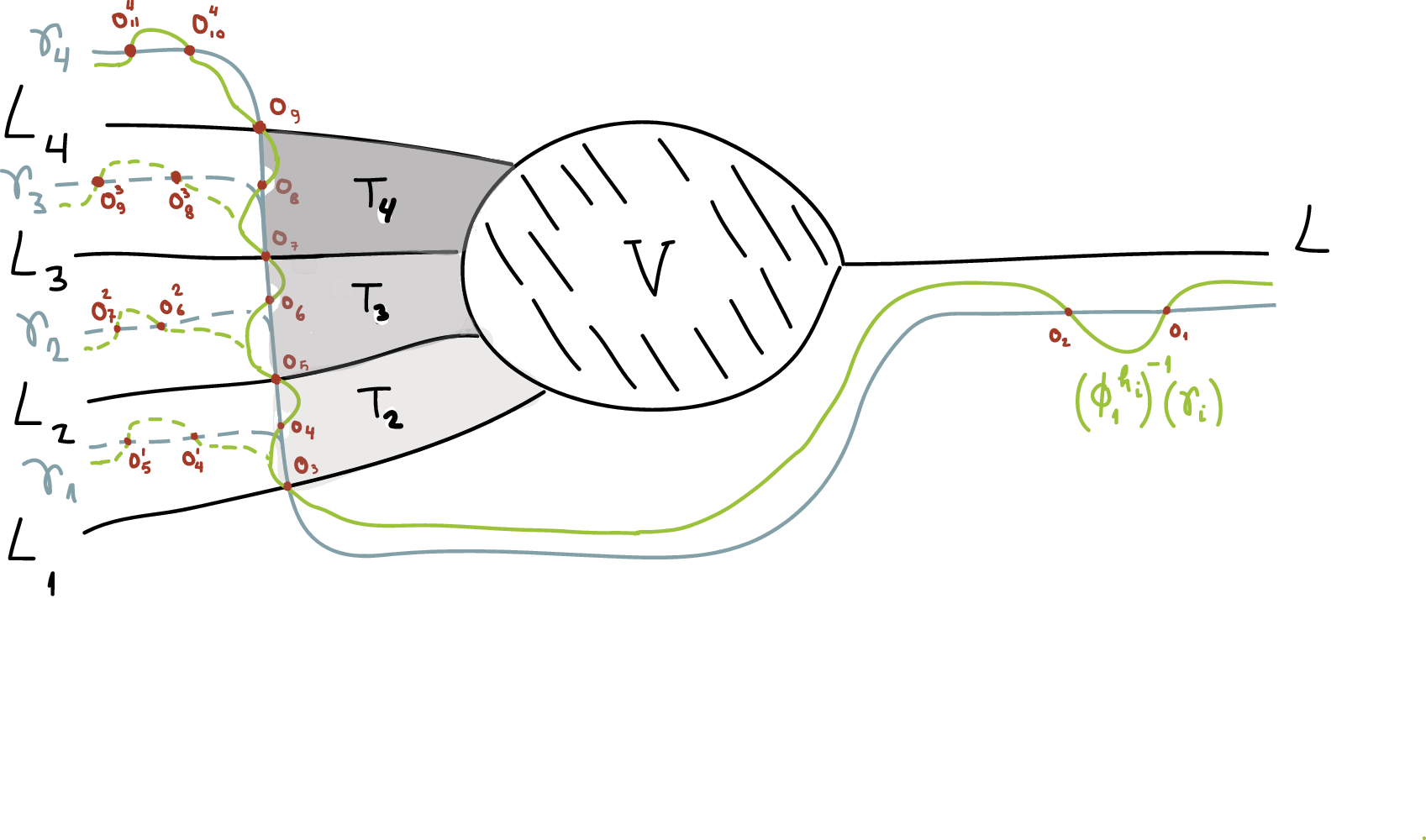, width=0.90\linewidth}
   \end{center}
   \caption{\label{fig:many-snakes} The cobordism $V$, the curves
     $\gamma_{j}$, and $(\phi_{1}^{h'_{j}})^{-1}(\gamma_{j})$, the
     points $o^{j}_{i}$ and the regions $T_{i}$.}
\end{figure}
More precisely, our choices are as follows. As before, we identify
$\mathbb{C} \cong \mathbb{R}^2$ in the standard way. All the curves
$\gamma_i$ will be of the (general) shape depicted in
Figure~\ref{fig:curve} but with different heights for their ends. We
first pick $\gamma_{m}$ so that the height of its negative end is
$m+\frac{1}{2}$ and the height of its positive end is $\frac{3}{2}$.
Additionally we require that $\gamma_{m}$ coincides with a vertical
line of real coordinate $\eta_{\gamma_{m}}<0$ in between the
intersections with the $L_{1}$ end of $V$ and the intersection with
the $L_{m}$ end. We now pick $h_{m}$ so there are $2m+3$ intersection
points $\{o^{m}_{i}\}=(\phi_{1}^{h'_{m}})^{-1}(\gamma_{m})\cap
\gamma_{m}$ that are so that $o^{m}_{1}$ is a bottleneck and $o^m_1 =
(\frac{5}{2},\frac{1}{2})$. Similarly, we require that $o^{m}_{2m+3}$
is a bottleneck and $o^{m}_{2m+3} = (-\frac{3}{2},m+\frac{1}{2})$. The
positive intersections satisfy $o^{m}_{2i+1}= \gamma_{m}\cap (i
\textnormal{'th negative end of } V)$, $1\leq i\leq m$. Further, we
assume that the negative intersections $o^{m}_{2i}$ are of height
$i-\frac{1}{2}$, $1\leq i\leq m+1$. The choice of the curves
$\gamma_{j}$ for $j<m$ is such that this curve has essentially the
same shape as $\gamma_{m}$, it coincides with $\gamma_{m}$ below the
height $i+\frac{1}{4}$ and its negative end is of height
$i+\frac{1}{2}$.  Moreover, $h'_{j}$ is such that
$(\phi_{1}^{h'_{j}})^{-1}(\gamma_{j})$ coincides with
$(\phi_{1}^{h'_{m}})^{-1}(\gamma_{m})$ for all points $(x,y)\in\C$
with $y\leq j+\frac{1}{4}$. This means, in particular, that we have
$o^{j}_{i}=o^{m}_{i}$ for all $1\leq i\leq 2j+1$.  Finally,
$o^{j}_{2j+3}$ is a bottleneck with $o^{j}_{2j+3} = (-\frac{3}{2},
j+\frac{1}{2})$ and $o^{j}_{2j+2}$ is of height $j+\frac{1}{2}$.

Next we pick the Floer and perturbation data required to define the
categories $\mathcal{B}_{j}=\mathcal{B}_{\gamma_{j},h_{j}}$, as in the
construction of the inclusion functors in~\S\ref{subsec:inclusion}.
Additionally, we require that the data for
$\mathcal{B}_{\gamma_{j+1},h_{j+1}}$ extends in the obvious sense the
data for $\mathcal{B}_{\gamma_{j},h_{j}}$.  We now define a sequence
of auxiliary $A_{\infty}$-categories $\mathcal{B}'_{j}$ for each
$1\leq j\leq m$. These have the same objects as $\mathcal{B}_{j}$ and
their morphisms are defined algebraically as the following quotients:
\begin{equation} \label{eq:hom-B'}
   \hom^{\mathcal{B}'_j}(\widetilde{N}', \widetilde{N}'') = 
   \hom^{\mathcal{B}_j}(\widetilde{N}', \widetilde{N}'') / \langle
   o^j_{2j+2}, o^j_{2j+3} \rangle,
\end{equation}
where the ``denominator'' in the quotient stands for the subspace
generated by the intersection points \pbredb{that project to}
$o^j_{2j+2}$ and $o^j_{2j+3}$. There is an obvious projection
$$p_{j}:\mathcal{B}_{j}\to \mathcal{B}'_{j}$$
given by a formula analogous to that of $p^{h}_{g}$ from
Remark~\ref{rem:comparison}.  The description of the higher
compositions in the proof of Propositions~\ref{prop:htpy-proj} shows
that $\mathcal{B}'_{j}$ so defined and endowed with the obvious
compositions inherited from $\mathcal{B}_{j}$ is an
$A_{\infty}$-category and that $p_{j}$ is an $A_{\infty}$-functor.
Notice also that there is a further, similar, projection
$$q_{j}:\mathcal{B}'_{j}\to \mathcal{B}'_{j-1}$$
that sends to $0$ the morphisms of type $o^{j}_{i}$ with $i= 2j,
2j+1$.  Again this is well defined due to the description of the
higher compositions in $\mathcal{B}_{\gamma,h}$, the fact that the
points $o^{j}_{i}$ coincide with $o^{j-1}_{i}$ for $1\leq i\leq 2j-1$
together with our choices of Floer and perturbation data.  Recall also
the map $e_{\gamma, h}$ from (\ref{eq:functor_I-gamma-h}) and let
$e_{j}=e_{\gamma_{j},h_{j}}:\fuk^{d}(M)\to \mathcal{B}_{j}$ and
$e'_{j}=p_{j}\circ e_{j}$. It is immediate to see that:
\begin{equation}\label{eq:e-commute}
e'_{j-1}=q_{j}\circ e'_{j}~.~
\end{equation}
We recall also the projection functors $c_{\gamma, h, i}$ from
Propositions~\ref{prop:htpy-proj}. Specifically in our case, the
$c_{\gamma_{j}, h_{j}, i}:\mathcal{B}_{j}\to \fuk^{d}(M)$ is the
projection on the morphisms of type $o^{j}_{i}$. Clearly, it descends
to $c'_{\gamma_{j}, h_{j}, i}:\mathcal{B}'_{j}\to \fuk^{d}(M)$ as long
as $i\leq 2j+1$.  For each $j$, we will mainly be interested in the
last of these functors which we denote for brevity by
$\sigma_{j}=c'_{\gamma_{j}, h_{j}, 2j+1}$.

The fundamental result here concerns the structure of the modules
$\overline{\mathcal{M}}_{j}=\overline{\mathcal{M}}_{V,\gamma_{j},h_{j}}$.
\begin{prop}\label{prop:str-mod} For all $1\leq j\leq m$,
   there are $\mathcal{B}'_{j}$-modules $\mathcal{M}'_{j}$ with the
   following properties:
\begin{itemize}
  \item[i.] $p^{\ast}_{j}(\mathcal{M}'_{j})=\overline{\mathcal{M}}_{j}$.
  \item[ii.] $\mathcal{M}'_{1}=\sigma_1^{\ast} \, \mathcal{Y}(L_{1})$
  \item[iii.] For $j\geq 2$, there is a $\mathcal{B}'_{j}$-module
   morphism $$\phi_{j}: \sigma_{j}^{\ast} \,
   \mathcal{Y}(L_{j})\longrightarrow q^{\ast}_{j} \,
   \mathcal{M}'_{j-1}$$ so that
   $\mathcal{M}'_{j}=\textnormal{Cone}(\phi_{j})~.~$
\end{itemize}
Here $\mathcal{Y}(L_{j})$ is the $\fuk^{d}(M)$-module associated to
$L_{j}$ by the Yoneda embedding and $\textnormal{Cone}(\phi_{j})$ is
the cone over the module morphism $\phi_j$
(see~\S\ref{sbsb:triang-derived}).
\end{prop}
\begin{proof}
   Fix $j\in\{1,\ldots, m\}$. The proof is based on the properties of
   the compositions $\mu^{\overline{\mathcal{M}}_{j}}$.

   We will use the notation in \S\ref{prop:htpy-proj} that we first
   adapt to the module $\overline{\mathcal{M}}_{j}$. For each $N\in
   \mathcal{L}^{\ast}_{d}$, we notice that
   $\overline{\mathcal{M}}_{j}(N)=CF^{cob}(\gamma_{j}\times N, V)$
   splits as a vector space as
   \begin{equation}\label{eq:split-mod0}
      \overline{\mathcal{M}}_{j}( N)=\bigoplus_{i=1}^{j} CF(N, L_{i})~.~
   \end{equation}
   We will denote the summand $CF(N,L_{i})$ in this decomposition by
   $CF^{\overline{\mathcal{M}}_j}(N,L_{i})$ and we denote the partial
   sum $\oplus_{i=1}^{s}CF(N,L_{i})$ by
   $\overline{\mathcal{M}}_{j}(N)^{\leq s}$.  The generators of
   $CF^{\overline{\mathcal{M}}_j}(N,L_{i})$ project to the point
   $o^{j}_{2i+1}\in \C$ and thus we will refer to the elements in
   $CF^{\overline{\mathcal{M}}_j}(N,L_{i})$ as the elements of
   $\overline{\mathcal{M}}_{j}(N)$ of type $o^{j}_{2i+1}$.

   Fix $N_{1},\ldots, N_{k}\in\mathcal{L}^{\ast}_{d}$ and $$z_{i}\in
   CF(N_{i},N_{i+1})^{o_{l_{i}}}\subset
   CF^{\mathcal{B}_{\gamma_{j},h_{j}}}(\widetilde{N}_{i},\widetilde{N}_{i+1})
   \ , \ b\in CF^{\overline{\mathcal{M}}_j}(N_{k}, L_{s})\subset
   CF^{cob}(\widetilde{N}_{k},V)$$ where $1\leq l_{i}\leq 2j+3$,
   $1\leq s \leq j$ and $\widetilde{N}_{1}=\gamma_{j}\times
   N_{1},\ldots, \widetilde{N}_{k}=\gamma_{j}\times N_{k}$.

   \begin{lem}\label{lem:technique2} With the notation above we have:
      \begin{itemize}
        \item[a.]  $\mu^{\overline{\mathcal{M}}_{j}}(z_{1},\ldots,
         z_{k-1}, b)\in \overline{\mathcal{M}}_{j}(N)^{\leq s}~.~$
        \item[b.] If $\mu^{\overline{\mathcal{M}}_{j}}(z_{1},\ldots,
         z_{k-1}, b)\not=0$, then $l_{i}\leq 2s+1$ for all
         $i=1,\ldots, k-1$.
      \end{itemize}
   \end{lem}
   \begin{proof}[Proof of Lemma~\ref{lem:technique2}]
      Consider a perturbed $J$-holomorphic polygon $u:S_{r}\to
      \C\times M$ that contributes to
      $\mu^{\overline{\mathcal{M}}_{j}}_{k}$. Thus $u$ satisfies the
      equation~\eqref{eq:perturb-module} with boundary conditions
      along $\widetilde{N}_{1},\ldots, \widetilde{N}_{k}, V$. It has
      entries $z_{1},\ldots, z_{k-1}, b$ and denote its exit by $c\in
      CF^{\overline{\mathcal{M}}_j}(N_{1}, L_{t})$, where $1 \leq t
      \leq j$.

       {To simplify the argument, we will choose the
         profile functions $h_j$ and the perturbation data for all
         tuples of the type $\gamma_j \times N_1, \ldots, \gamma_j
         \times N_k, V$ to be so that the Hamiltonian flow
         $\phi_t^{h_j'}$ of the extended profile functions $h'_j$ keep
         $V$ invariant.  Moreover, choose the $\Theta_0$-part of the
         perturbation data $(\Theta, \mathbf{J})$ to be so that the
         Hamiltonian functions associated to $\Theta_0$ all vanish
         near $V$, similar to the choices made for the category
         $\mathcal{B}_{\gamma, h}$ in~\S\ref{subsec:inclusion}. This
         is all possible since $\gamma_j$ intersects $\pi(V)$
         transversly.  See Figure~\ref{fig:many-snakes}. Note also
         that the negative intersection points between $\gamma_j$ and
         $(\phi_1^{h_j'})^{-1}(\gamma_j)$ are away from $\pi(V)$.}

       We again denote by $v$ the polygon obtained from $u$ by the
       naturality equation~\eqref{eq:nat-v} and we let $v'=\pi\circ v$
       be its planar projection.  We now notice that $v'$ is
       holomorphic in a neighborhood of $\partial S_{r}$. Moreover,
       $v'$ is also holomorphic at all points $z\in S_{r}$ so that
       $v'(z)$ belongs to the complement of the region $K_{j}\subset
       \C$, where
       $$K_{j}=  U_{h_{j}}\ \cup\ 
       \bigcup_{\tau=0}^{1}(\phi_{\tau}^{h'_{j}})^{-1}(\gamma_{j}) \
       \cup \ \pi(V).$$ Here $U_{h_{j}}$ stands for the union of small
       neighborhoods of the negative points $o^{j}_{2i}$ used in the
       construction of $\mathcal{B}_{\gamma_{j},h_{j}}$,
       see~\S\ref{subsec:inclusion}.  By Proposition~\ref{p:open-map},
       as soon as a component of $\C \setminus K_{j}$ intersects the
       image of $v'$, it is completely contained in it.  In
       particular, $v'(S_{r})$ can not intersect any unbounded such
       component. To continue, it is useful to identify the bounded
       connected components of $\C\backslash K_{j}$: \pbredb{we denote
         by $T_{i}$, $i=2, \ldots, j$, the (bounded) connected
         component whose closure contains the points $o^{j}_{2i-1}$
         and $o^{j}_{2i+1}$;} there are in all $j-1$ such components,
       see again Figure~\ref{fig:many-snakes}.  An important property
       of these components is that if $T_{i}\subset v'(S_{r})$ then
       one of the following four possibilities occurs:
      \begin{itemize}
        \item[1.] {there is some point $z\in S_{r}$ so that
           $v'(z)=o^{j}_{2i+1}$.}
        \item[2.]  {$o^{j}_{2i+1}$ is the projection of one of
           the $z_{q}$'s and the strip-like end corresponding to this
           $z_{q}$ has an image that covers the ``great angle''
           between $\gamma_{j}$ and
           $(\phi_{1}^{h'_{j}})^{-1}(\gamma_{j})$, see Figure
           \ref{fig:big-angle}.}
        \item[3.] {$\pi(c) = o^{j}_{2i+1}$.}
        \item[4.] {$\pi(b) = o^{j}_{2i+1}$.}
      \end{itemize}
   
      \begin{figure}[htbp]
         \begin{center}
            \epsfig{file=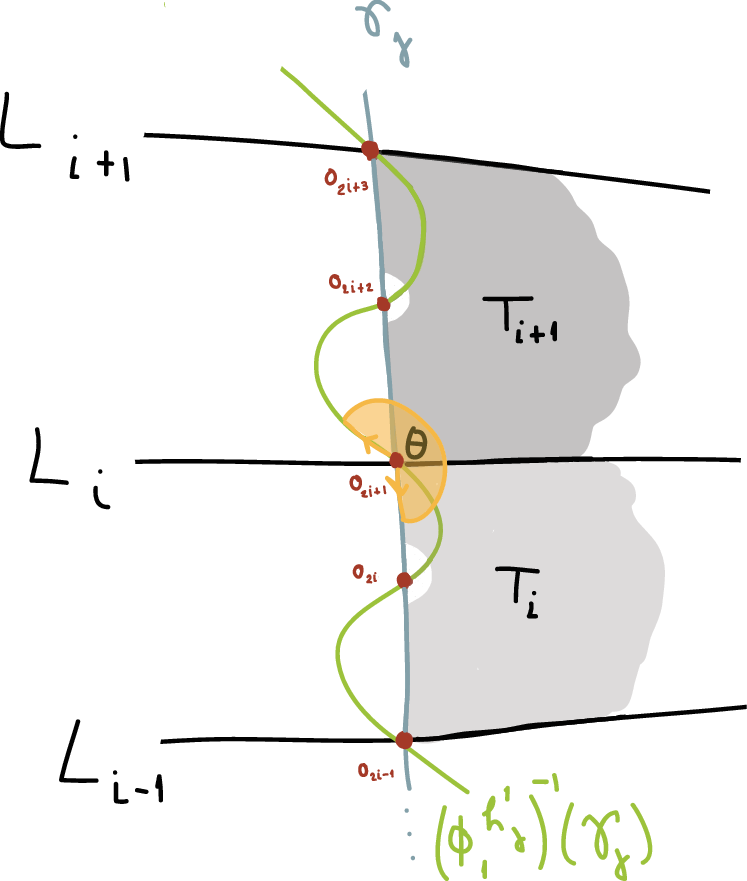, width=0.50\linewidth}
         \end{center}
         \caption{\label{fig:big-angle} The ``great angle'' between
           $\gamma_{j}$ and $(\phi_{1}^{h'_{j}})^{-1}(\gamma_{j})$. }
      \end{figure} 
      Given that $c$ is an exit point for $u$, we notice that if one
      of 1, 2, or 3 takes place, then the connected region of
      $\C\backslash K_{j}$ above $T_{i}$ is also contained in
      $v'(S_{r})$. We now claim that if possibility 4 does not occur
      then:
      \begin{equation} \label{eq:first-end} \mathrm{if\ } T_{i}\subset
         v'(S_{r}) \; \; \textnormal{then} \; \; i+1 \leq s~.~
      \end{equation}
      In other words, if $v'$ intersects $T_{i}$, then the point
      $o^{j}_{2s+1}$, which is the projection of the entry $b$, is
      ``above'' $T_{i}$. To see this assume $s\leq i$. Then the
      argument above shows that case 4 does not apply and thus the
      region above $T_{i}$ is also in $v'(S_{r})$. In case $i=j-1$ we
      have already arrived to a contradiction as the region above
      $T_{j-1}$ is not bounded. If $i< j-1$ this region is $T_{i+1}$
      and by repeating the same argument now applied to $T_{i+1}$ we
      deduce that the region above $T_{i+1}$ is again in $v'(S_{r})$.
      Recursively, we arrive to a contradiction.  The properties
      1,2,3,4 also apply to the point $o^{j}_{2i+1}$ and an argument
      similar to the above, but using the fact that $b$ is an entry
      point, shows that
      \begin{equation} \label{eq:second-end} \textnormal{if} \;\;
         T_{i}\subset v'(S_{r}) \;\; \textnormal{then} \;\; t\leq i.
      \end{equation}
      If $v'$ is not constant, then $T_{s-1}\subset v'(S_{r})$. Thus,
      putting together~\eqref{eq:first-end}
      and~\eqref{eq:second-end}, we obtain $t<s$ which shows the
      point a in the statement of our lemma. Notice that
      \begin{equation}\label{eq:cst-curves-mod}
         \textnormal{if} \;\; t=s \;\; \textnormal{then} \;\;
         v' \;\; \textnormal{is constant}~.~ 
      \end{equation}
      Indeed, if $v'$ is not constant, then $v'(S_{r})$ has to contain
      $T_{s-1}$ and this implies by~\eqref{eq:first-end}
      and~\eqref{eq:second-end} that $t< s$.

      To show the point b we clearly may assume that the curve $v'$ is
      not constant.  Suppose that there is some $q\in {1,\ldots, k-1}$
      so that $o_{j_{q}}> 2s+1$. Given that $t<s$ and using that the
      points $z_{i}$ are entries, it follows that one of the following
      two possibilities takes place:
      \begin{itemize}
        \item[1'.] there is a point $z$ belonging to one of the first
         $k-1$ edges of $S_{r}$ (these are the edges corresponding to
         the boundary conditions $\widetilde{N}_{1},\ldots,
         \widetilde{N}_{k-1}$) so that $v'(z) = o^{j}_{2s+1}$.
        \item[2'.] there is some $q'$ so that $\pi(z_{q'})=o^{j}_{2s+1}$
         and the image by $v'$ of the strip-like end associated to
         $z_{q'}$ covers the ``great angle'' between $\gamma_{j}$ and
         $(\phi_{1}^{h_{j}})^{-1}(\gamma_{j})$ at $o^{j}_{2s+1}$.
      \end{itemize}
      Both 1' and 2' show that the connected region of $\C\backslash
      K_{j}$ above $o^{j}_{2s+1}$ intersects $v'(S_{r})$ which
      contradicts (\ref{eq:first-end}) and ends the proof of the
      lemma.
   \end{proof}

   We now return to the proof of Proposition~\ref{prop:str-mod}. The
   first step is to define the $\mathcal{B}'_{j}$-modules
   $\mathcal{M}'_{j}$. First recall that the $A_{\infty}$-category
   $\mathcal{B}'_{j}$ has the same objects as $\mathcal{B}_{j}$ given
   by $\widetilde{N}=\gamma_{j}\times N$ with $N\in
   \mathcal{L}^{\ast}_{d}$ and its morphisms are obtained by taking
   the quotients~\eqref{eq:hom-B'}. In particular, there is a
   canonical identification of vector spaces:
   $$CF^{\mathcal{B}_{j}}(\widetilde{N},\widetilde{N}')^{\leq
     2j+1}=CF^{\mathcal{B}'_{j}}(\widetilde{N},\widetilde{N}')~.~$$
   Let $N\in \mathcal{L}^{\ast}_{d}$ and put
   $\mathcal{M}'_{j}(N)=CF^{cob}(\gamma_{j}\times N,V)=
   \overline{\mathcal{M}}_{j}(N)$.  If $N_{1},\ldots, N_{k}\in
   \mathcal{L}^{\ast}_{d}$ and $z_{i}\in
   CF^{\mathcal{B}_{j}}(\widetilde{N}_{i},\widetilde{N}_{i+1})^{\leq
     2j+1}$, $b\in CF^{cob}(\widetilde{N}_{k}, V)$ we put:
   $$\mu^{\mathcal{M}'_{j}}(z_{1},\ldots, z_{k-1},b)=
   \mu^{\overline{\mathcal{M}}_{j}}(z_{1},\ldots, z_{k-1},b)~.~$$ It
   is an immediate consequence of Lemma \ref{lem:technique2} b that
   this provides indeed a $\mathcal{B}'_{j}$-module. It also follows
   that $p_{j}^{\ast}(\mathcal{M}'_{j})=\overline{\mathcal{M}}_{j}$
   thus showing the first point of the proposition.

   For the second point first fix the notation $\mathcal{M}_{L_{i}}$
   for the $\fuk^{d}(M)$-module corresponding to $\mathcal{Y}(L_{i})$.
   Next notice that the module $\mathcal{M}'_{1}$ has the form
   $\mathcal{M}'_{1}(N)=\overline{\mathcal{M}}_{j}(N)^{\leq 1} =
   \mathcal{M}_{L_1}(N)$.  Using~\eqref{eq:cst-curves-mod} we deduce
   that all perturbed $J$-holomorphic polygons $u$ computing
   $\mu^{\mathcal{M}'_{1}}$ have the property that the associated
   curves $v'$ are constant.  This means that
   $\mu^{\mathcal{M}'_{1}}_{k}$ can be identified with
   $\mu^{\mathcal{M}_{L_{1}}}_{k}$ as long as the first $k-1$ inputs
   in $\mu^{\mathcal{M}'_{1}}_{k}$ are all elements in
   $CF^{\mathcal{B}_{1}}(\gamma_{j}\times N_{i}, \gamma_{j}\times
   N_{i+1})$ of type $o^{1}_{3}$. Moreover,
   $\mu^{\mathcal{M}'_{1}}_{k}$ is zero as soon as one of these inputs
   is homogeneous of a different type.  As $\sigma_{1}$ is the
   projection on $o^{1}_{3}$ this implies claim ii.

   We now prove iii.  The claim follows if we can show that the module
   $\mathcal{M}'_{j}$ satisfies
   \begin{equation}\label{eq:split-module-vs}
      \mathcal{M}'_{j}(N)=\mathcal{M'}_{j-1}(N)\oplus \mathcal{M}_{L_{j}}(N)
   \end{equation}
   for each $N\in\mathcal{L}^{\ast}_{d}$ and that, with respect to
   this splitting, the structure maps $\mu^{\mathcal{M}'_{j}}$ have
   the form:
   \begin{equation}\label{eq:split-module-mult}
      \mu^{\mathcal{M'}_{j}} = (\mu^{q_{j}^{\ast}\mathcal{M}'_{j-1}}\ ,\  
      \phi_{j}\oplus \mu^{\sigma_j^{\ast}\mathcal{M}_{L_{j}}} )
   \end{equation} 
   Indeed, the fact that $\mathcal{M}'_{j}$ is a
   $\mathcal{B}_{j}$-module then implies that the generalized
   multilinear map $\phi_{j}$ is in fact a module morphism
   $\phi_{j}:(c^{j})^{\ast}\mathcal{M}_{L_{j}}\to
   q_{j}^{\ast}\mathcal{M}'_{j-1}$ and the claim follows from the
   definition of the cone over a module morphism.  Both properties
   (\ref{eq:split-module-vs}) and (\ref{eq:split-module-mult}) depend
   in an essential way on the fact that our choices of data and
   perturbations used in the construction of
   $\mathcal{B}_{j}=\mathcal{B}_{\gamma_{j},h_{j}}$ extend the choices
   for $\mathcal{B}_{j-1}$. We first deal with
   (\ref{eq:split-module-vs}).  Recall that
   $\mathcal{M}_{L_{j}}(N)=CF(N,L_{j})$ and that for $i\leq 2j-1$ we
   have $o^{j}_{i}=o^{j-1}_{i}$. These facts together with the
   definition of $\mathcal{M}'_{j}$ and (\ref{eq:split-mod0}) imply
   immediately (\ref{eq:split-module-vs}).  Passing now to
   (\ref{eq:split-module-mult}) we see that the two points of Lemma
   \ref{lem:technique2} used together imply that the first term of
   $\mu^{\mathcal{M}'_{j}}$ is indeed the pull-back to
   $\mathcal{B}'_{j}$ of the multiplication
   $\mu^{\mathcal{M}'_{j-1}}$. The second term is the part of the
   multiplication involving elements of $\mathcal{M}_{L_{j}}$. This
   term decomposes in two parts, the first, $\phi_{j}$, with values in
   $\mathcal{M}'_{j-1}(-)$.  The second, $\psi_{j}$, with values in
   $\mathcal{M}_{L_{j}}(-)$. From this point on the argument is
   similar to that at the point ii: by (\ref{eq:cst-curves-mod}) the
   perturbed $J$-holomorphic curves $u$ computing $\psi_{j}$ are so
   that the associated planar curves $v'$ are constant equal to
   $o^{j}_{2j+1}$. This shows that $\psi_{j}(z_{1},\ldots, z_{k-1},
   -)$ can be identified with $\mu^{\mathcal{M}_{L_{j}}}$ as long as
   all $z_{i}\in CF^{\mathcal{B}_{j}}(\gamma_{j}\times N_{i},
   \gamma_{j}\times N_{i+1})$ are of type $o^{j}_{2j+1}$ and
   $\psi_{j}$ vanishes if one of the $z_{i}$'s is homogeneous and of a
   different type. Given that $\sigma_{j}$ is the projection on
   $o^{j}_{2j+1}$ this means precisely that $\psi_{j}$ is the higher
   composition in $\sigma_{j}^{\ast} \mathcal{M}_{L_{j}}$ which
   concludes the proof of the proposition.
\end{proof}

Having proved Proposition~\ref{prop:str-mod}, the next step is to use
it to relate the associated relevant $\fuk^{d}(M)$-modules (rather than
$\mathcal{B}'_j$-modules). We will make again use of the
functors
$$e_{j}:\fuk^{d}(M)\to \mathcal{B}_{j}, \quad  e'_{j}=p_{j}\circ
e_{j}:\fuk^{d}(M)\to\mathcal{B}'_{j}$$ defined just before
Proposition~\ref{prop:str-mod} (see also~\eqref{eq:e-commute}).
Recall also from~\S\ref{subsubsec:modules} that
$$\mathcal{M}_{V,\gamma_{j}, h_{j}}=e^{\ast}_{j}(\overline{\mathcal{M}}_{j})~.~$$

\begin{cor} \label{cor:exact-triangles} The $\fuk^{d}(M)$-modules
   $\mathcal{M}_{V, \gamma_{j}, h_{j}}$ have the following properties:
   \begin{itemize}
     \item[i.] $\mathcal{M}_{V,\gamma_{1},h_{1}}=\mathcal{Y}(L_{1})$
     \item[ii.] there are exact triangles
      \pbredb{
        \begin{equation} \label{eq:exact-tr-mod}
           \mathcal{Y}(L_{s})\stackrel{\phi_{V,s}\ }{\longrightarrow}
           \mathcal{M}_{V,\gamma_{s-1},h_{s-1}}\to
           \mathcal{M}_{V,\gamma_{s},h_{s}}, \ 2\leq s\leq m ~.~
        \end{equation}
      }
   \end{itemize}
\end{cor}
\begin{proof}
   To simplify notation we put
   $\mathcal{M}_{j}=\mathcal{M}_{V,\gamma_{j}, h_{j}}$. For the first
   claim, notice that $c^{1}\circ p_{1}\circ
   e_{1}=c_{\gamma_{1},h_{1},3}\circ e_{\gamma_{1},h_{1}}=
   id_{\fuk^{d}(M)}$. Therefore, using Proposition \ref{prop:str-mod}
   ii, we get
   $$\mathcal{M}_{V,\gamma_{1},h_{1}}=
   (c^{1}\circ p_{1}\circ
   e_{1})^{\ast}(\mathcal{Y}(L_{1}))=\mathcal{Y}(L_{1})~.~$$

   We now pass to the second point. First, recall
   from~\cite{Se:book-fukaya-categ} that a pull-back over an $A_{\infty}$-functor maps
   exact triangles to exact triangles. By
   Proposition~\ref{prop:str-mod} we know that there is an exact
   triangle of $\mathcal{B}'_{j}$-modules:
   \begin{equation}\label{eq:seq-mod}(c^{s})^{\ast}\mathcal{Y}(L_{s})\to
      q^{\ast}_{s}(\mathcal{M}'_{s-1})\to \mathcal{M}'_{s}~.~
   \end{equation}\
   We pull-back this triangle by the functor $e'_{s}=p_{s}\circ
   e_{\gamma_{s},h_{s}}$.  We recall that $c^{s}=c_{\gamma_{s},h_{s},
     2s+1}$ and that $c_{\gamma_{s}, h_{s},2s+1}\circ
   e_{\gamma_{s},h_{s}}=id_{\fuk^{d}(M)}$. Therefore, the first module
   on the left in (\ref{eq:seq-mod}) pulls-back to
   $\mathcal{Y}(L_{s})$. By the definition of
   $\mathcal{M}_{V,\gamma_{s}, h_{s}}$ and the point i of Proposition
   \ref{prop:str-mod}, we get that the pull-back of the module on the
   right in (\ref{eq:exact-tr-mod}) is precisely
   $\mathcal{M}_{V,\gamma_{s},h_{s}}$.  Finally, for the the middle
   module in (\ref{eq:exact-tr-mod}) we use the identity
   (\ref{eq:e-commute}) together with Proposition \ref{prop:str-mod} i
   to write:
   $$\mathcal{M}_{V,\gamma_{s-1},h_{s-1}}=(e'_{s-1})^{\ast}(\mathcal{M}'_{s-1})=
   (e'_{s})^{\ast}(q_{s}^{\ast} \mathcal{M}'_{s-1})$$ and $\phi_{V,j}$
   is defined as the pull-back of $\phi_{j}$.
\end{proof}

\begin{rem} \label{rem:invariance-triang} It is also useful to discuss
   the invariance properties of the triangles~\eqref{eq:exact-tr-mod}.
   It is not difficult to show - see also
   Proposition~\ref{prop:inv-incl} - that if $V'$ is horizontally
   isotopic to $V$ and the system of curves $\gamma_{j}$ and profile
   functions $h_{j}$ are replaced by, respectively, $\gamma'_{j}$,
   $h'_{j}$, then we have a diagram that commutes in homology:
   \begin{equation}\label{eq:comp-triang-mod}
      \begin{aligned}
         \xymatrix@-2pt{ \mathcal{Y}(L_{s+1})\ar[r]\ar[d] &
           \mathcal{M}_{V,\gamma_{s-1},h_{s-1}}\ar[r]\ar[d] &
           \mathcal{M}_{V,\gamma_{s},h_{s}}\ar[d] \\
           \mathcal{Y}(L_{s+1})\ar[r] & \mathcal{M}_{V',\gamma'_{s-1},
             h'_{s-1}}\ar[r] & \mathcal{M}_{V',\gamma'_{s},h'_{s}} }
      \end{aligned}
   \end{equation} 
   so that the first vertical arrow on the left is the identity and
   the two other arrows are quasi-isomorphisms. This is to be
   interpreted as follows: if all the additional structures (almost
   complex structures, Floer and perturbation data for $M$ and
   $\C\times M$ etc.) are fixed for both $(V, \{\gamma_{j}\},
   \{h_{j}\})$ as well as for $(V', \{\gamma'_{j}\}, \{h'_{j}\})$,
   subject to the restrictions required to establish
   Corollary~\ref{cor:exact-triangles}, then
   (\ref{eq:comp-triang-mod}) holds.  The dependence of the auxiliary
   data is in the sense of coherent systems. Assuming for a moment the
   auxiliary data fixed, the technique to prove
   (\ref{eq:comp-triang-mod}) is based on using the appropriate {\bf
     tot} category - as in \S\ref{sb:inv-fuk-2} - and repeating in
   this setting the arguments used to show
   Proposition~\ref{prop:str-mod}. There are no new compactness issues
   because $V$ and $V'$ are horizontally isotopic as are $\gamma_{j}$
   and $\gamma'_{j}$.  Finally, if $V$ varies inside a given
   horizontal isotopy class, but all the rest of the data is fixed,
   then the two vertical maps at the middle and right in
   (\ref{eq:comp-triang-mod}) can be described by the moving boundary
   method. In other words, they are of the type of the morphism $\Phi^{V}_{V'}$ at
   the end of~\S\ref{subsubsec:modules}.
\end{rem}

\subsection{Comparing the ends of a cobordism}
\label{subsubsec:ends-cob} We continue here in the setting and with
the notation of the previous subsection. Thus $V: L\cobto
(L_{1},\ldots, L_{m})$ is as before a Lagrangian cobordism belonging
to $\mathcal{L}_{d}(\C\times M)$.

\begin{prop}\label{prop:comp-morphisms}
   There exists a quasi-isomorphism of $\fuk^{d}(M)$-modules:
   $$\phi_{V,\gamma_{m},h_{m}}: \mathcal{Y}(L)\to \mathcal{M}_{V,\gamma_{m},h_{m}}$$
   that, in homology, depends only on the horizontal isotopy
   type of $V$.
\end{prop}
\begin{proof}
   We have seen before in \S\ref{subsubsec:modules} that the modules
   $\mathcal{M}_{V,\gamma_{m},h_{m}}$ have the property that if $V'$
   is horizontally isotopic to $V$, then there is a quasi-isomorphism
   (given by counting curves satisfying moving boundary conditions
   along the last edge of the underlying polygons):
   $$\Phi^{V'}_{V}: \overline{\mathcal{M}}_{V',\gamma_{m},h_{m}} \to 
   \overline{\mathcal{M}}_{V,\gamma_{m},h_{m}}.$$ To prove the
   statement of the proposition, let $V'$ be obtained by a horizontal
   translation of $V$ in such a way that the only intersection of
   $\gamma_{m}$ with $V'$ is along the single positive end of $V'$
   (see Figure~\ref{fig:translated-cob} for $m=4$).
   \begin{figure}[htbp]
      \begin{center}
         \epsfig{file=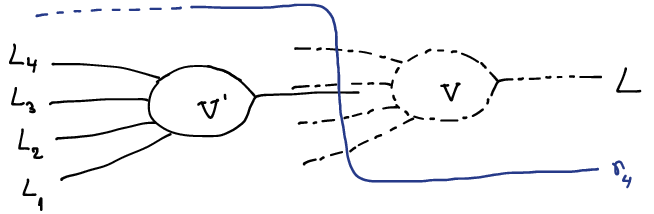, width=0.80\linewidth}
      \end{center}
      \caption{\label{fig:translated-cob} The cobordism $V$, its
        translation $V'$ and the curve $\gamma_{4}$. }
   \end{figure} 
   In this case, by the same argument showing the point i of
   Corollary~\ref{cor:exact-triangles}, we deduce
   $\mathcal{M}_{V',\gamma_{m},h_{m}}=\mathcal{Y}(L)$. We put
   $\phi_{V,\gamma_{m},
     h_{m}}=e^{\ast}_{\gamma_{m},h_{m}}(\Phi^{V'}_{V})$. The
   invariance part of the statement follows in a way similar to
   Remark~\ref{rem:invariance-triang}.
\end{proof}

\begin{rem} \label{rem:comparison-mor-discussion} The map
   $\phi_{V,\gamma_{m},h_{m}}$ is identified in homology with a map
   based on counting holomorphic polygons of exactly the same type as
   the maps $\phi_{V,j}$ from Corollary \ref{cor:exact-triangles}.
   This identification is useful in applications. Therefore we will
   justify this statement in some detail.  Consider a curve $\gamma'$
   as in Figure \ref{fig:super-bending}.
   \begin{figure}[htbp]
      \begin{center}
         \epsfig{file=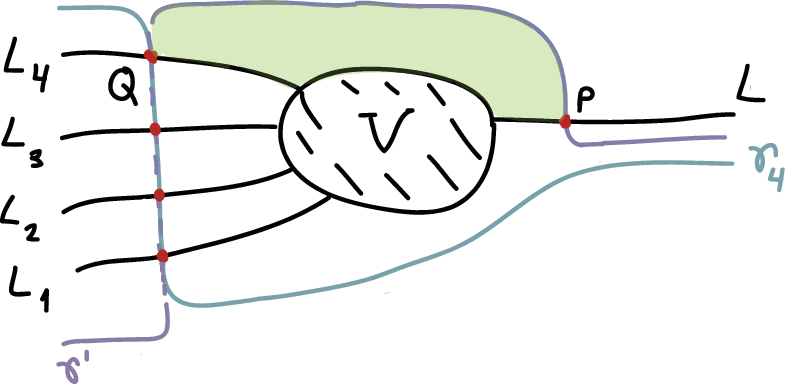, width=0.70\linewidth}
      \end{center}
      \caption{\label{fig:super-bending} The cobordism $V$, and the
        curves $\gamma_{4}$ and $\gamma'$. }
   \end{figure} 
   In particular, $\gamma'$ coincides with $\gamma_{m}$ as far as the
   intersections with the negative ends of $V$ are concerned.  We will
   also consider an associated profile function $h'$ that is also an
   extension of $h_{m}$. By the same methods as in the proof of
   Corollary \ref{cor:exact-triangles} we deduce that there is an
   additional exact triangle:
   $$\mathcal{Y}(L)\stackrel{\phi_{V,m+1}\ }{\longrightarrow}
   \mathcal{M}_{V,\gamma_{m},h_{m}}
   \stackrel{i}{\longrightarrow}Z\stackrel{p}{\longrightarrow}
   \mathcal{Y}(L)$$ so that the module $Z=cone (\phi_{V,m+1})$ is
   acyclic and thus $\phi_{V,m+1}$ is a quasi-isomorphism ($i$ is the
   inclusion here and $p$ the projection).  Geometrically, the module
   morphism $\phi_{V,m+1}$ counts $J$-polygons with the last entry
   over $P$ and ending over one of the intersections of $\gamma'$ with
   the negative ends of $V$.  It is not hard to see that
   $\phi_{V,\gamma_{m}, h_{m}}$ and $\phi_{V,m+1}$ are homologous.
   For this we translate again $V$ to $V'$ and compare the resulting
   structures associated to $V, \gamma'$ and to $V', \gamma'$.  More
   precisely, by the same compactness arguments as those used earlier
   in the paper, we obtain that the following diagram commutes in
   homology:
   \begin{eqnarray} \label{eq:commut-hats}
      \begin{aligned}
         \xymatrix{\mathcal{M}_{V,\gamma_{m},h_{m}}
           \ar[r]^{i} & Z \ar[r]^{p}& \mathcal{Y}(L)_{P}\\
           \mathcal{Y}(L)_{Q'}\ar[r]^{i}\ar[u]^{\phi_{V_{m},\gamma_{m},h_{m}}}
           & cone(id_{\mathcal{Y}(L)}) \ar[u]\ar[r]^{p}&
           \mathcal{Y}(L)_{P}\ar[u]^{id} }
      \end{aligned}
   \end{eqnarray}
   Here the vertical arrows are induced by translating from $V'$ to
   $V$ and $\mathcal{Y}(L)_{P}$ represents the module $\mathcal{Y}(L)$
   associated to the point $P\in\gamma'$ and $\mathcal{Y}(L)_{Q'}$
   represents the module $\mathcal{Y}(L)$ associated to the point $Q'$
   which is the leftmost intersection of $\gamma'$ and $V'$ - as in
   Figure \ref{fig:super-bending2}. This notation is based on the
   decomposition (as vector spaces) $CF(\gamma'\times N,
   V)=CF(N,L)\oplus \mathcal{M}_{V,\gamma_{m},h_{m}}(N)$ for
   $N\in\mathcal{L}^{\ast}_{d}(M)$ where the first term corresponds to
   generators that project to $P$ and so we denote it by
   $\mathcal{Y}(L)_{P}$. A similar convention is applied to $V'$ and
   $\gamma'$ relative to the point $Q'$.
   \begin{figure}[htbp]
      \begin{center}
         \epsfig{file=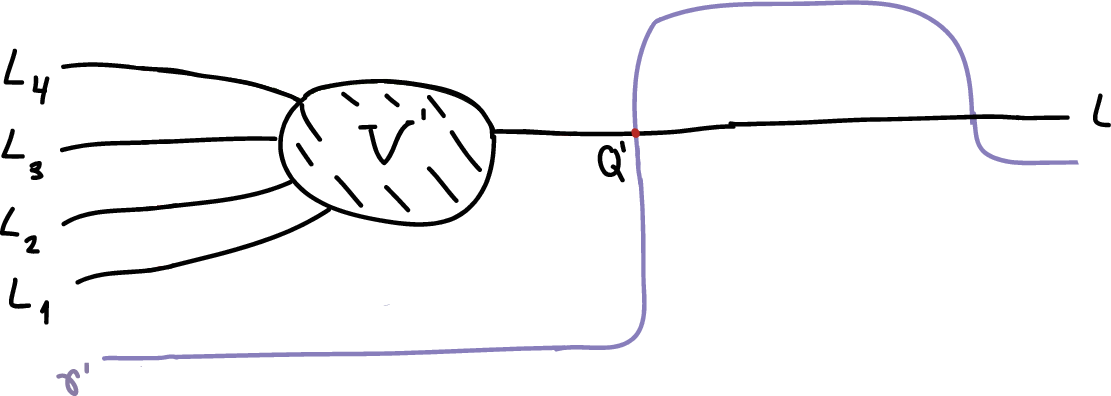, width=0.70\linewidth}
      \end{center}
      \caption{\label{fig:super-bending2} The cobordism $V'$, and the
        curves $\gamma'$. }
   \end{figure}

   By extending the top and bottom exact sequences one more step to
   the right we obtain that in homology $[\phi_{V,\gamma_{m},
     h_{m}}]=[\phi_{V,m+1}]$.
\end{rem}

\subsection{Definition of $\widetilde{\mathcal{F}}$}
\label{subsec:def-functor-tilde}

The functor
$$\widetilde{\mathcal{F}}:\mathcal{C}ob_{0}^{d}(M)\to T^{s}D\fuk^{d}(M)$$
is defined on objects by $\widetilde{\mathcal{F}}(L_{1},\ldots, L_{m})
= (L_{1},\ldots, L_{m})$.  To define it on morphisms, recall that the
morphisms in $\mathcal{C}ob_{0}^{d}(M)$ are obtained from cobordisms
of the type $V: N \cobto (K_{1}, \ldots , K_{r})$, $V\in
\mathcal{L}_{d}(\C\times M)$ in the following way. We consider
disjoint unions of such $V$'s, say $W=V_{1}\cup V_{2}\ldots \cup
V_{r}$ and take the horizontal isotopy class $[W]$ of $W$. An
additional equivalence relation is used so as to identify cobordisms
in case some of the ends of $W$ are void - as described in
\S\ref{subsec:category-cob}: in short, each morphism of
$\mathcal{C}ob_{0}^{d}(M)$ is represented by a unique isotopy class of
a cobordism $W$ as above so that all the negative ends and all
positive ends of $W$ are non-void and are as in Definition
\ref{def:Lcobordism} except if there is a single positive end which is
void and/or there is a single negative end which is void (both can
occur at the same time, for instance if $W$ is void).
  
To define $\widetilde{\mathcal{F}}$ on morphisms we will use the
definition of the category $T^{s}D\fuk^{d}(M)$ as given
in~\S\ref{subsec:cones}. However, we emphasize that we work here in an
ungraded setting.  We will first assume that $W$ has a single
connected component $W=V$ and that $V:L\cobto (L_{1},\ldots, L_{m})$.
By the results in \S\ref{subsubsec:triangles} and
\S\ref{subsubsec:ends-cob}, we can associate to this cobordism a
triple:
$$(\phi_{V,\gamma_{m},h_{m}}, \mathcal{M}_{V,\gamma_{m},h_{m}}, 
\eta_{V,\gamma_{m},h_{m}})$$ where $\phi_{V,\gamma_{m},h_{m}}:
\mathcal{Y}(L)\to \mathcal{M}_{V,\gamma_{m},h_{m}}$ is the
quasi-isomorphism given by Proposition~\ref{prop:comp-morphisms} and
$\eta_{V,\gamma_{m},h_{m}}$ is the
cone-decompositions~\eqref{eq:exact-tr-mod}.  This triple depends on
$\gamma_{m}$, $h_{m}$ as well as on all the other auxiliary data
required in the construction. However, the invariance properties of
these structures (see Remark \ref{rem:invariance-triang} and
Proposition~\ref{prop:comp-morphisms}) imply that if we take the image
$([\phi_{V,\gamma_{m},h_{m}}], [\mathcal{M}_{V,\gamma_{m},h_{m}}],
[\eta_{V,\gamma_{m},h_{m}}]) $ of this triple inside the derived
Fukaya category (or $T^{s}D\fuk^{d}(M)$) and compare it with the
triple $([\phi_{V',\gamma'_{m},h'_{m}}],
[\mathcal{M}_{V,'\gamma'_{m},h'_{m}}], [\eta_{V'
  ,\gamma'_{m},h'_{m}}]) $ associated to some $V'$ horizontally
isotopic to $V$ as well as to $\gamma'_{m}$, $h'_{m}$ etc, then the
two triples are related precisely by the equivalence described
by~\eqref{eq:cones-equiv} in~\S\ref{subsec:cones}.  Notice also that
the linearization of $[\eta_{V,\gamma_{m},h_{m}}]$ is $(L_{1},\ldots,
L_{m})$ again by (\ref{eq:exact-tr-mod}) (here $L_{i}$ and
$\mathcal{Y}(L_{i})$ are identified, of course).

Therefore, $([\phi_{V,\gamma_{m},h_{m}}],
[\mathcal{M}_{V,\gamma_{m},h_{m}}], [\eta_{V,\gamma_{m},h_{m}}])$
determines a morphism
$$\overline{([\phi_{V,\gamma_{m},h_{m}}], 
  [\mathcal{M}_{V,\gamma_{m},h_{m}}],
  [\eta_{V,\gamma_{m},h_{m}}])}:L\to (L_{1},\ldots, L_{m})$$ in
$T^{S}D\fuk^{d}(M)$ that depends only on the horizontal isotopy class
of $V$.

We then put: 
\begin{equation}\label{eq:functor-morph}
   \widetilde{\mathcal{F}}(V)=\overline{([\phi_{V,\gamma_{m},h_{m}}], 
     [\mathcal{M}_{V,\gamma_{m},h_{m}}], [\eta_{V,\gamma_{m},h_{m}}])}~.~
\end{equation}

The next step is to pass to the more general case of a disconnected
$W$. In this case we notice that both the domain and target categories
are strictly monoidal and we extend the definition of
$\widetilde{\mathcal{F}}$ monoidally.

\begin{rem}\label{rem:geom-dec} It is immediate to see that  this definition 
   is in fact compatible with the geometry of the cone-decompositions
   given in Corollary \ref{cor:exact-triangles} and of that of the
   maps $\phi_{{V},\gamma_{m},h_{m}}$ from
   Proposition~\ref{prop:comp-morphisms}. In other words, if we apply
   the constructions in~\S\ref{subsubsec:triangles}
   and~\S\ref{subsubsec:ends-cob} to a disconnected cobordism we
   obtain cone-decompositions and comparison maps between the ends
   that are the sums of the respective structures for each component
   at a time.  Indeed, this is a direct reflection of the fact that a
   complex of the form $CF^{cob}(\gamma_{m}\times N, W)$ for $W=
   V_{1}\cup \cdots \cup V_{r}$ with $V_{i}$ connected splits as
   $$CF^{cob}(\gamma_{m}\times N,W)= CF^{cob}(\gamma_{m}\times N,V_{1})
   \oplus \cdots \oplus CF^{cob}(\gamma_{m}\times N,V_{r}).$$
\end{rem}

\subsection{Compatibility with composition}
\label{subsec:comp-comp-functor}
It is easy to see that $\widetilde{\mathcal{F}}$ takes a trivial
cobordism to the identity (see~\S\ref{subsec:cones}). Thus, to show
that $\widetilde{\mathcal{F}}$ is a functor, it remains to prove that
it is compatible with composition.

Consider two connected cobordisms: $V:L\to (L_{1},\ldots, L_{m})$ and
$V':L_{i}\to (L'_{1},\ldots, L'_{q})$, where $1 \leq i \leq m$.

To fix ideas, assume that $V \subset [0,1]\times \R\times M$ and that
(after a possible rescaling) $V'\subset [0,a]\times [1,q]\times M$
with $a>2$. We also assume that $V'$ is cylindrical over $[a-1,a]
\times \mathbb{R}$ as well as over $[0,1] \times \mathbb{R}$. We now
consider the cobordism: $V'':L\to (L_{1},\ldots, L_{i-1},
L'_{1},\ldots, L'_{q}, L_{i+1},\ldots, L_{m})$, $V''\subset [-a,
1]\times \R\times M$ obtained by gluing $V' + (-a,i-1)$ (this is just
the translation of $V'$ by the vector $(-a,i-1) \in \mathbb{R}^2$) to
$V$ along the $L_{i}$ end and extending the other negative ends of $V$
by trivial cobordisms of the type $\eta_{j}\times L_{j}$ for $j\not=i$
- see Figure \ref{fig:comp-cob}.
\begin{figure}[htbp]
   \begin{center}
      \epsfig{file=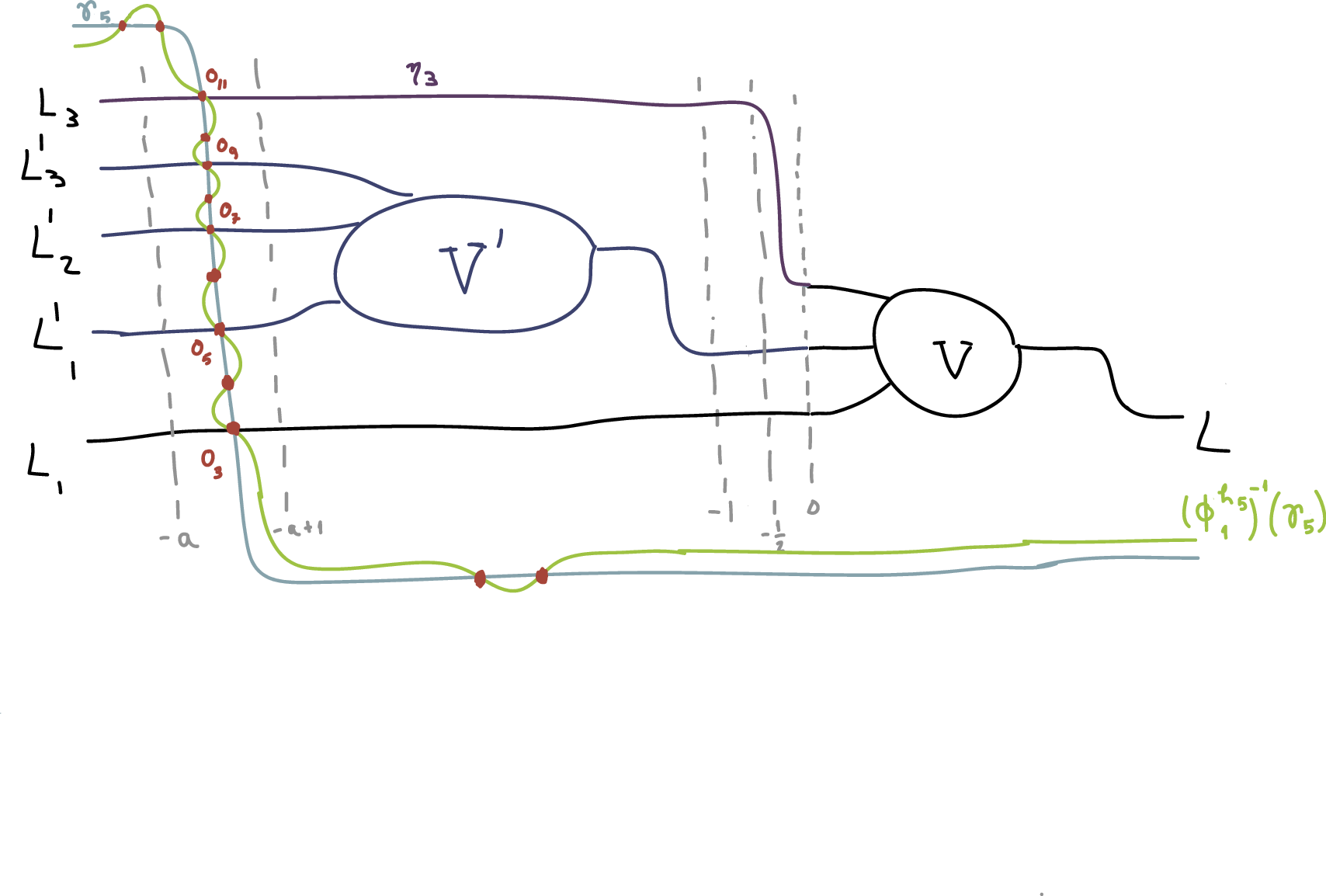, width=0.85\linewidth}
   \end{center}
   \vspace{-1.2in} \caption{\label{fig:comp-cob} The cobordism
     $V''$ together with the curve $\gamma_{5}$ and the curve
     $(\phi^{h_{5}}_{1})^{-1}(\gamma_{5})$.}
\end{figure}
Here the curves $\eta_{j}:[-a,0]\to [-a ,0]\times \R$ are so that
$\eta_{j}(0)=(0,j)$ for all $j$; for $j\geq i$ we have
$\eta_{j}(t)=(t,j+q)$ for all $t\leq -\frac{1}{2}$;
$\eta_{j}(t)=(t,j)$, for all $t\leq -\frac{1}{2}$ in case $j< i$.
Moreover, we assume that all $\eta_{j}$'s are graphs of functions that
have a derivative that is negative or null at all points.

As earlier in the paper, we will actually work in practice with the
$\R$-extensions of all the cobordisms involved -
see~\S\ref{sub:cob-def}.

In $\mor (\mathcal{C}ob^{d}_{0}(M))$ we have:
\begin{equation}\label{eq:comp-cob-compat}[V'']=(\mathrm{id}_{L_{1}}+\ldots
   +\mathrm{id}_{L_{i-1}}+ [V']+ \mathrm{id}_{L_{i+1}}+\ldots +
   \mathrm{id}_{L_{m}})\circ [V] ~.~
\end{equation}
Given the definition of the composition in $T^{S}\fuk^{d}(M)$ - see
(\ref{eq:Phi-circ-Phi'-1}) - it follows that it is enough to show that
$\widetilde{\mathcal{F}}$ is compatible with compositions of the type
in (\ref{eq:comp-cob-compat}). Put
$\mathcal{V}'=\mathrm{id}_{L_{1}}+\ldots +\mathrm{id}_{L_{i-1}}+ [V']+
\mathrm{id}_{L_{i+1}}+\ldots + \mathrm{id}_{L_{m}}$.  Thus our aim is
to show:

\begin{prop}\label{prop:compat-comp} With the notation above we have:
   $$\widetilde{\mathcal{F}}(V'')=\widetilde{\mathcal{F}}(\mathcal{V}')
   \circ \widetilde{\mathcal{F}}(V)~.~$$
\end{prop}

\begin{proof} The first step in the proof is to consider a different
   cobordism $V''_{0}$ that is obtained as follows: first consider the
   cobordism $V_{0}$ obtained by extending each negative end of $V$ by
   $[-a,0]\times \{j\}\times L_{j}$ if $j\leq i$ and by
   $\eta_{j}([-a,0]))\times L_{j}$ if $j>i$.  The cobordism $V''_{0}$
   is obtained by gluing $V'$ to $V_{0}$ along the $i$-th end and
   extending the other ends of $V_{0}$ trivially.  In other words:

   $$V''_{0}=V_{0}\cup \Bigl( V'+(-2a,i-1) \Bigr) \cup 
   \bigcup_{ j<i\ \mathrm{or}\ j>i+q } ([-2a, -a]\times \{j\}\times
   L_{j}),$$ see Figure~\ref{fig:comp-cob-far}.
   \begin{figure}[htbp]
      \begin{center}
         \epsfig{file=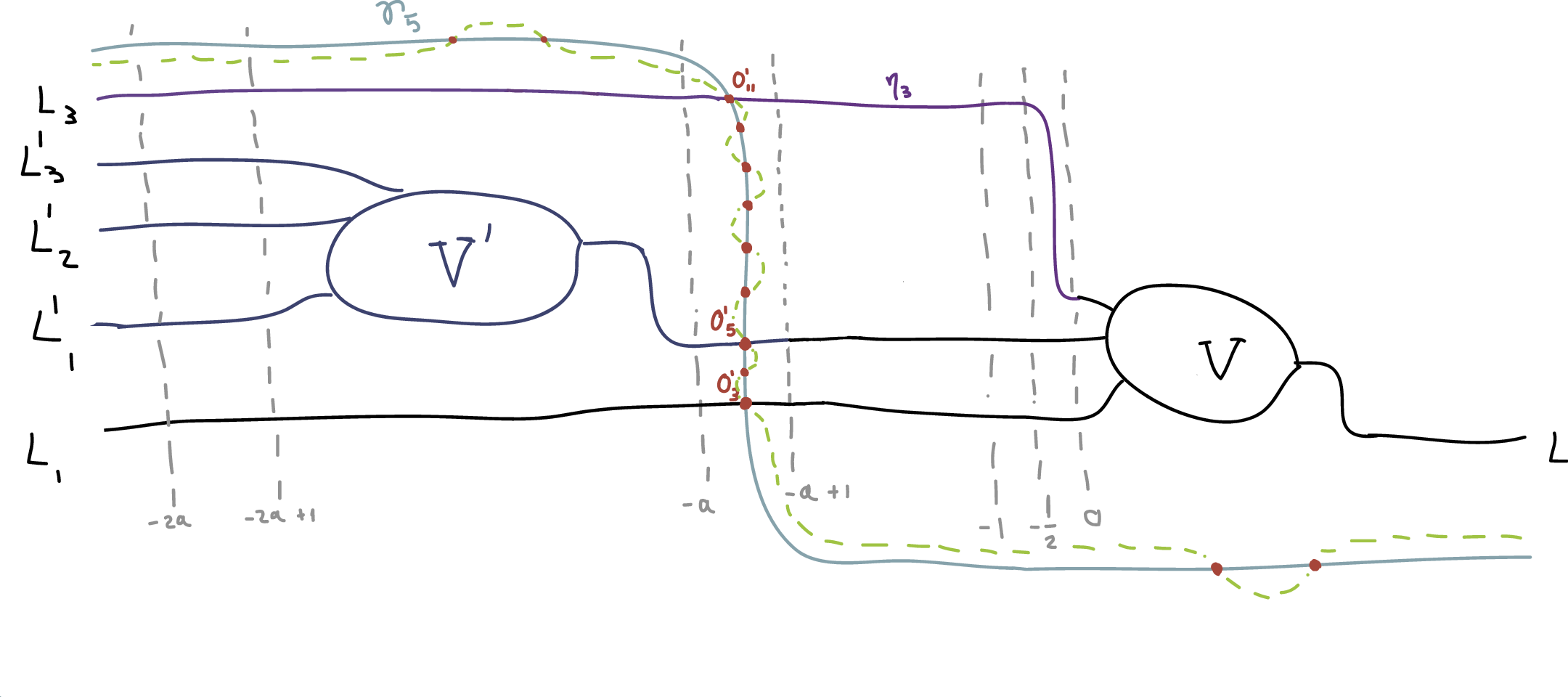, width=0.9\linewidth}
      \end{center}
      \vspace{-0.4in}\caption{\label{fig:comp-cob-far} The cobordism
        $V''_{0}$ obtained by translating $V'$ to the left by $a$. }
   \end{figure} 
   Let the $\R$-extensions of $V''_{0}$ and $V''$ be denoted by
   $\overline{V}''_{0}$ and $\overline{V}''$.

   It is not difficult to see that there is a horizontal Hamiltonian
   isotopy $\overline{\psi}_{t}:\C\times M\to \C\times M$, $t\in
   [0,1]$, of the form $\overline{\psi}_{t}(z,x)=(\psi_{t}(z),x)$ with
   $\psi_{t}:\C\to \C$ a Hamiltonian isotopy, and with the following
   properties:
   \begin{itemize}
     \item[i.]
      $\overline{\psi}_{1}(\overline{V}'')=\overline{V}''_{0}$.
     \item[ii.] $\psi_{t}$ is the identity on $[-\frac{1}{2},\infty)\times\R$.
     \item[iii.] $\psi_{t}$ is a translation in the negative, real
      direction for all points belonging to $\pi(V')\cup
      [-1,-\frac{1}{2}]\times \{i\}$.
     \item[iv.] $\psi_{t}((-\infty,\frac{1}{2}]\times \{j\})=
      (-\infty,\frac{1}{2}]\times \{j\}$ for all $j> i+q$ as well as
      all $j< i$. (Recall that $q$ is the number of negative ends of
      $V'$.)
   \end{itemize}

   For the next step, put $w=m+q-1$ and we consider a curve
   $\gamma_{w}$ together with a profile function $h_{w}$ as in
   \S\ref{subsubsec:triangles} so that the vertical part of
   $\gamma_{w}$ (i.e. the line that contains the points $o_{j}$ for
   $3\leq j\leq 2w+1$) is of real coordinate $-a+\frac{1}{2}$ (see
   Figures \ref{fig:comp-cob} and \ref{fig:comp-cob-far}). By Remark
   \ref{rem:invariance-triang}, this assumption is not restrictive.
   We also fix all the choices of auxiliary data required to define
   the category $\mathcal{B}_{\gamma_{w},h_{w}}$ as
   in~\S\ref{subsec:inclusion}. Recall that we also have the
   associated category $\mathcal{B}'_{w}$ whose objects are the same
   as those of $\mathcal{B}_{\gamma_{w},h_{w}}$ and whose morphism
   spaces are quotients of those of $\mathcal{B}_{\gamma_{w},h_{w}}$
   as described by~\eqref{eq:hom-B'} in~\S\ref{subsubsec:triangles}.

   By the constructions in \S\ref{subsubsec:modules}, in particular
   Proposition \ref{prop:str-mod}, we associate to $V''$ a sequence of
   $\mathcal{B}'_w$-modules, $\mathcal{M}^{*}_{j}$, $j$, $1\leq j\leq
   w$.  With the notation in Proposition~\ref{prop:str-mod}, these
   modules are defined by $\mathcal{M}^{*}_{w}=\mathcal{M}'_{w}$ and
   $$\mathcal{M}^{*}_{j}=(q_{j+1}\circ \ldots  \circ 
   q_{w})^{\ast}(\mathcal{M}'_{j}) \ \mathrm{for}\ 1\leq j < w~.~$$ In
   view of Lemma \ref{lem:technique2} these modules admit a direct
   description: \pbredb{$\mathcal{M}^{\ast}_{w}=\mathcal{M}'_{w}$ is
     the obvious $\mathcal{B}'_{w}$-module with}
   $\mathcal{M}^{\ast}_{w}(N)=CF^{cob}(\gamma_{w}\times N, V'')$ and,
   for $j<w$, $\mathcal{M}^{\ast}_{j}$ is the sub-module of
   $\mathcal{M}^{\ast}_{w}$ so that $\mathcal{M}^{\ast}_{j}(N)$ is
   generated by all the elements in $\mathcal{M}^{\ast}_{w}(N)$ lying
   over the points $o_{2s+1}$, $s\leq j$. In other words, the
   generators of $\mathcal{M}^{\ast}_{j}(N)$ correspond to the
   intersections of $\gamma_{m}$ with the first $j$ ends of
   $\pi(V'')$.  Thus we obtain inclusions:
   \begin{equation} \label{eq:exact-seq-ser} \mathcal{M}^{\ast}_{1}\to
      \ldots\to
      \mathcal{M}^{\ast}_{j-1}\stackrel{\kappa_{j}}{\longrightarrow}
      \mathcal{M}^{\ast}_{j}\stackrel{\kappa_{j+1}}{\longrightarrow}
      \mathcal{M}^{\ast}_{j+1}\to \ldots \to \mathcal{M}^{\ast}_{w}
   \end{equation}
   and each $\kappa_{j}$ fits into an exact triangle (that is
   determined by $\kappa_{j}$ up to isomorphism):
   \begin{equation}
      \label{eq:exact-triangle-star}
      Z_{j}\stackrel{\phi_{j}}{\longrightarrow}\mathcal{M}^{\ast}_{j-1}
      \stackrel{\kappa_{j}}{\longrightarrow}\mathcal{M}^{\ast}_{j}
   \end{equation}
   with:
   \begin{equation}
      Z_j = 
      \begin{cases}
         (c'_{\gamma_{w},h_{w},2j+1})^{\ast}(\mathcal{Y}(L_{j}))
         & \text{if $j<i$;} \\
         (c'_{\gamma_{w},h_{w},2j+1})^{\ast}(\mathcal{Y}(L'_{j-i+1}))
         & \text{if $i\leq j\leq i+q-1$;} \\
         (c'_{\gamma_{w},h_{w},2j+1})^{\ast}(\mathcal{Y}(L_{j-q+1})) &
         \text{if $i+q\leq j\leq w$.}
      \end{cases}
   \end{equation}

   See~\S\ref{prop:htpy-proj} for the definition of the projection
   $c'_{\gamma_{w},h_{w},j}$.

   Recall the map $e'_{w}:\fuk^{d}(M)\to \mathcal{B}'_{w}$ from
   equation (\ref{eq:e-commute}). Notice that, in view of Proposition
   \ref{prop:str-mod}, we have that the sequence of
   $\fuk^{d}(M)$-modules,
   $\mathcal{M}_{j}=(e'_{w})^{\ast}(\mathcal{M}^{\ast}_{j})$ is
   identified with the sequence $\mathcal{M}_{V'',\gamma_{j},h_{j}}$
   (where $\gamma_{j}, h_{j}$ are constructed, essentially by
   restricting appropriately $\gamma_{w}, h_{w}$). Moreover, the
   pull-back over $e'_{w}$ of the exact triangles
   (\ref{eq:exact-triangle-star}) are precisely the exact triangles
   (\ref{eq:exact-tr-mod}). In short, both the module
   $[\mathcal{M}_{V'',\gamma_{w},h_{w}}]$ \pbredb{as well as the
     iterated cone decomposition $[\eta_{V'',\gamma_{w},h_{w}}]$ are}
   determined by the sequence\eqref{eq:exact-seq-ser}. \pbredb{Next,
     we define $[\phi_{V'',\gamma_{w},h_{w}}]$ as follows. Consider
     the cobordism $V''_1$ obtained from $V''$ by translating it to
     the left by $2a$.  The map $\phi_{V'',\gamma_{w},h_{w}}$ is then
     defined to be the pullback over $e'_{w}$ of the morphism defined
     via the ``moving boundary'' map
   $$\Phi^{V''_{1}}_{V''}:c_{\gamma_{m},h_{m},1}^{\ast}
   (\mathcal{Y}(L))\to \mathcal{M}^{\ast}_{w}$$ induced by the
   horizontal isotopy moving $V''$ to $V''_1$.}

   The sequence of modules (\ref{eq:exact-seq-ser}) has an additional
   important property.  For $1\leq s\leq q$ let $K_{s}$ be the
   $\mathcal{B}'_{w}$ quotient module
   $K_{s}=\mathcal{M}^{\ast}_{i+s-1}/\mathcal{M}^{\ast}_{i-1}$. They
   also fit in a sequence of increasing submodules:
   \begin{equation}\label{eq:exactr-tr-K}
      K_{1}\to  \ldots K_{s} 
      \stackrel{\l_{s}}{\longrightarrow} K_{s+1}\to  \ldots \to K_{q}
   \end{equation}
   so that there is a cone decomposition $\eta_{V'}$ with exact
   triangles:
   $$Z_{s+i} \to K_{s}\to K_{s+1}~.~$$
   The key property of this cone decomposition is that
   $\widetilde{\mathcal{F}}(V')$ is the equivalence class of the
   triple $((e'_{w})^{\ast}(\phi_{V',\gamma_{w}, h_{w}}),
   (e'_{w})(K_{q}), (e'_{w})(\eta_{V'}))$ where
   $\phi_{V,'\gamma_{w},h_{w}}$ is obtained by a moving boundary
   morphism induced by translating $V'$ to the right by $a$ in
   Figure~\ref{fig:comp-cob-far}, for instance.

   We now intend to show that the cone decomposition
   $\eta_{V'',\gamma_{w},h_{w}}$ matches with the cone decomposition
   associated to $\widetilde{\mathcal{F}}(\mathcal{V}')\circ
   \widetilde{\mathcal{F}}(V)$. To this end, we consider the sequence
   of $\mathcal{B}'_{w}$-modules
   \begin{equation} \label{eq:exact-triangle-starN}
      \mathcal{N}^{\ast}_{1}\to \ldots
      \mathcal{N}^{\ast}_{j}\stackrel{h_{j}}{\longrightarrow}
      \mathcal{N}^{\ast}_{j+1}\to\ldots \to \mathcal{N}_{w}^{\ast}
   \end{equation}
   that are constructed just as the sequence
   in~\eqref{eq:exact-triangle-star} but by using $V''_{0}$ instead of
   $V''$. It is easy to see that all the compactness arguments
   necessary for the construction of the $\mathcal{M}^{\ast}$'s still
   apply in this case and that, by an appropriate choice of all the
   Floer and perturbation data involved we can insure that the modules
   in (\ref{eq:exact-triangle-starN}) satisfy the following
   properties:
   \begin{itemize}
     \item[-] For $2\leq j\leq w$ there are exact triangles
      \begin{equation*} \label{eq:exact-triangN}
         Z'_{j}\stackrel{\alpha'_{j}}{\longrightarrow}
         \mathcal{N}^{\ast}_{j-1}\stackrel{\tau_{j}}{\longrightarrow}
         \mathcal{N}^{\ast}_{j}
      \end{equation*}
      so that for $2\leq j< i$ or $i+q\leq j\leq w$ we have
      $Z'_{j}=Z_{j}$.
     \item[-]
      $Z'_{i}=(c'_{\gamma_{w},h_{w},2i+1})^*(\mathcal{Y}(L_{i}))$ and
      for $i+1\leq j < i+q$ we have $Z'_{j}=0$ and $\tau_{j}=\id$.
     \item[-] The sequence 
      $$\mathcal{N}^{\ast}_{1}\to \ldots \to \mathcal{N}^{\ast}_{i}\to 
      \mathcal{N}^{\ast}_{i+q}\to \ldots \to \mathcal{N}^{\ast}_{w}$$
      pulls back over $e'_{w}$ to the respective sequence
      corresponding to $\widetilde{\mathcal{F}}(V)$.
   \end{itemize}
   In view of the composition rule described in \S\ref{sub:cob-def},
   to show that
   $\widetilde{\mathcal{F}}(V'')=\widetilde{\mathcal{F}}(\mathcal{V}')\circ
   \widetilde{\mathcal{F}}(V)$ it is enough to prove that the
   sequences (\ref{eq:exact-seq-ser})
   and~\eqref{eq:exact-triangle-starN} are related by morphisms:
   $$\xi_{j}:\mathcal{M}^{\ast}_{j}\to \mathcal{N}^{\ast}_{j}$$ 
   \pbredb{so that $\xi_j \circ \kappa_{j} = \tau_{j}\circ \xi_{j-1}$}
     and additionally:
   \begin{itemize}
     \item[i.'] For $1\leq j < i$, $\xi_{j}=\id$.
     \item[ii.''] For $j > i+q$, the quotient morphism
      $\hat{\xi_{j}}:\mathcal{M}^{\ast}_{j}/\mathcal{M}^{\ast}_{j-1}\to
      \mathcal{N}^{\ast}_{j}/ \mathcal{N}^{\ast}_{j-1}$ is the
      identity.
     \item[iii.'] For $j=i+q-1$, the quotient morphism
      $$\hat{\xi}_{i+q-1}: K_{q}=
      \mathcal{M}^{\ast}_{i+q-1}/\mathcal{M}^{\ast}_{i-1}\to
      \mathcal{N}^{\ast}_{i+q-1}/\mathcal{N}^{\ast}_{i-1}=
      (c'_{\gamma_{w},h_{w},2i+1})^*(\mathcal{Y}(L_{i}))$$ is, in
      homology, the inverse of the morphism
      $\phi_{V',\gamma_{w},h_{w}}$.
   \end{itemize}
   We now want to remark that the module morphisms $\xi_{j}$ with the
   desired properties are induced by the Hamiltonian isotopy
   $\overline{\psi}_{t}$. More precisely, consider the morphism:
   $$\Psi^{V''}_{V''_{0}}:\mathcal{M}^{\ast}_{w}\to \mathcal{N}^{\ast}_{w}$$
   given by counting $J$-holomorphic polygons with the last edge
   satisfying a moving boundary condition along the isotopy
   $\overline{\psi}_{t}$.  The principal geometric input at this stage
   of the proof is that this morphism respects the filtration
   (\ref{eq:exact-seq-ser}).  This follows in view of the properties
   ii, iii, iv of $\psi_{t}$ combined with arguments similar to those
   in the proof of Lemma \ref{lem:technique2}. In essence, we need to
   show that a curve $u$ belonging to $0$-dimensional moduli space
   contributing to $\Psi^{V''}_{V''_{0}}$ that has its last input (the
   one belonging to $CF^{cob}(\gamma_{m}\times N, V'')$) of type
   $o_{k}$ can not have the exit of type $o_{k'}$ with $k'>k$. Indeed,
   if such would be the case we see, e.g. by looking at
   Figures~\ref{fig:comp-cob},~\ref{fig:comp-cob-far}, that either
   $o_{k}$ or $o_{k'}$ is left fixed by $\psi_{t}$. Therefore the
   arguments from the proof of Lemma~\ref{lem:technique2} apply here
   and prove analogues of~\eqref{eq:first-end}
   and~\eqref{eq:second-end}.  The properties i',ii' follow
   immediately.  Finally, from the definition of $\overline{\psi}_{t}$
   it follows that this Hamiltonian isotopy restricted to $V'$ is an
   inverse of the translation inducing $\phi_{V',\gamma_{m},h_{m}}$.
   This implies iii' and concludes the proof.
\end{proof}


\subsection{The diagram~\eqref{eq:commut-diag1} and the other
  Corollaries in~\S\ref{subsec:cor}} \label{sb:diag-cor}

For the convenience of the reader we repeat here
Diagram~\ref{eq:commut-diag1} with a slight addition at its bottom
that will be explained shortly:
\begin{eqnarray} \label{eq:commut-diag2} 
   \begin{aligned}
      \xymatrix{ \cob^{d}_{0}(M)\ar[r]^{\widetilde{\mathcal{F}}}
        \ar[d]_{\mathcal{P}}
        & T^{S}D\fuk^{d}(M)\ar[d]^{\mathcal{P}}  \\
        S\cob^{d}_{0}(M) \ar[r]^{\mathcal{F}} \ar[d] &
        D\fuk^{d}(M)\ar[d]^{\hom(N,-)}\\
        M\cob^{d}_{0}M\ar[r]^{\mathcal{H}F_{N}}& (\mathcal{V},\times)}
   \end{aligned}
\end{eqnarray}
The top line is the functor $\widetilde{\mathcal{F}}$ that was
constructed in \S\ref{subsec:def-functor-tilde} and
\S\ref{subsec:comp-comp-functor}.

We first provide a detailed description of the categories and functors
in diagram~\eqref{eq:commut-diag2}. We end the section - and the main
body of the paper - with the proof of Corollary
\ref{cor:special-short} which is the only statement from the
introduction that does not directly follow from the statement of
Theorem~\ref{thm:main}.

\subsubsection{The top square in ~\eqref{eq:commut-diag2}}
\label{subsubsec:top-square} The second row in the diagram is obtained
from the first one as follows.  The category
$S\mathcal{C}ob^{d}_{0}(M)$ has as objects Lagrangians $L\in
\mathcal{L}^{\ast}_{d}(M)$. The morphisms from $L$ to $L'$ are
connected cobordisms $V:L\cobto (L_{1},\ldots, L_{m-1}, L')$ (with $m
\geq 1$ arbitrary) modulo horizontal isotopy. The composition is
induced in the obvious way from that in $\mathcal{C}ob^{d}_{0}(M)$.
The projection $\mathcal{P}:\mathcal{C}ob^{d}_{0}(M)\to
S\mathcal{C}ob^{d}_{0}(M)$ is the projection on the last component,
i.e.  $\mathcal{P}(K_{1}, \ldots, K_{m}) = K_{m}$ and similarly for
morphisms. (Recall that in $\cob^d_0(M)$ some morphisms are
represented by disjoint unions of cobordisms.)

On the algebraic side the projection $\mathcal{P}$ is defined
in~\eqref{eq:proj}. On objects it is again the projection on the last
component and, in the language of Corollary \ref{cor:exact-triangles},
on morphisms it associates to a triple $(\phi_{V,\gamma_{m},h_{m}},
\mathcal{M}_{V,\gamma_{m},h_{m}},\eta_{V,\gamma_{m},h_{m}})$ the
morphism
$$w_{m}\circ\phi_{V, \gamma_{m},h_{m}}:\mathcal{Y}(L) 
\longrightarrow \mathcal{Y}(L_{m})$$ where
$w_{m}:\mathcal{M}_{V,\gamma_{m},h_{m}} \longrightarrow
\mathcal{Y}(L_{m})$ is the morphism fitting in the last exact triangle
from ~\eqref{eq:exact-tr-mod}:
$$\mathcal{M}_{V,\gamma_{m-1},h_{m-1}} \longrightarrow
\mathcal{M}_{V,\gamma_{m},h_{m}} \xrightarrow{w_{m}}
\mathcal{Y}(L_{m})~.~$$ The functor $\mathcal{F}$ is the identity on
objects. To define it on morphisms we notice that each morphism
$V:L\to L'$ in $S\mathcal{C}ob^{d}_{0}(M)$ provides also a morphism
$V:L\to (L_{1}, \ldots, L_{m-1}, L')$ in $\mathcal{C}ob^{d}_{0}(M)$ and
the functor $\mathcal{F}$ can be defined as
$\mathcal{F}([V])=\mathcal{P}\circ \widetilde{\mathcal{F}}([V])$.
More explicitly, the morphism $\mathcal{F}([V]):\mathcal{Y}(L)\to
\mathcal{Y}(L')$ can be also described in one of the following three
equivalent ways:
\begin{itemize}
  \item[a.] It is given by the composition
   $$\mathcal{Y}(L) \xrightarrow{\phi_{V,\gamma_{m},h_{m}}}
   \mathcal{M}_{V,\gamma_{m}, h_{m}} \longrightarrow
   \mathcal{M}_{V,\gamma_{m}, h_{m}}/\mathcal{M}_{V,\gamma_{m-1},
     h_{m-1}}$$ where the second map is the projection (see
   Corollary~\ref{cor:exact-triangles} and
   Proposition~\ref{prop:comp-morphisms}).
  \item[b.] As we are considering here a morphism
   $\mathcal{Y}(L)\to\mathcal{Y}(L')$, it follows from the general
   properties of the derived Fukaya category that this morphism is
   determined by a Floer homology class $\alpha_{V}\in HF(L,L')$. In
   view of the moving boundary description of the morphism
   $\phi_{V,\gamma_{m},h_{m}}$, this class can be defined as the image
   of the unit in $HF(L,L)$ by the moving boundary morphism
   $$\phi_{V}:CF(L,L)=CF(\gamma_{m}\times L, V')\to CF(\gamma_{m} \times L , V) 
   \xrightarrow{\textnormal{proj}} CF (L, L')$$ where the isotopy
   $V'\to V$ is a translation like in Figure~\ref{fig:translated-cob}
   ($L'= L_{4}$ in that picture) and $proj$ is the projection on the
   last term in the vector space decomposition
   $$CF(\gamma_{m}\times L,V)=\oplus_{i} CF(L,L_{i})$$
   (by Lemma~\ref{lem:technique2} $proj$ is a chain map). Thus,
   $\mathcal{F}(V)=\phi_{V}([L])\in HF(L,L')$.

  \item[c.] In view of Remark~\ref{rem:comparison-mor-discussion} and of the point a.
   the  morphism $\phi_{V}$ can also be described in terms of counting the
   $J$-holomorphic Floer strips with boundary conditions along $\gamma'\times L$ and $V$
   and having as input an intersection point
   that projects to the point $P$ in Figure~\ref{fig:super-bending} 
   and having as exit a point that projects to $Q$ in the same picture
   (these strips cover the region filled in color there; the same
   morphism is discussed in Figure~\ref{fig:MorphCob2}). \end{itemize} 

\subsubsection{The bottom square in diagram~\eqref{eq:commut-diag2}}
The category $M\mathcal{C}ob^{d}_{0}(M)$ is the monoidal completion of
$S\mathcal{C}ob^{d}_{0}(M)$, its objects are families $(L_{1}, \ldots,
L_{k})$ and the morphisms are similar families of morphisms from
$S\mathcal{C}ob^{d}_{0}(M)$. We remark that while the objects of
$\mathcal{C}ob^{d}_{0}(M)$ and $M\mathcal{C}ob^{d}_{0}(M)$ are the
same {\em the morphisms are different}.

Given $N\in \mathcal{L}^{\ast}_{d}(M)$, the functor $\mathcal{H}F_N$
associates to $(L_{1},\ldots, L_{k})$ the product of Floer homologies
\pbred{$HF(N, L_{1})\times \ldots \times HF(N, L_{k})$.} Further, for
each cobordism $V:L\to (L_{1},\ldots, L_{i}, L')$,
$\mathcal{H}F_{N}(V)$ is the morphism
\pbred{$$\mathcal{H}F_{N}(V)=\phi^{N}_{V} : HF(N,L) = 
HF(\gamma_{m}\times N, V') \longrightarrow HF(\gamma_{m} \times N , V)
\xrightarrow{\textnormal{proj}} HF(N, L')$$} defined by a moving
boundary condition as at the point b above. In particular,
$\phi_{V}=\mathcal{H}F_{L}(V)$. Alternatively, we can again define
this morphism by counting strips as indicated at the point c. above.
The commutativity of the bottom square is immediate in this setting.


\subsubsection{Proof of Corollary \ref{cor:special-short}}
\label{subsubsec:proof-cor}
We will make use of the equivalent descriptions of the functor
$\mathcal{F}$ as given at the points~a,~b,~c
in~\S\ref{subsubsec:top-square}. The key point is that for a cobordism
$V:L\to (L_{1},\ldots, L')$ the morphism $\mathcal{F}(V)\in
HF(L,L')=\mor_{D\fuk^{d}(M)}(L,L')$ has two equivalent descriptions.
One - described at the point~a - is a projection composed with the
moving boundary comparison morphism associated to a translation of $V$
as in Proposition \ref{prop:comp-morphisms} and Figure
\ref{fig:translated-cob}. The other - described at the point~c - is
expressed in terms of counting Floer strips with boundary conditions
along $\gamma'\times L$ and $V$ where $\gamma'$ is a ``hat''-like
curve as in Figure \ref{fig:super-bending} (see also Remark
\ref{rem:comparison-mor-discussion}).

We now specialize to the setting of the Corollary.  We use the
notation in the statement, in particular, $V,V', V''$ are as in Figure
\ref{fig:hairy-cob} (for the convenience of the reader, we repeat that
Figure below).
\begin{figure}[htbp]
      \begin{center} \vspace{-0.4in} \epsfig{file=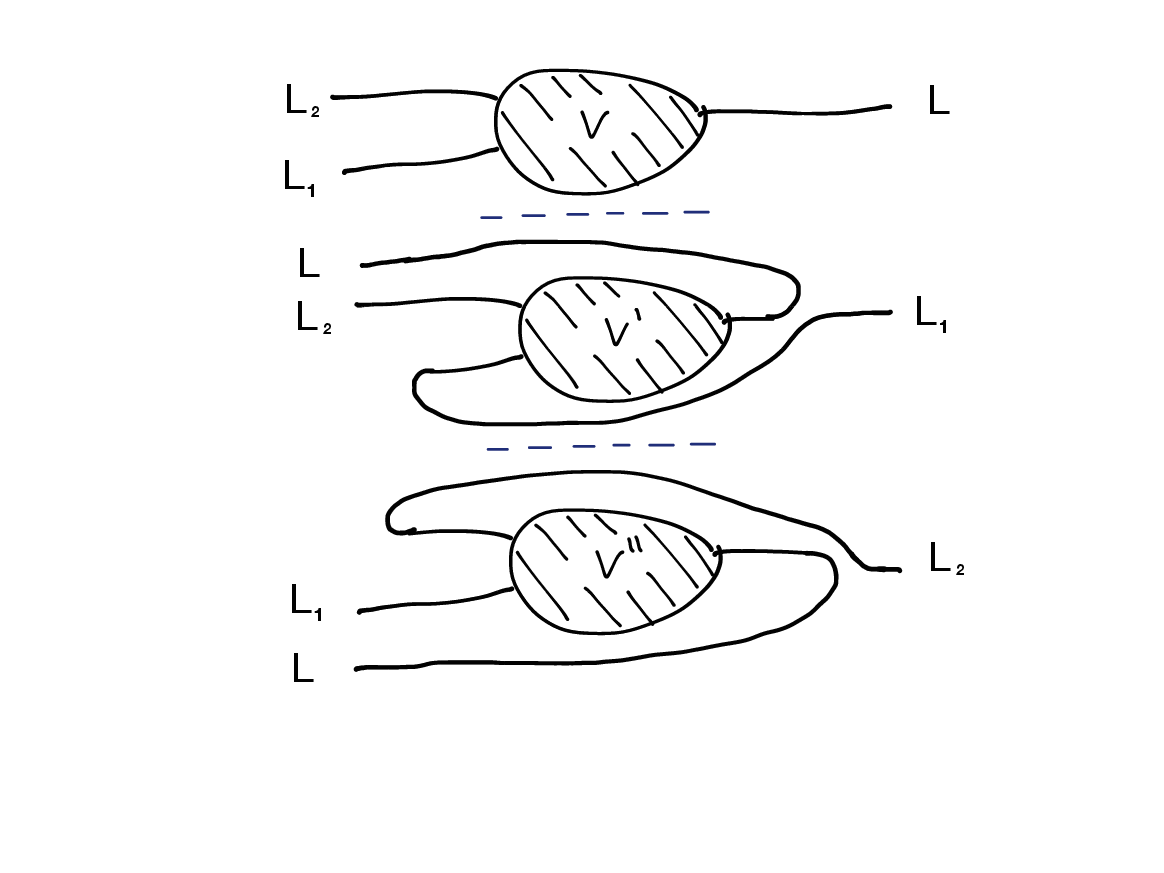,
           width=0.8\linewidth, height=0.6\linewidth }
      \end{center}
      \vspace{-0.8in}
   \end{figure}
\

We consider the cobordism $V$ and a curve $\gamma_{2}$
as  in Figure \ref{fig:hairy1}.
\begin{figure}[htbp]
      \begin{center} \vspace{-0in} \epsfig{file=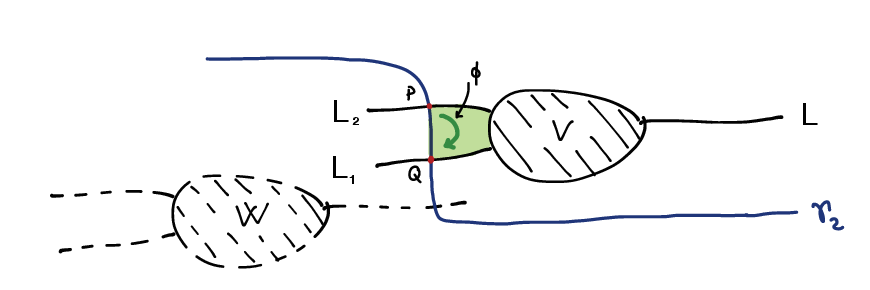,
           width=0.78\linewidth, height=0.33\linewidth }
      \end{center}
      \vspace{-0in}
      \caption{The cobordism $V$, \pbredb{its translated version $W$},
        the curve $\gamma_{2}$ and the strips giving
        $\phi$.\label{fig:hairy1}}
   \end{figure}
   As a consequence of Theorem~\ref{cor:multiple-triang} we deduce the
   existence of an exact triangle in $D\fuk^{d}(M)$, given as the
   horizontal line in the next diagram
   \begin{eqnarray} \label{eq:commut-hairy}
      \begin{aligned}
         \xymatrix{ L_{2}\ar[r]^{\phi} & L_{1} \ar[r]^{i \ \ } &
           \mathrm{Cone}(\phi) \ar[r]^{\ \ p}  & L_{2}\\
           & & L \ar[u]^{\sigma} & }
      \end{aligned}
   \end{eqnarray}
   together with an isomorphism $\sigma: L\to \mathrm{Cone}(\phi)$
   \pbredb{induced by isotoping $V$ to a another cobordism $W$,} from
   one side to the other of the curve $\gamma_{2}$, as in
   Figure~\ref{fig:hairy1} (see also
   Proposition~\ref{prop:comp-morphisms}).  Further, the morphism
   $\phi$ is induced by counting strips with boundary conditions along
   $\gamma_{2}\times L_{2}$ and $V$, starting at $P$ and ending at
   $Q$.  Moreover, $i$ is the inclusion of $L_{1}$ in $\mathrm{Cone}
   (\phi)=\mathcal{M}_{V,\gamma_{2}}$ and $p$ is the projection of the
   cone onto $\mathcal{Y}(L_{2})$ (recall that we neglect grading).

   By the description of $\mathcal{F}(-)$ at the point~a
   in~\S\ref{subsubsec:top-square}, $p\circ \sigma=\mathcal{F}(V)$.
   To finish the proof of the corollary it suffices to show that in
   $D\fuk^{d}(M)$ we have $\phi= \mathcal{F}(V'')$ and
   $\sigma^{-1}\circ i= \mathcal{F}(V')$.

   To show the identity $\sigma^{-1}\circ i=\mathcal{F}(V')$ we first
   notice that the composition $\psi=\sigma^{-1}\circ i$ is induced by
   \pbredb{translating $V$ to $W$} (the inverse of the translation
   used to define $\sigma$). Further, the same methods as those used
   in Remark~\ref{rem:comparison-mor-discussion} imply that $\psi$ can
   also be described by counting strips that start at $R$ and end at
   $K$ and with boundary conditions along $V$ and $\lambda \times
   L_{1}$ as in Figure~\ref{fig:inverted-hat}.
   \begin{figure}[htbp]
      \vspace{-0.2in}
      \begin{center} \vspace{-0in} \epsfig{file=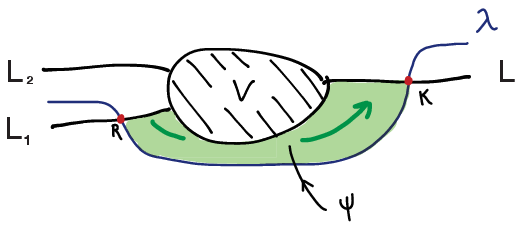,
           width=0.5\linewidth, height=0.25\linewidth }
      \end{center}
      \vspace{-0.2in}
      \caption{ The curve $\lambda$ and the morphism
        $\psi$.\label{fig:inverted-hat}}
   \end{figure}
   By the ``counting strips'' description of $\mathcal{F}(-)$ as at
   the point~c in~\S\ref{subsubsec:top-square} it is easy to see,
   using an appropriate transformation of $\lambda$ as in
   Figure~\ref{fig:twist-again},
\begin{figure}[htbp]
   \begin{center} \vspace{-0in} \epsfig{file=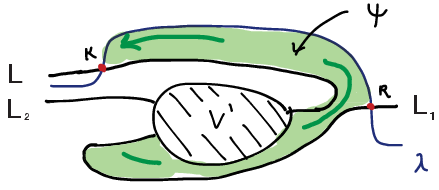,
        width=0.45\linewidth, height=0.23\linewidth }
   \end{center}
   \vspace{-0.1in}
   \caption{The curve $\lambda$ and the morphism $\psi$ after an 
     appropriate transformation.
     \label{fig:twist-again}}
\end{figure}
that in homology $[\phi_{V'}]=[\psi]$.
 
The identity $\phi= \mathcal{F}(V'')$ follows by the description of
$\mathcal{F}(-)$ in terms of ``strip counting'', as at the point~c
in~\S\ref{subsubsec:top-square}, and an appropriate transformation of
$\gamma_{2}$.

\appendix
\section{Generalities on $A_{\infty}$-categories }
\label{subsec:more-alg} 
We review here basic $A_{\infty}$ notions.  All of these are explained in detail
in~\cite{Se:book-fukaya-categ}.  We recall them for completeness, and
also because we use in the paper slightly different conventions (mainly homology
rather than cohomology).

In what follows all vector spaces are assumed to be over the field
$\mathbb{Z}_2$ and are generally ungraded. The algebra below has a
graded version (which also works over arbitrary base rings). This
requires inserting signs in most of the formulae below
(see~\cite{Se:book-fukaya-categ}).

\subsection{Extended multilinear maps} \label{sbsb:ext-mult} Let
$I$, $J$ be two sets of indices. Let $\mathbf{C} = \{C_{i,j}\}_{i,j
  \in I}$, $\mathbf{D} = \{D_{s,r}\}_{s,r \in J}$ be two collections
of vector spaces indexed by pairs of indices in $I$ and $J$. An {\em
  extended multilinear map} $F:\mathbf{C} \to \mathbf{D}$ consists of
the following. First we have a map $F: I \to J$ which we denote by
abuse of notation by $F$. We call it the index map.  Next we have a
collection of multilinear maps $$F^{i_0, \ldots, i_k} : C_{i_0, i_1}
\otimes \cdots \otimes C_{i_{k-1}, i_k} \longrightarrow D_{F(i_0),
  F(i_k)}$$ defined for every $k \geq 1$ and every $i_0, \ldots, i_k
\in I$. When it is obvious, we do not specify the index map, e.g. when
$I=J$, $\mathbf{C}=\mathbf{D}$ and the index map is the identity (we
call this the trivial index map). We often omit the superscripts from
$F^{i_0, \ldots, i_k}$ to simplify the notation. We sometimes denote
the latter by $F_k$ to indicate that it gets $k$ inputs, calling it
the degree $k$ part of $F$.

There are two ways to compose extended multilinear maps. Let $F:
\mathbf{C} \to \mathbf{D}$ and $G: \mathbf{D} \to \mathbf{E}$ be
extended multilinear maps. Define $G \circ F : \mathbf{C} \to
\mathbf{E}$ by composing the index maps in an obvious way and
$$(G \circ F) (a_1, \dots, a_k) = 
\sum G(F(a_{1},\ldots, a_{i_{1}-1}), F(a_{i_{1}}, \ldots,
a_{i_{2}-1}),\ldots, F(a_{i_{l}},\ldots, a_{k}))~.~$$ The sum is taken
over all the possibilities to write terms as on the right-hand side.

For the second type of composition we need to assume that $F:
\mathbf{D} \to \mathbf{D}$ (i.e. $\mathbf{C} = \mathbf{D}$) and that
the index map of $F$ is the identity. Let $G: \mathbf{D} \to
\mathbf{E}$. We define $G \star F : \mathbf{D} \to \mathbf{E}$ by
$$(G \star F) (a_1, \dots, a_k) = \sum G(a_{1},\ldots, a_{i-1}, F(a_i,
\ldots, a_{l}), a_{l+1}, \ldots, a_{k}))~,~$$ where the sum is over
all the possibilities to write terms as on the right side above: we
apply first $F$ to an ordered subset of the $a_{j}$'s (formed by
successive elements) and we use the output of $F$ as one of the inputs
for $G$.

We will also need the following generalization. Let $\mathbf{C} = \{
C_{i,j} \}_{i,j \in I}$ be a collection of vector spaces as before and
$\mathbf{M} = \{ M_i\}_{i \in I}$, $\mathbf{N} = \{ N_s\}_{s \in J}$
be two collections of vector spaces indexed by $I$ and $J$. By a {\em
  mixed extended multilinear} map $Q : \mathbf{C} \otimes \mathbf{M}
\to \mathbf{N}$ we mean an index map $Q: I \to J$ and a collection of
multilinear maps:
$$Q: C_{i_0, i_1} \otimes \cdots \otimes C_{i_{k-2}, i_{k-1}} 
\otimes M_{i_{k-1}} \longrightarrow N_{Q(i_0)},$$ defined for every $k
\geq 1$ and indices $i_0, \ldots, i_{k-1} \in I$. (Note that for $k=1$
we just have a map $Q: M_{i_0} \to M_{i_0}$.)

Let $Q: \mathbf{C} \otimes \mathbf{M} \to \mathbf{N}$ be such a map
with $\mathbf{N}$ being indexed by the same indices as $\mathbf{M}$
and with identity index map. Let $P: \mathbf{C} \otimes \mathbf{N} \to
\mathbf{R}$ be another mixed extended map. We can compose them to get
a new mixed extended multilinear map $P \dashv Q: \mathbf{C} \otimes
\mathbf{M} \to \mathbf{R}$ as follows:
$$(P \dashv Q) (a_1, \ldots, a_{k-1}, b) = \sum P(a_1, \ldots, a_{i-1}, 
Q(a_i, \ldots, a_{k-1},b)).$$

If $F : \mathbf{C} \to \mathbf{C}$ is an extended multilinear map with
identity index map and $Q: \mathbf{C} \otimes \mathbf{M} \to
\mathbf{N}$ is a mixed extended multilinear map, they can be composed
to give a new mixed map $Q \star F : \mathbf{C} \otimes \mathbf{M} \to
\mathbf{N}$ as follows:
$$(Q \star F)(a_1, \dots, a_{k-1},b) = \sum Q(a_{1},\ldots, a_{i-1}, F(a_i,
\ldots, a_{l}), a_{l+1}, \ldots, a_{k-1},b).$$

\subsection{$A_{\infty}$-categories} 
\subsubsection{Definition}\label{sbsb:A-infty-categ} An
$A_{\infty}$-category $\mathcal{A}$ consists of the following data. A
collection of objects $\textnormal{Ob}(\mathcal{A})$, a collection of
vector spaces $\mathbf{C}_{\mathcal{A}} = \{ C(L', L'') \}_{L', L''
\in \textnormal{Ob}(\mathcal{A})}$ indexed by pairs of objects $L',
L'' \in \textnormal{Ob}(\mathcal{A})$, and an extended multilinear map
$\mu: \mathbf{C}_{\mathcal{A}} \to \mathbf{C}_{\mathcal{A}}$ (called
composition map) with identity index map and which satisfies the
identity:
$$\mu \star \mu = 0.$$

The space $C(L',L'')$ is called the space of morphisms between $L'$
and $L''$ and is sometimes denoted by $\hom(L',L'')$. When we want to
emphasize the relation to the category $\mathcal{A}$ we write
$C_{\mathcal{A}}(L', L'')$ or $\hom_{\mathcal{A}}(L', L'')$.

Note that $\mu_1 : C(L', L'') \to C(L', L'')$ is a differential, i.e.
$\mu_1 \circ \mu_1 = 0$, hence $C(L', L'')$ becomes a chain complex
(ungraded in our case). We denote by $H(L', L'')$ the homology of
$(C(L', L''), \mu_1)$. The degree $2$ component $\mu_2$ of $\mu$
descends to homology and induces an associative product $$H(L_0, L_1)
\otimes H(L_1, L_2) \longrightarrow H(L_0,L_2), \quad a_1 \otimes a_2
\mapsto a_1 * a_2.$$ This means that by passing to homology we obtain
a (non-unital) category $H(\mathcal{A})$, called the homology category
associated to $\mathcal{A}$. Its objects are
$\textnormal{Ob}(\mathcal{A})$, and its morphisms are
$\hom_{H(\mathcal{A})}(L,K) = H(L,K)$. The composition of $\alpha \in
H(L_0, L_1)$, $\beta \in H(L_1, L_2)$ is defined to be $\beta \circ
\alpha := \alpha * \beta \in H(L_0, L_2)$.

\begin{rem}\label{rem:notation} Notice that  our notation is slightly different from that in
 \cite{Se:book-fukaya-categ}. The multiplication $\mu_{k}$ in an $A_{\infty}$ algebra
 is defined for us as $$\mu_{k}:C(X_{1},X_{2})\otimes C(X_{2},X_{3})\ldots \otimes C(X_{k},X_{k+1})\to C(X_{1},X_{k+1})$$
 and in  \cite{Se:book-fukaya-categ} it is defined as a map:
 $$C(X_{k},X_{k+1})\otimes C(X_{k-1},X_{k})\ldots \otimes C(X_{1},X_{2})\to C(X_{1},X_{k+1})~.~$$
As we do not care 
 about signs here the two conventions give the same notion and it is easy to move from one to the other.
 \end{rem}

\subsubsection{Units.}
An $A_{\infty}$-category $\mathcal{A}$ is called {\em strictly unital}
if for every object $L \in \textnormal{Ob}(\mathcal{A})$ there exists
a distinguished element $e_L \in C(L,L)$ with $\mu_1(e_L)=0$,
$\mu_2(e_L, \cdot) = \mu_2(\cdot, e_L) = \id$ and $\mu_k(a_1, \ldots,
a_{i-1}, e_L, a_{i+1}, \ldots, a_k) = 0$ for every $k \geq 3$ and $1
\leq i \leq k$.  $\mathcal{A}$ is called {\em homologically unital} if
for every object $L \in \textnormal{Ob}(\mathcal{A})$ there exists a
distinguished element $1_L \in H(L,L)$ which behaves like an identity
with respect to composition. In most cases considered in the paper
the $A_{\infty}$ categories will be homologically unital. Below we
will use the following abbreviations: ``hu'' stands for homologically
unital, ``su'' for strictly unital and ``nu'' for non-unital (or more
precisely, not necessarily unital).

\subsubsection{Functors}
Let $\mathcal{A}$, $\mathcal{B}$ be two $A_{\infty}$-categories. A nu
$A_{\infty}$-functor $\mathcal{F} : \mathcal{A} \to \mathcal{B}$
consists of the following pieces of data. A correspondence which
associates for every object $L \in \textnormal{Ob}(\mathcal{A})$ an
object $\mathcal{F}(L) \in \textnormal{Ob}(\mathcal{B})$. The second
piece is an extended multilinear map $\mathcal{F}:
\mathbf{C}_{\mathcal{A}} \to \mathbf{C}_{\mathcal{B}}$ with index map
$L \mapsto \mathcal{F}(L)$ and which satisfies the following identity
$$\mathcal{F} \star \mu^{\mathcal{A}} = \mu^{\mathcal{B}} \circ
\mathcal{F}.$$ Here $\mu_{\mathcal{A}}$, $\mu_{\mathcal{B}}$ are the
composition maps of $\mathcal{A}$, $\mathcal{B}$. Of course, nu
$A_{\infty}$-functors $\mathcal{F}$ descend to usual nu functors
$H(\mathcal{F}): H(\mathcal{A}) \to H(\mathcal{B})$ on the associated
homology categories.

Nu $A_{\infty}$-functors $\mathcal{F} : \mathcal{A} \to \mathcal{B}$,
$\mathcal{G}: \mathcal{B} \to \mathcal{C}$ can be composed to give a
new nu functor $\mathcal{G} \circ \mathcal{F}: \mathcal{A} \to
\mathcal{C}$. On the level of objects the composition is obvious and
the corresponding extended multilinear maps are composed by using the
$\circ$ composition, i.e. $\mathcal{G} \circ \mathcal{F}:
\mathbf{C}_{\mathcal{A}} \to \mathbf{C}_{\mathcal{C}}$.

We now describe natural transformations between nu
$A_{\infty}$-functors $\mathcal{F}, \mathcal{G} : \mathcal{A} \to
\mathcal{B}$.  First, define a collection of vector spaces
$\mathbf{C}_{\mathcal{B}}^{\mathcal{F}, \mathcal{G}} = \{
C_{\mathcal{A}}(\mathcal{F}(L'), \mathcal{G}(L''))\}_{L', L'' \in
  \textnormal{Ob}(\mathcal{A})}$ indexed by pairs of objects of
$\mathcal{A}$. A {\em pre-natural} transformation $T:\mathcal{F} \to
\mathcal{G}$ consists of a pair $T = (T^0, T')$, where $T'$ is an
extended multilinear map $T':\mathbf{C}_{\mathcal{A}} \to
\mathbf{C}_{\mathcal{B}}^{\mathcal{F}, \mathcal{G}}$, with identity
index map, and $T^0$ is a collection of morphisms indexed by
$\textnormal{Ob}(\mathcal{A})$, i.e. $T^0_L \in
C_{\mathcal{B}}(\mathcal{F}(L), \mathcal{G}(L)), \, \forall L \in
\textnormal{Ob}(\mathcal{A})$.

In order to describe which pre-natural transformations are actually
natural we need the following notation.  Define $(\mu^{\mathcal{B}}
\circ \star \circ (\mathcal{F}, T, \mathcal{G})) :
\mathbf{C}_{\mathcal{A}} \to \mathbf{C}_{\mathcal{B}}^{\mathcal{F},
\mathcal{G}}$ to be the following extended multilinear map:
\begin{equation} \label{eq:mu-FTG}
   \begin{aligned}
      (\mu^{\mathcal{B}} \circ \star \circ (\mathcal{F}, T,
      \mathcal{G}))(& a_1, \ldots, a_k) = \\
      \sum \mu^{\mathcal{B}}(& \mathcal{F}(a_1, \ldots, a_{q_1}),
      \ldots, \mathcal{F}(a_{q_{j-1}+1}\ldots, a_{q_{j}}), T(a_{q_j +
        1},
      \ldots, a_{q_{j+1}}), \\
      & \mathcal{G}(a_{q_{j+1}+1}, \dots, a_{q_{j+2}}), \ldots,
      \mathcal{G}(a_{q_r+1}, \ldots, a_k)).
   \end{aligned}
\end{equation}
This sum is taken over all $j$, $r$ and all partitions $$1 \leq q_1 <
\cdots < q_j \leq q_{j+1} < \cdots < q_r < k.$$ The convention is that
if $q_j = q_{j+1}$ then the operator $T$ on the right-hand side
receives no input so it should be interpreted as $T^0_{L_{q_j}} \in
C(L_{q_j}, L_{q_j})$, while if $q_j < q_{j+1}$ the operators $T'$ are
used.

A pre-natural transformation $T = (T^0, T')$ is called a natural
transformation if the following two conditions are satisfied:
\begin{equation} \label{eq:nat-tr}
   \begin{aligned}
      & \mu_1^{\mathcal{B}}(T^0_L) = 0 \quad \forall L \in
      \textnormal{Ob}(\mathcal{A}), 
      \\
      & (\mu^{\mathcal{B}} \circ \star \circ (\mathcal{F}, T,
      \mathcal{G})) + T' \star \mu^{\mathcal{A}} = 0.
   \end{aligned}
\end{equation}

As before, natural transformations between nu $A_{\infty}$-functors
induce natural transformations (in the usual sense) between the
cohomological functors.

Pre-natural transformations $T = (T^0, T') : \mathcal{F} \to
\mathcal{G}$ and $S = (S^0, S'): \mathcal{G} \to \mathcal{H}$ can be
composed as follows:
\begin{equation} \label{eq:pre-nat-comp}
   \begin{aligned}
      & (S \circ T)^0 = \mu_2^{\mathcal{B}}(T^0, S^0), \\
      & (S \circ T)' = \mu^{\mathcal{B}} \circ \star \circ \star \circ
      (\mathcal{F}, T, \mathcal{G}, S, \mathcal{H}),
   \end{aligned}
\end{equation}
where the operation $\circ \star \circ \star \circ$ should be
interpreted in a similar way to~\eqref{eq:mu-FTG}.

The collection of nu $A_{\infty}$-functors $\mathcal{A} \to
\mathcal{B}$ can be given the structure of an $A_{\infty}$-category
$\textnormal{\sl nu-fun}(\mathcal{A}, \mathcal{B})$. The objects are
the nu functors and the morphism spaces $C(\mathcal{F}',
\mathcal{F}'')$ are pre-natural transformations between $\mathcal{F}'$
and $\mathcal{F}''$. The operation $\mu_1$ is given by the left-hand
side of~\eqref{eq:nat-tr} and $\mu_2$ by the right-hand side
of~\eqref{eq:pre-nat-comp}. The higher $\mu_k$'s are given by an
obvious generalization of the left-hand side of~\eqref{eq:nat-tr}.
See~\cite{Se:book-fukaya-categ} for more details. Note that in this
language natural transformations are precisely those pre-natural ones
$T$ with $\mu_1(T)=0$.

Next, we have the notion of a homotopy between two nu
$A_{\infty}$-functors. Let $\mathcal{F}, \mathcal{G}:\mathcal{A} \to
\mathcal{B}$ be two nu $A_{\infty}$-functors with the same action on
objects. Consider the pre-natural transformation $D = (0, D')$ with
$D'_k = \mathcal{F}_k - \mathcal{G}_k$ for every $k \geq 1$. A simple
calculation shows that $D$ is a natural transformation. We say that
$\mathcal{F}$ is homotopic to $\mathcal{G}$ (or that $D$ is homotopic
to $0$) if there exists a pre-natural transformation $T = (T^0, T') :
\mathcal{F} \to \mathcal{G}$ with $T^0 = 0$ and such that $D =
\mu_1(T)$. Homotopy is an equivalence relation and is preserved by
composition with a given nu functor. Two homotopic nu functors
$\mathcal{F}$, $\mathcal{G}$ induce the same functors on the homology
categories.

\subsubsection{Homologically unital functors} \label{sbsb:hu-functors}
A (nu) $A_{\infty}$-functor $\mathcal{F}: \mathcal{A} \to \mathcal{B}$
between su $A_{\infty}$-categories $\mathcal{A}$, $\mathcal{B}$, is
called su if $\mathcal{F}(e_L) = e_{\mathcal{F}(L)}$ for every $L \in
\textnormal{Ob}(\mathcal{A})$ and $$\mathcal{F}(a_1, \ldots, a_{i-1},
e_L, a_{i+1}, \ldots, a_k) = 0$$ for every $k \geq 2$ and $1 \leq i
\leq k$.

If $\mathcal{A}$, $\mathcal{B}$ are only hu, a (nu) functor
$\mathcal{F}$ is called hu if $H(\mathcal{F}): H(\mathcal{A}) \to
H(\mathcal{B})$ is a unital functor. We denote by $\textnormal{\sl
  hu-fun}(\mathcal{A}, \mathcal{B})$ the full subcategory of
$\textnormal{\sl nu-fun}(\mathcal{A}, \mathcal{B})$ consisting of the
hu functors. A hu functor $\mathcal{F}$ is called a quasi-equivalence
if its homological functor $H(\mathcal{F})$ is an equivalence of
categories.

If $\mathcal{A}, \mathcal{B}$ are $A_{\infty}$-categories with
$\mathcal{B}$ being su then the $A_{\infty}$-categories
$\textnormal{\sl hu-fun}(\mathcal{A}, \mathcal{B})$ and
$\textnormal{\sl nu-fun}(\mathcal{A}, \mathcal{B})$ are strictly
unital.

\subsection{$A_{\infty}$-modules} \label{sbsb:A-infty-mod}
\subsubsection{Basic definition.} Let $Ch$
be the $A_{\infty}$-category of (ungraded) chain complexes of
$\mathbb{Z}_2$-vector spaces. Objects are chain complexes and
morphisms spaces are linear maps between their underlying vector
spaces. The operation $\mu_1$ is given by the induced differential on
maps between two chain complexes, so that $\mu_1(f)=0$ iff $f$ is a
chain map. The operation $\mu_2$ is given by the obvious composition.
The higher order composition maps are trivial, i.e. $\mu_k = 0, \;
\forall k \geq 3$, so that $Ch$ is actually a dg-category. We denote
by $Ch^{\textnormal{opp}}$ the opposite category. Denoting
$C^{\textnormal{opp}}(-,-)$ the morphism spaces for
$Ch^{\textnormal{opp}}$ this means that for two chain complexes $C'$,
$C''$ we have $C^{\textnormal{opp}}(C', C'') = \hom(C'', C')$ (the
$\hom$ being of vector spaces) endowed with the obvious differential
induced from $d_{C'}$, $d_{C''}$. The operation
$\mu_2^{\textnormal{opp}}: C(C', C'') \otimes C(C'', C''') \to C(C',
C''')$ is given by $\mu_2^{\textnormal{opp}}(f,g) = g \circ f$.  Note
that $Ch$ and $Ch^{\textnormal{opp}}$ are strictly unital.

Let $\mathcal{A}$ be an $A_{\infty}$-category. A nu
$\mathcal{A}$-module is a nu $A_{\infty}$-functor $\mathcal{M}:
\mathcal{A} \to Ch^{\textnormal{opp}}$. The collection of nu
$\mathcal{A}$-modules forms a category $\textnormal{\sl
  nu-mod}_{\mathcal{A}}$ whose objects are nu $\mathcal{A}$-modules,
and morphism spaces are pre-natural transformations endowed with the
differential $\mu_1$ described in the preceding section. The operation
$\mu_2$ is given by obvious composition and $\mu_k = 0, \; \forall k
\geq 3$.

An alternative more convenient way to describe nu modules is as
follows. On the level of objects a nu $\mathcal{A}$-module
$\mathcal{M}$ prescribes a collection $\mathbf{M} =
{\mathcal{M}(L)}_{L \in \textnormal{Ob}(\mathcal{A})}$ of vector
spaces indexed by $\textnormal{Ob}(\mathcal{A})$. The second
ingredient is a mixed extended multilinear map $\mu_{\mathcal{M}}:
\mathbf{C}_{\mathcal{A}} \otimes \mathbf{M} \to \mathbf{M}$ which
satisfies the following identity:
\begin{equation}
   (\mu^{\mathcal{M}} \dashv \mu^{\mathcal{M}}) + 
   (\mu^{\mathcal{M}} \star \mu^{\mathcal{A}}) = 0.
\end{equation}
Note that $\mu_1^{\mathcal{M}} : \mathcal{M}(L) \to \mathcal{M}(L)$ is
a differential, hence $\mathbf{M}$ becomes a collection of chain
complexes.

\subsubsection{Morphisms.}\label{subsubsec:morphisms} Pre-module morphisms $\nu: \mathcal{M}' \to \mathcal{M}''$ between two
nu $\mathcal{A}$-modules are given by mixed extended multilinear maps
$\nu: \mathbf{C}_{\mathcal{A}} \otimes \mathbf{M}' \to \mathbf{M}''$.
A pre-module morphism $\nu: \mathcal{M}' \to \mathcal{M}''$ is called
a module morphism if
\begin{equation} \label{eq:mod-morph}
   (\mu^{\mathcal{M}''} \dashv \nu) + (\nu \dashv \mu^{\mathcal{M}'}) + 
   (\nu \star \mu^{\mathcal{A}}) = 0.
\end{equation}
Pre-module morphisms $\nu : \mathcal{M}' \to \mathcal{M}''$, $\eta:
\mathcal{M}'' \to \mathcal{M}'''$ can be composed by
\begin{equation} \label{eq:comp-mod-morph} \eta \dashv \nu :
   \mathcal{M}' \to \mathcal{M}'''.
\end{equation}

The collection of nu $\mathcal{A}$-modules $\textnormal{\sl
  nu-mod}_{\mathcal{A}}$ becomes a dg-category with objects being nu
$\mathcal{A}$-modules and morphism spaces $C(\mathcal{M}',
\mathcal{M}'')$ being pre-module morphisms $\nu: \mathcal{M}' \to
\mathcal{M}''$. The differential $\mu_1$ on $C(\mathcal{M}',
\mathcal{M}'')$ is defined by the left-hand side
of~\eqref{eq:mod-morph} and the composition $\mu_2$ is defined
by~\eqref{eq:comp-mod-morph}. The higher operations $\mu_k$, $k\geq 3$
on $\textnormal{\sl nu-mod}_{\mathcal{A}}$ are defined to be $0$.

There is an identification of $A_{\infty}$-categories $\textnormal{\sl
  nu-fun}(\mathcal{A}, Ch^{\textnormal{opp}}) \approx \textnormal{\sl
  nu-mod}_{\mathcal{A}}$.  This goes as follows. If $\mathcal{M}$ is
an nu $\mathcal{A}$-module it corresponds to the nu functor
$\mathcal{F}: \mathcal{A} \to Ch^{\textnormal{opp}}$ which associates
to $L$ the chain complex $(\mathcal{M}(L), \mu_1^{\mathcal{M}})$ and
with extended multilinear map defined by:
\begin{equation} \label{eq:mod-func} 
   \begin{aligned}
      & \langle \mathcal{F}(a_1, \ldots, a_k), b \rangle =
      \mu^{\mathcal{M}}(a_1, \ldots, a_k, b), \\
      & a_j \in C(L_{j-1},L_j), \quad b \in
      C^{\textnormal{opp}}\bigl(\mathcal{M}(L_0),
      \mathcal{M}(L_k)\bigr) = \hom (\mathcal{M}(L_k),
      \mathcal{M}(L_0)).
   \end{aligned}
\end{equation}
On the level of morphisms the identification is as follows. If $\nu:
\mathcal{M}' \to \mathcal{M}''$ is a pre-module morphism and
$\mathcal{F}'$, $\mathcal{F}''$ are two nu-functors $\mathcal{A} \to
Ch^{\textnormal{opp}}$ corresponding to $\mathcal{M}'$,
$\mathcal{M}''$ respectively, then the pre-natural transformation
$T=(T^0, T'): \mathcal{F}' \to \mathcal{F}''$ corresponding to $\nu$
is defined by:
\begin{equation} \label{eq:pre-mod-fun-morph} 
   \begin{aligned}
      & T^0 \in C^{\textnormal{opp}} (\mathcal{M}'(L),
      \mathcal{M}''(L)) = \hom(\mathcal{M}''(L), \mathcal{M}'(L)),
      \quad T^0(b) =
      \nu(b), \\
      & \langle T'(a_1, \ldots, a_{k-1}), b \rangle = \nu(a_1, \ldots,
      a_{k-1}, b), \quad \forall\, k \geq 2.
   \end{aligned}
\end{equation}
\subsubsection{Pull back of $A_{\infty}$-modules}\label{subsubec:pull-back}
Let $\phi: \mathcal{A}'\to \mathcal{A}$ be a morphism of $A_{\infty}$ categories
and let $\mathcal{M}$ be an $A_{\infty}$-module over $\mathcal{A}$.
The pull back $\mathcal{M}'=\phi^{\ast}\mathcal{M}$ of $\mathcal{M}$ is an $A_{\infty}$ module
over $\mathcal{A}'$ defined by $\mathcal{M}'(X)=\mathcal{M}(\phi(X))$ for each object $X$ 
of $\mathcal{A}'$ and with the higher compositions given by
$$\mu_{\mathcal{M}'}(a_{1},\ldots, a_{k}, m)=\sum_{i_1<i_{2}\ldots} \mu_{\mathcal{M}}(\phi^{i_{1}}(a_{1},\ldots, a_{i_{1}}),
\phi^{i_{2}-i_{1}}(a_{i_{1}+1},\ldots, a_{i_{2}}), \ldots, m)~.~$$
It is easy to check that this is indeed an $A_{\infty}$-module.

\subsubsection{Homologically unital modules} \label{sbsb:hu-mod} Let
$\mathcal{A}$ be a hu $A_{\infty}$-category. A hu $\mathcal{A}$-module
is a hu $A_{\infty}$-functor $\mathcal{M}: \mathcal{A} \to
Ch^{\textnormal{opp}}$. Alternatively, a nu $\mathcal{A}$-module
$\mathcal{M}$ is called hu if for every $L \in
\textnormal{Ob}(\mathcal{A})$ the homology $H(\mathcal{M}(L),
\mu_1^{\mathcal{M}})$ is a unital module over the (unital) ring
$H(L,L)$. The full subcategory of $\textnormal{\sl
  nu-mod}_{\mathcal{A}}$ consisting of hu $\mathcal{A}$-modules will
be denoted by $\textnormal{hu-mod}_{\mathcal{A}}$. A pull back of
a hu module over a morphism of hu $A_{\infty}$ categories is also homologically unital.

Note that since $Ch^{\textnormal{opp}}$ is strictly unital, both
$A_{\infty}$-categories $\textnormal{\sl hu-mod}_{\mathcal{A}}$ and
$\textnormal{\sl nu-mod}_{\mathcal{A}}$ are strictly unital. The
identity module morphism $\nu: \mathcal{M} \to \mathcal{M}$ being the
one with $\nu(b) = b$ and $\nu(a_1, \ldots, a_{k-1}, b) = 0$ for every
$k \geq 1$.

\subsection{The Yoneda embedding} \label{sbsb:Yoneda} Let
$\mathcal{A}$ be an $A_{\infty}$-category. The Yoneda functor is the
nu $A_{\infty}$-functor $\mathcal{Y} : \mathcal{A} \to \textnormal{\sl
  nu-mod}_{\mathcal{A}}$, defined as follows. On the level of objects,
define $\mathcal{Y}(L):= \mathcal{M}_L$, where $\mathcal{M}_L$ is the
nu module that associates to the object $K \in
\textnormal{Ob}(\mathcal{A})$ the vector space $\mathcal{M}_L(K) =
C_{\mathcal{A}}(K, L)$. The mixed extended multilinear map of
$\mathcal{M}_L$ is defined as
$$\mu^{\mathcal{M}_L}(a_1, \ldots, a_{k-1}, b) = 
\mu^{\mathcal{A}}(a_1, \ldots, a_{k-1}, b), \quad \forall k \geq 1.$$
The extended multilinear map of the nu functor $\mathcal{Y}$ is
described as follows. Let $k\geq 2$ and $L_0, L_1, \ldots, L_k \in
\textnormal{Ob}(\mathcal{A})$. Let $a_j \in C_{\mathcal{A}}(L_{j-1},
L_j)$, $j=1, \ldots, k$. We need to specify a pre-module morphism
$\mathcal{Y}(a_1, \ldots, a_k): \mathcal{M}_{L_0} \to
\mathcal{M}_{L_{k}}$.  For simplicity of notation we write $\nu =
\mathcal{Y}(a_1, \ldots, a_k)$. To define $\nu$, let $d \geq 1$ and
$K_0, \ldots, K_{d-1} \in \textnormal{Ob}(\mathcal{A})$. We put
\begin{equation} \label{eq:yoneda-nu}
   \begin{aligned}
      & \nu: C_{\mathcal{A}}(K_0, K_1) \otimes \cdots \otimes
      C_{\mathcal{A}}(K_{d-2}, K_{d-1}) \otimes C(K_{d-1}, L_0)
      \longrightarrow C(K_0, L_k), \\
      & \nu(c_1, \ldots, c_{d-1}, b) := \mu^{\mathcal{A}}(c_1, \ldots,
      c_{d-1}, b, a_1, \ldots, a_k).
   \end{aligned}
\end{equation}

The Yoneda functor has better properties for homologically unital
categories. Assume that $\mathcal{A}$ is hu. Then the Yoneda functor
takes values inside the subcategory of hu modules, $\mathcal{Y}:
\mathcal{A} \to \textnormal{\sl hu-mod}_{\mathcal{A}}$.  Moreover, the
homological functor $H(\mathcal{Y}): H(\mathcal{A}) \to
H(\textnormal{\sl hu-mod}_{\mathcal{A}})$ is unital and moreover full
and faithful. This does not follow immediately from the definitions,
and we refer the reader to~\cite{Se:book-fukaya-categ}, Section~2g,
for the proofs.

Due to its properties in the hu case, we call $\mathcal{Y}$ the Yoneda
{\em embedding}. Note also that for the $A_{\infty}$-categories that
will occur in our applications (Fukaya categories) the Yoneda
embedding will generally be injective on objects.

From now on, we will implicitly assume our $A_{\infty}$-categories to
be homologically unital unless otherwise stated. For hu functors and
modules, we will drop from now on the wording ``hu'', calling then
simply functors and modules and denoting their respective categories
by $\textnormal{\sl fun}(\mathcal{A}, \mathcal{B})$ and
$\textnormal{\sl mod}_{\mathcal{A}}$.

\subsection{Exact triangles and derived categories}
\label{sbsb:triang-derived}

Let $\mathcal{A}$ be an $A_{\infty}$-category. Let $\mathcal{M}',
\mathcal{M}''$ be $\mathcal{A}$-modules and $\nu: \mathcal{M}' \to
\mathcal{M}''$ a module morphism. We define a new $\mathcal{A}$-module,
$\mathrm{Cone}(\nu)$, called the mapping cone of $\nu$ as follows.  On
the level of objects $\mathrm{Cone}(\nu)(L) = \mathcal{M}'(L) \oplus
\mathcal{M}''(L)$. The mixed extended multilinear map
$\mu^{\mathrm{Cone}(\nu)}$ is defined as
\begin{equation} \label{eq:cone-nu}
   \begin{aligned}
      \mu^{\mathrm{Cone}(\nu)} & (a_1, \ldots, a_{k-1}, (b', b'')) = \\
      & \Bigl(\mu^{\mathcal{M}'}(a_1, \ldots, a_{k-1}, b'),
      \mu^{\mathcal{M}''}(a_1, \ldots, a_{k-1}, b'') + \nu(a_1, \ldots,
      a_{k-1}, b') \Bigr).
   \end{aligned}
\end{equation}

The mapping cone comes with two module
morphisms: $i: \mathcal{M}'' \to \mathrm{Cone}(\nu)$ and $\pi:
\mathrm{Cone}(\nu) \to \mathcal{M}'$ defined by :
\begin{equation} \label{eq:maps-i-pi}
   \begin{aligned}
      & i_1(b'') = (0, b''), \quad i_k=0 \; \forall\, k \geq 2, \\
      & \pi_1(b', b'') = b', \quad \pi_k = 0 \; \forall\, k \geq 2.
   \end{aligned}
\end{equation}

We call the diagram $$\xymatrix{ \mathcal{M}' \ar[r]^{\nu} &
  \mathcal{M}'' \ar[r]^i & \mathrm{Cone}(\nu) \ar[r]^{\pi} &
  \mathcal{M}'}$$ an exact triangle. We extend the notion of exact
triangles to other diagrams as follows.  A diagram of
$\mathcal{A}$-modules
\begin{equation} \label{eq:ex-tr-1}
   \xymatrix{ \mathcal{M}' \ar[r]^{\nu} &
     \mathcal{M}'' \ar[r]^j & \mathrm{Cone} \ar[r]^{p} & \mathcal{M}'}
\end{equation}
(where $\nu$, $j$, $p$ are module morphisms) is called exact if there
exists a module morphism $t: \mathcal{C} \to \mathrm{Cone}(\nu)$ such
that $[t]: \mathcal{C} \to \mathrm{Cone}(\nu)$ is an isomorphism in the
homological category $H(\textnormal{\sl mod}_{\mathcal{A}})$ and such
that the following diagram commutes in $H(\textnormal{\sl
  mod}_{\mathcal{A}})$:
\begin{equation}
   \xymatrix{
     \mathcal{M}'' \ar[r]^{[j]} \ar[rd]_{[i]}
     & \mathcal{C} \ar[r]^{[p]} \ar[d]^{[t]}_{\approx}  & \mathcal{M}' \\
     & \mathrm{Cone}(\nu) \ar[ru]_{[\pi]}
   }
\end{equation}

Note that the $A_{\infty}$-category $\textnormal{\sl
  mod}_{\mathcal{A}}$ is triangulated in the sense that every module
morphism $\nu : \mathcal{M}' \to \mathcal{M}''$ can be completed to an
exact triangle~\eqref{eq:ex-tr-1}.

\begin{rem}
   As explained in~\cite{Se:book-fukaya-categ} one can generalize the
   notion of exact triangles in any $A_{\infty}$-category
   $\mathcal{B}$ (not just for modules). Let $L', L'', L''' \in
   \textnormal{Ob}(\mathcal{B})$ and $\nu \in
   C_{\mathcal{B}}(L',L'')$, $j \in C_{\mathcal{B}}(L'', L''')$, $p
   \in C_{\mathcal{B}}(L''', L')$ be cycles (i.e.
   $\mu_1^{\mathcal{B}}$ vanishes on them). The diagram $$\xymatrix{
     L' \ar[r]^{\nu} & L'' \ar[r]^j & \mathrm{Cone}(\nu) \ar[r]^p &
     L'}$$ is called an exact triangle if its image under the Yoneda
   embedding is an exact triangle in the sense defined above. In fact,
   being an exact triangle depends only on the homology classes of
   $\nu$, $j$, $p$ hence it is a property of diagrams in the
   homological category $H(\mathcal{B})$.
   See~\cite{Se:book-fukaya-categ} for more details on that.

   Apriori this definition of exact triangles might contradict with
   the previous one, e.g. if we take $\mathcal{B} = \textnormal{\sl
     mod}_{\mathcal{A}}$ to start with, then we have two apriori
   different definitions of exact triangles in $\mathcal{B}$. However,
   this is not the case (see Corollaries~3.9 and~3.10
   in~\cite{Se:book-fukaya-categ}) and the two definitions are
   actually compatible.

   An $A_{\infty}$-category $\mathcal{B}$ is called triangulated if
   every cycle $\nu \in C_{\mathcal{B}}(L',L'')$ can be completed to
   an exact triangle. Usually one adds to this an axiom regarding a
   shift functor. As we are working in the paper in an ungraded framework we
   can ignore this axiom. 
\end{rem}

We now briefly recall how to derive an $A_{\infty}$-category. Let
$\mathcal{A}$ be an $A_{\infty}$-category (recall that we implicitly
assume $\mathcal{A}$ to be homologically unital). Let $\mathcal{Y}:
\mathcal{A} \to \textnormal{\sl mod}_{\mathcal{A}}$ be the Yoneda
embedding and denote by $\mathcal{Y}(\mathcal{A})$ the image of
$\mathcal{A}$. 

We now take the triangulated closure of $\mathcal{Y}(\mathcal{A})$ in
$\textnormal{\sl mod}_{\mathcal{A}}$, namely a minimal full
subcategory $\mathcal{Y}(\mathcal{A})^{\wedge} \subset \textnormal{\sl
  mod}_{\mathcal{A}}$ with the following properties:
\begin{enumerate}
  \item $\textnormal{Ob}(\mathcal{Y}(\mathcal{A})^{\wedge})$ is closed
   under quasi-isomorphisms.
  \item $\mathcal{Y}(\mathcal{A})^{\wedge}$ is triangulated.
\end{enumerate}
A constructive realization of $\mathcal{Y}(\mathcal{A})^{\wedge}$ is
to first to take the full subcategory (not just homologically) of
$\textnormal{\sl mod}_{\mathcal{A}}$ containing the objects of
$\mathcal{Y}(\mathcal{A})$ and add all quasi-isomorphic objects to it
(plus all morphisms between these new objects and the old ones). Next,
form all possible mapping cones between objects of the previous
category and then iterate this procedures inductively arbitrary number
of times. Finally, one defines the derived category $D(\mathcal{A})$
of $\mathcal{A}$ as the homological category
$H\bigl(\mathcal{Y}(\mathcal{A})^{\wedge}\bigr)$. Note that
$D(\mathcal{A})$ is a triangulated category (in the usual sense).

As explained in~\cite{Se:book-fukaya-categ} there are many other
realizations of $D(\mathcal{A})$, but they all lead to equivalent
categories.

\subsection{Families of $A_{\infty}$-categories and equivalences}
\label{sbsb:families-A-infty}

The discussion here follows Chapter~10 of~\cite{Se:book-fukaya-categ}
where one can find more details and proofs.

Let $\mathcal{A}^i$, $i \in \mathcal{I}$, be a family of
$A_{\infty}$-categories indexed by a set $\mathcal{I}$. Suppose we
also have one $A_{\infty}$-category $\mathcal{A}^{\textnormal{tot}}$
such that for every $i \in \mathcal{I}$, $\mathcal{A}^i$ is a full
$A_{\infty}$-subcategory of $\mathcal{A}^{\textnormal{tot}}$. Denote
the embedding functor by $\mathcal{H}^i : \mathcal{A}^i \to
\mathcal{A}^{\textnormal{tot}}$. We assume further that each
$\mathcal{H}^i$ is a quasi-equivalence. Under these assumption there
exists a family of quasi equivalences $\mathcal{K}^i:
\mathcal{A}^{\textnormal{tot}} \to \mathcal{A}^i$, with $\mathcal{K}^i
\circ \mathcal{H}^i = \textnormal{Id}_{\mathcal{A}^i}$ for every $i
\in \mathcal{I}$. Moreover, the family $\mathcal{A}^i$, $i \in
\mathcal{I}$, admits a structure of a so called {\em coherent system
of $A_{\infty}$-categories} (see~\cite{Se:book-fukaya-categ} for the
precise definition). In particular we obtain a family of
quasi-equivalences $\mathcal{F}^{i_1, i_0}: \mathcal{A}^{i_0} \to
\mathcal{A}^{i_1}$ for every $i_0, i_1 \in \mathcal{I}$ with
$\mathcal{F}^{i,i} = \textnormal{Id}_{\mathcal{A}^i}$ and such that
$\mathcal{F}^{i_2, i_1} \circ \mathcal{F}^{i_1, i_0} \cong
\mathcal{F}^{i_2, i_0}$, where $\cong$ means that the two functors are
isomorphic via a natural transformation. We call the
$A_{\infty}$-functors $\mathcal{F}^{i_1, i_0}$ comparison functors.

Passing to the derived categories $D(\mathcal{A}^i)$, $i \in
\mathcal{I}$, we obtain a coherent system of (ordinary) categories
with equivalences $F^{i_1, i_0}: D(\mathcal{A}^{i_0}) \to
D(\mathcal{A}^{i_1})$ induced from the $\mathcal{F}^{i_1,
i_0}$'s. Moreover, the homology level functors $H(\mathcal{K}^i):
H(\mathcal{A}^{\textnormal{tot}}) \to H(\mathcal{A}^i)$ induced by the
$\mathcal{K}^i$'s uniquely determine the equivalences $F^{i_1, i_0}$.

Of course, the same applies not only to the derived categories
$D(\mathcal{A}^i)$, $i \in \mathcal{I}$, but also to the homological
categories $H(\mathcal{A}^i)$ (which can be viewed as full
subcategories of $D(\mathcal{A}^i)$). Namely, $H(\mathcal{A}^i)$, $i
\in \mathcal{I}$, becomes a coherent system of categories via the same
comparison functors $F^{i_1, i_0}$.


\bibliography{bibliography}

%
%
%

\end{document}